\documentclass[fleqn]{amsart}

\usepackage{imakeidx}

\makeindex[name=Notation,title=Index of Notation]

\usepackage[dvipsnames]{xcolor}

\usepackage{tikz-cd}

\usepackage{amssymb,comment}

\usepackage{todonotes}

\setuptodonotes{inline,color=teal!20}

\theoremstyle{plain}
\newtheorem{theorem}{Theorem}[section]
\newtheorem{claim}[theorem]{Claim}
\newtheorem{subclaim}[theorem]{Subclaim}

\newtheorem{fact}[theorem]{Fact}
\newtheorem{corollary}[theorem]{Corollary}

\newtheorem{lemma}[theorem]{Lemma}
\newtheorem{definition}[theorem]{Definition}

\theoremstyle{remark}
\newtheorem{remark}[theorem]{Remark}

\newcommand{\forces}{\mathrel{\Vdash}}

\newcommand{\cat}{^\frown}

\newcommand{\I}{\mathbb{I}}
\newcommand{\J}{\mathbb{J}}
\newcommand{\A}{\mathbb{A}}
\newcommand{\B}{\mathbb{B}}
\newcommand{\C}{\mathbb{C}}
\newcommand{\E}{\mathbb{E}}

\renewcommand{\P}{{\mathbb{P}}}    
\newcommand{\Q}{\mathbb{Q}}
\newcommand{\QTT}{\mathbb{QTT}}
\newcommand{\TT}{\mathbb{TT}}
\newcommand{\TBC}{\mathbb{TBC}}
\newcommand{\FL}{\mathbb{L}}
\newcommand{\FM}{\mathbb{M}}
\newcommand{\K}{\mathbb{K}}
\newcommand{\R}{\mathbb{R}}
\renewcommand{\S}{\mathbb{S}}
\newcommand{\U}{\mathbb{U}}
\newcommand{\T}{\mathbb{T}}
\newcommand{\D}{\mathbb{D}}
\newcommand{\X}{\mathbb{X}}

\newcommand{\termspace}{{\mathcal A}}

\newcommand{\laux}{ {\mathbb{L}_{\mathrm{aux}}} }
\newcommand{\raux}{ {\mathbb{R}_{\mathrm{aux}}} }
\newcommand{\lauxgen}{ L_{\mathrm{aux} } }
\newcommand{\rauxgen}{ R_{\mathrm{aux} } }

\newcommand{\Agg}{{A^{gg}}}
\newcommand{\Aggforcing}{{\A^{gg}}}

\newcommand{\Vi}{V_{\mathrm{inn}}}
\newcommand{\Vd}{V_{\mathrm{def}}}
\newcommand{\Vint}[1]{V_{\mathrm{int},#1}}

\newcommand{\Vl}{V^l}
\newcommand{\Vlb}{V^{lb}}
\newcommand{\Vlbi}{V^{lbi}}

\newcommand{\formerlyrho}{\sigma} 

\DeclareMathOperator{\Add}{Add}

\DeclareMathOperator{\cof}{cof}
\DeclareMathOperator{\dom}{dom}

\DeclareMathOperator{\stem}{stem}
\DeclareMathOperator{\rge}{rge}
\DeclareMathOperator{\crit}{crit}
\DeclareMathOperator{\ot}{ot}
\DeclareMathOperator{\Coll}{Coll}
\DeclareMathOperator{\East}{East}
\DeclareMathOperator{\rk}{rk}
\DeclareMathOperator{\lh}{lh}
\DeclareMathOperator{\cf}{cf}
\DeclareMathOperator{\s}{s}
\DeclareMathOperator{\lev}{lev}

\title{The tree property on long intervals of regular cardinals}

\author[J.~Cummings]{James Cummings}

\author[Y.~Hayut]{Yair Hayut}

\author[M.~Magidor]{Menachem Magidor}

\author[I.~Neeman]{Itay Neeman}

\author[D.~Sinapova]{Dima Sinapova}

\author[S.~Unger]{Spencer Unger}

\usepackage{bookmark}
\usepackage{hyperref}

\begin{document}

\maketitle


\begin{abstract}
 In this paper we prove that the tree property can hold on regular
 cardinals in an interval which overlaps a strong limit cardinal. This is
 a crucial milestone in the long term project, tracing back to a question
 raised by Foreman and Magidor in the 1980s, of obtaining the tree property at every regular
 cardinal above $\omega_1$.
\end{abstract}

\section{Introduction}

Let $\kappa$ be a regular infinite cardinal.
A {\em $\kappa$-tree} is a tree of height $\kappa$ where every level has cardinality
less than $\kappa$, and $\kappa$ has the {\em tree property} if every $\kappa$-tree has a branch of length $\kappa$.
A {\em $\kappa$-Aronszajn tree} is a counterexample to the tree property at $\kappa$, that is to say a $\kappa$-tree with
no branch of length $\kappa$. A $\lambda^+$-tree $T$ is {\em special} if there is a function $f: T \rightarrow \lambda$ such that
$u <_T v \implies f(u) \neq f(v)$: such a tree is a robust counterexample to the tree property, in the sense that it
is a $\lambda^+$-Aronszajn tree in any outer model where $\lambda^+$ remains a cardinal.

The tree property belongs to a class of {\em compactness properties}, which are of great interest in combinatorial set theory.
Significant results about the tree property include:
\begin{itemize} 
\item (K\H{o}nig \cite[1927]{Konig})  $\omega$ has the tree property. 
\item (Specker \cite[1949]{Specker}) If $\kappa^{<\kappa} = \kappa$ then there is a special $\kappa^+$-tree.
\item If $\kappa$ is strongly inaccessible then:
\begin{itemize}  
\item (Keisler and Tarski \cite[1963]{KeislerTarski}) If $\kappa$ has the tree property then $\kappa$ is weakly compact.
\item (Monk and Scott \cite[1964]{MonkScott}) If $\kappa$ is weakly compact then $\kappa$ has the tree property. 
\end{itemize}
\item (Silver \cite[Theorem 5.9, 1972]{Mitchell})
  If $\kappa$ is uncountable and has the tree property, then $\kappa$ is weakly compact in $L$.
\item (Magidor and Shelah \cite[1996]{MagidorShelah})
  If $\lambda$ is a singular limit of cardinals which are $\lambda^+$-strongly compact, then $\lambda^+$ has the tree property.
\item (Combining results of Foreman, Magidor and Schindler \cite[2001]{ForemanMagidorSchindler},
  Schimmerling and Zeman \cite[2004]{SchimmerlingZeman}, and Jensen and Steel \cite[2013]{JensenSteel})
  If $\kappa$ and $\kappa^+$ are successive regular cardinals with the tree property,
  then there is an inner model with a Woodin cardinal.
\item (Combining results of Schimmerling and Zeman \cite[2004]{SchimmerlingZeman}, and Jensen and Steel \cite[2013]{JensenSteel})
  If $\lambda$ is a singular cardinal such that $\lambda^+$ has the tree property, then there
  is an inner model with a Woodin cardinal. 
\end{itemize}

It is known to be consistent that certain small regular cardinals can have the tree property. Mitchell \cite{Mitchell}
showed that if $\lambda < \kappa$ with $\lambda$ regular and $\kappa$ weakly compact, then there
is a generic extension by $<\lambda$-closed $\kappa$-cc forcing 
in which $2^\lambda = \kappa = \lambda^{++}$ and the tree property of $\kappa$ is preserved.
Magidor and Shelah \cite{MagidorShelah} showed it to be consistent modulo a hypothesis at the level of huge cardinals that
$\aleph_{\omega+1}$ is strong limit and has the tree property.

A natural question, raised by Foreman and by Magidor among others,
asks whether it is consistent that all regular $\kappa > \aleph_1$ should simultaneously have the tree property.
There are many obstacles to be overcome in resolving this question: in particular we need a model where
GCH fails everywhere and Jensen's ``weak square'' principle $\square^*_\lambda$ fails for every $\lambda$.
On a closely related point, we seem to need instances of strong compactness in order to violate weak square for singular $\lambda$,
but there is tension here with Solovay's theorem that SCH holds above a strongly compact cardinal. 

Our main result is:
\begin{theorem} \label{mainthm}
  Modulo a suitable large cardinal assumption, it is consistent that
  $\aleph_{\omega^2}$ is strong limit and the tree property holds for all
  regular cardinals $\kappa$ such that $\aleph_2 \le \kappa \le \aleph_{\omega^2 + 3}$.
\end{theorem}

Theorem \ref{mainthm} is the first known instance where the tree property holds on
regular cardinals in an interval which overlaps a strong limit cardinal.
To be more precise, observe that if $\kappa$ is a
singular strong limit cardinal and SCH holds at $\kappa$,
then by Specker's theorem there is a special $\kappa^{++}$-tree.
It follows to get a model where all regular cardinals above $\omega_1$
have the tree property, we are required to produce a singular strong limit cardinal 
$\kappa$ where regular cardinals between $\omega_1$ and $\kappa$ all have the tree property,
$2^\kappa \ge \kappa^{++}$, and $\kappa^+$ has the tree property. 

 This has long been considered the next key step in the longstanding
 goal of obtaining the tree property everywhere. Notice that in Theorem \ref{mainthm} the strong
 limit cardinal which is overlapped is $\aleph_{\omega^2}$, not $\aleph_\omega$. This
 sidesteps another key question, which is still open, as to whether the failure of the
 Singular Cardinals Hypothesis at $\aleph_\omega$ is consistent with the tree property at $\aleph_{\omega+1}$.

The history behind Theorem \ref{mainthm} and the ingredients that go into its proof is a long one.
We survey this history very briefly, where the price of brevity is that some contributions are omitted.
In the light of the preceding discussion, we will be rather specific about cardinal arithmetic.
\begin{itemize}
\item Building on work of Abraham \cite{Abraham},
  Cummings and Foreman \cite{CummingsForeman} showed that consistently $2^{\aleph_n} = \aleph_{n+2}$
  and $\aleph_{n+2}$ has the tree property for all $n < \omega$ simultaneously. 
  They also showed that the tree property can hold at $\kappa^{++}$ where $\kappa$
  is strong limit of cofinality $\omega$ and $2^\kappa = \kappa^{++}$. 
\item Neeman \cite{NeemanUpto} showed that the tree property can hold at $\aleph_{n+2}$ for $n < \omega$
  and at $\aleph_{\omega+1}$. In this model $2^{\aleph_n} = \aleph_{n+2}$ for $n < \omega$
  and $2^{\aleph_\omega} = \aleph_{\omega+1}$. 
  Unger \cite{UngerUpto} showed it can hold for all regular
  cardinals in the interval $[\aleph_2, \aleph_{\omega + \omega})$. In this model
    $2^{\aleph_0} = \aleph_2$, $2^{\aleph_i} = \aleph_{\omega+2}$ for $1 \le i < \omega$,
    and $2^{\aleph_{\omega+i}} = \aleph_{\omega+i+2}$ for $i < \omega$.
  \item  Building on work of Gitik and Sharon \cite{GitikSharon}, Neeman \cite{ItaySCH} showed that
    the tree property can hold at $\kappa^+$ where $\kappa$
  is strong limit of cofinality $\omega$ and $2^\kappa = \kappa^{++}$. 
\item Sinapova \cite{SinapovaPrikry} produced a model of GCH where $\aleph_{\omega + 1}$
  has the tree property, using  different methods
  from those of Magidor and Shelah together with weaker hypotheses.
\item Sinapova \cite{sinapova-omega^2} produced a model where
  $\aleph_{\omega^2}$ is strong limit, $2^{\aleph_{\omega^2}} = \aleph_{\omega^2 +2}$,
  and  $\aleph_{\omega^2 + 1}$ has the tree property. 
  Sinapova and Unger \cite{SinapovaUnger} produced a model
  where  $\aleph_{\omega^2}$ is strong limit, $2^{\aleph_{\omega^2}} = \aleph_{\omega^2 +2}$,
  and both  $\aleph_{\omega^2 + 1}$ and $\aleph_{\omega^2 + 2}$ have the tree property.
\item Unger \cite{UngerIJM} proved a result closely related to Theorem \ref{mainthm}, producing
  a model where there is no {\em special} $\kappa$-Aronszajn tree for regular $\kappa$ with
  $\aleph_2 \le \kappa \le \aleph_{\omega^2 + 3}$.
  In this model none of the cardinals $\aleph_{\omega \cdot n}$ for $n$ finite
  is strong limit, in fact  $2^{\aleph_{(\omega \cdot n) + 4}} =  \aleph_{\omega \cdot (n+1) + 3}$.
  However $\aleph_{\omega^2}$ is strong limit, and $2^{\aleph_{\omega^2}} = \aleph_{\omega^2 + 3}$. 
\end{itemize}

The proof of Theorem \ref{mainthm} has several steps, which we outline here with many technicalities omitted.

\begin{itemize}

\item (Section \ref{preparation}) We start with a model $V_0$ such that $\theta$ and $\kappa$ are the first two
  supercompact cardinals, there are $\theta^+$ supercompact cardinals above $\kappa$, and if $\delta$ is the
  supremum of the first $\theta^+$ supercompact cardinals then there is $j_0 : V_0 \rightarrow M_0$
  which witnesses that $\kappa$ is $\delta^+$-supercompact and is such that supercompact cardinals
  up to $\delta$ are supercompact in $M_0$. 

\item (Section \ref{preparation})   We build a generic extension $V$ of $V_0$, 
  in which $\theta$ is the continuum and exhibits a strong form of generic supercompactness.
 We lift $j_0$ to obtain $j: V \rightarrow M$ with similar properties. 
  
\item (Section \ref{Rlambda}) Working in $V$, we use the generic supercompactness of $\theta$ to show that
  for every supercompact cardinal $\lambda$ with $\kappa < \lambda < \delta$,
  there exist an $\omega$-successor cardinal $\rho < \theta$ 
  and a forcing poset uniformly defined from $\rho$ and $\lambda$, forcing (among other things) the following conclusions:
  $\rho^+$ is $\aleph_1$, the successor of the supremum of the first $\omega$ supercompact cardinals above
  $\lambda$ is $\aleph_{\omega+1}$, and $\aleph_{\omega+1}$ has the tree property.
  We will never actually force with this forcing,
  rather we will use it as a device to show that certain cardinals in our final model have the tree property. 
  Informally we can think of $\rho$ as being ``good for $\lambda$''.

\item  (Section \ref{choiceofrho}) Still working in $V$, we select cardinals $\rho$, $\lambda^a$, $\lambda^b$ such that
  $\rho < \theta < \kappa < \lambda^a < \lambda^b < \delta$ and $\rho$ is good for both $\lambda^a$ and $\lambda^b$ in the
  sense described above. In the final model $\rho^+$ will become $\aleph_1$, $\theta$ will become $\aleph_2$, and
  $\kappa$ will be $\aleph_{\omega^2}$.

  Using the supercompactness of $\kappa$, 
  we argue that there are many triples $(\tau, \Lambda^a(\tau), \Lambda^b(\tau))$
  where $\tau < \Lambda^a(\tau) < \Lambda^b(\tau) < \kappa$ and  $(\tau, \Lambda^a(\tau), \Lambda^b(\tau))$ reflects
  the properties of $(\kappa, \lambda^a, \lambda^b)$. In particular
  $\rho$ is good for both $\Lambda^a(\tau)$ and $\Lambda^b(\tau)$. 

\item (Section \ref{moreprep}) We build a generic extension $V[L]$ of $V$ in which $\kappa$ is still highly supercompact 
  and certain cardinals above $\lambda^a$ are collapsed: in particular the cardinals $(\lambda^b)^{+n}$
  for $n \in \omega \cup \{ \omega+2, \omega+3 \}$
  as computed in $V[L]$ were all supercompact in $V$. A similar situation holds below $\kappa$ in $V[L]$ for the reflected cardinals $\Lambda^b(\tau)$.
  Working in $V[L]$ we carefully choose an embedding $j^*$ witnessing that $\kappa$
  is $<(\lambda^b)^{+\omega+3}$-supercompact,
  and derive supercompactness measures $U_n$ on $P_\kappa (\lambda^b)^{+n}$ for large enough $n < \omega$. 
  
\item (Section \ref{prikryforcing})  Working in $V[L]$ we define a forcing poset $\Aggforcing * \bar\P$
  where $\Aggforcing$ is a highly distributive auxiliary forcing, and $\bar\P$ is a diagonal supercompact Prikry forcing
  with some complex forcing posets interleaved between successive points of the generic $\omega$-sequence.
  Our final model is the extension of $V[L]$ by $\Aggforcing * \bar\P$.
  
  The definition of $\bar\P$  uses the measures $U_n$, and a ``guiding generic'' $K$
  whose definition involves  $\Aggforcing$.
  $\bar\P$  has the effect of making $\rho^+ = \aleph_1$, $\theta = \aleph_2$,
  and $\kappa = \aleph_{\omega^2}$. Above $\kappa$ all cardinals up to and including $(\lambda_b)^{+\omega}$ are collapsed
  to have cardinality $\kappa$, while cardinals above this point are preserved,
  so that $(\lambda^b)^{+\omega +n}$ becomes $\aleph_{\omega^2 + n}$.

\item (Section \ref{treeprop})
  We verify that in our final model all regular
  cardinals in $[\aleph_2, \aleph_{\omega^2})$ have the tree property.
    
\item (Section \ref{treeprop2}) 
  We verify that in our final model the cardinals $\aleph_{\omega^2+1}$,
  $\aleph_{\omega^2+2}$ and $\aleph_{\omega^2+3}$ all have the tree property. 

\end{itemize}

Our notational conventions are fairly standard. When $p$ and $q$ are forcing conditions
 we write ``$q \le p$'' when $q$ is stronger than $p$.
A poset is {\em $\tau$-closed} if every decreasing $\tau$-sequence has a lower bound,
and {\em $<\tau$-closed} if every decreasing $<\tau$-sequence has a lower bound: note that
some authors call these properties {\em $\tau^+$-closed} and {\em $\tau$-closed} respectively.
Our convention for directed closure is similar, so that a poset is {\em $<\tau$-directed closed}
if every directed subset of size less than $\tau$ has a lower bound.
When the decreasing sequences have greatest lower bounds we describe posets
 as being {\em canonically closed}: in particular
a poset is {\em canonically $\tau$-closed} if every decreasing $\tau$-sequence has a greatest lower bound,
and {\em canonically $<\tau$-closed} if every decreasing $<\tau$-sequence has a greatest lower bound.
Of course the Cohen poset $\Add(\tau, \rho)$ and the Levy collapse posets $\Coll(\tau, \rho)$ and
$\Coll(\tau, < \rho)$ are examples of canonically $<\tau$-closed posets. 
When $p_0$ and $p_1$ are compatible conditions we will sometimes abuse notation
and write ``$p_0 \wedge p_1 \forces \phi$'', when we should more properly write
``$p \forces \phi$ for every common refinement $p \le p_0, p_1$''.
Most of the forcing posets appearing in this paper have a top element, but we do not demand this.

In general we will name forcing posets with blackboard bold letters
(for example $\A$) and the associated generic objects with the corresponding
upper case italic letter (for example $A$).
When this naming convention would cause confusion we may call the $\A$-generic object $G_\A$.
If $\dot\tau$ is an $\A$-name then
${\dot\tau}[A]$ is the interpretation of $\dot \tau$ by $A$. For $x$
in the ground model, $\check x$ is the canonical name for $x$, where the
forcing for which $\check x$ is a name should always be clear from the context.

When $\kappa$ is inaccessible and $\lambda \ge \kappa$ we abuse
notation and write $P_\kappa \lambda$ for the set of $x \subseteq \lambda$ with
$x \cap \kappa \in \kappa$ and $\vert x \vert < \kappa$.
When $x, y \in P_\kappa \lambda$  we write $x \prec y$ for the relation
``$x \subseteq y$ and $\ot(x) < y \cap \kappa$. Of course we will also 
use $\prec$ for the relation ``is an elementary substructure of'' but
in practice there is no possibility of confusion.

Once the main construction begins at the start of Section \ref{setup},
we will begin to introduce many objects which are then fixed for the whole
duration of the construction. To help the reader keep track, all these ``global''
objects will be flagged as they appear and will correspond to entries in the
``Index of Notation'' section.

\section{Preliminaries}

\subsection{A  fact about $\kappa$-cc forcing}

The following Lemma is often useful.
\begin{lemma}\label{kappaccfact}
  Let $\kappa$ be regular and uncountable, let $\P$ be $\kappa$-cc and let
  $(p_\alpha)_{\alpha < \kappa}$ be a $\kappa$-sequence of conditions in $\P$. Then
  there exists $\alpha < \kappa$ such that
  $p_\alpha$ forces $\{ \beta < \kappa: p_\beta \in P \}$ to be unbounded in $\kappa$.
\end{lemma}

\begin{proof}
  If not, for each $\alpha$ we choose $r_\alpha \le p_\alpha$ such that
  $r_\alpha$ forces $\{ \beta < \kappa: p_\beta \in P \}$ to be bounded in $\kappa$,
  and then use $\kappa$-cc to find an ordinal $\eta_\alpha < \kappa$ such that
  $r_\alpha$ forces $\{ \beta < \kappa: p_\beta \in P \} \subseteq \eta_\alpha$.
  Let $C = \{ \delta: \forall \gamma < \delta \; \eta_\gamma < \delta \}$, so
  that $C$ is club in $\gamma$.
  If $\gamma, \delta \in C$ with $\gamma < \delta$ then
  $r_\gamma$ forces $p_\delta \notin P$, so that
  $r_\gamma$ forces $r_\delta \notin P$, which is to say that $r_\gamma$ is incompatible with
  $r_\delta$. So $\{ r_\gamma : \gamma \in C \}$ is an antichain in $\P$,
  contradicting $\kappa$-cc for $\P$.
\end{proof}

\begin{remark} An easy variation on this argument shows that there is
  $\alpha$ such that $p_\alpha$ forces $\{ \beta < \kappa: p_\beta \in P \}$ to be stationary
  in $V[P]$, but this is more than we need.
\end{remark}

\subsection{Laver functions and Laver indestructibility} \label{laverfunctions}

Recall that if $\kappa$ is supercompact, there is a {\em Laver function} $f: \kappa \rightarrow V_\kappa$, 
that is to say a function such that for all $x$ and all $\lambda$ there is a $\lambda$-supercompactness embedding
$j: V \rightarrow M$ such that $\crit(j) = \kappa$ and $j(f)(\kappa) = x$.

\begin{definition} Let $f$ be a partial function defined on ordinals. A {\em closure point of
    $f$} is an ordinal $\gamma$ such that $f(\alpha) \in V_\gamma$ for all $\alpha \in \dom(f) \cap \gamma$.
\end{definition}
Thinning the domain of a Laver function $f$, we may assume that $\dom(f)$ consists
of inaccessible closure points of $f$.

Given a Laver function
$f$ the {\em Laver iteration} is an Easton support iteration $\FL$ of length $\kappa$, where
we force with $f(\alpha)[L_\alpha]$ whenever $f(\alpha)$ is a $\FL_\alpha$-name for a $<\alpha$-directed closed
forcing poset: the poset $\FL$ is $\kappa$-cc and has cardinality
$\kappa$. Laver \cite{Laver} showed that if $\kappa$ 
is supercompact then a Laver function exists, and that the Laver iteration forces the supercompactness of $\kappa$
to be indestructible by subsequent $<\kappa$-directed closed forcing.
With our conventions the Laver iteration defined from $f$ preserves the inaccessibility
of all points in $\dom(f)$.

\begin{definition} Let $I$ be an interval of cardinals, then a partial function $f$ on $I$ is a
  {\em universal Laver function on $I$} if and only if $\dom(f) \subseteq \sup(I)$,
  and $f \restriction \kappa$ is a Laver function on $\kappa$ for every supercompact $\kappa \in I$.
\end{definition} 

Adapting the standard argument for the existence of a Laver function, it is easy to see that
every interval has a universal Laver function. Since the construction of a Laver function $f$ proceeds
by choosing $f(\alpha)$ as the least counterexample to $f \restriction \alpha$ being a Laver function,
we may (and will) assume that the domain of a universal Laver function contains no supercompact cardinals.
If $f$ is a universal Laver function on $I$ then the {\em standard Laver iteration defined from $f$} is the Easton support iteration
$\FL$ which runs from $\min(I)$ to $\sup(I)$, forcing as before with $f(\alpha)[L_\alpha]$ whenever
$f(\alpha)$ is a $\FL_\alpha$-name for a $<\alpha$-directed closed
forcing poset. The poset $\FL$ is $<\min(\dom(f))$-directed closed and makes every supercompact cardinal in $I$ indestructible. 

 We will need a strengthening of the concept of Laver indestructibility
 due to Neeman \cite{NeemanUpto}. 
 
 \begin{definition} \label{indestructibleLaverfunction}
   Let $\kappa$ be a supercompact cardinal. An {\em indestructible Laver function for $\kappa$}
   is a partial function $\phi$ from $\kappa$ to $V_\kappa$ such that for every $x \in V$, $\lambda \ge \kappa$
   and $<\kappa$-directed closed forcing extension $V[E]$, there is an elementary embedding
   $\pi: V[E] \rightarrow N$ such that:
   \begin{enumerate}
   \item The embedding $\pi$ is defined in $V[E]$, and witnesses that $\kappa$ is $\lambda$-supercompact in $V[E]$.
   \item $\pi \restriction ON$ is definable in $V$. 
   \item $\kappa \in \dom(\pi(\phi))$ and $\pi(\phi)(\kappa) = x$.
   \item The first point in $\dom(\pi(\phi))$ past $\kappa$ is greater than $\lambda$.
   \end{enumerate}
 \end{definition}

 Note that an indestructible Laver function for $\kappa$ can only exist when $\kappa$
 is indestructibly supercompact. Adapting the arguments of \cite{NeemanUpto} to use
 a universal Laver function, one can readily get a {\em universal indestructible Laver function}.
 
 \begin{fact} \label{universalindestructibleLaverfunction} Let $I$ be an interval of cardinals.
   Then there is a forcing poset $\FL$ such that in the extension by $\FL$,
   there exists a partial function $\phi$  such that $\phi \restriction \kappa$ is
   an indestructible Laver function for every $V$-supercompact cardinal $\kappa \in I$.
\end{fact}

\begin{proof} We do a straightforward adaptation of the argument from the beginning of \cite[Section 4]{NeemanUpto}.
  Let $f$ be a universal Laver function on $I$,
  and derive functions $f_0$ and $f_1$ from $f$ such that
  $f(\alpha) = (f_0(\alpha), f_1(\alpha))$ when $f(\alpha)$ is an ordered pair
  and the values $f_i(\alpha)$ are undefined otherwise. Let $\FL$ be the
  standard Laver iteration defined from $f_0$, and let $L$ be $\FL$-generic
  over $V$. Define $\phi(\alpha) = f_1(\alpha)[L_\alpha]$ at
  every point $\alpha$ such that $f_1(\alpha)$ is an $\FL_\alpha$-name.
 \end{proof}

\begin{remark} The poset $\FL$ does not create any new instances of supercompactness, and
   by convention the domain of a universal Laver function does not include any supercompact cardinals.
   It follows that in the extension by $\FL$, $\kappa \notin \dom(\phi)$ and 
   $\phi \restriction \kappa$ is an indestructible Laver function for every
   supercompact $\kappa \in I$.
\end{remark}
   
Unfortunately the property of Laver indestructibility is quite fragile:
\begin{fact}[Hamkins \cite{JoelSup}] \label{superdestructible}
  If $\kappa$ is supercompact and $\Q$ is a non-trivial
  forcing poset with $\vert \Q \vert < \kappa$,
  then $\kappa$ is not indestructible in the extension
  by $\Q$. In fact $\kappa$ becomes ``superdestructible'',
  that is to say its supercompactness (even its weak compactness)
  is destroyed by any further $<\kappa$-closed forcing which adds a new subset of $\kappa$.
\end{fact}
Since indestructibility plays a central role in our arguments, we will need to make repeated appeals
to Fact \ref{universalindestructibleLaverfunction}. 

\begin{fact} \label{laversmallfact}
  Let $\Q$ be a forcing poset, let $\vert \Q \vert < \mu$ and let $f$ be a universal Laver function
  defined up to $\mu$. Let $Q$ be $\Q$-generic over $V$ and define $F \in V[Q]$ by setting
  $F(\alpha) = f(\alpha)[Q]$ for all $\alpha$ such that $f(\alpha)$ is a $\Q$-name. Then
  $F$ is a universal Laver function on the interval $(\vert {\Q} \vert, \mu)$ in $V[Q]$.
\end{fact}

\begin{proof} Let $\vert \Q \vert < \kappa < \mu$ with $\kappa$ supercompact in $V[Q]$,
  so that $\kappa$ is supercompact in $V$ and $f \restriction \kappa$ is a Laver function.
  Let $\gamma > \kappa$, let $x \in H_\gamma^{V[Q]}$ and let $x = {\dot x}[Q]$ for some
  $\dot x \in H_\gamma$. Choose $j:V \rightarrow M$ witnessing $\kappa$ is $\gamma$-supercompact in $V$
  with $j(f)(\kappa) = \dot x$, then $j$ lifts to an embedding $j:V[Q] \rightarrow M[Q]$
  such that $j$ witnesses $\kappa$ is $\gamma$-supercompact in $V[Q]$ and  $j(F)(\kappa) = x$. 
\end{proof} 

\begin{remark} Note that the Laver functions $\phi$ and $F$ from Facts 
  \ref{universalindestructibleLaverfunction} and \ref{laversmallfact} are derived from an initial
  Laver function $f$ in such a way that $\rk(\phi(\alpha)), \rk(F(\alpha)) \le \rk(f(\alpha))$. It follows
  that closure points of $f$ are automatically closure points of its derived Laver functions.
\end{remark}

\subsection{Trees and systems} \label{branchlemmas}

We will sometimes be in a situation where $T$ is a tree, we know that $T$ has a branch in some
 generic extension, and we want to conclude that $T$ has a branch in $V$. 
In this situation we will often use one of the following {\em preservation lemmas} 
or {\em branch lemmas}. 

\begin{fact}[Unger \cite{ungersucc}]\label{chain} Let $\kappa$ be regular and uncountable.  If $\P \times \P$ is
  $\kappa$-cc, then $\P$ has the $\kappa$-approximation property. In particular
  forcing with $\P$ cannot add a branch through a tree
of height $\kappa$. \end{fact}

\begin{fact}[{Unger \cite[Lemma 6]{ungerind}}] \label{formerlyclosed}
  Let $\kappa$ and $\eta$ be regular and uncountable with $\kappa < \eta \le 2^{<\kappa}$.
  Let $\P$ be $<\kappa$-closed in $V$ and let $V'$ be a $\kappa$-cc extension of $V$.
  Then forcing with $\P$ over $V'$ cannot add a branch through an $\eta$-tree in $V'$. 
\end{fact} 

If the universe is a $\kappa$-cc generic extension of a submodel in which
$\P$ is $<\kappa$-closed, we sometimes say that $\P$ is {\em formerly $<\kappa$-closed}.

\begin{fact}[{Magidor and Shelah \cite[Theorem 2.1]{MagidorShelah}}] \label{MS preservation}
  Suppose that $\mu < \nu$ where $\mu$ is an infinite cardinal
  and $\nu$ is a singular cardinal of cofinality $\omega$.
  Let $V[G]$ be a $\mu$-closed generic extension, and let $E$ be generic over $V[G]$ for a poset in $V$ of size $\mu$.
  If $T \in V[E]$ and $T$ is a $\nu^+$-tree, then any branch through $T$ in $V[E][G]$ is already in
  $V[E]$.
\end{fact}

The concepts of {\em system}
and a {\em system of branches} will play a central role.
Typically a system arises from a name for a tree $T$ in some generic extension,
and a system of branches arises from a name for a branch of $T$ in a further generic extension.

\begin{definition} \label{systemdef}
  Let $D$ be a set of ordinals and $\tau$ be a cardinal. A {\em system on $D \times \tau$}
  is an indexed collection $(R_i)_{i \in I}$ of transitive reflexive
  relations on $D \times \tau$ such that:
\begin{itemize}
\item $(\alpha, \eta) R_i (\beta, \zeta)$ and $(\alpha, \eta) \neq (\beta, \zeta)$ implies $\alpha < \beta$.
\item $(\alpha, \eta) R_i (\beta, \zeta)$ and $(\alpha', \eta') R_i (\beta, \zeta)$ implies that
  $(\alpha, \eta)$ and $(\alpha', \eta')$ are $R_i$-comparable.
\item For $\alpha < \beta$ both in $D$ there exist $\eta, \zeta < \tau$ and $i \in I$ such that
  $(\alpha, \eta) R_i (\beta, \zeta)$.
\end{itemize}

A {\em system of branches} through such a system is an indexed collection $(b_j)_{j \in J}$ of partial functions
from $D$ to $\tau$ such that:
\begin{itemize}
\item $b_j$ is a {\em branch through $R_i$} for some $i$, that is for every $\beta \in \dom(b_j)$ and every $\alpha \in D \cap \beta$,
  $\alpha \in \dom(b_j)$ if and only if there is $\eta$ with $(\alpha, \eta) R_i (\beta, b_j(\beta))$ and
  in this case $b_j(\alpha)$ is the unique such $\eta$.
\item For every $\alpha \in D$ there is $j$ such that $\alpha \in \dom(b_j)$.
\end{itemize}
\end{definition} 

We will need the following technical fact about systems and systems of branches,
which appears in a slightly different form as \cite[Remark 3.4]{NeemanUpto}.

\begin{fact} \label{Itay3.4}
  Let $(R_i)_{i \in I}$ be a system on $D \times \tau$ and let $\nu$ be a cardinal
  such that $D$ is a cofinal subset of $\nu^+$. Let $\P$ be a poset which adds a system
  of branches $(b_j)_{j \in J}$ through the system $(R_i)_{i \in I}$, and let $\lambda$ be a regular cardinal
  such that:
  \begin{itemize}
  \item $\max(\vert I \vert, \vert J \vert, \tau) < \lambda < \nu$.
  \item There is a forcing $\Q$ which adds $\lambda$ mutually generic filters
    for $\P$, without  collapsing  $\lambda$ or forcing that $\cf(\nu^+) \le \lambda$.
  \end{itemize} 
 Then there is $j \in J$ such that $b_j \in V$ and $\dom(b_j)$ is cofinal in $\nu^+$.  
\end{fact} 

Since the proof is quite short we sketch it here. 

\begin{proof}[Proof sketch] Towards a contradiction, 
  we may assume without loss of generality that $\P$ forces
  ``$\dom(b_j)$ cofinal implies $b_j \notin V$'' for all $j$.
  Force with $\Q$ and let $b^\alpha_j$ be the realisation of $\dot b_j$ by the $\alpha^{\rm th}$ $\P$-generic filter.
  If $\alpha \neq \beta$ and both $\dom(b_j^\alpha)$ and $\dom(b_j^\beta)$ are cofinal, then by mutual
  genericity $b^\alpha_j \neq b^\beta_j$. Since $\cf(\nu^+) > \lambda$ we may choose $\eta < \nu^+$ so large such that
  $\dom(b^\alpha_j)$ bounded implies $\dom(b^\alpha_j) \subseteq \eta$ for all $j \in J$ and $\alpha < \lambda$, and  
  also $\dom(b_j^\alpha)$ and $\dom(b_j^\beta)$ both cofinal and $\alpha \neq \beta$ implies
  $b_j^\alpha \restriction \eta \neq b_j^\beta \restriction \eta$ for all $j \in J$ and distinct $\alpha, \beta < \lambda$.
  Let $\gamma \in D \setminus \eta$, then for all $\alpha < \lambda$ there exist $j \in J$, $i \in I$ and $\zeta < \tau$
  such that $\gamma \in \dom(b_j^\alpha)$ (in particular $\dom(b_j^\alpha)$ is cofinal),
  $b^j_\alpha$ is a branch through $R_i$ and $b_j^\alpha(\gamma) = \zeta$.
  Since $\lambda$ is a cardinal we may choose $\alpha \neq \beta$ which give the same values for $(j, i, \zeta)$,
  but then $b^\alpha_j(\gamma) = b^\beta_j(\gamma)$ and both $b^\alpha_j, b^\beta_j$ are branches through $R_i$,
  so that $b^\alpha_j \restriction \gamma = b^\beta_j \restriction \gamma$ in contradiction to the choice
  of $\eta$. 
\end{proof}
  
\subsection{A branch lemma} \label{spencerlemma} 

In Fact \ref{Itay3.4} it is important that the ``width'' $\tau$ of the system is considerably less than the ``height''
$\nu^+$. In Section \ref{plustwo} we are forced to consider systems where the height is the successor of the width,
and to handle these we will use an alternative branch lemma (due to Unger) whose proof is
similar to that of Fact \ref{formerlyclosed}. 

\begin{lemma} \label{widersystems} 
Let $V \subseteq W$.
Let $\delta < \nu < \mu < \lambda$ be cardinals in $W$ where $\mu$ and $\lambda$ are regular.
Assume that $2^\delta \ge \lambda$ in $V$, and $W = V[E]$ where $\E$ is $\mu$-cc in $V$.
Let $\P$ be $<\mu$-closed in $V$, where we note that by Easton's Lemma $\cf(\lambda) \ge \mu$ in $W[P]$. 

Let $\mathcal R$ be a system on $\lambda \times \mu$ in $W$, with relations $R_i$ for $i < \nu$. 
Assume that forcing with $\P$ over $W$ adds a system of branches $(b_i)_{i \in \nu}$
where $b_i$ is a branch through $R_i$. Then there is
$i$ such that $b_i \in W$ and $\dom(b_i)$ is cofinal in $\lambda$.
\end{lemma} 

\begin{proof}
We work in $V$ until further notice. For each $i < \nu$ we fix $\dot R_i$ an $\E$-name for $R_i$ and $\dot b_i$
an $\E \times \P$-name for $b_i$.
Assume for a contradiction that $\E \times \P$ forces that $b_i \notin W$ for every $i$ with
$\dom(b_i)$ cofinal. 
Since $\P$ is $<\mu$-closed and $\E$ is $\mu$-cc in $V[P]$,
it is easy to find $e^* \in \E$, $p^* \in \P$ and $\eta < \lambda$ such that
$(e^*, p^*) \forces^V_{\E \times \P} \mbox{``$\dom(b_i)$ bounded implies $\dom(b_i) \subseteq \eta$''}$ for all $i<\nu$.
Going forward we work below $(e^*, p^*)$.

 Let $e \in \E$, $p_0, p_1 \in \P$, $i < \nu$ and $\gamma \in [\eta, \lambda)$.
Then $(e, p_0, p_1)$ {\em forces divergence for $b_i$ at $\gamma$} if
both $(e, p_0)$ and $(e, p_1)$
decide ``$\dom(b_i)$ is cofinal'' and one of the following holds:
\begin{itemize}
\item At least one of $(e, p_0)$ and $(e, p_1)$ forces ``$\dom(b_i)$ is bounded''.  
\item Both $(e, p_0)$ and $(e, p_1)$ force ``$\dom(b_i)$ is cofinal'', and
  one of the following holds:
  \begin{itemize}
  \item $(e, p_0) \forces \gamma \in \dom(b_i)$, $(e, p_1) \forces \gamma \notin \dom(b_i)$.
  \item $(e, p_0) \forces \gamma \notin \dom(b_i)$, $(e, p_1) \forces \gamma \in \dom(b_i)$.  
  \item $(e, p_j) \forces b_i(\gamma) = \zeta_j$ for $j \in 2$, and $\zeta_0 \neq \zeta_1$.   
  \end{itemize}
\end{itemize}

\begin{claim} \label{spencerclaimone}
  Let $e \in \E$, $p_0, p_1 \in \P$ and $i < \nu$. Then there exist
$\gamma \in [\eta, \lambda)$ and $(e', p_0', p_1') \le (e, p_0, p_1)$ such that
$(e', p_0', p_1')$ forces divergence for $b_i$ at $\gamma$.
\end{claim}

\begin{proof} Extending if necessary, we may as well assume that
both $(e, p_0)$ and $(e, p_1)$ decide ``$\dom(b_i)$ is cofinal''. There is
nothing to do unless $(e, p_0)$ and $(e, p_1)$ force ``$\dom(b_i)$ is cofinal''.  In this
case force with $\E$ below $e$, and then force over $V[E]$ with the formerly closed forcing $\P \times \P$ below $(p_0, p_1)$
to obtain $P_{left} \times P_{right}$. Now since $b_i$ is forced over $V[E]$ by $\mathbb P$ not to lie in $V[E]$, 
$b_i[E \times P_{left}]$ and $b_i[E \times P_{right}]$  
are distinct partial functions with cofinal domains, and it is easy to choose
$(e', p_0', p_1') \in \E \times \P \times \P$ and $\gamma$ as required.
\end{proof} 

\begin{claim} \label{spencerclaimtwo}
  Let $p_0, p_1 \in \P$ and $i < \nu$. Then there exist $p_0' \le p_0$, $p_1' \le p_1$,
 $\gamma^* \in [\eta, \lambda)$ and a maximal antichain $A$ in $\E$ such that for all $e \in A$,
 $(e, p_0', p_1')$ forces divergence for $b_i$ at $\gamma$ for some $\gamma \in [\eta, \gamma^*)$.
\end{claim}

\begin{proof} We construct pairwise incompatible $e^\alpha \in \E$, decreasing
$p_0^\alpha$ below $p_0$ and $p_1^\alpha$ below $p_1$, and $\gamma_\alpha < \lambda$. 
Since $\E$ is $\mu$-cc the construction halts before $\mu$ steps.
If $(e^\alpha)_{\alpha < \beta}$ does not enumerate a maximal antichain in $\E$,
then we choose $f^\beta \in \E$ incomparable with all $e^\alpha$ for $\alpha < \beta$
and lower bounds $q^\beta_0$ and $q^\beta_1$ in $\P$ for the sequences
$(p_0^\alpha)_{\alpha < \beta}$ and
$(p_1^\alpha)_{\alpha < \beta}$, and then apply Claim \ref{spencerclaimone} to the condition
$(f^\beta, q^\beta_0, q^\beta_1)$. This gives $(e^\beta, p_0^\beta, p_1^\beta) \le (f^\beta, q^\beta_0, q^\beta_1)$ 
and $\gamma_\beta \in [\eta, \lambda)$ such that $(e^\beta, p_0^\beta, p_1^\beta)$
forces divergence at $\gamma_\beta$. Once the construction terminates after $\beta$ stages, we
let $A = \{ e_\alpha : \alpha < \beta \}$, let $p_0'$ and $p_1'$ be lower bounds for 
$(p_0^\alpha)_{\alpha < \beta}$ and $(p_1^\alpha)_{\alpha < \beta}$ respectively, and let $\gamma^* = \sup_{\alpha < \beta} \gamma_\alpha$.
\end{proof}

\begin{claim} \label{spencerclaimthree}
  Let $p_0, p_1 \in \P$. Then there exist $p_0' \le p_0$, $p_1' \le p_1$ and
$\gamma^* \in [\eta, \lambda)$ such that for all $i < \nu$,
there is a maximal antichain $A_i$ in $\E$ such that
for all $e \in A_i$,
 $(e, p_0', p_1')$ forces divergence for $b_i$ at $\gamma$ for some $\gamma \in [\eta, \gamma^*)$.
\end{claim}

\begin{proof} Apply Claim \ref{spencerclaimtwo} to each $i < \nu$ in turn, using the $<\mu$-closure of $\P$
and the regularity of $\lambda$ to find $p_0'$, $p_1'$ and $\gamma^*$ that work for all $i$.
\end{proof}

\begin{claim}
  Assuming that $\delta$ is minimal with $2^\delta \ge \lambda$,
  there exist a binary tree of decreasing sequences $(p_\sigma)_{\sigma \in {}^{<\delta} 2}$
and an ordinal $\gamma^* \in [\eta, \lambda)$ with the following property:
for all $\sigma \in {}^{<\delta} 2$ and all $i < \nu$,
there is a maximal antichain $A^\sigma_i$ such that 
for all $e \in A^\sigma_i$, $(e, p_{\sigma^\frown 0}, p_{\sigma^\frown 1})$
forces divergence for $b_i$ at $\gamma$ for some $\gamma < \gamma^*$.
\end{claim}

\begin{proof} For each $\sigma \in {}^{<\delta } 2$ we appeal to Claim \ref{spencerclaimthree}
  with $p_0 = p_1 = p_\sigma$ to find $p_{\sigma^\frown 0}$, $p_{\sigma^\frown 1}$, antichains
  $A^\sigma_i$ for $i < \nu$, and an ordinal $\gamma_\sigma$ as in the conclusion. At limit stages we take lower bounds.
  Since $2^{<\delta} < \lambda$
  and $\lambda$ is regular, we may set $\gamma^* = \sup_\sigma \gamma_\sigma$.
\end{proof}

  For each $f \in {}^\delta 2$, let $p_f$ be a lower bound for $(p_{f \restriction j})_{j < \delta}$.
  Now force with $\E$ and start to work in $V[E]$, so that ${\mathcal R} = \dot{\mathcal R}[E]$ is a
  system with relations $R_i = {\dot R}_i[E]$,  and $(\dot b_i[E])_{i < \nu}$ is a $\P$-name for a system of branches
  with $b_i$ a branch through $R_i$. 
  
  For each $f \in {}^\delta 2$, let $q_f \le p_f$ decide a value of $i$ such that $\gamma^* \in \dom b_i[E]$,
  and let $q_f$ also decide the value $b_i[E](\gamma^*)$ for this $i$. Since $\lambda$ is still a cardinal in $V[E]$,
  there exist $f \neq g$ and values $i < \nu$ and $\zeta < \mu$ such that $q_f$ and $q_g$ both force that
  $b_i[E](\gamma^*) = \zeta$. Let $\sigma$ be the longest common initial segment
  of $f$ and $g$, so that without loss of generality $q_f \le p_{\sigma^\frown 0}$ and 
  $q_g \le p_{\sigma^\frown 1}$. By construction there exist a condition $e \in E$ and $\gamma < \gamma^*$ such that:
\begin{itemize}
\item  $(e, p_{\sigma^\frown 0}, p_{\sigma^\frown 1})$   forces divergence for $b_i$ at $\gamma$.
\item  $e$ forces that both $q_f$ and $q_g$ force  $b_i[E](\gamma^*) = \zeta$.
\end{itemize}
  This is impossible as both $q_f$ and $q_g$
  force that $b_i[E](\gamma) = \zeta'$ for the unique $\zeta'$ such that $(\gamma, \zeta') R_i (\gamma^*, \zeta)$.
\end{proof}

\subsection{Another branch lemma} \label{newbranchlemma}

We will require a branch lemma with the same general flavour as Fact \ref{MS preservation}.
This will be used in Section \ref{plusone} to help establish the tree property
at $\aleph_{\omega^2 + 1}$ in our final model. This branch lemma is quite general
and has some independent interest, so we prove it here axiomatising the needed assumptions.  
It is a descendant of a branch lemma due to Sinapova and Unger \cite{SinapovaUnger}.

Let $\P$ and $\R$ be forcing posets and let $\mu$ and $\nu$ be cardinals.
 We assume that:
\begin{enumerate}
\item \label{preslemmahyp1}    $\nu$ is a cardinal of cofinality $\omega$ and $\mu = \nu^+$.
\item \label{preslemmahyp2} There is a cardinal $\kappa \le \nu$ 
  such that $\forces^V_\P \mu = \kappa^+$ and  $\forces^V_{\P \times \R} \mu = \kappa^+$. 
\item \label{preslemmahyp3} $\R$ is $< \mu$-distributive and countably closed.
\item \label{preslemmahyp4} Every condition $p$ in $\P$ has a {\em stem} (we write it as $\stem(p)$)
  and there are at most $\nu$ stems.
\item \label{preslemmahyp5}
If $\stem(p) = \stem(p') = h$ then there is $q \le p, p'$ with $\stem(q) = h$.
\item  \label{preslemmahyp7} If $(p_n)_{n \in \omega}$ is a decreasing sequence of conditions with stem $h$
  then there is a lower bound with stem $h$. 
\end{enumerate}

We note for the record that by assumptions \ref{preslemmahyp1},
\ref{preslemmahyp4} and  \ref{preslemmahyp5} the poset $\P$ is $\mu$-cc.   

We say that stem $h'$ {\em extends} stem $h$ if there are conditions $p, p' \in \P$
such that $\stem(p) = h$, $\stem(p') = h'$ and $p' \le p$. 

The motivating idea is that $\P$ is some type of Prikry forcing, and $\R$ is a ``mild'' forcing poset.
Our assumptions on $\P$ are quite weak, in particular we
do not need to assume any form of the Prikry lemma.
In the intended application $\P$ will be a complex Prikry-type forcing where 
taking a direct extension can change the stem, and the direct extension ordering
is not countably closed.

\begin{lemma}\label{preslemma}
Let $P \times R$ be $\P \times \R$-generic and let  $T \in V[P]$ be a $\mu$-tree.  
 If $T$ has a cofinal branch in $V[P][R]$, then $T$ has a cofinal
 branch in $V[P]$. 
\end{lemma}

\begin{proof}
  Suppose that $\dot{T} \in V$ is a $\P$-name of a $\mu$-tree. As usual, for each $\alpha < \mu$ we
assume that level $\alpha$ in the tree consists of pairs in $\{ \alpha \} \times \kappa$. We refer to elements
of $\mu \times \kappa$ as {\em nodes}, and if $u$ is a node often we call $u_0$ the {\em level of $u$}
and write it as $\lev(u)$. Of course $\dot T \in V[R]$ and can be viewed as a $\P$-name for a $\mu$-tree
in this model. 

A note on notation: $\alpha$ and $\beta$ will typically be levels of nodes.
$h$ will typically be a stem. $p$ and $q$ will typically
be conditions in $\P$. $r$ and $s$ will typically be conditions in $\R$. $u$ and $v$ will typically be nodes.
Of course these letters
may be decorated with subscripts and superscripts as needed.

Without loss of generality, let  $\dot{b}$ be an $\R \times \P$-name
which is forced by the empty condition to be a cofinal branch though $\dot{T}$.
 Let $\dot{b}[R]$ be the $\P$-name in $V[R]$
 for such a branch obtained by partially realising $\dot b$.
 
 \begin{claim} \label{usefulclaim} 
   Let $p \in {\P}$.  
   Let $u$ and $v$ be nodes with $\lev(u)  < \lev(v)$.
     If $p \forces^{V[R]}_{\P} u, v \in {\dot b}[R]$ then
     $p \forces^V_{\P} u <_{\dot T} v$. 
\end{claim}
   
\begin{proof}
   Let $r$ force that $p \forces^{V[R]}_{\P} u, v \in {\dot b}[R]$, so that
   $(r, p) \forces^V_{\R \times \P} u,v \in {\dot b}$. If there is $p' \le p$ such that
   $p' \forces^V_{\P} u \nless_{\dot T} v$, then 
   $(r, p') \forces^V_{\R \times \P} u \nless_{\dot T} v$ and $(r, p') \forces^V_{\R \times \P} u,v \in {\dot b}$,
   which is impossible as $\dot b$ is forced to be a branch. So $p \forces^V_{\P} u <_{\dot T} v$ as claimed.
\end{proof}

 The following definition takes place in $V[R]$.

\begin{definition} \label{dagger}
  For a stem $h$, we say that {\em $\dagger_h$ holds} if there are an unbounded $J \subseteq \mu$,
  $\xi < \kappa$ and $\langle p_\alpha \mid \alpha \in J \rangle$ such that for all $\alpha \in J$: 
  \begin{itemize}
    
  \item The condition $p_\alpha$ has stem $h$.

  \item $p_\alpha \forces^{V[R]}_{\P} \langle \alpha, \xi \rangle \in \dot{b}[R]$.

  \end{itemize}
\end{definition}

Before reading the following remark, the reader should recall our convention that
when we write ``$p_0 \wedge p_1 \forces \phi$'' we mean only that $p_0$ and $p_1$ are compatible
and every common lower bound $r \le p_0, p_1$ is such that $r \forces \phi$. We are not asserting that
$p_0$ and $p_1$ have a greatest lower bound.

\begin{remark}  Let $J$, $\xi$ and $\langle p_\alpha \mid \alpha \in J \rangle$
  witness $\dagger_h$ as in Definition \ref{dagger} and let $\alpha, \beta \in J$ with $\alpha < \beta$.
 Then $p_\alpha$ and $p_\beta$ are compatible in $\P$ by item \ref{preslemmahyp5} of our hypotheses, and    
  $p_\alpha \wedge p_\beta \forces^V_{\P} \langle \alpha, \xi \rangle <_{\dot{T}} \langle \beta, \xi \rangle$ by  Claim \ref{usefulclaim}. 
\end{remark}

\begin{claim} \label{daggerstemsdense}
In $V[R]$ every stem $h$ can be extended to a stem $h'$ such that $\dagger_{h'}$ holds.
\end{claim}

\begin{proof}
  Work in $V[R]$. Let $h$ be a stem, and let $p \in \P$ be a condition with stem
  $h$. For each $\alpha < \mu$, let $p_\alpha \leq p$ and $u_\alpha= \langle \alpha, \xi_\alpha \rangle$
  be such that $p_\alpha \forces_{\P} u_\alpha\in \dot{b}[R]$.
  Since the number of stems is less than $\mu$, there exist an unbounded set $J \subseteq \mu$,
  an ordinal $\xi < \kappa$ and a stem $h'$ such that $p_\alpha$ has stem $h'$ and $\xi_\alpha = \xi$
  for all $\alpha \in J$.  Then $h'$ extends $h$ and $\dagger_{h'}$ holds.
\end{proof}

If $\dagger_h$ holds and $h'$ extends $h$, it does not follow in general that $\dagger_{h'}$ will hold.
The issue is that in general not every condition with stem $h$ can be extended to a condition with stem $h'$.
However we do have the following in $V[R]$:
\begin{claim} \label{daggerhandy}
  If $\dagger_h$ holds then there is a condition $p$ with stem $h$ such that
   $\{ p' : \mbox{$\dagger_{\stem(p')}$ holds} \}$ is dense below $p$. 
\end{claim}

\begin{proof} Let $J$, $\xi$ and $(p_\alpha)_{\alpha \in J}$ witness $\dagger_h$.
  As $\P$ is $\mu$-cc, it follows from Lemma \ref{kappaccfact} that there is $\alpha$ such that $p_\alpha$ forces
  the set of $\beta \in J$ with $p_\beta \in P$ to be unbounded. Set $p$ equal to
  $p_\alpha$ and let $\bar p \le p$. Then $\bar p$ is compatible with $p_\beta$ for
  every $\beta$ in some unbounded $J' \subseteq J$, and we may choose $p'_\beta \le \bar p, p_\beta$
  for all $\beta \in J'$. 
  Thinning out $J'$ we may assume that for some stem $h'$,  $\stem(p'_\beta) = h'$ for all $\beta \in J'$.
  Then the conditions $(p'_\beta)_{\beta \in J'}$ together with $\xi$ and $J'$ witness $\dagger_{h'}$,
  and for any $\beta \in J'$ we have that $p'_\beta \le \bar p$ and $\dagger_{\stem(p'_\beta)}$ holds.
\end{proof}

The following definition takes place in $V$.

\begin{definition} \label{hsplitting}
    Let $h$ be a stem and let $s \in \R$. There is an {\em $(h,s)$-splitting} if there are a
condition $p \in \P$ with stem $h$, conditions $s^0, s^1$ in
$\R$, and nodes $u^0, u^1 \in \mu \times \kappa$ such that:
\begin{enumerate}
\item $s^0, s^1 \le s$. 
\item  $(s^k, p) \forces^V_{\R \times \P} u^k \in \dot{b}$ for $k \in 2$.
\item  $p \forces^V_{\P} \mbox{``$u^0$ and $u^1$ are incomparable in $\dot{T}$.''}$
\end{enumerate}
\end{definition}

We note that the witnessing conditions $s^0$ and $s^1$ for an $(h, s)$-splitting must be incompatible.
The issue is that if $s^* \le s^0, s^1$ then $(s^*, p)$ forces that both $u^0$ and $u^1$ lie on $\dot b$,
while $p$ forces them to be incomparable in $T$.

\begin{definition} \label{hnew}
  $\dot b$ is {\em $h$-new below $s$} if and only if the set of $s'$ such that there is
  an $(h, s')$-splitting is dense below $s$.
\end{definition}

The following key claim takes place in $V$. 
\begin{claim} \label{splittingclaim} 
  Suppose that $s$ forces ``$\dagger_h$ holds'', and $\dot b$ is
  $h$-new below $s$. There are sequences $\langle s_i \mid i < \nu\rangle$,
  $\langle p_i \mid i < \nu \rangle$, and $\langle v_i \mid i< \nu\rangle$, such that
\begin{enumerate}
\item For all $i < \nu$, $s_i \leq s$ and the stem of $p_i$ is $h$.
\item For all $i < \nu$, $(s_i, p_i) \forces^V_{\R\times\P} v_i \in \dot{b}$.
\item For $i < j < \nu$, 
     $p_i \wedge p_j \forces^V_\P \mbox{``$v_i$ and $v_j$ are incomparable in $\dot{T}$''}$.
\end{enumerate}
\end{claim}

\begin{proof}
  Suppose that $s$ forces  $\dagger_h$  as witnessed by $\xi$, $\dot{J}$, and ${\dot p}_\alpha$ for
  $\alpha \in \dot J$.  
   Forcing below $s$ we pass to a generic extension $V[R]$ where $\xi$, $J$ and
  $\langle p_\alpha : \alpha \in J \rangle$ witness $\dagger_h$. 

  The following subclaim takes place in $V[R]$. 
\begin{subclaim} \label{splittingsubclaim}
  For every $\gamma \in J$ there exist $p \in \P$ with stem $h$, conditions $r^0, r^1$ in $\R$ below $s$
  and nodes $v^0, v^1 \in \mu \times \kappa$
  such that:
  \begin{itemize}
  \item $(r^0, p) \forces^V_{\R \times \P} (\gamma, \xi) \in \dot b$.   
  \item For $k \in 2$, $(r^k, p) \forces^V_{\R \times P} v^k \in \dot{b}$.
  \item $p \forces^V_\P \mbox{``$v^0$ and $v^1$ are incomparable elements above $(\gamma, \xi)$ in $\dot T$''}$.
  \item $r^0 \in R$.
  \end{itemize}
\end{subclaim}

\begin{proof} We will do a density argument in $V$ to show that suitable values for $r^0$ are dense
  below $s$. Let $r \le s$ force that $\gamma \in J$, and decide the value of $p_\gamma$ as
  $q$. Then $q$ has stem $h$ and $(r, q) \forces^V_{\R \times \P} (\gamma, \xi) \in \dot b$.
  Since $r \le s$ and $\dot b$ is $h$-new below $s$,
  we may extend $r$ if needed and assume that there is an $(h, r)$-splitting.
  
  Fix $r^0, r^1 \le r$, nodes $v^0, v^1$ and a condition $q'$ with stem $h$
  such that:
\begin{enumerate}
\item  $(r^k, q') \forces^V_{\R \times \P} v^k \in \dot{b}$ for $k \in 2$.
\item  $q' \forces^V_{\P} \mbox{``$v^0$ and $v^1$ are incomparable in $\dot{T}$.''}$.
\end{enumerate}

Since $\stem(q) = \stem(q') = h$, by item \ref{preslemmahyp5} of our hypotheses
 we may find $p \le q, q'$ with $\stem(p) = h$. 
 Since $(r^k, p)$ forces that both $(\gamma, \xi)$
and $v^k$ are in $\dot b$, $p$ forces that $(\gamma, \xi)$ and $v^k$ are comparable in
$\dot T$. Since $p$ also forces that $v^0$ and $v^1$ are incomparable,
it follows that they are both on levels above $\gamma$.  
\end{proof}

Still working in $V[R]$, choose a club $C \subseteq \mu$ such that for all $\beta \in C$ and all
$\gamma \in J \cap  \beta$,  the conclusion of Subclaim \ref{splittingsubclaim} holds
with witnessing nodes $v^0,v^1$ having levels below $\beta$.

We select increasing sequences $\gamma_i$ and $\beta_i$ for $i < \nu$ such that
\begin{enumerate}
\item $\beta_i \in C$,
\item $\gamma_i \in J$,
\item $\gamma_i < \beta_i \leq \gamma_{i+1}$.
\end{enumerate}

Now for each $\gamma_i$ the conclusion of Subclaim \ref{splittingsubclaim} holds,
with witnessing nodes on levels below $\gamma_{i+1}$.
We record the witnesses to this splitting as $p_i$, $r_i^k$ and $v_i^k$. 
Let $u_i = (\gamma_i, \xi)$. 

For $i < j$ we claim that $p_i \wedge p_j$ forces that $v_i^1$ and $v_j^1$ are incomparable.
The point is that we can choose a lower bound $r^*$ for  $r_i^0$, $r_j^0$ as both are in $R$. 
Now $(r^*, p_i \wedge p_j) \forces^V_{\R \times \P} v_i^0, u_j \in \dot b$, so
$p_i \wedge p_j \forces^V_\P v_i^0 <_{\dot T} u_j <_{\dot T} v_j^1$. Since $v_i^0, v_i^1$ are on levels
below $\gamma_j$ and $p_i \wedge p_j$ forces that $v_i^0$ is  incomparable with $v_i^1$,
$p_i \wedge p_j$ forces that $v_i^1$ is incomparable with $v_j^1$. 

Now let $v_i = v_i^1$ and $s_i = r_i^1$. 
By the distributivity of $\R$, the sequence $\langle s_i, p_i,v_i \mid i < \nu \rangle$ is in $V$
and by construction it satisfies the desired properties.  
\end{proof}

\begin{claim} \label{newclaim}
  If $s$ forces $\dagger_h$, then $\dot b$ is not $h$-new below $s$.
\end{claim}

\begin{proof}
  Assume for a contradiction that $s$ forces $\dagger_h$ and $\dot b$ is $h$-new below $s$. Note that
  these properties also hold for conditions below $s$. 
  
Using Claim \ref{splittingclaim}, we will construct a tree of conditions
$\langle (r_\sigma, p_\sigma) \mid \sigma \in \nu^{<\omega} \rangle$ in $\R \times \P$
and nodes $\langle v_\sigma \mid \sigma\in\nu^{<\omega}\rangle$ such that:
\begin{enumerate}
\item For all $\sigma$, $p_\sigma$ has stem $h$.
\item If  $\sigma'\supset \sigma$, then $(r_{\sigma'},p_{\sigma'})\leq (r_\sigma, p_\sigma)$.
\item  For all $\sigma$, $(r_\sigma, p_\sigma) \forces^V_{\R \times \P} v_\sigma\in \dot{b}$.
\item  For all $\sigma$ and all  $i \neq j$ in $\nu$,
   $p_{\sigma\cat i} \wedge p_{\sigma\cat j}$ forces that $v_{\sigma\cat i}$ and $v_{\sigma\cat j}$
are incomparable in $\dot{T}$.
\end{enumerate}
 
Given $r_\sigma$ and $s_\sigma$, we appeal to Claim \ref{splittingclaim}
with $s_\sigma$ in place of $s$ to produce $s_i$, $v_i$ and $p_i$.
We then set $r_{\sigma\cat i} = s_i$ and
$v_{\sigma\cat i} = v_i$. Finally we set $p_{\sigma\cat i} =  p_\sigma \wedge p_i$.

 When the construction is done we choose $\gamma < \mu$ such that all the nodes $v_\sigma$
 have levels below $\gamma$. We use the countable closure of $\R$ and item \ref{preslemmahyp7} of  our hypotheses
 on $\P$ to choose $(r_f, p_f)$ for $f \in \nu^{\omega}$ such that $p_f$ has stem $h$,
 and  $(r_f, p_f) \le (r_{f \restriction n}, p_{f \restriction n})$ for all $n < \omega$.
 We then choose $(r'_f, p'_f) \le (r_f, p_f)$ so that $(r'_f, p'_f)$ determines  
 the node on level $\gamma$ in the branch $\dot b$ as $u_f$.

   Since $\nu^\omega > \nu \ge \kappa$, there exist $f \neq g$ such that
   $u_f = u_g = u^*$  and $\stem(p'_f) = \stem(p'_g) = h^*$ for some node $u^*$ and stem $h^*$. Let $n$ be least such that $f(n) \neq g(n)$. Let
   $\sigma = f \restriction n = g \restriction n$, $i = f(n)$ and $j = g(n)$,
   so that $f \restriction n + 1= \sigma\cat i$ and $g \restriction n + 1= \sigma\cat j$.

   By construction
   $p'_f \wedge p'_g$ forces that $v_{\sigma\cat i}$ and $v_{\sigma\cat j}$ are incomparable in $\dot{T}$.
   Also $(r_f, p'_f \wedge p'_g)$ forces that both $v_{\sigma\cat i}$ and $u^*$ are in $\dot b$,
   so $p'_f \wedge p'_g \forces^V_{\P} v_{\sigma\cat i} < u^*$. Similarly 
   $p'_f \wedge p'_g \forces^V_{\P} v_{\sigma\cat j} < u^*$. This is a contradiction. 
 \end{proof}

For each stem $h$, let $D_h$ be the set of $s$ such that
either $s$ forces $\neg \dagger_h$ or there is no $(h, s)$-splitting.
It is easy to see that $D_h$ is open, and we claim that it is also
dense. To see this let $s$ be arbitrary, where by extending $s$
we may assume that $s$ decides $\dagger_h$. If $s$ forces $\neg\dagger_h$
then $s \in D_h$ by definition. If $s$ forces $\dagger_h$ 
then $\dot b$ is not $h$-new below $s$ by Claim \ref{newclaim},
in which case by definition there is $s' \le s$ with no $(h, s')$-splitting
and $s' \in D_h$. 
 
Since $\R$ is $\nu$-distributive, $\bigcap_h D_h$ is dense and open.
Let $s^* \in \bigcap_h D_h$. By Claims \ref{daggerstemsdense} and
\ref{daggerhandy}, and extending $s^*$ if necessary, we may assume that:
\begin{itemize}
\item For some stem $h$,  $s^*$ forces that $\dagger_h$ holds. 
\item There is a condition $p$ such that $s^*$ forces
 that  $\{ p' : \dagger_{\stem(p')} \}$ is dense below $p$. 
\end{itemize}

Now we force below $p$ to obtain a $\P$-generic filter $P$ with $p \in P$.
Working in $V[P]$, let
\[
d = \{ u \in \mu \times \kappa : \exists q \in P \; \exists s' \le s^* \;
(s', q) \forces^V_{\R \times \P} u \in \dot b \}.
\]

\begin{claim} $d$ is a cofinal branch in $T$.
\end{claim}

\begin{proof} Since $\R \times \P$ forces that $\dot b$ is a cofinal branch, it is routine to check that
  $d$ contains nodes with unboundedly high levels. The key remaining point is that $d$ is a chain in $T$.
  To see this, suppose for a contradiction that $u^0$ and $u^1$ are incomparable members of $d$.

  We may choose $q \in P$ with $q \le p$, together with $s^0, s^1 \le s^*$, such that:
  \begin{itemize}
  \item $q \forces^V_\P \mbox{``$u^0$ and $u^1$ are incomparable in $\dot T$.''}$.
  \item $(s^i, q) \forces^V_{\R \times \P} u^i \in \dot b$.
   \end{itemize} 

  Forcing with $\R$ below $s^*$ over $V[P]$ we obtain $R$ mutually generic with $P$,
  such that in $V[R]$ the set $\{ p' : \dagger_{\stem(p')} \}$ is dense below $p$.
  So we may choose $p' \in P$ with $p' \le q$ and $s^{**} \le s^*$ such that
  $s^{**}$ forces $\dagger_h$ where $h = stem(p')$. So $s^*$ does not force $\neg\dagger_h$,
  and since $s^* \in D_h$ it follows that there is no $(h, s^*)$-splitting.
  However $p'$ together with $s^i$ and $u^i$ form an example of
  an $(h, s^*)$-splitting, which is a contradiction.  
\end{proof}

This completes the proof of Lemma \ref{preslemma}.
\end{proof}

\subsection{Term forcing} \label{termforcing}

Let $\P$ be a forcing poset and $\dot {\Q}$ be a $\P$-name for a forcing poset.
Then $\termspace(\P, \dot\Q)$ is the set of $\P$-names
for elements of $\dot\Q$, where we identify names $\dot q_0$ and $\dot q_1$
if $\forces_\P \dot q_0 = \dot q_1$.
$\termspace(\P, \dot\Q)$ is 
ordered as follows: $\dot q_1 \le \dot q_0$ in $\termspace(\P, \dot\Q)$ 
if and only if $\forces_{\P} \dot q_1 \le_{\dot {\Q}} \dot q_0$. Term forcing was introduced by Laver,
and the theory was elaborated by Foreman.

The following Lemmas are standard:
\begin{lemma} \label{standardtermforcinglemma}
Let $\P$ be a forcing poset, $\dot {\Q}$ be a $\P$-name for a forcing poset, and let $\R = \termspace(\P, \dot \Q)$. 
\begin{enumerate}  
\item The identity function is a projection from $\P \times \R$
  to $\P * \dot\Q$.
\item If $P \times R$ is $\P \times \R$-generic over $V$
  and  $Q = \{ {\dot q}[P]: \dot q \in R \}$, then $Q$ is $\dot {\Q}[G]$-generic over $V[P]$.
\item  \label{quotient_to_term}  If $P * Q$ is ${\P} * {\Q}$-generic and we force over $V[P * Q]$ with
  $\{ \dot q \in R : {\dot q}[P] \in Q \}$ using the ordering inherited
  from $\R$, then we obtain $R$ such that $P \times R$ is
  ${\P} \times \R$-generic and induces $P * Q$.  
\item If $\lambda$ is forced by $\P$ to be a regular uncountable cardinal and $\dot\Q$ is forced by $\P$ 
  to be $<\lambda$-closed (resp $<\lambda$-directed closed, canonically $<\lambda$-closed),
  then $\R$  is $<\lambda$-closed (resp $<\lambda$-directed closed, canonically $<\lambda$-closed).
\end{enumerate}
\end{lemma}

\begin{lemma} \label{standardtermforcinglemma2}
  Let $\P$ be a forcing poset.
\begin{enumerate}  
\item \label{termspaceproduct} If it is forced by $\P$ that $\dot \Q$ and $\dot \R$ are forcing posets, then 
  $\termspace(\P, \dot \Q \times \dot \R)$ is canonically
  isomorphic to $\termspace(\P, \dot \Q) \times \termspace(\P, \dot \R)$,
  identifying names for pairs with pairs of names. 
\item \label{termspaceprojection}   If it is forced by $\P$ that $\dot \Q_1$ and $\dot \Q_0$ are forcing posets and
  $\dot \pi$ is a projection from $\dot \Q_1$ to $\dot \Q_0$,
   then $\dot q \mapsto {\dot \pi}(\dot q)$ is a projection from $\termspace (\P, \dot Q_1)$ to $\termspace (\P, \dot \Q_0)$. 
\end{enumerate}
\end{lemma}

We can view $\termspace(\P, \dot \Q)$ as adding a ``universal generic object''
for $\dot \Q$, which can be realised using any $V$-generic filter $P$ on $\P$ as
a $V[P]$-generic filter on ${\dot Q}[P]$. 
For use later we record some more easy facts about term forcing.

\begin{lemma} \label{termforcingcclemma}
 If $\kappa$ is weakly compact, $\vert \P \vert < \kappa$ and  $\forces_\P \mbox{``$\dot \Q$ is $\kappa$-cc''}$,
 then $\termspace(\P, \dot \Q)$ is $\kappa$-cc. 
\end{lemma}

\begin{proof} Suppose for a contradiction that $(\tau_i)_{i < \kappa}$ enumerates an antichain in
  $\termspace(\P, \dot \Q)$, so that for $i < j$ we have
  $\not\forces_\P \mbox{``$\tau_i$ and $\tau_j$ are compatible in $\dot\Q$''}$.
  Define a colouring of $[\kappa]^2$ in $\vert \P \vert$ colours, by
  colouring $(i, j)$ with some condition $p(i, j)$ such that
  $p(i, j) \forces \mbox{``$\tau_i$ and $\tau_j$ are incompatible in $\dot\Q$''}$.

  Since $\kappa$ is weakly compact, there exist $H \in [\kappa]^\kappa$ and $p$ such that
  $p(i, j) = p$ for all $(i, j) \in [H]^2$. But then $p$ forces that
  $(\tau_i)_{i \in H}$ enumerates an antichain in $\dot \Q$, contradicting the hypothesis.
\end{proof}  
  
\begin{lemma} \label{iteratedtermforcing}  
 If $\dot \R$ is a $\P * \dot \Q$-name for a forcing poset then
 $\termspace(\P * \dot \Q, \dot \R) \simeq \termspace(\P, \termspace^{V[P]}(\dot \Q, \dot \R))$.
\end{lemma}

\begin{proof} This is immediate using the canonical identification between
  $\P * \dot \Q$-terms for elements of $\R$ on the one hand,  and $\P$-terms 
  for $\Q$-terms for elements of $\R$ on the other hand.
\end{proof}

It is also useful to analyse $\termspace(\P, \dot \Q * \dot \R)$ 
where $\Q * \dot \R \in V[P]$. 
\begin{lemma} \label{termiteratedforcing}
  There is a projection from $\termspace(\P, \dot \Q) \times \termspace(\P * \dot \Q, \dot \R)$
  to $\termspace(\P, \dot \Q * \dot \R)$ 
\end{lemma} 

\begin{proof} In $V[P]$ there is a projection from $\Q \times \termspace(\Q, \dot \R)$ to $\Q * {\dot \R}$,
  and by item \ref{termspaceprojection} of Lemma \ref{standardtermforcinglemma2}
  this induces a projection from $\termspace(\P, \dot \Q) \times \termspace(\P, \termspace^{V[P]}(\Q, \dot \R))$
  to $\termspace(\P, \dot \Q * \dot \R)$.
  By Lemma \ref{iteratedtermforcing}  
  the posets $\termspace(\P, \termspace^{V[P]}(\Q, \dot \R))$
  and $\termspace(\P * \dot \Q, \dot \R)$ are canonically isomorphic.
\end{proof}

\begin{remark} With suitable identifications, the projection map
  from the proof of Lemma \ref{termiteratedforcing} 
  is the identity map.
\end{remark}


Let $P * Q$ be $\P * \dot \Q$-generic. 
In a mild abuse of notation, we sometimes denote by ``$\termspace(\P, \dot\Q)/P * Q$'' the forcing
 from item \ref{quotient_to_term} of Lemma \ref{standardtermforcinglemma}, which is defined in $V[P * Q]$
to produce an $\termspace(\P, \dot\Q)$-generic filter  $R$ such that $P \times R$ projects
to $P * Q$.
We will call this kind of forcing poset a {\em quotient to term} poset.
We will often  say ``force to remove the dependence of $Q$ on $P$''
or ``force to refine $P * Q$ to $P \times R$'' as shorthand for ``force with the quotient to term poset $\termspace(\P, \dot\Q)/P * Q$''. 
The forcing $\termspace(\P, \dot\Q)/P * Q$ is defined in $V[P * Q]$ but we may force with it over
generic extensions of this model:

\begin{lemma} \label{rearrangeqtot} 
  Let $P * Q$ be $\P * \dot \Q$-generic over $V$, let $\K \in V[P * Q]$ and let $K$ be
  $\K$-generic over $V[P * Q]$. Forcing with $\termspace(\P, \dot\Q)/P * Q$ over $V[P * Q * K]$ produces
  $R$ such that $P \times R$ induces $P * Q$ and $K$ is $\K$-generic over $V[P \times R]$.
\end{lemma} 

\begin{proof} Let $A$ be $\termspace(\P, \dot\Q)/P * Q$-generic  over $V[P * Q * K]$ and
  let $V[P * Q * A] = V[P \times R]$. 
  Since $\K$ and $\termspace(\P, \dot\Q)/P * Q$ are both in $V[P * Q]$, $K$ and $A$ are mutually generic over
  $V[P * Q]$, so $K$ is $\K$-generic over $V[P \times R]$.   
\end{proof}

\begin{remark} \label{rearrangeqtotremark}
  In the sequel we sometimes replace $\termspace(\P, \dot \Q)$ by more elaborate posets
  which have the similar effect of adding a $\P$-name for a ${\dot \Q}[P]$-generic object:
  the analogue of Lemma \ref{rearrangeqtot} is true for such posets by the same argument.
\end{remark}  

  We record some easy but useful equivalences involving quotient to term posets.

\begin{lemma} \label{obviousinhindsight}
   In $V$,  $\P * \Q * \termspace(\P, \dot \Q)/P * Q$ is equivalent to
$\P \times \termspace(\P, \dot \Q)$. 
   In $V[P]$,  $\Q * (\termspace(\P, \dot \Q)/P * Q)$ is equivalent to $\termspace(\P, \dot \Q)$.
\end{lemma}

The idea of term forcing extends in a natural way to iterations with more than two steps.
Suppose that $\langle \P_\alpha, \dot \Q_\alpha: \alpha < j \rangle$ is an iteration with
limit $\P_j$. Then we may form a product of term posets $\prod_{\alpha < j} \termspace(\P_\alpha, \dot \Q_\alpha)$,
using the same supports that were used to form $\P_j$. We note that the poset $\P_0$ is trivial,
so the first term poset in the product is equivalent to $\Q_0$.

It is easy to see that:
\begin{itemize}
\item the underlying set of $\prod_{\alpha < j} \termspace(\P_\alpha, \dot \Q_\alpha)$ is the underlying set of $\P_j$. 
\item The identity function is a projection from $\prod_{\alpha < j} \termspace(\P_\alpha, \dot \Q_\alpha)$ to $\P_j$.
\item There is a natural quotient to term forcing defined in $V[P_j]$ to produce a $\prod_{\alpha < j} \termspace(\P_\alpha, \dot \Q_\alpha)$-generic
  object which projects to $\P_j$. 
\end{itemize}

We will need some lemmas relating Cohen posets computed in different models.

\begin{lemma} \label{oldlemma}
  Let $\kappa^{<\kappa} = \kappa \le \lambda$ and let $\P$ be a $\kappa$-cc forcing poset
  of cardinality at most $\kappa$. Let $\dot \Q$ be a $\P$-name
  for $\Add^{V[\P]}(\kappa, \lambda)$.
  Then $\termspace(\P, \dot\Q)$ is equivalent to $\Add^V(\kappa, \lambda)$.
\end{lemma}

\begin{proof} We can view $\Add^{V[P]}(\kappa, \lambda)$ as the $<\kappa$-support product in $V[P]$ of $\lambda$ copies
  of $2$, considered as a poset where $0$ and $1$ are incomparable elements. By the chain condition
  $\termspace(\P, \dot\Q)$ is equivalent to the $<\kappa$-support product of $\lambda$ copies of
  $\termspace(\P, \dot 2)$. An element $\dot t$ of $\termspace(\P, \dot 2)$ is determined by the Boolean
  value $b_{\dot t}$ of ``$\dot t = 0$'', and easily
  $\forces {\dot t_0} \le {\dot t_1} \iff \forces {\dot t_0} = {\dot t_1} \iff b_{{\dot t}_0} = b_{{\dot t}_1}$.    
  So $\termspace(\P, \dot 2)$ is a poset with at most $\kappa$ pairwise incomparable conditions, and then easily
  $\termspace(\P, \dot\Q)$ is equivalent to $\Add^V(\kappa, \lambda)$.
\end{proof} 

In the situation of Lemma \ref{oldlemma}, let $\R = \termspace(\P, \dot\Q)$.
If $P \times  R$ induces $P * Q$
then $V[P * Q]$ and $V[P \times R]$ have the same $<\kappa$-sequences of ordinals:
  to put it another way the associated quotient to term forcing $\R/P * Q$
  is $<\kappa$-distributive in $V[P * Q]$.  Since both $\P * \Q$ and $\P \times \R$ are
  $\kappa^+$-cc, $\R/P * Q$ is $\kappa^+$-cc in $V[P * Q]$. 

  For use in the proofs of Claims \ref{GroupIInequals1} and \ref{GroupIInequals0} from Section \ref{GroupII},
  we need an easy lemma about $\R/P * Q$.
  The point of Lemma \ref{refinement1} is that  chain condition properties of $\R/P * Q$
   can be deduced from corresponding properties for $\R$. 
  
\begin{lemma} \label{refinement1} 
  Let the hypotheses of Lemma \ref{oldlemma} hold and let $\R = \termspace(\P, \dot\Q)$.
  Then in any outer model $W$ of $V[P * Q]$ where $\R$ is $(\kappa^+)^W$-Knaster, $\P \times (\R/P * Q)$ 
 is also $(\kappa^+)^W$-Knaster. 
\end{lemma}

\begin{proof} Since $\vert \P \vert = \kappa$ it is enough to show that $\R/P * Q$ is 
  $(\kappa^+)^W$-Knaster. We use the description of $\R$ from the proof of Lemma \ref{oldlemma}.
  Let $(r_\alpha)_{\alpha < (\kappa^+)^W}$ be a sequence in $W$ such that $r_\alpha \in \R/P * Q$.
  Since $\R$ is $(\kappa^+)^W$-Knaster in $W$, we may find $B \in W$ such that $B$ is unbounded
  in $(\kappa^+)^W$ and $(r_\alpha)_{\alpha \in B}$ is a sequence of conditions which
  are pairwise compatible in $\R$. Let $\alpha, \beta \in B$, then by definition
  for every $\eta \in \dom(r_\alpha) \cap \dom(r_\beta)$ the same term appears
  at coordinate $\eta$ in $r_\alpha$ and $r_\beta$. It is easy to see that
  $(r_\alpha \cup r_\beta)[P] = r_\alpha[P] \cup r_\beta[P] \in Q$, so that
  $r_\alpha \cup r_\beta \in \R/P * Q$ and is a common lower bound
  in $\R/P*Q$ for $r_\alpha$ and $r_\beta$.
\end{proof}

We also record an easy fact about the closure of quotient-to-term posets. 

\begin{lemma} \label{qtot-closed}
  Let $\P$ be $\rho$-distributive and let $\P$ force ``$\dot \Q$ is canonically $\rho$-closed''.
  Let $P * Q$ be $\P * \dot \Q$-generic over $V$. Then the quotient-to-term poset
  $\termspace(\P, \dot \Q)/P * Q$ is canonically $\rho$-closed in $V[P * Q]$.

  More generally, if $\langle \P_\alpha, \dot \Q_\alpha: \alpha < j \rangle$ is an iteration whose
  supports are closed under increasing $\rho$-sequences, $\Q_0$ is $\rho$-distributive 
  and $\forces_\alpha$ ``$\Q_\alpha$ is $\rho$-canonically $\rho$-closed'' for $0 < \alpha < j$ then
  the associated quotient-to-term poset is canonically $\rho$-closed in $V[P_j]$. 
  
\end{lemma}

\begin{proof} Let $(\tau_i)_{i < \rho} \in V[P * Q]$ be a decreasing sequence in $\termspace(\P, \dot \Q)/P * Q$.
  Since $\P * \dot \Q$ is $\rho$-distributive, $(\tau_i)_{i < \rho} \in V$. By the definition of $\termspace(\P, \dot \Q)/P * Q$,
  $\tau_i[P] \in Q$ for all $i$ and $\forces_\P \tau_j \le \tau_i$ for $i < j < \rho$. Let $\tau$ be a name
  for a greatest lower bound for $(\tau_i)_{i < \rho}$ in $\dot \Q$, then $\tau[P]$ is a greatest lower bound for $(\tau_i[P])_{i < \rho}$ in $\dot Q[P]$,
  so $\tau[P] \in Q$ and hence $\tau \in \termspace(\P, \dot \Q)/P * Q$.
  It follows easily that $\tau$ is a greatest lower bound
  for $(\tau_i)_{i < \rho}$ in $\termspace(\P, \dot \Q)/P * Q$. The argument for longer iterations is essentially the same.
\end{proof}

\subsection{Projection and absorption} \label{projections}
  
We also collect some facts about projections between forcing posets
and absorbing forcing posets by collapses which will be used in the sequel.
We refer the reader to \cite{MagidorShelah} for a careful discussion of these
matters.

\begin{definition}
 Let $\P$ and $\Q$ be canonically $<\kappa$-closed.
 A projection map $\pi:\P \to \Q$ is
 {\em $<\kappa$-continuous} if it preserves the greatest lower bounds assured by the canonical closure.
 That is to say if $(p_i)_{i < \alpha}$ is decreasing in $\P$ for some $\alpha < \kappa$,
 and $p$ is the greatest lower bound in $\P$ for $(p_i)_{i < \alpha}$, then
 $\pi(p)$ is the greatest lower bound in $\Q$ for $(\pi(p_i))_{i < \alpha}$.
\end{definition}

Facts \ref{closedquotient} and \ref{absorb} both form part of \cite[Lemma 2.6]{MagidorShelah}.

\begin{fact} \label{closedquotient} Suppose that $\P$ and $\Q$ are canonically $<\kappa$-closed
  and $\pi:\P \to \Q$ is a $<\kappa$-continuous projection.
  If $Q$ is $\Q$-generic, then in
$V[Q]$ the quotient forcing $\P/Q$ is canonically $<\kappa$-closed. \end{fact}

\begin{fact} \label{absorb} Suppose that $\kappa < \mu$ are inaccessible
cardinals.  Suppose that $\Q$ is a canonically $<\kappa$-closed forcing of size
at most $\mu$.  Then there is a $<\kappa$-continuous projection from $\Coll(\kappa,\mu)$
to $\Q$. \end{fact}

\begin{definition} \label{Easton-collapse} Let $\kappa < \lambda$ where $\kappa$ is inaccessible and $\lambda$ is Mahlo.
  Let $E$ be a set of inaccessible cardinals such that $\kappa, \lambda \in E$, and
  $E \cap [\kappa, \lambda)$ is the intersection of a club subset of $[\kappa, \lambda)$ with
      the set of inaccessible cardinals in this interval. For each $\alpha \in E$, let
    $\alpha^* = \min(E \setminus (\alpha +1))$.   
  
Let $\East^E(\kappa,<\lambda)$  be the collection of partial functions $f$ with $\dom(f) \subseteq E \cap [\kappa, \lambda)$ such that
\begin{enumerate}
\item \label{ecitem1}  $\dom(f)$ is an {\em Easton set}, that
  is to say it is bounded in every inaccessible cardinal. 
\item For all $\alpha \in \dom(f)$, $f(\alpha) \in \Coll(\alpha, <\alpha^*)$.
\end{enumerate}
$\East^E(\kappa,<\lambda)$ is ordered coordinatewise. 
\end{definition}

Note that by the hypotheses on $E$ and $\lambda$,
$E \cap [\kappa, \lambda)$ is stationary in $\lambda$, and the Easton support condition for $f$ is equivalent to demanding that
$\dom(f)$ is bounded in every cardinal in $E \cup \{ \lambda \}$. 

\begin{lemma}  $\East^E(\kappa,<\lambda)$ is canonically $<\kappa$-closed and $\lambda$-Knaster.
\end{lemma}

\begin{proof} The closure is immediate since each component is canonically $<\kappa$-closed and the union of fewer than
  $\kappa$ Easton subsets of $[\kappa, \lambda)$ is Easton. Given $(p_i)_{i < \lambda}$ we may
    find a stationary set $E' \subseteq E \cap [\kappa, \lambda)$ such that $p_i \restriction i$ is constant for $i \in E'$,
    and then a stationary $E'' \subseteq E'$ such that $\dom(p_i) \subseteq j$ for $i, j \in E''$ with
    $i < j$. The conditions $p_i$ for $i \in E''$ are pairwise compatible. 
\end{proof}

\begin{lemma} \label{Easton-collapse-absorb}
  With the same hypotheses as in Definition \ref{Easton-collapse}, let  $(\U(\alpha))_{\alpha \in E \cap [\kappa, \lambda)}$
    be such that  $\U(\alpha)$ is a canonically $<\alpha$-closed poset (which may be trivial) in $V_{\alpha^*}$,
    and let  $\U$ be the Easton support product of the $\U(\alpha)$'s.
    Then there is a $<\kappa$-continuous projection from $\East^A(\kappa,<\lambda)$ to $\U$. 
\end{lemma} 

\begin{proof} By Fact \ref{absorb}, for every $\alpha \in A \cap [\kappa, \lambda)$
 there is a $<\alpha$-continuous projection $\pi_\alpha: \Coll(\alpha, <\alpha^*) \rightarrow \U(\alpha)$.
  We define a projection $\pi$ from $\East^A(\kappa,<\lambda)$ to $\U$
  by defining $\dom(\pi(f)) = \dom(f)$ and $\pi(f)(\alpha) = \pi_\alpha(f(\alpha))$ for all $\alpha$.  
  It is easy to check that  the map $\pi$ is a $<\kappa$-continuous projection since each of the maps  $\pi_\alpha$ is a $<\alpha$-continuous
  projection.
\end{proof}

  The point of  $\East^E(\kappa, <\lambda)$ is that it can absorb suitable Easton support iterations in a reasonable way.
  The following lemma is a prototype for the arguments in Section \ref{GroupIV}. 
  
\begin{lemma} \label{Easton-collapse-lemma}
 With the same hypotheses as in Definition \ref{Easton-collapse}, 
  let $\langle \P_\alpha, \dot \Q_\alpha: \alpha \in E \cap [\kappa, \lambda) \rangle$ be an
    Easton support iteration, assume that 
    $\forces_\alpha \mbox{``$\dot \Q_\alpha \in V_{\alpha^*}$''}$
    and  $\forces_\alpha \mbox{``$\dot \Q_\alpha$ is canonically $<\alpha$-closed''}$ for all $\alpha$, 
    and let $\P_\lambda$ be the direct limit.
Then $\P_\lambda$ can be absorbed into $\East^E(\kappa, \lambda)$ so that the quotient forcing is canonically $<\kappa$-closed.
\end{lemma}

\begin{proof}
    Let $\U(\alpha) = \termspace(\P_\alpha, \dot \Q_\alpha)$ and let $\U$ be the Easton support product of the posets
    $\U(\alpha)$. Then there is a natural projection from $\U$ to the limit poset $\P_\lambda$,
    and it routine to check that the projection is $<\kappa$-continuous. 
    Lemma \ref{Easton-collapse-absorb} gives a $<\kappa$-continuous projection from $\East^E(\kappa, \lambda)$ to $\U$.
    It follows from Fact \ref{closedquotient} 
    that the quotient forcings for
    absorbing $\P_\lambda$ into $\U$ and $\U$ into $\East^E(\kappa, \lambda)$ are both canonically $<\kappa$-closed. 
\end{proof}     

There is a parallel but simpler fact for the standard Levy collapse. 

\begin{lemma} \label{Levy-collapse-absorb} 
  Let $\kappa$ and $\lambda$ be inaccessible with $\kappa < \lambda$, let $(\U(\alpha))_{\alpha \in [\kappa, \lambda)}$
    be a sequence of canonically $<\kappa$-closed posets of cardinality less than $\lambda$ and let $\U$ be the product
    of the $\U(\alpha)$'s with $<\kappa$-supports. Then there is a $<\kappa$-continuous projection from $\Coll(\kappa, <\lambda)$ to $\U$.
\end{lemma} 

\subsection{Easton sets in Easton extensions}

In Section \ref{GroupIV} we will need to absorb some iterations of the form $\FL * \P$,
where $\FL$ and $\P$ are Easton support iterations done over the same set of cardinals, into a suitable
Easton support product of term forcings. To apply the ideas of Section \ref{projections} we need
to analyse the Easton sets in $V[L]$, because they will form the supports of conditions in $\P$.

Let $E$ be a set of inaccessible cardinals with limit order type such that for every inaccessible
$\alpha < \sup(E)$ with $\alpha = \sup(E \cap \alpha)$, $\alpha$ is in $E$.
For each $\alpha < \sup(E)$, let $\alpha^* = \min(E \setminus (\alpha + 1))$. 
Note that a subset of $E$ is an Easton set if and only if it is bounded in every element of $E$,
together with $\sup(E)$ in case this cardinal is inaccessible.

Let $\FL$ be an iteration with Easton supports such that:
\begin{itemize}
\item The support of $\FL$ is contained in $E$.
\item For every $\alpha \in E$, it is forced by $\FL \restriction \alpha$ that the iterand at $\alpha$ is $<\alpha$-closed.   
\item For every $\alpha \in E$, $\vert \FL \restriction \alpha +1 \vert < \alpha^*$.
\end{itemize}

  It is easy to see that for every $\beta < \sup(E)$, 
$\vert \FL \restriction \beta + 1 \vert < \beta^*$. As a consequence
it is forced by $\FL \restriction \beta + 1$ that the tail iteration
above $\beta$ is $<\beta^*$-closed. It follows easily that every $\alpha \in E$
remains inaccessible in $V[L]$. 

\begin{lemma} \label{eastonsets}
If $S \in V[L]$ is an Easton subset of $E$ then $S$ is covered
by an Easton subset of $E$ which lies in $V$.
\end{lemma}

\begin{proof}
Let $\dot S$ be an $\FL$-name for an Easton subset of $E$.  We will establish
that a stronger statement holds for all triples $(p, \beta, \gamma)$ where
$p \in \FL$, and $\beta < \gamma \le \sup(E)$:

\medskip

\noindent $R(p, \beta, \gamma)$ is the statement ``There exist $p' \le p$ 
with $p' \restriction \beta + 1 = p \restriction \beta + 1$ and
an Easton set  $T \subseteq [\beta, \gamma) \cap E$ such that
$p' \forces \dot S \cap [\beta, \gamma) \subseteq \check T$''. 

\medskip

Our desired conclusion will follow by setting $\beta = 0$ and $\gamma = \sup(E)$.
We prove that $R(p, \beta, \gamma)$ holds for all triples $(p, \beta, \gamma)$ by induction on $\gamma$.

\begin{itemize}

\item Case I: $E$ is bounded in $\gamma$, say $\bar \gamma = \sup(E \cap \gamma) < \gamma$.
  If $\beta = \bar \gamma$ then $[\beta, \gamma) \cap E = \emptyset$ and there is nothing to do.
    If $\beta < \bar \gamma$ then we appeal to $R(p, \beta, \bar \gamma)$, which is true by induction.

\item Case II: $\gamma = \sup(E \cap \gamma)$ and $\cf(\gamma) = \mu < \gamma$. 

  \begin{itemize}

  \item   Subcase IIa: $\mu \le \beta$. Note that the union of at most $\mu$ Easton subsets of
    $[\beta, \gamma) \cap E$ is Easton. 
    Choose an increasing sequence $(\gamma_i)_{i < \mu}$ of ordinals
    which is cofinal in $(\beta, \gamma)$. Let $p_0 = p$, and build a decreasing sequence $(p_i)_{i < \mu}$
    of conditions in $\FL$ with $p_i \restriction \beta + 1 = p_0 \restriction \beta + 1$,
    together with Easton sets $T_i \subseteq [\beta, \gamma_i) \cap E$, such that
      $p_{i+1} \forces \dot S \cap [\beta, \gamma_i) \subseteq \check T_i$.

    At successor steps we choose
    $p_{i+1}$ by appealing to $R(p_i, \beta, \gamma_i)$, at limits we may take lower bounds because all iterands
    past $\beta$ are forced to be $\mu$-closed and (by the remark about unions of Easton sets) there is
    no problem with the supports. After $\mu$ steps we let $p'$ be a lower bound for the conditions $p_i$ such that
    $p' \restriction \beta + 1 = p_0 \restriction \beta + 1$, and $T = \bigcup_{i < \mu} T_i$, where $p'$ can be chosen
    as in the choice of $p_i$ for $i$ limit and $T$ is Easton by the remark on unions of Easton sets.  
    
  \item   Subcase IIb: $\beta < \mu < \gamma$. Start by appealing to $R(p, \beta, \mu)$ to produce $p' \le p$ and 
  $T_0$ an Easton subset of $[\beta, \mu)$ such that $p' \restriction \beta + 1 = p \restriction \beta + 1$ 
  and $p' \forces \dot S \cap [\beta, \mu) \subseteq \check T_0$. Then replace $\beta$ by $\mu$ and argue as in Subcase IIa
    to produce $p'' \le p'$ and an Easton set $T_1 \subseteq [\mu, \gamma) \cap E$ such that $p'' \restriction \mu + 1 = p' \restriction \mu + 1$
      and $p'' \forces \dot S \cap [\mu, \gamma) \subseteq \check T_1$. Clearly $p''$ and $T_0 \cup T_1$ will serve to witness
        $R(p, \beta, \gamma)$.

  \end{itemize}
  
\item Case III: $\gamma = \sup(E \cap \gamma)$ and $\gamma$ is inaccessible, in particular $\gamma \in E$ or $\gamma = \sup(E)$.

  It is forced that $\dot S$ is bounded in $\gamma$. Since $\vert \FL \restriction \beta + 1 \vert < \beta^* < \gamma$ we may
  build $p' \le p$ such that $p' \restriction \beta + 1 = p \restriction \beta + 1$, and for every $q \le p'$
  such that $q$ decides $\sup(\dot S \cap \gamma)$ we have that $q \restriction \beta + 1^\frown p' \restriction (\beta + 1, \sup(A))$
  decides it: the key points are that all iterands past $\beta$ are $<\beta^*$-closed, $(\beta + 1, \gamma) \cap E = [\beta^*, \gamma) \cap E$,
  and the union of fewer than $\beta^*$ Easton subsets of $[\beta^*, \gamma) \cap E$  is Easton.
  If we let $\gamma' = \sup \{ \eta : \exists r \in \FL \restriction \beta + 1 \; r^\frown p' \restriction (\beta + 1, \gamma) \forces \sup(\dot S \cap \gamma) = \eta \}$
  then $\gamma' < \gamma$ and $p' \forces \dot S \cap \gamma \subseteq \gamma'$. Appealing to $R(p', \beta, \gamma')$ we find $p'' \le p'$ and
  an Easton set $T \subseteq [\beta, \gamma')$ such that $p'' \restriction \beta + 1 = p' \restriction \beta + 1$ and
  $p'' \forces \dot S \cap [\beta, \gamma) \subseteq \check T$. 

\end{itemize}
    
\end{proof}

\subsection{Robustness of chain condition}

In the sequel we will often force over $V$ using the Cohen poset $\Add(\kappa, \lambda)$
defined over some inner model of $V$. This idea is often useful in the situation where $\kappa = \mu^+$
and $2^\mu > \mu^+$: forcing with $\Add(\mu^+, \lambda)$ as defined in $V$ will collapse $\mu^+$,
 so instead we force with $\Add(\mu^+, \lambda)$ defined in some inner model where $2^\mu = \mu^+$. 
 The following Lemma shows that the chain condition (really the Knaster property) of Cohen forcing is quite robust.
 We will also need that the distributivity of Cohen forcing is  robust, but
 we will typically establish this by {\it ad hoc} arguments
 using Easton's Lemma and term forcing, see for instance Lemma \ref{meetingindyreqs} below.

\begin{lemma}\label{Cohenrobust}
  Let $\kappa$ be regular and let $\P = \Add(\kappa, \lambda)$.
  If $\eta^{<\kappa} < \mu$ for every $\eta < \mu$,
    and $W$ is an outer model in which $\mu$ is regular and  every set of ordinals
    of size less than $\kappa$ in $W$ is covered by a set of size less than $\kappa$ in $V$, then
    $W \models \mbox{``$\P$ is $\mu$-Knaster''}$.
\end{lemma}

\begin{proof}
  We work in $W$, noting that our hypotheses imply that
  $\kappa$ is still regular in $W$. Let $(p_i)_{i < \mu}$ be a $\mu$-sequence of conditions
  in $\P$. Let $X = \bigcup_{i < \mu} \dom(p_i)$, so that $X \subseteq \kappa \times \lambda$ with
  $\vert X \vert \le \mu$, and enumerate $X$ as $(x_i)_{i < \mu}$. Let $d_i = \{ j < \mu : x_j \in \dom(p_i) \}$.
  
  Let $S = \mu \cap \cof(\kappa)$, and for $i \in S$ define $f(i) = \sup(d_i \cap i)$. Since $f$ is regressive
  we may fix $S_0 \subseteq S$ stationary and $\eta < \mu$ such that $f(i) = \eta$ for all $i \in S$.
  Thinning out $S_0$ if necessary, we may assume that if $i, j \in S_0$ with $i < j$ then  $\sup d_i < j$. 
  Let $D_i \in V$ be such that $d_i \cap \eta \subseteq D_i \subseteq \eta$ and $\vert D_i \vert < \kappa$.
  Since $\eta^{<\kappa} < \mu$ in $V$, we may find $S_1 \subseteq S_0$ stationary and $D$ such that $D_i = D$ for every
  $i \in S_1$.

  Now let $z = \{ x_j : j \in D \}$, and use the covering hypothesis again to find $Z \in V$
  such that $z \subseteq Z \subseteq \kappa \times \lambda$ and $\vert Z \vert < \kappa$.
  Since $2^{\vert Z \vert} < \mu$ in $V$, we may find $S_2 \subseteq S_1$ stationary
  and a partial function $p$ from $Z$ to $2$ such that $p_i \restriction Z = p$ for all $i \in S_2$.

  We claim that the conditions $p_i$ for $i \in S_2$ are compatible. Let $i < j$ with $i, j \in S_2$,
  and let $(\alpha, \beta) \in \dom(p_i) \cap \dom(p_j)$. Since $(\alpha, \beta) \in X$,
  we find $k$ with $(\alpha, \beta) = x_k$, so that by definition $k \in d_i \cap d_j$.
  Since $\sup(d_i) < j$, $k \in d_j \cap j$, so $k < f(j) = \eta$.
  
  It follows that $k \in d_i \cap \eta$, so that $k \in  D_i =  D$.
  By definition $x_k = (\alpha, \beta) \in z \subseteq Z$,
  and since $(\alpha, \beta) \in \dom(p_i) \cap \dom(p_j)$ we have
  $p_i(\alpha, \beta) = p(\alpha, \beta) = p_j(\alpha, \beta)$.
\end{proof}

\begin{remark}  Similar lemmas with similar purposes appear in papers by
  Abraham \cite[Lemma 2.16]{Abraham} and Cummings and Foreman \cite[Lemma 2.6]{CummingsForeman}
\end{remark}

To streamline the process of applying Lemma \ref{Cohenrobust}, we encapsulate some of the hypotheses
in a definition.

\begin{definition} \label{kappamugood}  Let $\kappa$ and $\mu$ be regular cardinals. 
  Then an outer model $W \supseteq V$ is {\em $(\kappa, \mu)$-good} if and only if
  every set of ordinals of size less than $\kappa$ in $W$ is covered by a set of size less than $\kappa$ in $V$,
  and $\mu$ is regular in $W$.  
\end{definition}

\begin{remark} If $W_1$ is a $(\kappa, \mu)$-good outer model of $V$ and $W_2$ is a $(\kappa, \mu)$-good outer
  model of $W_1$, then $W_2$ is a $(\kappa, \mu)$-good outer model of $V$.
  If $W$ is a $(\kappa, \mu)$-good outer model of $V$ and $W'$ is an intermediate model,
  then $W'$ is also a $(\kappa, \mu)$-good outer model of $V$.
\end{remark}

\subsection{A technical fact}

We will need a version of a technical fact from \cite{NeemanUpto}. The exact statement is
slightly different but the proof will be essentially the same, see the discussion following the statement.

\begin{fact}[{essentially \cite[Lemma 3.10]{NeemanUpto}}] \label{Itay3.10}
   Let $(\kappa_m)_{2 \le m < \omega}$ be an increasing sequence of regular cardinals and
  let $\nu = \sup_m \kappa_m$. Let $Index \subseteq \kappa_2$, let $N < \omega$  and let $\FM(\rho)$ for
  $\rho \in Index$ be forcing posets such that $\vert \FM(\rho) \vert \le \kappa_N$ for all $\rho$.
  Let $R = V_\zeta$ where $\zeta > \nu^+$ and $R$ satisfies a large enough finite fragment of ZFC.  
  Assume that:
  \begin{itemize}
  \item For all sufficiently large $m < \omega$, there exist posets $\P$ and $\Q$ such that:
    \begin{itemize}
    \item $\P$ adds a generic embedding $\pi: V \rightarrow V^*$ such that $\crit(\pi) > \kappa_m$
      and $\pi$ is discontinuous at $\nu^+$.
    \item $\Q$ adds $\kappa_m$ mutually generic filters for $\P$.
    \item $\Q$ preserves cardinals up to and including $\kappa_m$, and forces
    $\cf(\nu^+) > \kappa_m$.   
    \end{itemize}
  \item There are stationarily many $X \prec R$ such that for some $\nu^+$-Knaster poset $\P_X$:
    \begin{itemize}
    \item $\nu^+ \subseteq X$ and $\vert X \vert = \nu^+$.
    \item Letting $M$ be the transitive collapse of $X$, $\P_X$ adds a generic embedding 
      $\pi: M \rightarrow M^*$ such that $\crit(\pi) = \kappa_2$, $\pi(\kappa_2) > \nu^+$, and $\pi$ is discontinuous at $\nu^+$.
    \item $\nu \in \pi(Index)$, and $\P_X$  adds $L$ which is  $\pi(\FM)(\nu)$-generic over $M^*$.
    \end{itemize}
  \end{itemize}
  
  Then there exists $\rho \in Index$ such that $\FM(\rho)$ forces  ``$\nu^+$ has the tree property''.
\end{fact}

The only differences between the proof here and in \cite[Lemma 3.10]{NeemanUpto} are that:
\begin{itemize}
\item  The forcing posets $\FM(\rho)$ are potentially larger (cardinality $\kappa_N$ rather than $\kappa_2$)
    which does not materially affect the argument for choosing $D$ and $n$.  
\item  Only a tail of the cardinals $\kappa_m$ for $m > 2$ are assumed to have the necessary properties,
 but we can still choose a suitable $m > n, N$. 
\end{itemize}

\section{$\A$, $\B$, $\U$, $\C$ and $\S$} \label{abcu}

We will use several versions of the main forcing construction from
Neeman's paper \cite{NeemanUpto}. To minimise repetition we  describe here the simplest
version that we will need, then later in the paper we modify the construction as needed.
See Section \ref{modify} for a discussion of how we modify the construction.

\subsection{The basic forcing} \label{basicaus} 

The version we describe in detail here is basically the forcing of \cite{NeemanUpto} with the
minor simplification that the cardinal $\mu_1$ is fixed from the start rather than
chosen generically. 
The initial setup involves an increasing $\omega$-sequence of regular cardinals $(\mu_n)_{n < \omega}$,
where $\mu_0^{<\mu_0} = \mu_0$, $\mu_1^{<\mu_1} = \mu_1$, and the $\mu_n$'s are indestructibly supercompact for $n \ge 2$, together with 
 a universal indestructible Laver function $\phi$ defined on $(\mu_1, \mu_\omega)$, where $\mu_\omega = \sup_n \mu_n$.
 We will define forcing posets $\A$, $\B$, $\U$, $\C$ and $\S$ along with a number of auxiliary forcing posets. 

 The basic idea is that we will define $\A, \B, \C \in V$ and then project $\mathbb A \times \mathbb B \times \mathbb C$
 to an iteration $\A * \U * \S$. 
 Forcing with $\A * \U * \S$ will produce an extension in which
 $2^{\mu_n} = \mu_{n+2}$ for all $n$, $\mu_{n+1} = \mu_n^+$ for all $n > 0$,
 and $\mu_n$ enjoys a highly indestructible version of the tree property
 for $n \ge 2$. Very roughly speaking $\A$ is responsible for making
 $2^\mu_n = \mu_{n+2}$ for all $n$, $\S$ is responsible for collapsing 
 cardinals so that $\mu_{n+1}$ becomes the successor of  $\mu_n$ for $n > 0$,
 and $\U$ is responsible for making the tree property at $\mu_n$ indestructible
 for $n \ge 2$. We discuss these points in more detail after defining
 $\A * \U * \S$.

 \begin{remark}   Readers of \cite{NeemanUpto} will notice that the definitions of $\A_n$ and $\A$ are slightly
   different here. This makes the definitions more uniform, and is possible because
   the value of $\mu_1$ is fixed. 
 \end{remark}

\begin{itemize}

\item  $\A$:  Conditions in $\A_n$ are partial functions from
  the interval $[\mu_{n+1}, \mu_{n+2})$ to $2$ with supports of size less than
    $\mu_n$, ordered by extension.
   We will sometimes write $\A_n$ as $\Add(\mu_n, [\mu_{n+1}, \mu_{n+2}))$.
    Of course $\A_n$ is equivalent to the standard 
    Cohen poset $\Add(\mu_n, \mu_{n+2})$. The poset $\A_n \restriction \alpha$ is defined
    in the obvious way. 

  $\A$ is the full support product of the posets $\A_n$ for $n < \omega$.
  Whenever it is convenient we will regard conditions in $\A$
  as partial functions $p$ from $[\mu_1, \mu_\omega)$ to
  $2$, such that $p \restriction [\mu_{n+1}, \mu_{n+2})$ has support of size less than $\mu_n$. 
  Intuitively $\A$ is set up so that we finish adding Cohen subsets of $\mu_n$
  before we begin to add Cohen subsets of $\mu_{n+1}$.
      
  For $\alpha < \mu_\omega$, $\A \restriction \alpha$ is equivalent
  to $\prod_{i < n} \A_i \times \A_n \restriction \alpha$ for the least $n$ such that $\alpha \le \mu_{n+2}$.
  We let $A$ be some $\A$-generic object, and define
  $A_n$ and $A \restriction \alpha$ in the obvious way.

\item  $\B$ and $\U$:  $\B$ and $\U$ are two posets with the same set of conditions but different orderings,
  with $\B \in V$ and $\U \in V[A]$. In a sense that we make precise later $\B$ is
  a term forcing for $\U$, but its definition involves a kind of ``self-reference'' not present
  in the simple term forcing of Section \ref{termforcing}.
  Conditions in $\B$ will be certain functions with domains contained in
  $(\mu_1, \mu_\omega)$, and $\B \restriction \alpha$ is the set of $b \in \B$ with $\dom(b) \subseteq \alpha$;
  more generally if $I$ is an interval then $\B \restriction I$ is the set of  $b \in \B$ with
  $\dom(b) \subseteq I$, and in all cases we will view $\B \restriction I$ as a poset with the ordering
  inherited from $\mathbb B$. 
  ${\U} \restriction \alpha$ has the same conditions as $\B \restriction \alpha$.

  Formally speaking we will define $\B \restriction \alpha$ and ${\U} \restriction \alpha$ by simultaneous
  induction on $\alpha$, in such a way that ${\U} \restriction \alpha \in V[A \restriction \alpha]$. 
  A condition $b \in \B \restriction \alpha$ is a function such that:
  \begin{itemize}
  \item $\dom(b)$ is an Easton subset of the set of $\alpha' \in \dom(\phi) \cap \alpha$
    such that  
    $\phi(\alpha')$ is a $\A \restriction \alpha' * \dot{\U} \restriction \alpha'$-name
    for a $<\alpha'$-directed closed forcing poset.
  \item For every $\alpha' \in \dom(b)$, $b(\alpha')$ is an
    $\A \restriction \alpha' * \dot{\U} \restriction \alpha'$-name
    for a condition in $\phi(\alpha')$.
  \end{itemize}

  $\B \restriction \alpha$ and ${\U} \restriction \alpha$ are ordered as follows:
  \begin{itemize}
  \item  $b_1 \le b_0$ in $\B \restriction \alpha$ if and only if $\dom(b_0) \subseteq \dom(b_1)$
    and
    $(0, b_1 \restriction \alpha') \forces_{\A \restriction \alpha' * \U \restriction \alpha'}
    b_1(\alpha') \le b_0(\alpha')$ for all $\alpha' \in \dom(b_0)$.
  \item  $u_1 \le u_0$ in ${\U} \restriction \alpha$ if and only if $\dom(u_0) \subseteq \dom(u_1)$
    and there is $a \in A \restriction \alpha$ such that
    $(a \restriction \alpha', u_1 \restriction \alpha')
    \forces_{\A \restriction \alpha' * \U \restriction \alpha'}  u_1(\alpha') \le u_0(\alpha')$ for all $\alpha' \in \dom(u_0)$.
  \end{itemize}

\end{itemize}

By going to a dense subset we may view $\A * \dot\U$ as consisting of pairs $(a, u)$ where $a \in \A$ and
$u \in \B$, ordered as follows: $(a_1, u_1) \le (a_0, u_0)$ if and only if $a_1 \le a_0$ in $\A$,
$\dom(u_0) \subseteq \dom(u_1)$,
and $(a_1 \restriction \alpha, u_1 \restriction \alpha) \forces u_1(\alpha) \le u_0(\alpha)$
for all $\alpha \in \dom(u_0)$. A similar remark applies to initial segments $\A \restriction \beta * \dot\U \restriction \alpha$
 where $\alpha \le \beta \le \mu_\omega$.

\begin{remark} \label{easyaubfacts}
  \leavevmode
\begin{enumerate}
\item   
  We see from the definition that $\U \in V[A]$, and that $\U$ may be viewed as some type of iteration
  in $V[A]$, where at every $\alpha$ in the domain of $\mathbb B$ we use the $\U \restriction \alpha$-name
  $\phi(\alpha)[A \restriction \alpha]$
\item
  The construction of $\B$ is also iterative, so that in particular for $\alpha < \beta \le \mu_\omega$
  the poset $\B \restriction \beta$ is {\em not} isomorphic to
  $\B \restriction \alpha \times \B \restriction [\alpha, \beta)$. However
    standard term forcing arguments show that the natural concatenation map is a
    projection from $\B \restriction \alpha \times \B \restriction [\alpha, \beta)$ to $\B \restriction \beta$.
\end{enumerate}
\end{remark}

 Let $\alpha \le \beta \le \mu_\omega$ and let $F \subseteq \A \restriction \beta * \dot\U \restriction \alpha$ be a filter,
 which we assume to be generated by pairs $(a, u)$ with $a \in \A \restriction \beta$ and
 $u \in \U \restriction \alpha$. 
 The reader is warned that $F$ may only exist in a generic extension
 of $V$.

 \begin{itemize}

 \item $\B^{+F} \restriction [\alpha, \beta)$: The underlying set of the poset $\B^{+F} \restriction [\alpha, \beta)$ 
   is $\B \restriction [\alpha, \beta)$, and it is ordered by feeding in information from $F$.
   Formally $b_1 \le b_0$ if and only if $\dom(b_0) \subseteq \dom(b_1)$ and there is $(a, u) \in F$ such that
   $(a \restriction \alpha', u \cup b_1 \restriction \alpha') \forces b_1(\alpha') \le b_0(\alpha')$ for all $\alpha' \in \dom(b_0)$.
   Note that the definition makes sense because $\dom(u) \subseteq \alpha$ and $\dom(b_1) \subseteq [\alpha, \beta)$,
     so that $u \cup b_1 \restriction \alpha' \in \B \restriction \alpha'$. Note also that $F$ being a filter
     generated by pairs $(a, u)$ as above
     is sufficient to show that the ordering on $\B^{+F} \restriction [\alpha, \beta)$ is transitive. 

 \end{itemize}

 A couple of examples may help to clarify this definition, where throughout $\alpha \le \beta < \mu_\omega$:
\begin{itemize} 
\item   $F = 0$ (the trivial filter): $\B^{+0} \restriction [\alpha, \beta) = \B \restriction [\alpha, \beta)$.
\item   $\alpha = 0$ and  $F = A \restriction \beta * 0$: $\B^{+A \restriction \beta * 0} \restriction [0, \beta) =
  \U \restriction \beta$.
\item   $F = A \restriction \beta * U \restriction \alpha$: $\B^{+A \restriction \beta * U \restriction \alpha} \restriction [\alpha, \beta)$
  is equivalent to the natural forcing for prolonging $A \restriction \beta * U \restriction \alpha$ to  $A \restriction \beta * U \restriction \beta$.
  We let $\U \restriction [\alpha, \beta) = \B^{+A \restriction \beta * U \restriction \alpha} \restriction [\alpha, \beta)$.
\end{itemize}

We quote without proof some facts from \cite{NeemanUpto}: the proofs are in every case the same, or slightly easier
because here we fixed a value for $\mu_1$ in advance.

\begin{fact}[{\cite[Claim 4.5]{NeemanUpto}}] \label{Itay4.5}
  Let $F_0 \subseteq F_1$ be two filters on $\A \restriction \beta * \dot\U \restriction \alpha$,
  and let $G_0$ be generic for $\B^{+F_0} \restriction [\alpha, \beta)$ over a universe $W \supseteq V$ with $F_0, F_1 \in W$. Let
    $G_1$ be the upwards closure of $G_0$ in $\B^{+F_1} \restriction [\alpha, \beta)$. Then 
  $G_1$ is generic for $\B^{+F_1} \restriction [\alpha, \beta)$ over $W$.
\end{fact}

\begin{remark}
Fact \ref{Itay4.5} explains our comment above that $\B$ is a kind of term forcing.
As an instructive example let $G_0$ be $\B^{+A \restriction \alpha * U \restriction \alpha} \restriction [\alpha, \beta)$-generic
over $V[A \restriction \alpha * U \restriction \alpha]$. If we force over $V[A \restriction \alpha * U \restriction \alpha][G_0]$ 
with $\A \restriction [\alpha, \beta)$ and prolong
  $A \restriction \alpha * U \restriction \alpha$ to $A \restriction \beta * U \restriction \alpha$,
  then in $V[A \restriction \beta * U \restriction \alpha][G_0]$ we may induce
  $G_1$ which is $\U \restriction [\alpha, \beta) = \B^{+A \restriction \beta * U \restriction \alpha} \restriction [\alpha, \beta)$-generic
      over $V[A \restriction \beta * U \restriction \alpha][G_0]$. So $\B^{+A \restriction \alpha * U \restriction \alpha} \restriction [\alpha, \beta)$
   serves as a kind of term poset,      
    adding an $\A \restriction [\alpha, \beta)$-name for a $\U \restriction [\alpha, \beta)$-generic object.
\end{remark}

Fact \ref{Itay4.5} has a kind of reversal: if $G_1$ is generic for $\B^{+F_1} \restriction [\alpha, \beta)$ over $W$
  then we can force over $W[G_1]$ with a suitable factor forcing to obtain $G_0$ which induces $G_1$ as above:
  the factor forcing is just $G_1$ with the ordering of $\B^{+F_0} \restriction [\alpha, \beta)$,
  and is a version of the ``quotient to term'' forcing discussed in Section \ref{termforcing}.

\begin{fact}[{\cite[Claim 4.7]{NeemanUpto}}] \label{Itay4.7}
  If $\alpha' \le \alpha$ and $F'$ is $\A \restriction \alpha' * \U \restriction \alpha'$-generic
  over $V$, then
  $\B^{+F'} \restriction [\alpha, \beta)$ is  $<\alpha$-directed closed in $V[F']$.
\end{fact}

\begin{remark} \label{bclosure} As a useful special case of Claim \ref{Itay4.7}, we may set $\alpha' = 0$ and $F' = 0$ to see that
  $\B \restriction [\alpha, \beta)$ is  $<\alpha$-directed closed in $V$.
\end{remark}

  To lend some insight into what the forcing $\A * \U$ is doing, we quote a fact
  from \cite{NeemanUpto}. We will not be appealing to this fact directly,
  but the ideas in its proof will be used heavily in the proof of Lemma \ref{indestructible} below. 
        
\begin{fact}[{\cite[Claim 4.12]{NeemanUpto}}] \label{Itay4.12}
  Let $A * U \restriction \mu_{n+2}$ be $\A * \U \restriction \mu_{n+2}$-generic over $V$.
  Then in $V[A * U \restriction \mu_{n+2}]$ the cardinal  $\mu_{n+2}$ is indestructibly generically supercompact  
  for $<\mu_{n+2}$-directed closed posets lying in $V[A \restriction \mu_{n+2} * U \restriction \mu_{n+2}]$,
    where the generic embeddings $\pi$ witnessing the generic supercompactness are added by posets
    of the form $\Add^V(\mu_n, \pi(\mu_{n+2})) \times \Add^V(\mu_{n+1}, \pi(\mu_{n+3}))$.
\end{fact} 

Now we define more posets ${\C} \in V$,  ${\C}^{+F}$ for a filter $F$ on $\A * \U$,
and $\S$:

\begin{itemize}  
  
\item $\C$: 
 The forcing poset $\C$ is the full support product of forcing posets
  ${\C}_n$ for $n < \omega$. Conditions in ${\C}_n$ are functions whose domains are subsets
  of $(\mu_{n+1}, \mu_{n+2})$ with domains of size less than $\mu_{n+1}$.
  If $c \in {\C}$ and $\alpha \in \dom(c) \cap (\mu_{n+1}, \mu_{n+2})$ then $c(\alpha)$ is an
  $\A \restriction \alpha * \dot \U \restriction \mu_{n+1}$-name
  for a condition in $\Add(\mu_{n+1}, 1)^{V[A \restriction \alpha * U \restriction \mu_{n+1}]}$.
  $\C$ is ordered like a pure term forcing, that is to say
  $c_1 \le c_0$ if and only if $\dom(c_0) \subseteq \dom(c_1)$ and
  $\forces c_1(\alpha) \le c_0(\alpha)$ for all $\alpha \in \dom(c_0)$. 
  
\item ${\C}^{+F}$: Let $F \subseteq \A * \U$ be a filter,
  and define a forcing poset ${\C}^{+F}$
  with the same set of conditions as $\C$ but a richer ordering:
  $c_1 \le c_0$ if and
  only if $\dom(c_0) \subseteq \dom(c_1)$ and there is $(a, u) \in F$ such that
  for all $n$ and all $\alpha \in \dom(c_0) \cap (\mu_{n+1}, \mu_{n+2})$,
  $(a \restriction \alpha, u \restriction \mu_{n+1}) \forces c_1(\alpha) \le c_0(\alpha)$.

\item  $\S = \C^{+ A * U}$.
  
 \end{itemize}

$\C$ serves as a term forcing for $\S$ in roughly the same way that $\B$ serves as a term forcing
for $\U$.     
Restrictions of the posets $\C$ and $\S$ to intervals are defined in the natural way, and there is an analogous
version of Fact \ref{Itay4.5} for $\C$ and $\S$. 

We let $\B_n = \B \restriction [\mu_{n+1}, \mu_{n+2})$, $\U_n = \U \restriction [\mu_{n+1}, \mu_{n+2})$,
    and  $\S_n = \S \restriction [\mu_{n+1}, \mu_{n+2})$.
      In connection with this we note that
      $\A_n = \A \restriction [\mu_{n+1}, \mu_{n+2})$ and $\C_n = \C \restriction [\mu_{n+1}, \mu_{n+2})$.
      It is easy to see that $\U \restriction \mu_n \in V[A \restriction \mu_n]$ and
      $\S \restriction \mu_n \in V[A \restriction \mu_n][U \restriction \mu_{n-1}]$ for all $n > 1$.

\begin{remark} Each of the posets $\A, \B, \C$ consists of partial functions
  with domains contained in $[\mu_1, \mu_\omega)$. It is useful to note that
  we are using different supports in each of these posets on the interval
  $[\mu_{n+1}, \mu_{n+2})$, which corresponds to the factors with index $n$:
    supports of size less than $\mu_n$ for $\A_n$, Easton supports for $\B_n$,
    supports of size less than $\mu_{n+1}$ for $\C_n$.
\end{remark}

In the current setting, $\S$ is just a product in $V[A * U]$ of the posets $\S_n$.
We emphasise that $\U$ is a not a product but an iteration.      
    We may view $\A * \U * \S$ as a projection of $\A \times \B \times \C$ in the natural way.   
    Much as in Remark \ref{easyaubfacts}  we may also view $\B$ as a projection of
    $\prod_n \B_n$, and so may view $\A * \U * \S$ as a projection
    of $\prod_n \A_n \times \B_n \times \C_n$. See Lemma \ref{portmanteaulemma} below for
    more on this. 
    
One small difference with \cite{NeemanUpto} is that here the definitions are valid for $n=0$, because we
fixed the value of $\mu_1$ in advance. The definitions for $n=0$ have some special features that will
be useful later, and which we record in the following remarks.

\begin{remark} \label{nzerospecial}
  $\B_0 \restriction \mu_1$ and $\U_0 \restriction \mu_1$ are trivial.
  $\U_0 = (\B_0)^{+A_0} \in V[A_0]$. Since $U_0 \restriction \mu_1$ is trivial,
  $U_0$ is irrelevant to the definition of $\S_0$, and $\S_0 = \C_0^{+A_0} \in V[A_0]$.
  $V[A_0 * U_0 * S_0] = V[A_0 * (U_0 \times S_0)]$, and we may view
  $\A_0 * \U_0 * \S_0$ as a projection of $\A_0 \times \B_0 \times \C_0$.
  $\A_0 * \S_0$ is essentially Mitchell forcing \cite{Mitchell},
  and $\C_0$ is essentially the term forcing from Abraham's product analysis of Mitchell
  forcing \cite{Abraham}.
\end{remark}

\begin{remark} \label{useful}  The natural forcing to add
  a $\B_0 \times \C_0$-generic object $B_0 \times C_0$ such that $A_0 \times (B_0 \times C_0)$
  induces $A_0 * U_0 * S_0$ is $<\mu_0$-closed in $V[A_0 * U_0 * S_0]$.
  The argument is essentially the same as that for Lemma \ref{qtot-closed}: conditions
  are pairs $(b, c)$ with $(b, c)[A_0] \in U_0 \times S_0$, $\A_0 * \U_0 * \S_0$
  is $<\mu_0$-distributive so that decreasing $<\mu_0$-sequences lie in $V$,
  hence it is easy to find a lower bound.
\end{remark}

We quote more facts from \cite{NeemanUpto}.

\begin{fact}[{\cite[Claim 4.15]{NeemanUpto}}] \label{Itay4.15}
\leavevmode  
  \begin{enumerate}
  \item Let $\alpha \le \mu_{n+1}$ and let $F = A \restriction \alpha * U \restriction \alpha$.
    Then in $V[F]$ the poset ${\C}^{+F} \restriction [\mu_{n+1}, \mu_\omega)$ is
        $<\mu_{n+1}$-directed closed.
  \item  Let $\alpha \in (\mu_{n+1}, \mu_{n+2})$ and let $F = A \restriction \alpha * U \restriction \mu_{n+1}$.     
    Then in $V[F]$ the poset ${\C}^{+F} \restriction [\alpha, \mu_{n+2})$ is
      $<\mu_{n+1}$-directed closed.
  \end{enumerate}
\end{fact}

\begin{fact}[{\cite[Claim 4.30]{NeemanUpto}}] \label{Itay4.30}
  Let $F = A \restriction \mu_{n+2} * U \restriction \mu_{n+2}$,
  let $\P_1$ be the poset to refine $U \restriction [\mu_{n+2}, \mu_\omega)$ to a generic
   object for $\B^{+F} \restriction [\mu_{n+2}, \mu_\omega)$, and 
 let $\P_2$ be the poset to refine $S \restriction [\mu_{n+2}, \mu_\omega)$ to a generic
   object for $\C^{+F} \restriction [\mu_{n+2}, \mu_\omega)$. 
     Then both $\P_1$ and $\P_2$ are $<\mu_{n+1}$-closed in $V[A][U][S \restriction [\mu_{n+1}, \mu_\omega)]$.
\end{fact}

As we already mentioned, if we  force with $\A * \U * {\S}$
we obtain an extension in which $2^{\mu_n} = \mu_{n+2}$ for all $n$,
and $\mu_{n+1} = \mu_n^+$ for all $n > 0$.
\begin{itemize}
\item $\A$ is responsible for blowing up the powersets of the $\mu_n$'s.
\item $\U$ is responsible for ensuring that
  $\mu_{n+2}$ has the indestructible generic supercompactness property from Fact \ref{Itay4.12}  in $V[A * U \restriction \mu_{n+2}]$. 
\item $\S$ is responsible for collapsing cardinals in the interval $(\mu_{n+1}, \mu_{n+2})$
   to have cardinality $\mu_{n+1}$. 
\end{itemize}

$\A * \U * {\S}$ is a descendant of Mitchell's original forcing \cite{Mitchell}
for collapsing a large cardinal while preserving the tree property. Exactly
as in that forcing the $\S$-coordinate is collapsing cardinals between $\mu_{n+1}$ and $\mu_{n+2}$
``in parallel'' with the $\A$ coordinate adding subsets of $\mu_n$, so that
there is no inner model where $2^{\mu_n} = \mu_{n+1}$ and $\mu_{n+2} = \mu_{n+1}^+$
and we do not run afoul of Specker's result from \cite{Specker}.

We record some information about $\A * \U * \S$ for use later.

\begin{lemma} \label{portmanteaulemma}
\leavevmode
\begin{enumerate}

\item \label{imaginary1}
  $\A \restriction \mu_{n+2}$ is $\mu_{n+1}$-Knaster and $\A \restriction [\mu_{n+2}, \mu_\omega)$
   is $< \mu_{n+1}$-directed closed. 

\item \label{imaginary2}  $\B \restriction \mu_{n+1}$ is $\mu_{n+1}$-Knaster. 

\item \label{imaginary3}  $\C \restriction \mu_{n+1}$ is $\mu_{n+1}$-Knaster.   

\item \label{imaginary4}  $\A \restriction \mu_{n+2} * \U \restriction \mu_{n+1} * \S \restriction \mu_{n+1}$ 
 is $\mu_{n+1}$-Knaster. 

\item \label{imaginary5}  $\B \restriction [\mu_{n+1}, \gamma)$ is $<\mu_{n+1}$-directed closed for all $\gamma$,
  in particular $\B_n$ is $<\mu_{n+1}$-directed closed.

\item \label{imaginary6}  $\C_n$ is $<\mu_{n+1}$-directed closed, as is $\C \restriction [\mu_{n+1}, \mu_\omega)$.

\item \label{portmanteau3}
    For each $n$, the forcing poset $\A* \U* \S$ is the projection of $\P_0 \times \P_1$, where
    $\P_0 = \A \restriction \mu_{n+2} * \U \restriction \mu_{n+1} * \S \restriction \mu_{n+1}$ 
    and $\P_1 = \A \restriction [\mu_{n+2}, \mu_\omega) \times \B \restriction [\mu_{n+1}, \mu_\omega) \times \C \restriction [\mu_{n+1}, \mu_\omega)$.  
    $\P_0$ is $\mu_{n+1}$-Knaster and $\P_1$ is $<\mu_{n+1}$-directed closed. 

\item \label{forrobustness} It is forced by $\P_1$ that $\P_0$ is $\mu_{n+1}$-cc.

\item \label{fortailanalysis} $\prod_{n < \omega} \A_n \times \B_n \times \C_n$
   adds no $<\mu_0$-sequences of ordinals, and  preserves the cardinals $\mu_n$ for
          $n < \omega$ together with $\mu_\omega^+$. 
   Since $\A * \U * \S$ is a projection of $\prod_{n < \omega} \A_n \times \B_n \times \C_n$,
   the same holds for $\A * \U * \S$.
    
\end{enumerate}
\end{lemma}

\begin{proof}

\leavevmode

\begin{enumerate}

\item  This is immediate since $\A \restriction \mu_{n+2} = \prod_{i \le n} \A_i$
  and $\A \restriction [\mu_{n+2}, \omega) = \prod_{i > n} \A_i$.

\item For $n = 0$, $\B \restriction \mu_1$ is trivial forcing. For $n > 0$ the supports of conditions
    in $\B \restriction \mu_{n+1}$ are Easton subsets of the Mahlo cardinal $\mu_{n+1}$,
    $\vert \B \restriction \alpha \vert < \mu_{n+1}$ for all $\alpha < \mu_{n + 1}$, and
    the $\mu_{n+1}$-Knaster property for $\B \restriction \mu_{n+1}$ follows by standard arguments in iterated forcing.

  \item  For $n = 0$, $\C \restriction \mu_1$ is trivial forcing. For $n > 0$,
    $\C \restriction \mu_{n+1} = (\C \restriction \mu_n) \times \C_{n-1}$, and $\vert \C \restriction \mu_n \vert < \mu_{n+1}$
    so this factor is trivially $\mu_{n+1}$-Knaster. $\C_{n-1}$ is the product taken with $<\mu_n$-supports of
    $\mu_{n+1}$ posets each with cardinality less than $\mu_{n+1}$,  and $\mu_{n+1}$ is inaccessible,
    so the $\mu_{n+1}$-Knaster property for $\B \restriction \mu_{n+1}$ follows by standard arguments in product forcing.
    
\item $\A \restriction \mu_{n+2} * \U \restriction \mu_{n+1} * \S \restriction \mu_{n+1}$ 
  is a projection of $\A \restriction \mu_{n+2} \times \B \restriction \mu_{n+1} \times \C \restriction \mu_{n+1}$,
  which is a product of $\mu_{n+1}$-Knaster posets. We note that the projection is the identity map
  between two posets with the same underlying set but different orderings.

\item  This follows from Remark \ref{bclosure}. 
  
 \item $\C_n$ is the product taken with $<\mu_{n+1}$-supports of $<\mu_{n+1}$-closed term forcing posets,
   and $\C \restriction  [\mu_{n+1}, \mu_\omega) = \prod_{i \ge n} \C_i$.

\item 
    It is routine to verify that the natural map from $\P_0 \times \P_1$
    to $\A * \U * \S$ is a projection. The claims about closure and chain condition
    follow immediately from what we already proved. 

\item This is immediate by Easton's lemma. 
    
  \item The preservation of $\mu_0$ and the claim about $<\mu_0$-sequences are immediate, as the product is $<\mu_0$-closed.
    $\prod_{m \le n} \A_m \times \prod_{m<n} \B_m \times \prod_{m < n} \C_m$ is
    $\mu_{n+1}$-Knaster and $\prod_{m > n} \A_m \times \prod_{m \ge n} \B_m \times \prod_{m \ge n} \C_m$ is
    $< \mu_{n+1}$-closed, so that $\mu_{n+1}$ is preserved by Easton's Lemma. If $\mu_{\omega}^+$ were collapsed
    we would have $\cf(\mu_\omega^+) < \mu_\omega$ in the extension, but this is impossible by Easton's Lemma.

\end{enumerate}

\end{proof}

\begin{remark} Item \ref{forrobustness} is immediate in our current setting
  but will hold and be useful in more general settings, as we discuss
  in Section \ref{modify} below.
\end{remark}

\begin{corollary} \label{rightafterportmanteau}

  \leavevmode

\begin{enumerate}

\item   
  Every $<\mu_{n+1}$-sequence of ordinals from $V[A * U * S]$ lies in the submodel
  $V[A \restriction \mu_{n+2}][U \restriction \mu_{n+1}][S \restriction \mu_{n+1}]$.

\item \label{aus-covering}  Every set of ordinals of cardinality less than $\mu_{n+1}$ in $V[A * U * S]$
    is covered by such a set lying in $V$.

 \item If $\Q$ is $\mu_\omega$-closed in $V$ and $Q$ is $\Q$-generic over $V[A * U * S]$,
    every set of ordinals of cardinality less than $\mu_{n+1}$ in $V[A * U * S][Q]$
    is covered by such a set lying in $V$.

\end{enumerate}
    
\end{corollary} 

\begin{proof} The first two claims are immediate. For the last claim write
  $\Q \times \A * \U * \S$ as the projection of $\Q \times \P_0 \times \P_1$ where
  $\P_0$ is $\mu_{n+1}$-cc and $\P_1$ is $< \mu_{n+1}$-closed, then argue as usual by Easton's Lemma.
\end{proof} 

In a similar spirit, we state some more easy projection and absorption facts about $\U_{[n, \omega)}$ and
  $\S_n$ for use in Section \ref{GroupIV}.

\begin{lemma} \label{u_into_easton} 
  $\A * \U$ can be viewed as the projection of the product of $\A * \U_{[0, n)}$ and
    the Easton support product of the term forcing posets
    $\termspace^V(\A \restriction \alpha, \U \restriction \alpha)$
    for $\alpha \in \dom(\B_{[n, \omega)})$. 
\end{lemma}

We note that by  Lemma \ref{Easton-collapse-absorb}, the product of term forcing in
Lemma \ref{u_into_easton} may be 
absorbed into a suitable Easton collapse.

\begin{lemma} \label{s_into_easton} 
  $\A * \U * \S_n$ can be viewed as the projection of the product of $\A * \U$ and $\C_n$.
\end{lemma}

We note that $\C_n$ may be viewed as the $<\mu_{n+1}$-support
 product of the term forcing posets $\termspace^V(\A \restriction \alpha * \U \restriction \mu_{n+1}, \dot\Add(\mu_{n+1}, 1))$
for $\alpha \in (\mu_{n+1}, \mu_{n+2})$.    
As such, by Lemma \ref{Levy-collapse-absorb} $\C_n$ may be
absorbed into a suitable $<\mu_{n+1}$-closed Levy collapse.

For use later (notably in Sections \ref{moreprep} and \ref{GroupIV})
we record the fact that $\U_{[n, \omega)}$ and
  $\S_{[n, \omega)}$ have a very modest degree of closure in the models where they
    are defined. These results are surely not optimal (and in some special cases we will
    need and prove more closure) but are all we need for the purposes of
    Sections \ref{moreprep} and  \ref{GroupIV}. The argument is similar to but easier than the proofs
    of Fact \ref{Itay4.7} or  Lemma \ref{q0info}. 

\begin{lemma} \label{modestclosure}
      \item   $\U_{[n, \omega)}$ is $<\mu_0$-closed in $V[A * U_{[0, n)}]$.
      \item   $\S_{[n, \omega)}$ is $<\mu_0$-closed in $V[A * U]$.
\end{lemma}

\begin{proof}  We only prove the closure of $\U_n$ in $V[A * U_{[0, n)}]$, which is enough to illustrate
    the idea. Recall that the underlying set of $\U_n$ is $\B_n = \B \restriction (\mu_{n+1}, \mu_{n+2})$,
    and the ordering is defined in $V[A_{[0, n]} * U_{[0, n)}]$. Let $\eta < \mu_0$ and
    let $\vec u = (u_i)_{i < \eta}$ be a decreasing sequence in $\U_n$, where by Lemma \ref{portmanteaulemma}
    we have $\vec u \in V$. We assume without loss of generality that it is forced by
    $\A_{[0, n]} * \U_{[0, n)}$ that $\vec u$ is decreasing in $\U_n$.

    We will construct $b \in \B_n$ inductively, where $b = \bigcup_i \dom(u_i)$ is easily seen to be an
    Easton set, and arrange that it is forced by $\A_{[0, n]} * \U_{[0, n)}$ 
    that $b$ is a lower bound for $\vec u$ in $\U_n$.
    Suppose that $\mu_{n+1} \le \alpha < \mu_{n+2}$,
    and we have defined $b \restriction \alpha$ which is forced by $\A_{[0, n]} * \U_{[0, n)}$ 
    to be a lower bound for
    $\vec u \restriction \alpha$ in $\U_n \restriction \alpha$.
    We force with $\A \restriction \alpha * \U \restriction \alpha$
    below $(0, b \restriction \alpha)$ to obtain $F^A_\alpha * F^U_\alpha$. 
    
    Let $d_i = u_i(\alpha)[F^A_\alpha * F^U_\alpha]$.
    We claim that    $(d_i)_{i < \eta}$ forms a decreasing sequence
    in $\phi(\alpha)[F^A_\alpha * F^U_\alpha]$. To see this let $i < j < \eta$,
    and force to prolong $F^A_\alpha$ to $F^A$ which is $\A_{[0, n]}$-generic.
    By our hypothesis on $\vec u$,
    $V[F^A * F^A_\alpha \restriction \mu_{n+1}] \models u_j \le_{\U_n} u_i$.
    By the definition of $\U_n$, 
    there are conditions $p \in F^A_\alpha$ and $q \in F^U_\alpha \restriction \mu_{n+1}$
    such that $(p, q \frown u_j \restriction \alpha) \forces u_j(\alpha) \le u_i(\alpha)$.
    
    Since we forced below $(0, b \restriction \alpha)$,
    $(0, b \restriction \alpha) \in F^A_\alpha * F^U_\alpha$. By hypothesis 
    $V[F^A * F^A_\alpha \restriction \mu_{n+1}] \models b \restriction \alpha \le_{\U_n \restriction \alpha}
    u_j \restriction \alpha$, so extending $p$ and $q$ if necessary
    we may assume that
    $(p, q) \forces b \restriction \alpha \le u_j \restriction \alpha$.
    So $(p, q \frown u_j \restriction \alpha) \in F^A_\alpha * F^U_\alpha$
    and $d_j \le d_i$ as required.
   
    Now since $\phi(\alpha)$ is forced to be $< \alpha$-closed
    we may choose an  $\A \restriction \alpha * \U \restriction \alpha$-name $b(\alpha)$ such that
    $(0, b \restriction \alpha) \forces b(\alpha) \le \dot d_i$ for all $i$.
    Let $F = A_{[0, n]} * U_{[0, n-1)}$ be an arbitrary $\A_{[0, n]} * \U_{[0, n-1)}$-generic object.
   By induction $b \restriction \alpha \le u_i \restriction \alpha$ for all $i < \eta$ in the version
   of $\U_n$ computed in $V[F]$, and we will show that $b \restriction \alpha + 1 \le u_i \restriction \alpha + 1$
   in this poset. This is easy because for  any condition $(a, u) \in F$, $(a \restriction \alpha, u \cup b \restriction \alpha)$
   forces $b(\alpha) \le u_i(\alpha)$. 
\end{proof}

\subsection{Further analysis} \label{further}

As we saw in Section \ref{basicaus},
$\A * \U * \S$ is naturally a projection of $\prod_n (\A_n \times \B_n \times \C_n)$,
and this latter forcing preserves all cardinals  $\mu_n$ together with $\mu_\omega^+$.
To get more information 
we will use a style of analysis sometimes called ``tail forcing'', which is 
often useful in the setting of a product of $\omega$ many increasingly closed forcing posets.

Let $f$ and $g$ be elements of $\prod_n (\A_n \times \B_n \times \C_n)$.
We say that $f =_\mathrm{finite} g$ if and only if $f(n) = g(n)$ for all large $n$,
and $f \le_\mathrm{finite} g$ if and only if $f(n) \le g(n)$ for all large $n$.
Then $=_\mathrm{finite}$ is an equivalence relation on
$\prod_n (\A_n \times \B_n \times \C_n)$.

If we let
$\prod_n (\A_n \times \B_n \times \C_n)/\mathrm{finite}$ be the set of equivalence classes
then $\le_\mathrm{finite}$ naturally induces a partial  ordering on 
$\prod_n (\A_n \times \B_n \times \C_n)/\mathrm{finite}$, and it is easy to see that
$f \mapsto [f]_{\rm finite}$ is a projection from
$\prod_n (\A_n \times \B_n \times \C_n)$ to $\prod_n (\A_n \times \B_n \times \C_n)/\mathrm{finite}$.
If we define $=_\mathrm{finite}$ and $\le_\mathrm{finite}$ on 
$\prod_{n \ge m} (\A_n \times \B_n \times \C_n)$ in the natural way, then
easily $\prod_n (\A_n \times \B_n \times \C_n)/\mathrm{finite}$ is isomorphic to
$\prod_{n \ge m} (\A_n \times \B_n \times \C_n)/\mathrm{finite}$,
and $f \mapsto [f]_{\rm finite}$ is a projection from
$\prod_{n \ge m} (\A_n \times \B_n \times \C_n)$ to $\prod_{n \ge m} (\A_n \times \B_n \times \C_n)/\mathrm{finite}$,

Now we may represent $\prod_n (\A_n \times \B_n \times \C_n)$  as a two-step iteration $\E_0 * \E_1$, where
 $\E_0 = \prod_n (\A_n \times \B_n \times \C_n)/\mathrm{finite}$, and $\E_1$ is the set of elements of 
$\prod_n \A_n \times \B_n \times \C_n$ whose classes modulo finite are in $E_0$, where the ordering of $\E_1$ is the the ordering
inherited from $\prod_n \A_n \times \B_n \times \C_n$.

\begin{claim} $\E_0$ is $\mu_m$-strategically closed for every $m < \omega$,
  so in particular it is $\mu_\omega$-distributive.
\end{claim}

\begin{proof} 
  By the discussion above, $\E_0$ is isomorphic to the projection via $f \mapsto [f]_\mathrm{finite}$
  of the $<\mu_m$-closed poset $\prod_{n \ge m} (\A_n \times \B_n \times \C_n)$.
\end{proof}

\begin{claim} $\E_1$ is $\mu_\omega^+$-cc in $V[E_0]$.
\end{claim}

\begin{proof}
   We will show the stronger assertion that $\E_1$ is the union of $\mu_\omega$ many filters in $V[E_0 * E_1]$. 
   Let $R$ be the $\prod_n (\A_n \times \B_n \times \C_n)$-generic object added by $\E_0 * \E_1$,
   so that $V[E_0 * E_1] = V[R]$,  $E_0$ is the set of equivalence classes of elements of $R$ modulo finite,
   $\E_1$ is the set of conditions which are equal mod finite to some element of $R$, and $E_1 = R$. 
   The key point is that
   $\mu_\omega^+$ is still an uncountable regular cardinal in $V[R]$.

   Now we work in $V[R]$. If $p \in \E_1$ then there are $n < \omega$ and $r \in R$ such that
   $p \restriction [n, \omega) = r \restriction [n, \omega)$. For all $n$ and
   $q \in \prod_{i < n} (\A_i \times \B_i \times \C_i)$, let
   $F_{n, q}$ be the set of $p \in \E_1$ such that $p \restriction n = q$ and
   $p \restriction [n, \omega) = r \restriction [n, \omega)$ for some $r \in R$.        
   It is easy to see that $F_{n, q}$ is a filter, $\E_1 = \bigcup_{n, q} F_{n, q}$
   and there are $\mu_\omega$ possibilities for $(n, q)$.        
 \end{proof} 

   We have proved: 
\begin{lemma} \label{tailforcinglemma}
     $V[A * U * S] \subseteq V[E_0 * E_1]$, where $\E_0$ is $\mu_\omega$-distributive in $V$ and
  $\E_1$ is $\mu_\omega^+$-cc in $V[E_0]$.
\end{lemma}

\begin{remark} \label{morobust}
  Since $\E_0$ is $\mu_m$-strategically closed for all $m$,  the distributivity of $\E_0$ is quite robust in mild forcing extensions of $V$.
  The argument we gave for the chain condition of $\E_1$ shows that
    the chain condition of $\E_1$ is also robust in mild forcing extensions of $V[E_0]$.
\end{remark}

\subsection{Modifying the forcing} \label{modify}

In the sequel we will need to use some modified forms of $\A * \U * {\S}$. The main modifications
will be:
\begin{itemize}

\item We sometimes choose the Cohen forcing $\A_n$ from an inner model $\bar V$ (which may depend on $n$), that is we
  set $\A_n = \Add^{\bar V}(\mu_n, \mu_{n+2})$.
  When we do this we will make sure to arrange that
  $\A_{[m, n]}$ is $<\mu_m$-distributive and $\mu_{n+1}$-Knaster for $m \le n < \omega$.

\item We sometimes weaken the assumptions on the cardinals $\mu_n$ for $n \ge 2$ and the function $\phi$.
  The $\mu_n$'s will still be supercompact but may not be indestructibly supercompact, and (relatedly) the
  function $\phi \restriction \mu_n$ may only be a Laver function rather than an indestructible Laver function.
  In practice there will typically be an inner model $V'$ such that $V$ is a small generic extension of $V'$, and
  $\phi$ is obtained from an indestructible Laver function in $V'$ using Lemma \ref{laversmallfact}.
  
\end{itemize}

  With these modifications the closure assertions from Facts \ref{Itay4.7},  \ref{Itay4.15},  and \ref{Itay4.30} 
  will remain true, since they only use chain condition and distributivity properties on the $\A$-coordinate.
   Most of the conclusions of Lemmas \ref{portmanteaulemma} and \ref{rightafterportmanteau} remain true, the only difference
  is that now $\A \restriction [\mu_{n+2}, \mu_\omega)$ and $\P_1$ are merely $< \mu_{n+1}$-distributive.
  The analysis from Section \ref{further} needs  to be slightly modified but the conclusion is the same: the modified version of
  $\A * \U * {\S}$ embeds into a two-step iteration where the first step has a robust
  form of $\mu_\omega$-distributivity, and the second step has a robust form of $\mu_\omega^+$-cc.

\begin{remark} \label{twostepaus}
  At certain points in the main construction (see Sections \ref{successiveprikry} and \ref{firstprikry}) we will
  start with a sequence of cardinal parameters $\mu_0, \mu_1, \mu_2 \ldots$, force with
  $\A_0 * \U_0 * \S_0$ (in Section \ref{successiveprikry}) or $\A_0 * \U_0 * \FL$
  for some preparation forcing $\FL$ (in Section \ref{firstprikry}),
  and then work over the extension  to define and force with $\A_{[1, \omega)} * \U_{[1, \omega)} * \S_{[1, \omega)}$
  (in Section \ref{successiveprikry}) or $\A_{[1, \omega)} * \U_{[1, \omega)} * \S$ (in Section \ref{firstprikry}).
  The resulting iteration is broadly similar to the  $\A * \U * \S$ construction defined from $\mu_0, \mu_1, \mu_2 \ldots$
  but is not equivalent: we will handle this situation by analysing the two parts of the construction
  seperately. 
  In the sequel we will call this kind of iteration a {\em two-phase $\A * \U * \S$ construction.}
  Since $\A_0$ forces $2^{\mu_0} = \mu_2$ it will be important to define $\A_1$ in an inner model, so that
  $\A_1$ does not collapse $\mu_2$.
\end{remark} 

\begin{remark}
Readers of this paper and \cite{NeemanUpto} will note a limited family resemblance between $\A * \U * \S$ and the constructions
of Abraham \cite{Abraham} and Cummings and Foreman \cite{CummingsForeman}, which also involve forcing posets with
an ``add coordinate'', a ``collapse coordinate'' and an ``indestructibility coordinate''. The key differences are that in
those earlier papers the supercompact cardinals are not assumed to be indestructible, all three coordinates are iterations,
and the ``indestructibility coordinate'' comes last and uses ordinary ground model Laver functions to guess names.
\end{remark}

\section{Indestructibility results}

In the proof of Theorem \ref{mainthm} we will produce a model which combines many different instances
of the construction of Section \ref{abcu}. Roughly speaking the double successor cardinals below $\aleph_{\omega^2}$
in the final model will be grouped into blocks of length $\omega$, where cardinals
in each block will be handled by an instance of that construction. Unfortunately in each block
there is interference caused by the instances that handle the neighbouring blocks. We will deal with
some of this interference by proving  general indestructibility results
(Lemmas \ref{indestructible} and \ref{indestructible2} below) stating that instances of the
tree property produced by the construction of Section \ref{abcu} are somewhat robust under further mild forcing.
All the ideas needed for the indestructibility results are already present in \cite{NeemanUpto},
 we just need some small adjustments to the proofs. 

 In some cases we would like to use Lemma \ref{indestructible} in situations where the hypotheses
 do not quite apply, and this issue will be addressed by going to a further generic extension where
 the hypotheses do apply, and using a mutual genericity argument to finish. See Remarks \ref{mutualgenericity1} and \ref{mutualgenericity2} 
 following Lemma \ref{indestructible} for more on this. Of course we could have incorporated this
 idea into the statement and proof of Lemma \ref{indestructible}, at the cost of further complicating
 the statement and the proof.

 Let $\A * \U * \S$ be a forcing poset of the type described in Section \ref{abcu}, but allowing
 for the possibility that some of the posets $\A_k$ may be chosen in submodels as discussed in
 Section \ref{modify}. 
The poset $\A * \U * \S$ will be constructed in a universe $\Vd$: this notation is perhaps a bit cumbrous
but makes it easier to specify which universe is to play the role of $\Vd$ in the sequel.
Let $n < \omega$, with a view to showing that the tree property holds at $\mu_{n+2}$ in a wide class of generic extensions
of $\Vd[A * U * S]$, and make the following assumptions:
\begin{itemize}
\item There is an inner model $\Vi$ of $\Vd$, such that $\Vd$ is a generic extension $\Vi[X]$
    and $\Vi \models \hbox{``$\X$ is $\omega$-distributive  and $\mu_1$-cc with $\vert \X \vert \le \mu_1$''}$. 
\item In $\Vi$,  $(\mu_i)_{i < \omega}$ is an increasing sequence of regular cardinals such that
  $\mu_i$ is indestructibly supercompact for $i \ge 2$, and there is a universal indestructible
  Laver function $\psi$ defined up to $\mu_\omega$. The cardinal $\mu_1$ need not be a large cardinal,
  and in fact is often the successor of a singular cardinal. 
\item The Laver function $\phi$ used to define $\U$ is obtained from the universal indestructible Laver function
  $\psi$ in $\Vi$ using Lemma \ref{laversmallfact}, that is $\phi(\alpha) = \psi(\alpha)[X]$
  whenever $\psi(\alpha)$ is an $\X$-name in $\Vi$.
  In $\Vd$, $\phi$ is a universal Laver function on the interval $(\mu_1, \mu_\omega)$, where $\mu_\omega = \sup_{i < \omega} \mu_i$.
\item Each of the posets $\A_k$ may be defined in some inner model of $\Vd$, with the
  constraints that:
\begin{itemize}
\item   $\A_k$ is defined in $\Vi$ for $k \ge n + 2$.
\item   $\prod_{i \le m} \A_i$ is $\mu_{m+1}$-Knaster in $\Vd$ for all $m$.
\item   $\prod_{i \ge m} \A_i$ is $<\mu_m$-distributive in $\Vd$ for all $m$.
\end{itemize}
\end{itemize}

\begin{remark} With an eye to future applications, these hypotheses are slightly more
  general than is needed for our purposes in this paper.
  \end{remark}

  From our hypotheses $\A_{[n+2, \omega)} =(\prod_{n+2 \le i < \omega} \Add(\mu_i, [\mu_{i+1},\mu_{i+2})))^{\Vi}$,
    so that $\A_{[n+2, \omega)}$ is defined and $<\mu_{n+2}$-directed closed in $\Vi$.
  For $k \le n + 1$, $\A_k$ will most often be defined
  in some model intermediate between $\Vi$ and $\Vd$, and in this case
  Lemma \ref{choosea} below will handle most of the work of checking the chain condition
  and distributivity of products of the $\A_k$'s.

  \begin{lemma} \label{choosea}   Suppose that $\A_k = \Add^{\Vint{k}}(\mu_k, [\mu_{k+1}, \mu_{k+2}))$
    for $1 \le k \le n+1$,
      where $\Vint{k}$ is intermediate between $\Vi$ and $\Vd$.
      Then
      $\A_{[1, m]} = \prod_{1 \le i \le m} \A_i$ is $\mu_{m+1}$-Knaster in $\Vd$ and
      $\A_{[m, \omega)} = \prod_{i \ge m} \A_i$ is $<\mu_m$-distributive in $\Vd$
        for all $m \ge 1$.
\end{lemma}

\begin{proof} We set $\Vint{k} = \Vi$ for $k \ge n+2$, so that
  $\A_k = \Add^{\Vint{k}}(\mu_k, [\mu_{k+1}, \mu_{k+2}))$ for $k \ge 1$. 
  As $\mu_k$ is inaccessible in $\Vd$ for $k \ge 2$, 
  $\Vint{k} \models \mbox{``$\forall \eta < \mu_{k+1} \; \eta^{<\mu_k} < \mu_{k+1}$''}$ for all $k \ge 1$.

   Let $\Vint{k} = \Vi[X_k]$ for a forcing poset $\X_k \in \Vi$, and let
  $\A_k^* = \termspace^{\Vi}(\X_k, \dot \A_k)$, so that $\A_k^*$ is $<\mu_k$-closed in $\Vi$ by Lemma \ref{standardtermforcinglemma}. 
  For $m \ge n+2$, $\A_{[m, \omega)}$ is $<\mu_m$-closed in $\Vi$, so it is
    $<\mu_m$-distributive in $\Vd$ by Easton's lemma.

    For $1 \le m \le n+1$, we may write
    $\Vd[A_{[m, \omega)}] = \Vi[(X * A_{[m, n+2)}) \times A_{[n+2, \omega)}]$,
    and by a suitable quotient to term forcing we may
    extend to obtain a generic extension
    $\Vi[X \times \prod_{m \le i < n+2} A^*_i \times A_{[n+2, \omega)}]$.
    Since  $\prod_{m \le i < n+2} \A^*_i \times \A_{[n+2, \omega)}$ is
    $<\mu_m$-closed in $\Vi$, by Easton's lemma it is 
    $<\mu_m$-distributive in $\Vd$, so that easily
    $\A_{[m, \omega)}$ is $<\mu_m$-distributive in $\Vd$.

    It is easy to see that $\Vd$ is a $(\mu_1, \mu_2)$-good extension of $\Vint{k}$, so
    by Lemma \ref{Cohenrobust} $\A_1$ is $\mu_2$-Knaster in $\Vd$. Now we show by
    induction on $m$ that $\A_{[1, m]}$ is $\mu_{m+1}$-Knaster in $\Vd$:
    if $\A_{[1, m]}$ is $\mu_{m+1}$-Knaster in $\Vd$ then
    $\Vd[A_{[1, m]}]$ is a $(\mu_{m+1}, \mu_{m+2})$-good extension of $\Vint{m+1}$,
    so that $\A_{m+1}$ is $\mu_{m+2}$-Knaster in $\Vd[A_{[1, m]}]$ and hence $\A_{[1,m+1]}$ is $\mu_{m+2}$-Knaster. 
\end{proof}

\begin{remark} Lemma \ref{choosea} leaves us only with the problem of showing that
  $\A_{[0, \omega]}$ is $<\mu_0$-distributive and $\A_{[0, m]}$ is $\mu_{m+1}$-cc.
\end{remark}

To make the hypotheses of the forthcoming Lemma \ref{indestructible} more digestible, we use some
notational conventions:
\begin{itemize}
\item We will show that the tree property at $\mu_{n+2}$ holds in generic extensions of  $\Vd[A * U * S]$ 
  by products of posets that can be written in the form  $\D^{\rm small} \times \D^0 \times \D^1 \times \D^2 \times \D^3$, where the factors satisfy
  some hypotheses to be listed later. We write this product $\D^{s, 0, 1, 2, 3}$,
  and denote subproducts and generic objects for subproducts in the natural way.
\item $W$ is $\Vd[A * U * S][D^{s,0,1,2,3}]$.  
\item  $j$ is any embedding witnessing the $\chi$-supercompactness of $\mu_{n+2}$
  in  $\Vd[H]$ where  $H = A_{[n+2, \omega)} \times D^2$
    and  $\chi = \max(\mu_{\omega}, \vert \D^{s, 0, 1, 2, 3} \vert)^+$.
    Our hypotheses will ensure that $\mu_{n+2}$ is supercompact in $\Vd[H]$, so that such embeddings
    $j$ will exist. Note that $\Vd[H] \subseteq W$.
\item If $\Q \in \Vd$ is a $\mu_{n+2}$-cc poset with $\vert \Q \vert \le \chi$ which remains $\mu_{n+2}$-cc in $\Vd[H]$,
    $j$ is an embedding as above (so that in $\Vd[H]$, $j \restriction \Q$ is a complete embedding of $\Q$ into $j(\Q)$),
  and $Q$ is $\Q$-generic over $\Vd[H]$,  then $j(\Q)/j[Q]$ is the natural poset defined in $\Vd[H][Q]$ to produce a
  $j(\Q)$-generic object $\hat Q$ with $j[Q] \subseteq \hat Q$.
\end{itemize}
We note that in the proof of Lemma \ref{indestructible} we will construct and lift a highly specific embedding $j$,
which is not known in advance and depends on the inputs to the Lemma.

\begin{remark} The posets $\D^{\rm small}, \D^0$, $\D^1$, $\D^2$, $\D^3$ are
  enumerated roughly in order of increasing distributivity. 
  They appear in a different order in the hypotheses of Lemma \ref{indestructible} because the hypotheses
  about $\D^0$ and $\D^1$ mention $\D^2$ and $\D^3$, and the hypothesis about $\D^0$ mentions $\D^1$.
\end{remark}   
  

\begin{lemma} \label{indestructible}
  With the hypotheses on $n$, $\Vi$, $\Vd$ and $\A * \U * \S$ as above,
  let $\D^{\rm small}$, $\D^0$, $\D^1$, $\D^2$, $\D^3$ be forcing posets such that,
  setting $W = \Vd[A * U * S][D^{s,0,1,2,3}]$ and $H = A_{[n+2, \omega)} \times D^2$:

  \begin{enumerate}

  \item \label{indhyp1}  $\mu_{n+1}$ and $\mu_{n+2}$ are regular cardinals in $W$. 

  \item \label{indhyp2} $\A_n$ is $\mu_{n+1}$-Knaster in $\Vd[H]$, and $\A_{n+1}$ is  $\mu_{n+2}$-Knaster in $\Vd[H]$.  
    
  \item \label{indhyp3}  $\D^2 \in \Vi$, and $\Vi \models \mbox{``$\D^2$ is $<\mu_{n+2}$-directed closed''}$.

\item \label{indhyp8}  $\D^3 \in \Vd[A \restriction \mu_{n+2} * U \restriction {\mu_{n+2}}]$ and 
  $\Vd[A \restriction \mu_{n+2} * U \restriction {\mu_{n+2}}] \models
  \mbox{``$\D^3$ is  $<\mu_{n+2}$-directed closed''}$.

  \item \label{indhyp4}  $\D^1 \in \Vd$.
    \begin{enumerate}
    \item \label{indhyp4a} $\D^1$ is $\mu_{n+2}$-Knaster in $\Vd[H]$.                      
    \item \label{indhyp4b} $\D^1$ is $<\mu_{n+1}$-distributive in $\Vd[A * U *S][D^{2,3}]$. 
    \end{enumerate}
    
\item \label{indhyp5}
  For any $j$ which is the unique lift
    \footnote{Note that by the indestructibility of $\mu_{n+2}$ in $\Vi$ and the hypotheses
    on $\A_{[n+2, \omega)}$ and $\D^2$, $\mu_{n+2}$ is $\chi$-supercompact in $\Vi[H]$. Since $\Vd[H] = \Vi[H][X]$
    and $X$ is generic for forcing of cardinality at most $\mu_1$, any embedding witnessing
    $\chi$-supercompactness for $\mu_{n+2}$ in $\Vi[H]$ lifts uniquely in a trivial fashion to an
    embedding witnessing $\chi$-supercompactness for $\mu_{n+2}$ in $\Vd[H]$.}
  to $\Vd[H]$ of an embedding witnessing
    the $\chi$-supercompactness of $\mu_{n+2}$ in $\Vi[H]$, if
    $\P_{2b} = j(\A_{n+1} \times \D^1)/j[A_{n+1} \times D^1]$ then: 
    \begin{enumerate}
    \item  \label{indhyp5a} $\P_{2b}$ is $\mu_{n+2}$-Knaster in $W$.            
    \item  \label{indhyp5b} $\P_{2b}$ is $<\mu_{n+1}$-distributive in $W$.      
    \end{enumerate}

  \item \label{indhyp6}  $\D^0 \in \Vd$, and $\D^0$ is $\mu_{n+1}$-Knaster in $\Vd[A * U * S][D^{1,2,3}][P_{2b}]$. 

  \item \label{indhyp6.5}  
    For any $W'$ which is 
    an extension of $W[P_{2 b}]$ by a forcing which is $<\mu_{n+1}$-closed in $\Vd[A * U * S \restriction [\mu_{n+1}, \mu_\omega)][D^2]$,  
    and any $j$ as in Hypothesis \ref{indhyp5},
    if $\P_{2a} = j(\A_n \times \D^0)/j[A_n \times D^0]$ then 
    $\P_{2a}$ is $\mu_{n+1}$-Knaster in $W'$.  
    
\item \label{indhyp7}
  $\D^{\rm small} \in \Vd[A * U * S]$, and $\Vd[A * U * S] \models \vert \D^{\rm small} \vert \le \mu_n$.

\end{enumerate}

  Then $W \models \mbox{``$\mu_{n+2}$ has the tree property''}$.
\end{lemma}

Before proving Lemma \ref{indestructible}, we make some remarks about its hypotheses
and show that these hypotheses entail some additional properties.

\begin{remark}

\leavevmode

\begin{itemize}

\item 
  In applications $\D^0$ will often be a Cohen poset adding subsets to $\mu_n$ defined in some inner model of $\Vd$,
  and similarly $\D^1$ will often be a Cohen poset adding subsets to $\mu_{n+1}$ defined in some inner model of $\Vd$.

\item $\P_{2a}$ and $\P_{2b}$ are so named because they will be used
  successively in Step 2 of the construction for Lemma \ref{indestructible}.
  $\P_{2a}$ and $\P_{2b}$ are defined respectively in the submodels $\Vd[H][A_n \times D^0]$
  and $\Vd[H][A_{n+1} \times D^1]$ of the model $W$.

\item In connection with $\D^{\rm small}$, we recall that $\mu_{n+1}$ is the successor of $\mu_n$ in  $\Vd[A * U * S]$
  for $n > 0$. In the intended applications it is often the case that $\D^{\rm small}$ is defined in
  a submodel of $\Vd[A * U * S]$ where $\mu_n < \vert \D^{\rm small} \vert < \mu_{n+1}$.

\item Some cardinals (notably $\mu_n$) may be collapsed in $W$, for example we could set $\D^{\rm small} = \Coll(\omega, \mu_n)$.

\item Hypotheses \ref{indhyp2},
  \ref{indhyp4a} and \ref{indhyp6} jointly imply that
  both $\A_n \times \D^0$ and $\A_{n+1} \times \D^1$ are $\mu_{n+2}$-cc
  in $\Vd[H]$.
  It follows that the posets $\P_{2 a}$ and $\P_{2b}$
  are guaranteed to be well-defined.   

\item It is implicit in the hypotheses that $j(\A_n)/j[A_n]$ and $j(\A_{n+1})/j[A_{n+1}]$ have rather robust
  chain condition and distributivity properties.

\end{itemize}
  
\end{remark}

The following auxiliary lemma, which we will use in the proof
of Lemma \ref{indestructible}, provides a good example of the use of term forcing and ``quotient to term forcing''
   to analyse complicated generic extensions.

   \begin{lemma} \label{Wgoodextn} Under the hypotheses of Lemma \ref{indestructible}:
\begin{itemize}     
\item  $\Vd[A * U * S] \models \mbox{``$\D^{2,3}$ is $\mu_{n+1}$-distributive''}$.
\item  $\Vd[A * U * S] \models \mbox{``$\D^{1,2,3}$ is $<\mu_{n+1}$-distributive''}$.
\end{itemize}  
\end{lemma}

\begin{proof}
  We begin by analysing the model $\Vd[A * U * S][D^{2,3}] = \Vd[A * U * S][D^2 \times D^3]$.
  Recalling that $\Vd = \Vi[X]$ and that $\D^2 \in \Vi$, this model is
  $\Vi[(X * A * U * S * D^3) \times D^2]$. 
  Since $\D^3 \in \Vd[A \restriction \mu_{n+2} * U \restriction \mu_{n+2}]$, we may form in
  $\Vi$ the term forcing $\T^3 = \termspace^{\Vi}(\X * \A \restriction \mu_{n+2} * \U \restriction \mu_{n+2}, \dot\D^3)$.
  By hypothesis \ref{indhyp8} of
  Lemma \ref{indestructible},
  $\D^3$ is $< \mu_{n+2}$-directed closed in $\Vd[A \restriction \mu_{n+2} * U \restriction \mu_{n+2}]$,
 and  it follows from Lemma \ref{standardtermforcinglemma} that (just like $\D^2$) the poset $\T^3$ is $<\mu_{n+2}$-directed closed in $\Vi$.

  Now we force over $\Vd[A * U * S][D^2 \times D^3]$
  with the ``quotient to term''  forcing $\T^3/(A \restriction \mu_{n+2} * U \restriction \mu_{n+2}) * D^3$,
  which is computed in the submodel $\Vd[A \restriction \mu_{n+2} * U \restriction \mu_{n+2} * D^3]$.
  By  Lemma \ref{rearrangeqtot} we obtain $T^3$ such that
  $(X * A \restriction \mu_{n+2} * U \restriction \mu_{n+2}) \times T^3$ induces
  $X * A \restriction \mu_{n+2} * U \restriction \mu_{n+2} * D^3$,
  and $D^2 \times T^3$ is generic over $\Vd[A * U * S]$ for $\D^2 \times \T^3$. 
  In particular $\Vd[A * U * S][D^2 \times D^3] \subseteq \Vd[A * U * S][D^2 \times T^3]$.

  Recall from Lemma \ref{portmanteaulemma} that in $\Vd$ we may write $\A * \U * \S$ as the projection
  of $\P^*_0 \times {\bar \P}^*_1 \times {\P}^*_1$,
  where $\P^*_0 = \A_{[0, n+1]} * \U_{[0,n]} * \S_{[0,n]}$,
   ${\bar \P}^*_1 = \A_{[n+2,\omega)}$, and ${\P}^*_1 = \B_{[n+1, \omega)} \times \C_{[n+1, \omega)}$.  
  Under our current hypotheses $\P^*_0$ is defined and $\mu_{n+2}$-cc in $\Vd$,
  ${\bar \P}^*_1$ is defined and $<\mu_{n+2}$-directed closed
  in $\Vi$, while ${\P}^*_1$ is  defined and $<\mu_{n+2}$-directed closed in $\Vd$. 

  Let $\T^* = \termspace^{\Vi}(\X, {\P}^*_1)$, so that $\T^*$ is defined and $<\mu_{n+2}$-directed closed in $\Vi$.
  With another round of quotient to term forcing and another appeal to Lemma \ref{rearrangeqtot},
  we produce a generic extension
  $\Vi[(X * P^*_0) \times A_{[n+2,\omega)} \times T^* \times D^2 \times T^3] \supseteq V[A * U * S][D^2 \times T^3]$.  
  Now $A_{[n+2,\omega)} \times T^* \times D^2 \times T^3$ is generic over $\Vi$ for $<\mu_{n+2}$-closed
  forcing and $X * P^*_0$ is generic for $\mu_{n+2}$-cc forcing.
  Appealing to Easton's Lemma, every $\mu_{n+1}$-sequence of ordinals in $\Vd[A * U * S][D^2 \times D^3]$
  is in $\Vi[X * P^*_0] = \Vd[A_{[0, n+1]} * U_{[0,n]} * S_{[0,n]}]$, in particular
  $\D^2 \times \D^3$ is $\mu_{n+1}$-distributive in $\Vd[A * U * S]$.   
  By hypothesis \ref{indhyp4b} of Lemma \ref{indestructible},
  $\D^1$ is $<\mu_{n+1}$-distributive  in $\Vd[A * U * S][D^2 \times D^3]$,
  so that $\D^1 \times \D^2 \times \D^3$ is $<\mu_{n+1}$-distributive  in $\Vd[A * U * S]$ as required.   

\end{proof}

\begin{remark} \label{benignhelix} For use in Lemma \ref{meetingindyreqs} below, we note that
  no hypotheses involving either $\D^0$ or $\D^1$ were used to prove 
  the distributivity of $\D^{2, 3}$, and that for the distributivity of
  $\D^{1,2,3}$ we used only that $\D^1$ is $<\mu_{n+1}$-distributive  in $\Vd[A * U * S][D^2 \times D^3]$.
\end{remark}

With these preliminaries out of the way, we are now ready to prove Lemma \ref{indestructible}.
      
\begin{proof} [Proof of Lemma \ref{indestructible}]
  Recall that $W = \Vd[A * U * S][D^{s,0,1,2,3}]$.
 We will  show that
 the cardinal $\mu_{n+2}$ has the tree property in $W$. This involves constructing a generic embedding with domain
 $W$ and critical point $\mu_{n+2}$,
 and then arguing that the forcing which adds the embedding will not add a branch to a $\mu_{n+2}$-tree. 
 The forcing to add the embedding will be constructed in several steps.
 
 Recall that $\mu_{n+2}$ is indestructibly supercompact in $\Vi$, and $H = A_{[n +2, \omega)} \times D^2$
   which is generic for $<\mu_{n+2}$-directed closed forcing in $\Vi$,
 so that  $\mu_{n+2}$ is supercompact in $\Vi[H]$. We will eventually choose an embedding
     $j$ defined in $\Vi[H]$ witnessing that  $\mu_{n+2}$ is sufficiently supercompact, and having some other desirable properties,
 but we defer this choice for the moment. When we choose $j$ it will trivially lift onto $\Vd[H]$, because
  $\Vd$ is a small generic extension of $\Vi$. 
       
      To help motivate the lifting construction below, we list relevant generic objects which must be added to
      $\Vd[H]$ to obtain $W$. In the following list the ``small'' group consists of generic objects for posets of size less than
      $\mu_{n+2}$ where the lifting is essentially trivial. 

      \begin{itemize}
      \item Small:  $A \restriction \mu_{n+1}$, $U \restriction \mu_{n+1}$, $S \restriction \mu_{n+1}$ and $D^{\rm small}$.
      \item Large or potentially large: $A \restriction [\mu_{n+1}, \mu_{n+3})$,
       $U \restriction [\mu_{n+1}, \mu_\omega)$,  $S \restriction [\mu_{n+1}, \mu_\omega)$ and the
       posets $D^i$ for $i =0, 1, 3$.
      \end{itemize}
      
      At several steps in the following construction we record some closure and chain condition information about the posets
      which appear in that step. This information will be used in the proof of the tree property.

 \begin{itemize}

 \item
   Step 1a: (``Remove dependence of $U \restriction [\mu_{n+2}, \mu_\omega)$  on $A \restriction [\mu_{n+2}, \mu_\omega)$'').
   Recall that $\U \restriction [\mu_{n+2}, \mu_\omega) = {\B}^{+F'}\restriction[\mu_{n+2}, \mu_\omega)$,
     where $F' = A * (U \restriction \mu_{n+2})$.
    Let $F = A \restriction \mu_{n+2} *  U \restriction \mu_{n+2}$, and note that $F \subseteq F'$.
     Let $\P_{1a}$ be the ``quotient to term'' forcing which adds a filter $B_{1 a}$
      on $\B_{1 a} = {\B}^{+F}\restriction[\mu_{n+2}, \mu_\omega)$, 
       so that $B_{1 a}$ induces  $U \restriction [\mu_{n+2}, \mu_\omega)$ as in Fact \ref{Itay4.5}. We force over $W$ with $\P_{1a}$.
         We see that
         \[
         \Vd[A * U][P_{1a}] = \Vd[F][A \restriction [\mu_{n+2}, \mu_\omega) \times B_{1 a}] = \Vd[A * U \restriction \mu_{n+2}][B_{1a}]
         \]
         and arguing as in
         \footnote{$\B_{1 a}$ is not literally the termspace forcing
         $\termspace^{\Vd[F]}(\A \restriction [\mu_{n+2}, \mu_\omega), \U \restriction [\mu_{n+2}, \mu_\omega))$,
          but it does add an $\A \restriction [\mu_{n+2}, \mu_\omega)$-name for a
            filter which is $\U \restriction [\mu_{n+2}, \mu_\omega))$-generic over $\Vd[A * U \restriction \mu_{n+2}]$
         and this is sufficient. See Remark \ref{rearrangeqtotremark}.}
         Lemma \ref{rearrangeqtot}
         $S * D^{s,0, 1,2,3}$ is $\S * \D^{s,0, 1,2,3}$-generic over $\Vd[A * U \restriction \mu_{n+2}][B_{1a}]$.
         We have
         \[
         W[P_{1a}] = \Vd[A * U \restriction \mu_{n+2}][B_{1a}][S * D^{s,0, 1,2,3}].
         \]

    Now  $\B_{1 a} \in \Vd[F]$, and it follows 
      from Fact \ref{Itay4.7} that $\B_{1 a}$ is $<\mu_{n+2}$-directed closed in this model. On the other hand $\P_{1 a} \in \Vd[A * U]$,
      and appealing to Fact \ref{Itay4.30} it is actually $<\mu_{n+1}$-closed in the larger model
      $\Vd[A * U * S \restriction [\mu_{n+1}, \mu_\omega)]$.
        This closure still holds in the further extension $\Vd[A * U * S \restriction [\mu_{n+1}, \mu_\omega)][D^{1,2,3}]$,
     since by Lemma \ref{Wgoodextn} $\D^{1,2,3}$ is $<\mu_{n+1}$-distributive in $\Vd[A * U * S]$.

   \item Step 1b: (``Remove dependence of $S \restriction [\mu_{n+2}, \mu_\omega)$  on
     $A \restriction [\mu_{n+2}, \mu_\omega) * U \restriction [\mu_{n+2}, \mu_\omega)$'').
 Recall that  $\S \restriction [\mu_{n+2}, \mu_\omega) = {\C}^{+F''}\restriction[\mu_{n+2}, \mu_\omega)$  where $F'' = A * U$, 
   and $F = A \restriction \mu_{n+2} *  U \restriction \mu_{n+2} \subseteq F''$.
  Let $\P_{1b}$ be the quotient forcing which adds a filter $C_{1b}$  
  on $\C_{1b} = {\C}^{+F}\restriction[\mu_{n+2}, \mu_\omega)$, inducing $S \restriction [\mu_{n+2}, \mu_\omega)$.

   We force over $W[P_{1a}]$ with $\P_{1b}$, and let $\P_1 = \P_{1a} \times \P_{1b}$.   
   As in Step 1a,
   \[
   W[P_1] = \Vd[A * U \restriction \mu_{n+2} * S \restriction \mu_{n+2}][B_{1a}][C_{1b}][D^{s,0, 1,2,3}],
   \]
   and $D^{s,0, 1,2,3}$ continues to be $\D^{s,0, 1,2,3}$-generic over the slightly larger model
  $\Vd[A * U \restriction \mu_{n+2} * S \restriction \mu_{n+2}][B_{1a}][C_{1b}]$.

    Similarly to step 1a,  $\C_{1 b} \in \Vd[F]$, and it follows from
      Fact \ref{Itay4.15} that $\C_{1 b}$ is $<\mu_{n+2}$-directed closed in this model.   
      On the other hand $\P_{1 b} \in \Vd[A * U * S \restriction [\mu_{n+2}, \mu_\omega)]$,
      and appealing to Fact \ref{Itay4.30} it is
      actually $<\mu_{n+1}$-closed in the larger model
      $V[A * U * S \restriction [\mu_{n+1}, \mu_\omega)]$.
   As in Step 1a, this closure still holds in the further extension 
   $\Vd[A * U * S \restriction [\mu_{n+1}, \mu_\omega)][D^{1,2,3}]$.     
    
\item  Choosing $j$: 
     Recall that we defined $\A * \U$ using a Laver function $\phi$ in $\Vd$, obtained
     from a universal indestructible Laver function $\psi$ in $\Vi$,
     setting $\phi(\alpha)=\psi(\alpha)[X]$ whenever $\psi(\alpha)$ is an $\X$-name. 
     Working in $\Vi[H]$ we choose $j$ such that
\begin{itemize}
\item $j \restriction ON$ is definable in $\Vi$.       
\item  $j$ witnesses $\mu_{n+2}$  is $\chi$-supercompact where $\chi  = \max(\mu_{\omega}, \vert \D^{0, 1, 2, 3} \vert)^+$.  
\item  The next point in $\dom(j(\psi))$ past $\mu_{n+2}$ is greater than $\chi$.
\item $j(\psi)(\mu_{n+2}) = \ddot Z$, where $\ddot Z$ is an $\X$-name in $\Vi$ for an
  $\A \restriction \mu_{n+2} * \U \restriction \mu_{n+2}$-name in $\Vd$ for $\B_{1a} \times \C_{1b} \times \D^3$.  
\end{itemize}

Since $\Vd$ is a generic extension of $\Vi$ by forcing of size at most
$\mu_1$, we trivially lift $j$ to obtain $j:\Vd[H] \rightarrow M_H$.
Note that  $j(\phi)(\mu_{n+2}) = \dot Z$, where $\dot Z$ is a
  $\A \restriction \mu_{n+2} * \U \restriction \mu_{n+2}$-name in $\Vd$ for $\B_{1a} \times \C_{1b} \times \D^3$.  

Since $\B_{1a} \times \C_{1b} \times \D^3 \in \Vd[A \restriction \mu_{n+2} * U \restriction \mu_{n+2}]$ and is
    $<\mu_{n+2}$-directed closed in this model, the choice of $j$ implies that
    $\mu_{n+2}$ is in the support of the $\U$-coordinate of $j(\A * \U)$ and the forcing
    which appears there is $\B_{1a} \times \C_{1b} \times \D^3$.  

\item Step 2a: (``Stretch $A_n \times D^0$'')

  Let $\P_{2 a} = j(\A_n \times \D^0)/j[A_n \times D_0]$. We force over $W[P_1]$ with $\P_{2 a}$ and add
  a $j(\A_n \times \D^0)$-generic object  ${\hat A}_n \times {\hat D}^0$ 
  such that $j[A_n \times D^0] \subseteq {\hat A}_n \times {\hat D}^0$.

    \item Step 2b: (``Stretch $A_{n +1} \times D^1$'')
  Let $\P_{2 b} = j(\A_{n+1} \times \D^1)/j[A_{n+1} \times D^1]$. We force over $W[P_1][P_{2a}]$ with $\P_{2 b}$ and add
  a $j(\A_{n+1} \times \D^1)$-generic object  ${\hat A}_{n+1} \times {\hat D}^1$ 
  such that $j[A_{n+1} \times D^1] \subseteq {\hat A}_{n+1} \times {\hat D}^1$.

  We let $\P_2 = \P_{2a} \times \P_{2b}$.
  Then
   \[
   W[P_1][P_2] = \Vd[\hat A * U \restriction \mu_{n+2} * S \restriction \mu_{n+2}][B_{1a}][C_{1b}][\hat D^{s,0, 1,2,3}],
   \]
   where $\hat A = A_{[0, n-1)} \times \hat A_n \times \hat A_{n+1} \times A_{[n+2, \omega)}$
   and $\hat D^{s,0, 1,2,3} = D^s \times \hat D^0 \times \hat D^1 \times D^2 \times D^3$.

    \item Step 3 (``Stretch the term forcing for $S \restriction [\mu_{n+1}, \mu_{n+2})$'')
      Recall that
      $\S \restriction [\mu_{n+1}, \mu_{n+2}) = \C^{+\bar F} \restriction [\mu_{n+1}, \mu_{n+2})$,
      where $\bar F = A \restriction \mu_{n+2} * U \restriction \mu_{n+1}$.
      Let $\P_3 = j(\C)^{+\bar F} \restriction [\mu_{n+2}, j(\mu_{n+2}))$.
      We force with $\P_3$ over $W[P_1][P_2]$.

        Note that $\P_3$ is defined in $\Vd[H][\bar F] \subseteq V[A * U \restriction \mu_{n+1}][D^2]$. 
        Modifying the proof of  \cite[Claim 4.31]{NeemanUpto} to account for $\D^2$,
        $\P_3$ is   $<\mu_{n+1}$-closed
        in $V[A * U *S \restriction [\mu_{n+1}, \mu_\omega)][D^2]$. 
          By distributivity, $\P_3$ retains this closure in the larger model
          $V[A * U * S \restriction [\mu_{n+1}, \mu_\omega)][D^{1,2,3}]$.

\end{itemize}

 We now extend $j:\Vd[H] \rightarrow M_H$ to a generic embedding with domain $W$, working (ultimately) in the extension
 $W' = W[P_{1, 2, 3}]$ where $P_{1, 2, 3}$ collects the generic objects we added in the steps above.

\begin{itemize}

\item Stage 1:
  Recall that $B_{1 a} \times C_{1 b}$ is the generic object added by $\P_1 = \P_{1 a} \times \P_{1b}$
  for $\B_{1 a} \times \C_{1 b}$, a poset which is  defined and is $<\mu_{n+2}$-directed closed
  in $\Vd[F]$, where  $F = A \restriction \mu_{n+2} * U \restriction \mu_{n+2}$.
  By our choice of $j$, $j(\A * \U * \S)$  has a ``$\U$-component''
  in which $\B_{1 a} \times \C_{1 b} \times \D^3$ appears at stage $\mu_{n+2}$.

  Modifying the proof of Fact \ref{Itay4.12} from \cite{NeemanUpto},
  we may lift $j$ onto $\Vd[A * U \restriction \mu_{n+2}][B_{1a} \times C_{1b}][D^{0,1,2,3}]$.
  We outline the modified proof, with a focus on where to find the compatible generic objects
 on the ``$j$-side''.
    
\begin{itemize}
\item $H = A_{[n+2, \omega)} \times D^2 \in \Vd[H]$,
so that if $\hat H = j(H) = j(A_{[n+2, \omega)}) \times j(D^2)$
then $\hat H \in M_H$ and $M_H=j(\Vd)[\hat H]$.
\item $j(A_{[0, n)}) = A_{[0, n)}$.
\item
$j(A_n \times D^0)$ is obtained by combining $A_n \times D^0$ and the generic object $P_{2a}$ for the ``stretching'' poset $\P_{2a}$.
\item 
$j(D^1 \times A_{n+1})$ is obtained by combining
 $D^1 \times A_{n+1}$ and the generic object $P_{2b}$ for the ``stretching'' poset $\P_{2b}$.
\item $j(U \restriction \mu_{n+2})$ is obtained by concatenating $U \restriction \mu_{n+2}$,
  $B_{1a} \times C_{1b} \times D^3$ (the generic object at $\mu_{n+2}$),
  and a generic object for $j(\U) \restriction (\mu_{n+2}, j(\mu_{n+2}))$
 which is constructed using closure under $\chi$-sequences.
\item $j(B_{1a} \times C_{1b} \times D^3)$ is constructed using closure under $\chi$-sequences and a master condition argument.   
\end{itemize}

\item  Stage 2:
  $S \restriction \mu_{n+1}$ is generic for forcing of size less than $\mu_{n + 2}$, 
  so we may trivially lift $j$
  onto $\Vd[A * U \restriction \mu_{n+2}][B_{1a} \times C_{1b}][D^{0,1,2,3}][S \restriction \mu_{n+1}]$.

\item  Stage 3:  As we noted above in the definition of $\P_3$,
      $\S \restriction [\mu_{n+1}, \mu_{n+2}) = \C^{+\bar F} \restriction [\mu_{n+1}, \mu_{n+2})$
      where $\bar F = A \restriction \mu_{n+2} * U \restriction \mu_{n+1}$.
      When we apply $j$ to $\S \restriction [\mu_{n+1}, \mu_{n+2})$ it is only
        the ``$A$-component'' which gets stretched: more precisely
        $j(A \restriction \mu_{n+2}) = A_{[0, n)} * \hat A_n$,        
        $j(\bar F) = (A_{[0, n)} * \hat A_n) * U \restriction \mu_{n+1}$ and
      $j(\S \restriction [\mu_{n+1}, \mu_{n+2})) = j(\C)^{+(A_{[0, n)} * \hat A_n) * U \restriction \mu_{n+1}} \restriction [\mu_{n+1}, j(\mu_{n+2}))$.

        Let $S_n = S \restriction [\mu_{n+1}, \mu_{n+2})$.
        Recall that  $\P_3 = j(\C)^{+\bar F} \restriction [\mu_{n+2}, j(\mu_{n+2}))$,
        and note that:
        \begin{itemize}
        \item $\bar F \subseteq j(\bar F) = (A_{[0, n)} * \hat A_n) * U \restriction \mu_{n+1}$.     
        \item $P_3$ is generic over $W[P_{1,2}]$ which contains all relevant generic objects.
        \item $j(\C)^{+ \bar F} \restriction [\mu_{n+1}, j(\mu_{n+2})) \simeq
        j(\C)^{+ \bar F} \restriction [\mu_{n+1}, \mu_{n+2}) \times \P_3 \simeq \S_n \times \P_3$.    
        \end{itemize}

        We may therefore form the upwards closure $\hat S_n$ of $S_n \times P_3$ in
        $j(\S \restriction [\mu_{n+1}, \mu_{n+2}))$, to produce $\hat S_n$
          such that $\hat S_n$ is generic
          for $j(\S \restriction [\mu_{n+1}, \mu_{n+2}))$.
            Since $\crit(j) = \mu_{n+2}$, and conditions in $\S_n$ have supports which are
            bounded subsets of $\mu_{n+2}$, 
         it is easy to see that $j[S_n] \subseteq \hat S_n$ and so we may lift $j$ onto 
         $\Vd[A * U \restriction \mu_{n+2}][B_{1a} \times C_{1b}][D^{0,1,2,3}][S \restriction \mu_{n+2}]$.

\item   Stage 4:
   Recall that $\P_{1a}$ added a filter $B_{1 a}$ on $\B_{1 a} = {\B}^{+F}\restriction[\mu_{n+2}, \mu_\omega)$,
     such that $B_{1 a}$ induces  $U \restriction [\mu_{n+2}, \mu_\omega)$.
    This used the description of $\U \restriction [\mu_{n+2}, \mu_\omega)$ as ${\B}^{+F'}\restriction[\mu_{n+2}, \mu_\omega)$,
     where $F' = A * (U \restriction \mu_{n+2})$.

     Since we have lifted $j$ onto a model which contains both $B_{1 a}$ and $F'$, we may use $j(B_{1a})$ and $j(F')$
     to induce a filter $\hat U_{[\mu_{n+2}, \mu_\omega)}$ on $j(\U \restriction [\mu_{n+2}, \mu_\omega))$ with
       $j[U \restriction [\mu_{n+2}, \mu_\omega))] \subseteq \hat U_{[\mu_{n+2}, \mu_\omega)}$.
         This lets us lift $j$ onto 
         $\Vd[A * U][B_{1a} \times C_{1b}][D^{0,1,2,3}][S \restriction \mu_{n+2}]$.

 \item   Stage 5:
 Similarly to Stage 4,  $\P_{1b}$ added a filter $C_{1b}$ 
  on $\C_{1b} = {\C}^{+F}\restriction[\mu_{n+2}, \mu_\omega)$, inducing $S \restriction [\mu_{n+2}, \mu_\omega)$.
  This used the description of  $\S \restriction [\mu_{n+2}, \mu_\omega)$ as ${\C}^{+F''}\restriction[\mu_{n+2}, \mu_\omega)$,
   where $F'' = A  * U$. 

     Since we have lifted $j$ onto a model which contains both $C_{1b}$ and $F''$, we may use $j(C_{1b})$ and $j(F'')$
     to induce a filter $\hat S_{[\mu_{n+2}, \mu_\omega)}$ on $j(\S \restriction [\mu_{n+2}, \mu_\omega))$ with
       $j[S \restriction [\mu_{n+2}, \mu_\omega))] \subseteq \hat S_{[\mu_{n+2}, \mu_\omega)}$.
             This lets us lift $j$ onto $\Vd[A * U * S][B_{1a} \times C_{1b}][D^{0,1,2,3}]$

           \item  Stage 6: Since $\D^{\rm small} \in \Vd[A * U * S]$ and $\vert \D^{\rm small} \vert \le \mu_n$, we may
             trivially lift
          $j$ onto $\Vd[A * U * S][B_{1a} \times C_{1b}][D^{s,0,1,2,3}] = W[B_{1a} \times C_{1b}]$. 

\end{itemize}

To verify the tree property, we need to check that the forcing posets used to extend $j$ onto $W$ can not
add a branch to a $\mu_{n+2}$-tree. Recall that the lifting of $j$ is defined in $W' = W[P_{1,2,3}]$. 
 As we already mentioned, $\mu_{n+1}$ and $\mu_{n+2}$ are preserved in $W$ but
it is possible that $\mu_n$ has been collapsed.

The proof which follows involves a number of auxiliary models. See the diagram
which follows the proof and its legend for a picture of how they are related.

\begin{itemize}

\item  Let $M_0 = W$ and $M_1 = M_0[P_{2b}]$.
  By hypothesis \ref{indhyp5a} of Lemma \ref{indestructible}, $\P_{2b}$ is $\mu_{n+2}$-Knaster in $M_0$, 
  so  by Lemma \ref{chain} no tree of height $\mu_{n+2}$ in $M_0$ has a new branch in $M_1$.
  By hypothesis \ref{indhyp5b}, $\P_{2 b}$ is $<\mu_{n+1}$-distributive in $M_0$, so that
  both $\mu_{n+1}$ and $\mu_{n+2}$ are regular in $M_1$. 
  
\item  Let $M_2 = M_1[P_1 \times P_3]$.
  We claim that no $\mu_{n+2}$-tree in $M_1$ has a new branch in $M_2$.

  Recall the closure property which we noted for $\P_{1a}$, $\P_{1b}$ and $\P_3$. They are all $<\mu_{n+1}$-closed in
  a certain submodel $M_{-}$ of $M_0$,
  where $M_{-} = \Vd[A * U][S \restriction [\mu_{n+1}, \mu_\omega)][D^{1,2,3}]$.
    Our aim is ultimately to make an appeal to Fact \ref{formerlyclosed} with
    $\tau = \vert \mu_n \vert$ and $\eta = \mu_{n+2}$.
  Note that in the model $M_{-}$ we have $2^{\vert \mu_n \vert} \ge \mu_{n+2}$.

  Since $M_0 = M_{-}[(D^0 \times S \restriction \mu_{n+1})*D^s]$,
  we have $M_1 = M_{-}[P_{2b}][(D^0 \times S \restriction \mu_{n+1})*D^s]$.
  Now $\P_{2b} \in M_{-}$ and by hypothesis \ref{indhyp5b}
  it is $<\mu_{n+1}$-distributive in $M_0$,
  so $\P_{2b}$  is $<\mu_{n+1}$-distributive in $M_{-}$ and hence
  $\P_{1 a} \times \P_{1 b} \times \P_3$ is $<\mu_{n+1}$-closed in $M_{-}[P_{2b}]$.
  
  By hypothesis \ref{indhyp6},
  $\D^0$ is $\mu_{n+1}$-Knaster in $\Vd[A * U * S][D^{1,2,3}][P_{2b}]$, so 
  it is $\mu_{n+1}$-cc in $M_{-}[P_{2b}]$.
  It is easy to see that $\S \restriction \mu_{n+1}$ is $\mu_{n+1}$-Knaster in $M_{-}[P_{2b}]$,
  and by hypothesis $\vert \D^{\rm small} \vert \le \mu_n$.
  So $(\D^0 \times \S \restriction \mu_{n+1})*\D^{\rm small}$ is $\mu_{n+1}$-cc in $M_{-}[P_{2b}]$.

  We are exactly in the situation of  Fact \ref{formerlyclosed} where the forcing posets live in  $M_{-}[P_{2b}]$:
 \begin{enumerate}
\item Since $M_{-}[P_{2 b}] \subseteq M_1$, $\mu_{n+1}$ and $\mu_{n+2}$ are regular in $M_{-}[P_{2 b}]$.    
\item $2^{\vert \mu_n \vert} \ge \mu_{n+2}$ in $M_{-}[P_{2b}]$. 
\item $\P_1 \times \P_3$ is $<\mu_{n+1}$-closed in $M_{-}[P_{2b}]$.
\item $M_1 = M_{-}[P_{2b}][Y]$, where $Y = (D^0 \times S \restriction \mu_{n+1})*D^s$ and $Y$ 
  is generic for $\mu_{n+1}$-cc  forcing over $M_{-}[P_{2b}]$. 
\end{enumerate}
 Since $M_2 = M_1[P_{1a} \times P_{1b} \times P_3]$, $M_2$ is an extension
 of $M_1$ by ``formerly closed'' forcing in the sense of Fact \ref{formerlyclosed}
 and we are done.

\item   
  Before the last step we need to analyse the cardinals of $M_2$. By Easton's Lemma,
  $\mu_{n+1}$ is preserved in this model.
  We claim that in $M_2$ the cardinal $\mu_{n+2}$ is collapsed so that (by Easton's Lemma
   again) it has cofinality $\mu_{n+1}$. To see this note that at coordinate $\mu_{n+2}$,
   conditions in $j(\C)$ have $\A \restriction \mu_{n+2} \times \U \restriction \mu_{n+1}$-terms 
   for conditions in $\Add(\mu_{n+1},1)^{V[A \restriction \mu_{n+2} \times \U \restriction \mu_{n+1}]}$. Since
   we are augmenting with $A \restriction \mu_{n+2} \times U \restriction \mu_{n+1}$ to form $\P_3$,
    we add a generic object for $\Add(\mu_{n+1},1)^{V[A \restriction \mu_{n+2} \times \U \restriction \mu_{n+1}]}$, and collapse
    $\mu_{n+2}$ because  $\mu_{n+2} = (2^{\mu_n})^{V[A \restriction \mu_{n+2} \times \U \restriction \mu_{n+1}]}$.

\item   We also need to analyse the chain condition of $\P_{2a}$ in $M_2$.
  $M_2 = M_1[P_1 \times P_3]$, and as we saw above $\P_1 \times \P_3$ is
  defined and $<\mu_{n+1}$-closed in $M_{-}$, hence it is formerly
  $<\mu_{n+1}$-closed in $M_{-}[S \restriction \mu_{n+1}] = V[A * U * S][D^{1,2,3}]$.
  By hypothesis \ref{indhyp6.5}, $\P_{2a}$ is
  $\mu_{n+1}$-Knaster in $M_2 = M_1[P_1 \times P_3]$.

\item
Let $M_3 = W' = M_2[P_{2a}]$. We claim that no tree of height $\mu_{n+1}$ in $M_2$ has a new branch in $M_3$.
This is immediate by Lemma \ref{chain} because $\P_{2a}$ is $\mu_{n+1}$-Knaster in $M_2$. 

\end{itemize}  

\end{proof}

\begin{tikzcd}[row sep=huge,column sep=huge]
M_- \arrow{r}{P_{2b}} \arrow{d}{Y} &   M_-[P_{2b}] \arrow{d}{Y}      &                          &                             \\ 
M_0 \arrow{r}{P_{2b}}              &   M_1 \arrow{r}{P_1 \times P_3}  &  M_2 \arrow{r}{P_{2a}}     &                   M_3       \\
\end{tikzcd} 
\begin{itemize}
\item $M_{-} = \Vd[A * U][S \restriction [\mu_{n+1}, \mu_\omega)][D^{1,2,3}]$.  
\item $M_0 = W = \Vd[A * U * S][D^{0,1,2,3,s}] = M_-[Y]$, where $Y = (D^0 \times S \restriction \mu_{n+1})*D^s$. 
\item $M_1 = M_0[P_{2b}]$.
\item $M_2=  M_1[P_1 \times P_3]$.
\item $M_3 = M_2[P_{2a}] = W[P_{1,2,3}] = W'$.
\end{itemize}

The following lemma will enable us to satisfy the hypotheses
of Lemma \ref{indestructible} in most instances. We are assuming all the background
hypotheses listed at the start of this section, notably $\A_{[s, t]}$ is $<\mu_s$-distributive
and $\mu_{t+1}$-cc in $\Vd$.

\begin{lemma} \label{meetingindyreqs}
  Let $V^a, V^b, V^c, V^d$ be inner models with $\Vi \subseteq V^x \subseteq \Vd$ for $x = a,b,c,d$.
  Assume that:
  \begin{itemize}  
  \item $n > 0$. 
  \item  $\D^{\rm small}$ is any poset in $\Vd[A * U * S]$ with  $\Vd[A * U * S] \models \vert \D^{\rm small} \vert \le \mu_n$. 
  \item   $\D^0   = \Add^{V^a}(\mu_{\bar n}, \sigma)$ for some $\sigma$ and some $\bar n \le n$.
  \item   $\A_n   = \Add^{V^b}(\mu_n, \sigma')$ for some $\sigma'$.
  \item   $\D^1 = \Add^{V^c}(\mu_{n+1}, \tau)$ for some $\tau$. 
  \item   $\A_{n+1} = \Add^{V^d}(\mu_{n+1}, \tau')$ for some $\tau'$. 
  \item   $\D^2$ is any poset in $\Vi$ with $\Vi \models \mbox{``$\D^2$ is $<\mu_{n+2}$-directed closed''}$.
  \item   $\D^3$ is any poset in $\Vd[A \restriction \mu_{n+2} * U \restriction \mu_{n+2}]$
  with $\Vd[A \restriction \mu_{n+2} * U \restriction \mu_{n+2}] \models \mbox{``$\D^3$ is $<\mu_{n+2}$-directed closed''}$.
  \end{itemize}

  Then the hypotheses of Lemma \ref{indestructible} are satisfied.
\end{lemma}

\begin{proof}
  Hypotheses \ref{indhyp3}, \ref{indhyp7} and \ref{indhyp8} are immediate. 
  Since $\D^{\rm small} \in \Vd[A * U * S]$ and $\vert \D^{\rm small} \vert \le \mu_n$, we may assume
  that $\D^{\rm small} \in \Vd[A \restriction \mu_{n+2} * U \restriction \mu_{n+1} * S \restriction \mu_{n+1}]$.
  As we noted in Remark \ref{benignhelix}, our hypotheses imply that $\D^{2,3}$ is $\mu_{n+1}$-distributive in $\Vd[A * U * S]$.
  
  Since $H$ is generic over $\Vd$ for forcing which is defined and $<\mu_{n+2}$-closed in $\Vi$, and $\Vd$
  is a $\mu_1$-cc generic extension of $\Vi$, by Easton's Lemma $\Vd[H]$ is a $<\mu_{n+2}$-distributive extension of $\Vd$.
  Since $\Vd$ is a $\mu_1$-cc generic extension of $V^b$ and $V^d$, and both $\mu_{n+1}$ and $\mu_{n+2}$ are inaccessible
  in any submodel of $\Vd$, it follows that
  $\Vd[H]$ is a $(\mu_n, \mu_{n+1})$-good extension of $V^b$ and a $(\mu_{n+1}, \mu_{n+2})$-good extension of $V^d$,
 so that by Lemma \ref{Cohenrobust} $\A_n$ and $\A_{n+1}$ are respectively $\mu_{n+1}$-Knaster and $\mu_{n+2}$-Knaster in $\Vd[H]$. 
 We have satisfied Hypothesis \ref{indhyp2}.  Similarly $\D^1$ is $\mu_{n+2}$-Knaster in $\Vd[H]$
 and we have satisfied Hypothesis \ref{indhyp4a}. 

  We need some analysis of $\P_{2 a}$ and $\P_{2 b}$.          
     Let $\X^x \in \Vi$ and $\X_x \in V^x$ be such that $V^x = \Vi[X^x]$ and $\Vd = V^x[X_x]$ for $x = a,b,c,d$. 
     Note that we may assume that $\X^x$ is $\mu_1$-cc in $\Vi$ and $\X_x$ is $\mu_1$-cc in $V^x$.
  Recall that $j$ is an embedding witnessing that $\mu_{n+2}$ is highly supercompact in the model
  $\Vd[A_{[n+2, \omega)} \times D^2]$,   and is the trivial lift (keeping in mind that $\vert \X \vert \le \mu_1 < \mu_{n+2}$)
  of such an embedding defined in $\Vi[A_{[n+2, \omega)} \times D^2]$.
  It is easy to see that $V^a$ and $j(V^a)$ have the same $<\mu_n$-sequences of ordinals,
  so that $j(\D^0) = \Add^{V^a}(\mu_{\bar n}, j(\sigma))$ and
  $j(D_0)/j[D_0] = \Add^{V^a}(\mu_{\bar n}, j(\sigma) \setminus j[\sigma])$
  and similarly for $\A_n$, $\D^1$ and $\A_{n+1}$.

  Now we revisit the argument for Lemma \ref{Wgoodextn}, but we need a slightly different decomposition
  for $\A * \U * \S$. $\A * \U * \S$ may be written in $\Vd$ as a projection of the product
  \[
  (\A_{[0, n]} * \U_{[0,n)} * \S_{[0, n)}) \times \A_{n+1} \times \A_{[n+2, \omega)} \times \B_{[n, \omega)} \times \C_{[n, \omega)}
  \]
  where:
  \begin{itemize}
  \item  $\A_{[0, n]} * \U_{[0,n)} * \S_{[0, n)}$ is defined and $\mu_{n+1}$-cc in $\Vd$.
  \item  $\A_{n+1}$ is defined and $<\mu_{n+1}$-closed in $V^d = \Vi[X_d]$.
  \item $\A_{[n+2, \omega)}$ is defined and $<\mu_{n+2}$-closed in $\Vi$.
  \item $\B_{[n, \omega)} \times \C_{[n, \omega)}$ is defined and $<\mu_{n+1}$-closed in
      $\Vd$.
  \end{itemize}    

  As in the proof of Lemma \ref{Wgoodextn} (and keeping in mind that $\D^1$ is defined and $< \mu_{n+1}$-closed in
  $V^c = \Vi[X^c]$) we may force with a series of quotient to term forcings to extend $\Vd[A * U * S][D^{1,2,3}]$ to
  a  model of the form $\Vi[(X * P''_0) \times T  \times A_{[n+2,\omega)} \times D^2 \times T^3]$,
    where:
    \begin{itemize}
    \item $\P''_0 = \A_{[0, n]} * \U_{[0,n)} * \S_{[0, n)}$, so that $\X * \P''_0$ is $\mu_{n+1}$-cc in $\Vi$.
    \item $\T = \termspace^{\Vi}(\X^d, \dot \A_{n+1}) \times \termspace^{\Vi}(\X^c, \dot \D_1)
    \times \termspace^{\Vi}(\X, \dot \B_{[n, \omega)} \times \dot \C_{[n, \omega)})$,     
      so that $\T$ is $<\mu_{n+1}$-closed in $\Vi$.
    \item As before $\T^3 = \termspace^{\Vi}(\X * \A \restriction \mu_{n+2} * \U \restriction \mu_{n+2}, \dot D^3)$,
      so that $\T^3$ is $<\mu_{n+2}$-closed in $\Vi$.
    \end{itemize}  

    By Easton's Lemma all $<\mu_{n+1}$-sequences of ordinals from $\Vd[A * U * S][D^{1,2,3}]$ lie in
    $\Vi[X * P''_0] = \Vd[P''_0] \subseteq \Vd[A * U * S]$, so that in particular
    $\D^1$ is $<\mu_{n+1}$-distributive in $\Vd[A * U * S][D^{2,3}]$ and we satisfied
    Hypothesis \ref{indhyp4b}. 

    By Lemma \ref{rightafterportmanteau} and the hypothesis that $\Vd$ is a $\mu_1$-cc extension of $\Vi$,
    $\Vd[A * U * S]$ is both a $(\mu_n, \mu_{n+1})$-good extension of $V^a$ and a $(\mu_{n+1}, \mu_{n+2})$-good extension
    of $V^c$. Since $\D^{2,3}$ is $<\mu_{n+2}$-distributive in $\Vd[A * U * S]$, $\Vd[A * U *S][D^{2, 3}]$ is a
    $(\mu_{n+1}, \mu_{n+2})$-good extension of $V^c$, so that $\D^1$ is $\mu_{n+2}$-Knaster in $\Vd[A * U *S][D^{2, 3}]$.
    Since $\D^1$ is $<\mu_{n+1}$-distributive in $\Vd[A * U *S][D^{2, 3}]$, $\Vd[A * U *S][D^{1, 2, 3}]$ is a
    $(\mu_n, \mu_{n+1})$-good extension of $V^a$, so that $\D^0$ is $\mu_{n+1}$-Knaster in $\Vd[A * U *S][D^{1, 2, 3}]$.
    In fact $\D^0 \times \D^{\rm small}$ is $\mu_{n+1}$-Knaster in $\Vd[A * U *S][D^{1, 2, 3}]$, from which it follows
    easily that both $\mu_{n+1}$ and $\mu_{n+2}$ are regular in $W$. We have satisfied Hypothesis \ref{indhyp1}.

    The analysis of the last paragraph also shows that $W$ is a $(\mu_{n+1}, \mu_{n+2})$-good extension 
    of $V^c$ and of $V^d$. From the analysis of $j(\A_{n+1})$ and $j(\D^1)$, it follows readily that
    $\P_{2 b}$ is $\mu_{n+2}$-Knaster in $W$.  We have satisfied Hypothesis \ref{indhyp5a}. 

    Now we do another analysis in the same style as Lemma \ref{Wgoodextn}, but this
    time we expand the model $W[P_{2b}] = \Vd[A * U * S][D^{s,0,1,2,3}][P_{2b}]$ to 
    $\Vi[(X * P''_0 * (D^0 \times D^s)) \times T  \times A_{[n+2,\omega)} \times D^2 \times T^3 \times T']$,
    where $\T' = \termspace^{\Vi}(\X^c, j(\D_1)/j[D_1]) \times \termspace^{\Vi}(\X^d, j(\A_{n+1})/j[A_{n+1}])$.
    We recall from our earlier analysis that $\D^{\rm small} \in \Vd[P''_0]$ and that $\D^0$ is
    $\mu_{n+1}$-cc in $\Vd[A * U * S]$, so that easily $\X * \P''_0 * (\D^0 \times \D^{\rm small})$
    is $\mu_{n+1}$-cc in $\Vi$. By Easton's lemma all $<\mu_{n+1}$-sequences of ordinals in
    $W[P_{2b}]$ lie in the submodel $\Vi[X * P''_0 * D^{0,s}]$ of $W$, so that
    $\P_{2 b}$ is $<\mu_{n+1}$-distributive in $W$. We have satisfied hypothesis \ref{indhyp5b}. 

    We saw already that $\Vd[A * U * S][D^{1,2,3}]$ is a $(\mu_n, \mu_{n+1})$-good extension
    of $V^a$. Since $\P_{2 b}$ is $<\mu_{n+1}$-distributive in $W$ it has this property
    in $\Vd[A * U * S][D^{1,2,3}]$, so $\Vd[A * U * S][D^{1,2,3}][P_{2b}]$ is
    a $(\mu_n, \mu_{n+1})$-good extension of $V^a$, and 
    thus  $\D^0$ is $\mu_{n+1}$-Knaster in $\Vd[A * U * S][D^{1,2,3}][P_{2b}]$.
    We have satisfied hypothesis \ref{indhyp6}.

    Finally let $\Q$ be defined and
    $<\mu_{n+1}$-closed in $\Vd[A * U][S \restriction [\mu_{n+1}, \mu_\omega)][D^2]$,
    and let $W' = W[P_{2b}][Q] = \Vd[A * U * S][D^{1,2,3}][D^{0, s}][Q][P_{2b}]$.
    Arguing as before we expand 
    $\Vd[A * U * S][D^{1,2,3}][Q][P_{2b}]$
    to a model
    $\Vi[(X * P''_0)  \times T  \times A_{[n+2,\omega)} \times D^2 \times T^3 \times T' \times T'']$,
    where $T'' = \termspace^{\Vi}(\X * \A * \U * \S \restriction [\mu_{n+1}, \mu_\omega) * \D^2, \dot \Q)$,
    and use this to argue that all $<\mu_{n+1}$-sequences of ordinals from
    $\Vd[A * U * S][D^{1,2,3}][Q][P_{2b}]$ lie in $\Vd[P''_0]$.    
    It follows that $\Vd[A * U * S][D^{1,2,3}][Q][P_{2b}]$ is a $(\mu_n, \mu_{n+1})$-good extension 
  of each model $V^x$, so that easily $\D^0 \times \D^{\rm small} \times \P_{2a}$ is
  $\mu_{n+1}$-Knaster in $\Vd[A * U * S][D^{1,2,3}][Q][P_{2b}]$. So
  $\P_{2 a}$ is $\mu_{n+1}$-Knaster in $\Vd[A * U * S][D^{1,2,3}][Q][P_{2b}][D^{0,s}] = W'$,
  and we have satisfied hypothesis \ref{indhyp6.5}.
\end{proof}

It will be useful later (in the $n=0$ cases from Sections \ref{GroupI} and \ref{GroupIII},
and again in Section \ref{plustwo})
to know that certain initial segments of
$\Vd[A * U * S]$ have similar indestructibility properties to those in Lemma \ref{indestructible}.
The following lemma is stated under the same hypotheses as that lemma, and as far as possible with the same notation.
Although the construction of the relevant generic embedding
is very similar to that for Lemma \ref{indestructible}, we have given it in some
detail as a service to readers of Section \ref{plustwo}. We have not stated the Lemma in the maximum
possible generality, in particular we have dispensed with $\D^1$ and have only
some specific instances of $\D^0$.

\begin{lemma}\label{indestructible2}
  Let $\mu_{n+2} \le \eta < \mu_\omega$ and let
  $V' = \Vd[A \restriction \eta * U \restriction \mu_{n+2} * S \restriction \mu_{n+2}]$.  
  Let $\D^2, \D^3, \D^0, \D^{small}$ be forcing posets such that
  \begin{enumerate}
  \item $\D^2 \in \Vi$ and $\D^2$ is $<\mu_{n+2}$-directed closed in $\Vi$.
  \item $\D^3 \in \Vd[A \restriction \mu_{n+2} * U \restriction \mu_{n+2}]$ and $\D^3$ is $<\mu_{n+2}$-directed closed in
    $\Vd[A \restriction \mu_{n+2} * U \restriction \mu_{n+2}]$
  \item $\D^0 = \Add^{\Vd}(\mu_n, \sigma)$ for some $\sigma$, or $\D^0 = Coll(\omega, \rho)$ for some $\rho < \mu_1$.  
  \item $\D^{small} \in V'$ and $V' \models \vert \D^{small} \vert \le \mu_n$.
\end{enumerate}
\end{lemma}
Then the tree property holds at $\mu_{n+2}$ in $V'[D^{s, 0, 2, 3}]$.

\begin{proof} Let $W = V'[D^{s, 0, 2, 3}]$.
  Let $\bar A = A \restriction [\mu_{n+2}, \eta)$,
    so that $\bar A$ is generic for $<\mu_{n+2}$-directed closed forcing defined in $\Vi$. 
  $\mu_{n+2}$ is indestructibly supercompact in $\Vi$,
  we will construct a supercompactness embedding $j$ defined in $\Vi[D^2 \times \bar A]$
  and lift it to such an embedding defined in $\Vd[D^2 \times \bar A]$.
  
  We can dispense with Steps 1a and 1b from the previous construction, so there is no $\P_1$.  We
  choose $\chi$ suitably large and then work in $\Vi[D^2 \times \bar A]$ to
  choose $j$ such that $j \restriction ON$ is defined in $\Vi$,
  $j$ witnesses $\mu_{n+2}$ is $\chi$-supercompact, the next point in $\dom(j(\psi))$ past
  $\mu_{n+2}$ is greater than $\chi$, and $j(\psi)(\mu_{n+2})$ is a name in $\Vi$ for an
  $\A \restriction \mu_{n+2} * U \restriction \mu_{n+2}$-name for $\D^3$. Then after lifting
  $j$ to $\Vd[D^2 \times \bar A]$,
  $j(\phi)(\mu_{n+2})$ is an $\A \restriction \mu_{n+2} * U \restriction \mu_{n+2}$-name in $\Vd$ for $\D^3$.
  
  $\P_{2 a}$ is chosen as before, with the simplification that it is now just a forcing in
  $\Vd$ adding Cohen subsets to $\mu_n$.
  As before $\P_{2 a}$ is $\mu_{n+1}$-Knaster in a robust way.
   $\P_{2 b}$ is also as before, with the simplification that there is no $\D^1$ and so this poset
  is just ``stretching'' $A_{n+1}$: as before $\P_{2 b}$ is $<\mu_{n+1}$-distributive and $\mu_{n+2}$-cc.

  $\P_3$ is chosen essentially as before. $\P_3$ is defined in $\Vd[A \restriction \eta * U \restriction \mu_{n+1}][D^2]$,
  and is $< \mu_{n+1}$-closed in
  $\Vd[A \restriction \eta * U \restriction \mu_{n+2}][S \restriction [\mu_{n+1}, \mu_{n+2})][D^{2, 3}]$.

      In the lifting argument we lift $j$ onto $\Vd[A \restriction \eta * U \restriction \mu_{n+2}][D^{0, 2, 3}]$
      (like Stage 1), extend to $\Vd[A \restriction \eta * U \restriction \mu_{n+2} * S \restriction \mu_{n+1}][D^{0, 2, 3}]$
      (like Stage 2), extend to $\Vd[A \restriction \eta * U \restriction \mu_{n+2} * S \restriction \mu_{n+2}][D^{0, 2, 3}]$
      (like Stage 3), and finally extend to
      $\Vd[A \restriction \eta * U \restriction \mu_{n+2} * S \restriction \mu_{n+2}][D^{s, 0, 2, 3}]$ (like Stage 6).
      The argument for the tree property is essentially identical, as we still have the relevant cardinal
      arithmetic and all the posets $\P_i$ are either missing or have the same properties as before. 
\end{proof}

\begin{remark} \label{mutualgenericity1}
  As we mentioned in the preamble to Lemma \ref{indestructible}, there are a couple of
  instances where we would like to apply the Lemma but the hypotheses are not quite
  satisfied. To be more precise, we want to prove that $\mu_{n+2}$
  has the tree property in some extension $W' = \Vd[A * U * S][D']$ where $\D'$ is a product of posets
  which does not quite meet the hypotheses of Lemma \ref{indestructible}.
  In this case we can sometimes use mutual genericity to our advantage.

  More specifically, assume that by forcing over $W'$ with some poset $\P'$,
  we obtain a generic embedding with domain $W'$ and critical point $\mu_{n+2}$.
  Let $T \in W'$ be a $\mu_{n+2}$-tree, so that $T$ has a branch $b \in W'[P']$.
  Assume further that $E$ is mutually generic with $P'$ over $W'$,
  and that our previous arguments can be adapted to show that every branch of $T$
  from $W'[E][P']$ lies in $W'[E]$. Then $b \in W'[E] \cap W'[P']$,
  and by the mutual genericity of $E$ and $P'$ we have $b \in W'$ as required.
\end{remark}

\begin{remark}  \label{mutualgenericity2}
  A particular instance of the idea of Remark \ref{mutualgenericity1}
  can be used to handle more posets of
  cardinality $\mu_{n+1}$ in the setting of Lemma \ref{indestructible}.
  It is clear that in general a forcing poset of size $\mu_{n+1}$ can destroy the tree property
    at $\mu_{n+2}$, for example $\Coll(\omega, \mu_{n+1})$ will always do this. In the language 
    of Lemma \ref{indestructible}, such a poset may not be a viable choice for $\D^0$ (insufficient
    chain condition) or $\D^1$ (insufficient distributivity).

    Suppose that $\D \in \Vd$ and let $W^* = W[D]$ where $W = \Vd[A * U * S][D^{s,0,1,2,3}]$
    as in Lemma \ref{indestructible}. Assume that: 
\begin{enumerate}
\item $\vert \D \vert = \mu_{n+1}$.
\item  $\D$ is the projection of a two-step iteration $\P * \dot \Q$
  where
    $\P$ forces that $\Q$ is the union of fewer than $\mu_{n+1}$ filters,
    and $\vert \P * \Q \vert = \mu_{n+1}$. 
  \item $\P_{2a}$ is $\mu_{n+1}$-Knaster in $W[P_1 \times P_{2b} \times P_3 \times P]$.
  \item $\P_{2b}$ is $<\mu_{n+1}$-distributive in $\Vd[A * U][S \restriction [\mu_{n+1}, \mu_\omega)][D^{1,2,3}][P]$. 
  \item $\P$ is $<\mu_{n+1}$-distributive in  $\Vd[A * U][S \restriction [\mu_{n+1}, \mu_\omega)][D^{1,2,3}]$.  
\end{enumerate}

Then we claim that Lemma \ref{indestructible} remains true if we add $\D$ as a factor to the
product of posets which preserves the tree property at $\mu_{n+2}$, that is to say we claim that $\mu_{n+2}$
has the tree property in $W^*$. As for Lemma \ref{indestructible}, the proof is followed by a
picture with a legend to help the reader keep track of all the models and forcing posets. 

    Let $T$ be a $\mu_{n+2}$-tree in $W^*$.
    We define the embedding $j$ and lift it to $W$ in the model $W[P_{1,2,3}]$ just as in Lemma \ref{indestructible}.
    Since $\vert \D \vert = \mu_{n+1}$, we may trivially lift the embedding onto
    $W[D] = W^*$, working in the model $W^*[P_{1,2,3}]$.
    As usual we obtain a branch $b$ in $W^*[P_{1,2,3}]$.
    
    To cope with the problem that $\D$ is not necessarily $\mu_{n+1}$-cc, force over 
    $W^*[P_{1,2,3,}]$ with $\P * \Q/D$ to obtain $P * Q$ which induces $D$,
    and is mutually generic over $W$ with $P_{1,2,3}$.
    Let 
    $W^{**}= W[P * Q]$, so that $W^* \subseteq W^{**}$ and $b \in W^{**}[P_{1,2,3}] = W[P_{1,2,3} \times P * Q]$. 

    Now let $M_0 = W^{**}$, $M_1 = M_0[P_{2b}]$, $M_2=M_1[P_1 \times P_{3}]$,
    $M_3 = M_2[P_{2a}] = W^{**}[P_{1,2,3}]$. We aim to argue that $b \in M_0$.
    Since $\vert \P * \Q \vert = \mu_{n+1}$, $\P_{2 b}$ is $\mu_{n+2}$-Knaster in $M_0$,
    so there is no change in the step from $M_1$ to $M_0$.
    
    Let $M_{-} = \Vd[A * U][S \restriction [\mu_{n+1}, \mu_\omega)][D^{1,2,3}][P]$.
      By hypothesis
      $\P$ is $<\mu_{n+1}$-distributive in  $\Vd[A * U][S \restriction [\mu_{n+1}, \mu_\omega)][D^{1,2,3}]$,
     so that $\P_1 \times \P_3$ is $<\mu_{n+1}$-closed in $M_{-}$.

     By hypothesis $\P_{2b}$ is $<\mu_{n+1}$-distributive in $M_{-}$. It follows that 
    $\P_1 \times \P_3$ is $<\mu_{n+1}$-closed in $M_{-}[P_{2b}]$. 
    Finally $(\S \restriction \mu_{n+1} \times \Q  \times \D^0) * \D^{\rm small}$   is $\mu_{n+1}$-cc in 
    $M_{-}[P_{2b}]$, where the factor $\Q$ causes no problems because $\Q$ is the union of fewer than $\mu_{n+1}$ filters, 
     so $\P_1 \times \P_3$ is formerly closed in
     $M_1 = M_{-}[P_{2b}][(S \restriction \mu_{n+1} \times Q \times D^0) * D^s]$ and we finish the step from
     $M_2$ to $M_1$  as before.
     
     Finally our hypotheses imply that $\P_{2a}$ is $\mu_{n+1}$-Knaster in $W[P_1 \times P_{2b} \times P_3 \times P][Q] = M_2$,
     and we finish the step from $M_3$ to $M_2$ as before.          
  It follows that $b \in W^{**} = W^*[P*Q/D]$. Since $b \in W^*[P_{1,2,3}]$ and
  $P*Q/D$ is mutually generic with $P_{1,2,3}$, $b \in W^*$ and we are done.

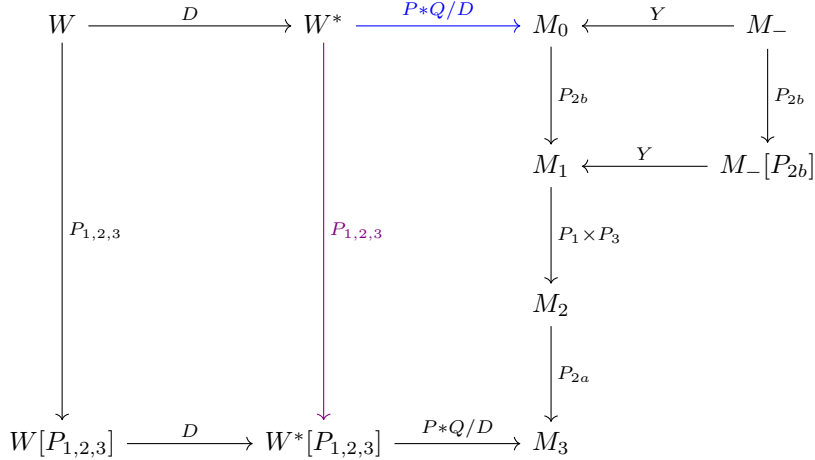
\begin{tikzcd}[row sep=huge,column sep=huge]
  W \arrow{r}{D} \arrow{ddd}{P_{1,2,3}} & W^* \arrow[blue]{r}{P*Q/D} \arrow[violet]{ddd}{P_{1,2,3}} & M_0 \arrow{d}{P_{2b}}       & M_{-} \arrow[swap]{l}{Y} \arrow{d}{P_{2b}}     \\
  {}                                &               {}                           & M_1 \arrow{d}{P_1 \times P_3} & M_{-}[P_{2b}] \arrow[swap]{l}{Y}                      \\
  {}                               &                {}                           & M_2 \arrow{d}{P_{2a}}                                 \\
  W[P_{1,2,3}]     \arrow{r}{D}       & W^{*}[P_{1,2,3}] \arrow{r}{P*Q/D}            & M_3                                                  \\   
\end{tikzcd}

The blue arrow for $P*Q/D$ and the violet arrow for $P_{1,2,3}$ indicate mutually generic objects
over $W^*$. 
\begin{itemize}
\item $W=\Vd[A * U * S][D^{0,1,2,3,s}]$. 
\item $W^* = W[D]$.
\item $M_{-} =  \Vd[A * U][S \restriction [\mu_{n+1}, \mu_\omega)][D^{1,2,3}][P]$. 
\item $M_0 = W^{**} =W^*[P*Q/D] = W[P * Q] =\Vd[A * U * S][D^{0,1,2,3,s}][P * Q] = M_{-}[Y]$, where $Y = (S \restriction \mu_{n+1} \times Q \times D^0) * D^s$.
\item $M_1 = M_0[P_{2b}]$.
\item $M_2 = M_1[P_1 \times P_3]$.
\item $M_3 = M_2[P_{2a}] = W^{**}[P_{1,2,3}]$.
\end{itemize}

\end{remark}

\begin{remark} The main construction for Theorem \ref{mainthm} contains many instances of
  ``quotient to term'' posets, for instance in the definitions of $\Q_1(\tau, \tau^*)$ and
  $\Q_2(\tau, \tau^*)$ in Section \ref{successiveprikry}.
  The role of these quotient to term posets is typically to produce generic objects
  which fit into one of the indestructibility schemes from the current section.
\end{remark}

\section{Initial hypotheses} \label{setup}

We are now ready to begin the main construction.  As we mentioned in the introduction,
we will be introducing many pieces of notation which will have fixed meanings for the rest
of the paper. Every time we introduce one or more of these important pieces of notation,
we will flag it as ``Global notation'' and add a corresponding entry in the index of notation.

\subsection{Preparing $V$} \label{preparation}

We start with a model $V_0$ with the following properties:
\begin{enumerate}
\item  $\theta$ is the least supercompact cardinal.
\item There exist cardinals $\kappa_\alpha$ for $\alpha < \theta^+$ 
  such that $\theta < \kappa_0$ and each $\kappa_\alpha$ is supercompact.
  We let $\kappa = \kappa_0$ and $\delta = \sup_{\alpha < \theta^+} \kappa_\alpha$. 
\item There is an elementary embedding $j_0 : V_0 \to M_0$ such that $j_0$ witnesses
 that  $\kappa$ is $\delta^+$-supercompact, and in addition $\kappa_\alpha$ is supercompact in $M_0$ for all $\alpha < \theta^+$.
\item There is a universal indestructible Laver function $\phi_0$ defined up to $\delta$,
  in particular $\theta$ and the $\kappa_\alpha$'s are all indestructibly supercompact.
  Every element of $\dom(\phi_0)$ is an inaccessible closure point of $\phi_0$.   
\item $j_0(\phi_0) \restriction \delta = \phi_0$.   
\end{enumerate}

\smallskip

\noindent Global notation: $V_0$\index[Notation]{$V_0$},
                 $\theta$\index[Notation]{$\theta$},
                 $\kappa_\alpha$\index[Notation]{$\kappa_\alpha$}, 
                 $\kappa$\index[Notation]{$\kappa$},
                 $\delta$\index[Notation]{$\delta$},
                 $j_0$\index[Notation]{$j_0$},
                 $M_0$\index[Notation]{$M_0$}, 
                 $\phi_0$\index[Notation]{$\phi_0$}.

\smallskip

\begin{remark}
  Given a model $V_0'$ where  hypotheses 1-3 hold we may arrange that hypotheses 1-5 hold
  in a suitable extension $V_0$ of $V_0'$.  To see this let 1-3 hold in $V_0'$ where
  hypothesis 3 is witnessed by $j_0' : V_0' \to M_0'$. The main point is that in $V_0'$ we
  may choose $\phi_0'$ a universal Laver function defined up to $\delta$ such that
  $j_0'(\phi_0') \restriction \delta = \phi_0'$: doing the corresponding Laver preparation
  will give a model $V_0$ for hypotheses 1-5. 
 
  To see that we can choose a suitable function $\phi_0'$, recall that to define $\phi_0'(\alpha)$
  we choose a counterexample to $\phi_0' \restriction \alpha$
    being a Laver function which is minimal with respect to some well-ordering. We will fix a well-ordering
    $\prec$ of $V_\kappa$ such that $\prec \restriction V_\alpha$ is an initial segment of $\prec$ for all $\alpha < \kappa$,
    and define $\phi_0'$ using $\prec^* = j_0'(\prec) \restriction V_\delta$: this works because $\prec^* \restriction V_\kappa = \prec$
    and $j_0'(\prec^*) \restriction V_\delta =j_0'(\prec^* \restriction V_\kappa) \restriction V_\delta =
    j_0'(\prec) \restriction V_\delta = \prec^*$.    
\end{remark}
  
It will be important later that all Laver functions used during the construction are derived
from the initial Laver function $\phi_0$ as in \ref{laversmallfact} and \ref{universalindestructibleLaverfunction}.
Let $E_0$ be defined in $V_0$ as the set of inaccessible closure points
of $\phi_0$ in the interval $(\theta, \delta)$, and let $\alpha^* = \min(E_0 \setminus (\alpha + 1))$
for $\alpha < \delta$.

\smallskip

\noindent Global notation: $E_0$\index[Notation]{$E_0$}, $\alpha^*$\index[Notation]{$\alpha^*$}. 

\smallskip

Our first step is to produce an extension $V$ of $V_0$ in which the $\kappa_\alpha$'s retain
the properties listed above, $\theta$ is the continuum,
and $\theta$ is ``generically indestructibly supercompact via Cohen reals''. 
To be more precise:

\begin{lemma} \label{fromVzerotoV}
  There is a generic extension $V$ of $V_0$ in which:
\begin{itemize}
\item  $2^{\aleph_0} = \theta$.
\item  For every $<\theta$-directed closed generic extension $V[H]$ and every $\gamma > \theta$,
  there exists a generic $\gamma$-supercompactness embedding $\pi: V[H] \to N$ with critical point $\theta$.
  The embedding $\pi$ exists in an extension of $V[H]$ obtained by adding $\pi(\theta)$ Cohen reals.
\item  The embedding $\pi$ lifts an embedding $i: V_0 \rightarrow N_0$ defined in $V_0$,
  where $\crit(i) = \theta$ and $i$ may be chosen to witness an arbitrarily high degree of supercompactness
  for $\theta$ in $V_0$. 
\item There is a universal indestructible Laver function $\phi$ defined on $(\theta, \delta)$, in
particular every supercompact cardinal up to $\delta$ is indestructible.
\item There is an elementary embedding $j : V \to M$ such that $j$ witnesses
  $\kappa$ is $\delta^+$-supercompact and in addition each $\kappa_\alpha$ is supercompact in $M$.
\end{itemize}
\end{lemma}

\smallskip

\noindent Global notation: $V$\index[Notation]{$V$},
                            $j$\index[Notation]{$j$},
                            $M$\index[Notation]{$M$},
                            $\phi$\index[Notation]{$\phi$}

\smallskip

\begin{proof}
    The construction will be reminiscent of that of $\A * \dot\U$ from Section \ref{abcu}.
    This resemblance is not coincidental and will be used later, see Section \ref{firstprikry}.
    We will freely use the notation and ideas of Section \ref{abcu}.
    Objects in the current construction typically have
    names as in Section \ref{abcu} decorated with a superscript $0$
    
    We work in $V_0$.
    Let $\A^0$ be the poset of finite partial functions $a$ from $\delta$ to $2$ with $\dom(a) \subseteq \theta$.
    Of course $\A^0$ is equivalent to $\Add(\omega, \theta)$, and $\A^0 \restriction \alpha = \A^0$
    for $\theta \le \alpha < \delta$: defining $\A^0$ in this artificial way just makes the following definitions
    more uniform.
    Let $A^0$ be $\A^0$-generic. As in Section \ref{abcu} we define posets $\B^0 \in V$ and $\U^0 \in V[A^0]$
    such that:
    \begin{itemize}
    \item $\B^0$ and $\U^0$ have the same set of conditions.
    \item 
      The support of ${\B}^0$ consists of $\alpha < \delta$ such that
      $\phi_0(\alpha)$
      is a pair $(\psi_0(\alpha), \psi_1(\alpha))$ with the following properties:
      \begin{itemize}
      \item $\psi_0(\alpha)$ is an 
      $\A^0 \restriction \alpha * \U^0 \restriction \alpha$-name for a $<\alpha$-directed closed
        forcing poset.
      \item $\psi_1(\alpha)$ is an 
        $\A^0 \restriction \alpha * \U^0 \restriction \alpha$-name.
      \end{itemize}  
    \item  
      An element $b \in \B^0$ is a function such that $\dom(b)$ is an Easton subset of the support of $\B^0$, and 
      $b(\alpha)$ is an $\A^0 \restriction \alpha * \U^0 \restriction \alpha$-name for an element of
      $\psi_0(\alpha)$.  
    \item For $b_0, b_1 \in \B^0$, $b_1 \le b_0$ if and only if $\dom(b_0) \subseteq \dom(b_1)$ and
      $(0, b_1 \restriction \alpha) \forces b_1(\alpha) \le b_0(\alpha)$ for all $\alpha \in \dom(b_0)$.
    \item For $u_0, u_1 \in  \U^0$, $u_1 \le u_0$ if and only if $\dom(u_0) \subseteq \dom(u_1)$ and 
      there is $a \in A^0$ such that $(a \restriction \alpha , u_1 \restriction \alpha) \forces u_1(\alpha) \le u_0(\alpha)$
      for all $\alpha \in \dom(u_0)$.
    \end{itemize}

\smallskip
    
    \noindent Global notation: $\A^0$\index[Notation]{$\A^0$},
                               $\B^0$\index[Notation]{$\B^0$},
                               $\U^0$\index[Notation]{$\U^0$}
     
\smallskip

    Let $V = V_0[A^0 * U^0]$ where $A^0 * U^0$ is $\A^0 * \U^0$-generic over $V_0$.
    We record a few remarks:
    \begin{itemize}
    \item $a \restriction \alpha = a \restriction \min(\alpha, \theta)$ for all $a \in \A^0$,
      and similarly for $\A^0$ and $A^0$.
    \item Since $\theta$ is supercompact, $\theta \notin \dom(\phi_0)$.
    \item We can view $\A^0 * \U^0$ as a two-step iteration, forcing first with
      $\Add(\omega, \theta) * \dot\U^0 \restriction \theta$ and then with a forcing poset
      ${\FL}^0$ defined in $V_0[A^0 * U^0 \restriction \theta]$.
    \item   The forcing poset $\FL^0$ is essentially a Laver preparation on the interval $(\theta, \delta)$,
      with the minor modification that the guessing function is just guessing names for forcing posets rather
      than pairs consisting of a name for a forcing poset and an ordinal: in the standard Laver preparation the role of the ordinals
      is to ``space out'' the support of the image of the preparation under supercompactness embeddings,
      and in our context this is handled by the properties of $\phi_0$.
    \item  In $V$ we have a universal indestructible Laver function $\phi$ on $(\theta, \delta)$, given
      by $\phi: \alpha \mapsto \psi_1(\alpha)[A^0 * U^0 \restriction \alpha]$. 
    \item $2^\omega = \theta$ in $V$.
    \item $\A^0 * \U^0 \restriction \theta$ is $\theta$-cc in $V_0$.
    \item The poset ${\FL}^0$ is  $\theta$-directed closed in $V_0[A^0 * U^0 \restriction \theta]$.
\end{itemize}

\smallskip
    
    \noindent Global notation: $\FL^0$\index[Notation]{$\FL^0$},

\smallskip

The main point is to establish that $\theta$ is indestructibly generically supercompact via
adding Cohen reals. Since the argument is essentially that for 
\cite[Claim 4.12]{NeemanUpto} with certain simplifications, we have relegated
it to Appendix \ref{appendixA}. For use in Lemma \ref{selection}, we note that
if $\mathbb Q \in V$ is $<\theta$-directed closed and our goal is generic supercompactness for $\theta$ via
Cohen reals in the extension by $\Q$, then we lift $i:V_0 \rightarrow N_0$ where $i$ witnesses 
a high degree of supercompactness for $\theta$ in $V_0$, and the forcing at $\theta$
in the second coordinate of $i(\A^0 * \U^*)$ is $\FL^0 * \Q$.

    It remains to lift $j_0$ onto $V$, which is comparatively straightforward. Let
    $V_1 = V_0[A^0 * U^0 \restriction \theta]$ and $M_1 = M_0[A^0 * U^0 \restriction \theta]$,
    so that easily $j_0$ lifts to $j_1: V_1 \rightarrow M_1$. It is easy to verify that
    $j_1( {\FL}^0 ) \restriction \delta = {\FL}^0$.
    We construct a compatible generic object $L^* \in V_1$ for $j_1( {\FL}^0 )$ as follows:
    \begin{itemize}
    \item $L^* \restriction \delta = L^0$.
    \item $L^* \restriction (\delta, j_0(\kappa))$ is constructed by counting antichains and closure.
    \item $L^* \restriction (j_0(\kappa), j_0(\delta))$ is constructed by counting antichains and closure,
      working below a master condition chosen as a lower bound for $j_0[L^0 \restriction (\kappa, \delta)]$.
    \end{itemize}
    
    Now we lift as usual to get $j: V \rightarrow M = M_1[L^*] = M_1[L^0][L^* \restriction (\delta, j(\delta))]$.
    Each $\kappa_\alpha$ is (indestructibly) supercompact in $M$  by the Levy-Solovay theorem and the fact that $L^*$ is generic for a 
    Laver preparation over $M_1$.

 \end{proof}

\begin{remark} If $\lambda < \delta$ and $\lambda$ is supercompact in either $V$ or $M$, and
  $\alpha$ is such that $\lambda < \kappa_\alpha$, then (by the agreement between $V$ and
  $M$) $\lambda$ is $\kappa_\alpha$-supercompact in both $V$ and $M$. Since $\kappa_\alpha$ is supercompact
  in both $V$ and $M$, $\lambda$ is supercompact both in $V$ and in $M$. 
\end{remark} 

\begin{remark} The reader may be wondering why we need the $\kappa_\alpha$'s to be supercompact in $M$.
  The point is that we will eventually be doing a version of Prikry forcing at $\kappa$, so that
  each Prikry point $\tau$ comes with reflections of the $\kappa_\alpha$'s which are fully supercompact.
  This is convenient because when $\tau$ and $\tau^*$ are successive Prikry points, so that $\tau^*$
  is far above the reflections of the $\kappa_\alpha$'s attached to $\tau$, we need these reflections to
  be supercompact for a long way past $\tau^*$. 
\end{remark}

\begin{remark} Our starting hypotheses are consistent relative to the existence of a $2$-huge cardinal,
  and in fact relative to the hypothesis that there is a cardinal which is both
  supercompact and huge, which we will show is weaker.
  \begin{itemize}
  \item Let $\kappa$ be $2$-huge, and fix $i: V \rightarrow N$ such that $\crit(i) = \kappa$,
    $i(\kappa) = \lambda$, $i(\lambda) = \mu$ and ${}^\mu N \subseteq N$. Then
    easily $V_\mu \models \mbox{``$\kappa$ is huge with target $\lambda$''}$.
    Also $\kappa$ is supercompact up to $\lambda$, by elementarity
    and closure $\lambda$ is supercompact up to $\mu$, so
    that $\kappa$ is supercompact up to $\mu$ and hence
    $V_\mu \models \mbox{``$\kappa$ is supercompact''}$.
  \item   Suppose now that $\kappa$ is supercompact and also is huge with target $\lambda$,
    as witnessed by $i: V \rightarrow N$ with $\crit(i) = \kappa$, $i(\kappa) = \lambda$,
    ${}^\lambda N \subseteq N$. By elementarity, $\lambda$ is supercompact in $N$.
    By the agreement between $V$ and $N$, $\kappa$ is supercompact up to $\lambda$ in $N$,
    so by reflection there are unboundedly many $\alpha < \kappa$ with $\alpha$ supercompact
    up to $\kappa$. Applying $i$, in $N$ there are unboundedly many $\beta < \lambda$ which
    are supercompact up to $\lambda$. Let $B$ be the set of such $\beta$, where 
    since $\lambda$ is supercompact in $N$ it follows that every $\beta \in B$ is supercompact in $N$.
    
    For any $\eta$ with $\kappa < \eta < \lambda$, let $U$ be the supercompactness measure
    on $P_\kappa \eta$ derived from $i$. It is easy to see that every $\beta \in B \cap \eta$
    is supercompact in $Ult(V, U)$. In the universe $V_\lambda$ every $\beta \in B$ is supercompact,
    and every $\beta \in B \cap \eta$ is supercompact in the ultrapower by $U$. It is now easy
    to get the starting hypotheses. 
  \end{itemize}
\end{remark}

\subsection{$\raux(\lambda)$ and $\laux(\rho, \lambda)$} \label{Rlambda}

We now work in the universe $V$ constructed in the last section, and construct auxiliary posets
$\raux(\lambda)$ and $\laux(\rho, \lambda)$. The subscript ``aux'' is to underline that we
will not actually force with these posets during the main construction. Their role is to help
us choose parameters for the main construction, which we will do in Section \ref{choiceofrho}.

Let $\lambda$ be a supercompact cardinal with $\lambda < \delta$.  
We define:
\begin{itemize}
\item $\lambda_0 = \lambda$.
\item For $n < \omega$, $\lambda_{n+1}$ is the least supercompact cardinal greater than $\lambda_n$. 
\item $\lambda_\omega = \sup_{n<\omega} \lambda_n$.
\item $\lambda_{\omega+1} = \lambda_\omega^+$.
\item For $0 < n < \omega$, $\lambda_{\omega+n+1}$ is the least supercompact cardinal greater than $\lambda_{\omega+n}$.
\end{itemize}

In $V$ we define a poset $\raux(\lambda)$ to be the product of the following three posets:
\begin{enumerate}
\item  $\East^{E_0}(\lambda_{\omega+1}, <\lambda_{\omega+2}) \times \prod_{n<\omega}\East^{E_0}(\lambda_n,<\lambda_{n+1})$.
  Here $\East$ is the Easton collapse defined above in Section \ref{projections} and $E_0$ is the set of inaccessible closure
  points in the interval $(\theta, \delta)$ of our initial Laver function $\phi_0$: note that elements of $E_0$
  are inaccessible closure points of $\phi$, $\dom(\phi) \subseteq E_0$, and $E_0$ is stationary in every supercompact cardinal up to $\delta$.

\item $\Coll(\lambda_{\omega+1},<\lambda_{\omega+2}) \times \prod_{n<\omega}
\Coll(\lambda_n,<\lambda_{n+1})$.

\item $\Add(\lambda_{17}, \lambda_{\omega+2})
\times \prod_{n < \omega} \Add(\lambda_n,\lambda_{n+2}) \times \Add(\lambda_{\omega+1}, \lambda_{\omega+2})$. 
\end{enumerate}

\smallskip

\noindent Global notation: $\lambda_i$\index[Notation]{$\lambda_i$},
$\raux(\lambda)$\index[Notation]{$\raux(\lambda)$}

\smallskip

Recall that in $V_0$ the cardinal $\theta$ is supercompact, and $\phi_0 \restriction \theta$ is
an indestructible Laver function for $\theta$.
We claim that there are many cardinals $\rho < \theta$ such that
in $V_0$ the cardinal $\rho$ is a limit of $\omega$ many inaccessible cardinals,
and there is an active stage $\bar \theta < \rho$ of the preparation forcing from Section \ref{preparation} 
such that $\rho$ becomes
an $\omega$-successor cardinal in $V_0[A^0 \restriction \bar\theta * (U^0 \restriction \bar\theta + 1)]$. 
To see this let $\rho'$ be the limit of the first $\omega$ $V_0$-inaccessible cardinals greater than $\theta$.
Use the guessing property of $\phi_0 \restriction \theta$ to anticipate a suitable $<\theta$-directed closed collapsing forcing
defined in $V_0[A^0 * U^0 \restriction \theta]$ which makes $\rho'$ into an $\omega$-successor cardinal.

Let $\mathrm{Index}$ be the set of all such $\rho$.
For each $\rho \in \mathrm{Index}$ let $\bar\theta(\rho)$ be the least ordinal such that
$\rho$ is an $\omega$-successor cardinal in $V_0[A^0 \restriction \bar\theta(\rho) * (U^0 \restriction \bar\theta(\rho) + 1)]$,
let $W(\rho) = V_0[A^0 \restriction \bar\theta(\rho) * (U^0 \restriction \bar\theta(\rho) + 1)]$,
and define $\laux(\rho,\lambda)$ to be the poset $\Coll(\omega,\rho) \times \Coll^{W(\rho)}(\rho^+,\lambda_1)$.
It is routine to check that $\Coll^{W(\rho)}(\rho^+,\lambda_1)$ is $\rho$-distributive in any $\lambda$-closed extension of $V$,
a fact which will be used in the proof of Lemma \ref{selection}. 

\smallskip

\noindent Global notation: 
$\mathrm{Index}$\index[Notation]{$\mathrm{Index}$},
$\bar\theta(\rho)$\index[Notation]{$\bar\theta(\rho)$},
$W(\rho)$\index[Notation]{$W(\rho)$},
$\laux(\rho, \lambda)$\index[Notation]{$\laux(\rho, \lambda)$} 

\smallskip

\begin{remark} The proof of Lemma \ref{selection} uses ideas from unpublished work of Hayut. 
\end{remark}

\begin{lemma}\label{selection} For each supercompact cardinal
 $\lambda$ with  $\kappa < \lambda < \delta$,  there is $\rho \in \mathrm{Index}$ such that
$\forces^V_{\laux(\rho,\lambda) \times \raux(\lambda)} \mbox{``the tree property holds at $\lambda_{\omega+1}$''}$  \end{lemma}

\begin{proof}
  For technical reasons we will prove a slightly different (but equivalent) version of the conclusion.
  Let $\raux'(\lambda)$ be the result of replacing $\Add(\lambda_{17}, \lambda_{\omega+2})$
  by $\Add(\lambda_{17}, \lambda_{\omega+1})$ in the product that defines $\raux(\lambda)$.
  We will show that for some $\rho$,
  $\forces^V_{\laux(\rho,\lambda) \times \raux'(\lambda)} \mbox{``the tree property holds at $\lambda_{\omega+1}$''}$.
  This is good enough because if $H \times G$ is $\laux(\rho,\lambda) \times \raux(\lambda)$-generic over $V$
  and $T \in V[H \times G]$ with $T$ a $\lambda_{\omega+1}$-tree,
  then by chain condition and homogeneity there is a submodel $V[H \times G']$ where
  $T \in V[H \times G']$ and $H \times G'$ is $\laux(\rho,\lambda) \times \raux'(\lambda)$-generic over $V$

Let $G$ be $\raux'(\lambda)$-generic over $V$. 
  We will use Fact \ref{Itay3.10} in $V[G]$  with the parameters set as follows:
  \begin{itemize}
  \item $\kappa_2$ is $\theta$
  \item $\kappa_n$ is $\lambda_{n-3}$ for $n > 2$.
  \item $\nu$ is $\lambda_\omega$ and $\nu^+$ is $\lambda_{\omega+1}$. 
  \item Index is the set Index we just defined. 
  \item $M(\rho)$ is $\laux(\rho,\lambda)$. 
  \end{itemize}
  
  Once we have verified that the hypotheses of Fact \ref{Itay3.10} hold in $V[G]$
  the conclusion will be immediate.
  For the first hypothesis, let $n \ge 18$ and observe that
  $\raux'(\lambda)$ can be factored as $\R_0 \times \R_1$ where $\R_0$ is $\lambda_{n+2}$-cc
  and $\R_1$ is $<\lambda_{n+2}$-directed closed. We decompose $G$ accordingly as $G_0 \times G_1$.
  By indestructibility, there is an embedding
  $j$ defined in $V[G_1]$ witnessing that $\lambda_{n+2}$ is $\nu^+$-supercompact in that model.
  Forcing over $V[G_1]$ with $\P = j(\R_0)/j[G_0]$ will add a generic object which enables us to lift
  $j$ onto $V[G]$.

  Clearly $\R_0$ is a product of terms of the form $\Coll(\lambda_m, < \lambda_{m+1})$, $East^{E_0}(\lambda_m, < \lambda_{m+1})$ and $\Add(\lambda_m, \lambda_{m+2})$ 
  for $m \le n+1$, together with $\Add(\lambda_{17}, \lambda_{\omega+1})$.
  Since many factors in $\R_0$ are fixed by $j$, the corresponding factors in $j(R_0)/G_0$ are
  trivial. It follows that  $j(\R_0)/G_0$ is the product of the factors $\Coll(\lambda_{n+1}, [\lambda_{n+2}, j(\lambda_{n+2}))$,
    $East^{j(E_0)}(\lambda_{n+1}, [\lambda_{n+2}, j(\lambda_{n+2}))$,
    $\Add(\lambda_n, j(\lambda_{n+2}) - \lambda_{n+2})$, $\Add(\lambda_{n+1}, j(\lambda_{n+3}) - j[\lambda_{n+3}])$  
   and $\Add(\lambda_{17}, j(\lambda_{\omega+1}) - j[\lambda_{\omega+1}])$.

   Now we let $\Q$ be the product of $\lambda_n$ copies of $j(\R_0)/G_0$ with the following supports:
   full support for the $<\lambda_{n+1}$-closed components,
   supports of size less than $\lambda_n$ for the
   components of form $\Add(\lambda_n, j(\lambda_{n+2}) - \lambda_{n+2})$  and
  supports of size less than $\lambda_{17}$ for the
  components of form $\Add(\lambda_{17}, j(\lambda_{\omega+1}) - j[\lambda_{\omega+1}])$.
  It is routine to check that $\Q$ preserves cardinals up to and including $\lambda_{n+1}$ and
  forces that $\cf(\lambda_{\omega+1}) = \lambda_{n+1}$, so that $\Q$ is as required.

  For the second hypothesis, we will use the indestructible generic supercompactness
  of $\theta$ in $V$ secured by Lemma \ref{fromVzerotoV}
  to define a certain generic embedding, and then reflect the
  existence of this embedding to a well-chosen elementary substructure $X \prec R$
  where $R$ is a suitable rank initial segment of $V[G]$.
  
  More precisely, let $\pi: V[G] \rightarrow V^*$ be a generic embedding added by
  the forcing poset $\Add(\omega, \pi(\theta) - \theta)$ such that:
  \begin{itemize}
  \item $\crit(\pi) = \theta$
  \item $\pi(\theta) > \lambda_{\omega+1}$
  \item $\pi$ is discontinuous at $\lambda_{\omega+1}$.
  \end{itemize}
  From the proof of Lemma \ref{fromVzerotoV} we recall that
\begin{itemize} 
\item $V = V_0[A^0 * U^0 \restriction \theta * L^0]$, so that
  $V[G] = V_0[A^0 * U^0 \restriction \theta * L^0 * G]$. 
\item  $\pi$ is a lift of a supercompactness embedding $i:V_0 \rightarrow N_0$
  with critical point $\theta$ defined  in $V_0$, with the property that 
  the forcing at coordinate $\theta$ in $i(\U^0)$ is  $\FL^0 * \raux'(\lambda)$.
\item  The embedding $i$  may witness an arbitrarily high degree of supercompactness for $\theta$
  in $V_0$.
\end{itemize}

  We claim that $\lambda_\omega \in \pi(\mathrm{Index})$: this is easy because $\lambda_\omega$ is a limit
  of supercompact cardinals in $N_0[A^0 * U^0]$ but becomes $\lambda^{+\omega}$ in
  $N_0[A^0 * U^0 * L^0 * G]$. By definition
  $\pi(\FM)(\lambda_\omega) = \pi(\laux)(\lambda_\omega, \pi(\lambda)) = \Coll(\omega, \lambda_\omega) \times \Coll^{N_0[A^0 * U^0 * L^0 * G]}(\lambda_{\omega+1}, \pi(\lambda_1))$.
  Recalling that $V[G] = V_0[A^0 * U^0 * L^0 * G]$ and that $i$ can witness arbitrary levels of supercompactness,
  we may arrange that $\Coll^{N_0[A^0 * U^0 * L^0 * G]}(\lambda_{\omega+1}, i(\lambda_1)) = \Coll^{V[G]}(\lambda_{\omega+1}, \pi(\lambda_1))$,
  in particular it is defined and $\lambda_\omega$-closed  in $V[G]$.

  In summary, we have shown that there is a generic embedding $\pi:V[G] \rightarrow V^*$ added
  by $\Add(\omega, \pi(\theta) - \theta)$ such that
  $\crit(\pi) = \theta$, $\pi(\theta) > \lambda_{\omega+1}$, $\pi$ is discontinuous at $\lambda_{\omega+1}$,
  $\lambda_\omega \in \pi(\mathrm{Index})$ and $\pi(\laux)(\lambda_\omega, \pi(\lambda)) =  \Coll(\omega, \lambda_\omega) \times \Coll^{V[G]}(\lambda_{\omega+1}, \pi(\lambda_1))$.
  Observe that $2^{\lambda_\omega} = \lambda_{\omega+1}$ in $V[G]$ (it was for this reason that we replaced
  $\raux(\lambda)$ by $\raux'(\lambda)$). We choose $R$ a long enough rank initial segment of $V[G]$ that 
  for any algebra of finitary functions on $R$, 
  we may find $X \prec R$ which has size $\lambda_{\omega+1}$, is closed
  under $\lambda_\omega$-sequences, and reflects the statement asserting the existence of a suitable generic embedding $\pi$
 
  Let $M$ be the collapse of $X$, and let $A \times h$ be generic over $V[G]$ for
  $\P_X = \Add(\omega, \pi(\theta) - \theta) \times \Coll(\omega, \lambda_\omega)$. 
  Using $A$, we may define in $M[A]$ a generic embedding $\pi_X: M \rightarrow M^* \subseteq M[A]$
  such that $\crit(\pi_X) = \theta$, $\pi_X(\theta) > \lambda_{\omega+1}$, $\pi_X$ is discontinuous at $\lambda_{\omega+1}$,
  $\lambda_\omega \in \pi_X(\mathrm{Index})$ and $\pi_X(\laux)(\lambda_\omega, \pi_X(\lambda)) =  \Coll(\omega, \lambda_\omega) \times \Coll^M(\lambda_{\omega+1}, \pi_X(\lambda_1))$.
  Since $\vert M \vert = \lambda_{\omega+1}$ and ${}^{\lambda_\omega} M \subseteq M$, we may build
  a filter $C \in V[G]$ which is generic over $M$  for $\Coll^M(\lambda_{\omega+1}, \pi_X(\lambda_1))$.

  Now $M[C] \subseteq V[G]$ and $A \times h$ is generic over $V[G]$, so $A \times h$ is generic over $M[C]$,
  and since $M^* \subseteq M[A]$ we see that $h \times C$ is generic over $M^*$. 
  It follows that forcing over $V[G]$ with the $\lambda_{\omega+1}$-Knaster poset $\P_X =\Add(\omega, \pi_X(\theta) - \theta) \times \Coll(\omega, \lambda_\omega)$
  has added the embedding $\pi_X:M \rightarrow M^*$ and a filter $h \times C$ which is $\pi_X(\laux)(\lambda_\omega, \pi_X(\lambda))$-generic over $M^*$. 
  Since we constructed $X$ to be closed under an arbitrary algebra on $R$, there are stationarily many $X$ and we have fulfilled
  the second clause in the hypotheses of Fact \ref{Itay3.10}.   
\end{proof} 

\subsection{Selecting $\rho$} \label{choiceofrho} 
Using the fact that there are $\theta^+$
supercompact cardinals above $\kappa$, we choose supercompact cardinals $\lambda^a$ and $\lambda^b$
above $\kappa$ such that $\lambda^a_{\omega+3} < \lambda^b$
and the cardinals $\lambda^a, \lambda^b$ select the same cardinal $\rho$
from Lemma \ref{selection}.  We can assume that $(\rho,\lambda^{a},\lambda^{b})$
is the lexicographically least such triple with this property: recalling that
$j : V \to M$ is a $\delta^+$-supercompactness embedding with critical point $\kappa$,
 we see that $(\rho,\lambda^{a},\lambda^{b})$ is definable from $\kappa$ in $M$ using the same definition.

Having fixed $\rho$, we also fix some related parameters. We set  $\bar\theta = \bar\theta(\rho)$
and ${\bar W} = W(\rho)$. It follows that for $\lambda = \lambda^a, \lambda^b$
we have:
\begin{itemize}
\item $\laux(\rho,\lambda) = \Coll(\omega,\rho) \times \Coll^{\bar W}(\rho^+,\lambda_1)$.
\item It forced by $\laux(\rho, \lambda) \times \raux(\lambda)$ that the tree property
  holds at $\lambda_{\omega+1}$. 
\end{itemize}

It follows that there is  a measure one set of points $\tau$ below
$\kappa$ with reflected versions $\Lambda^a(\tau)$ and $\Lambda^b(\tau)$
of the cardinals $\lambda^{a}$ and $\lambda^{b}$.
To be more specific:
\begin{enumerate}
\item    $j(\Lambda^x_i)(\kappa) = \lambda^x_i$ 
  for $x \in \{a, b\}$. 
\item $\theta < \tau < \Lambda^a(\tau) < \Lambda^b(\tau) < \kappa$.
  where $\Lambda^a(\tau)$ and $\Lambda^b(\tau)$ are supercompact.
\item Setting $\Lambda^x_i(\tau) = \Lambda^x(\tau)_i$ for $x \in \{ a, b \}$ and
  $i < \omega + \omega$, $\Lambda^a_{\omega + 3}(\tau) < \Lambda^b(\tau)$.
\item  It is forced by
  $\laux(\rho, \Lambda^x(\tau)) \times \raux(\Lambda^x(\tau))$ that the tree property holds at
  $\Lambda^x_{\omega+1}(\tau)$ for $x \in \{a, b\}$. 
\end{enumerate}

\smallskip

\noindent Global notation: $\lambda^a$\index[Notation]{$\lambda^a$},
$\lambda^b$\index[Notation]{$\lambda^b$},
$\rho$\index[Notation]{$\rho$},
$\bar\theta$\index[Notation]{$\bar\theta$},
$\bar{W}$\index[Notation]{$\bar{W}$},
$\Lambda^a(\tau)$\index[Notation]{$\Lambda^a(\tau)$},
$\Lambda^b(\tau)$\index[Notation]{$\Lambda^b(\tau)$},
$\Lambda^a_i(\tau)$\index[Notation]{$\Lambda^a_i(\tau)$},
$\Lambda^b_i(\tau)$\index[Notation]{$\Lambda^b_i(\tau)$}  
  
\smallskip

\section{More preparation} \label{moreprep}

Let  $Y$ be the set of supercompact cardinals $\tau$ less than $\kappa$
which are such that $\Lambda^a(\tau)$ and $\Lambda^b(\tau)$ are defined,
and 
are closed under the function $\sigma \mapsto \Lambda^b_{\omega+3}(\sigma)$.
 We define an Easton support iteration $\FL$ which is nontrivial only at each $\tau \in Y \cup \{\kappa\}$.
 For ease of notation we specify the forcing at step $\kappa$ and
note that the forcing at $\tau$ can be obtained by replacing $\kappa$ with
$\tau$ and $\lambda^z_i$ by $\Lambda^z_i(\tau)$ (for $z \in \{a , b \}$ and
$i \in \omega + \omega$).

\smallskip

\noindent Global notation: $Y$\index[Notation]{$Y$},  $\FL$\index[Notation]{$\FL$}

\smallskip

The forcing $\FL(\tau)$ at stage $\tau$ will be  $<\Lambda^a_{17}(\tau)$-closed.
Since $2^\omega = \theta$ in $V$, and $Y$ is a set of supercompact cardinals,
the forcing $\FL$ will be much more than $\theta^+$-closed. This will be important
in Section \ref{GroupIII}.

Of course we define the forcing at $\kappa$ in $V[L \restriction \kappa]$ where $L \restriction \kappa$ is $\FL_\kappa$-generic.
The preparation forcing at $\kappa$ will be defined in stages, and will ultimately
have components $\FL^b$, $\I^b$, and $\A_e \times \J^c$. 

\smallskip

\noindent Global notation: $\FL^b$\index[Notation]{$\FL^b$},
                           $\I^b$\index[Notation]{$\I^b$},
                           $\A_e$\index[Notation]{$\A_e$},
                           $\J^c$\index[Notation]{$\J^c$}
\smallskip

Note that by Fact \ref{superdestructible},
the cardinals $\lambda^b_n$ are supercompact but no longer indestructible in $V[L \restriction \kappa]$.
Let ${\FL}^b \in V[L \restriction \kappa]$ be a Laver preparation
for the interval $(\lambda^a_{\omega+1}, \lambda^b_\omega)$,  defined using
the Laver function $\alpha \mapsto \phi(\alpha)[L \restriction \kappa]$ on this interval.
Let $L^b$ be ${\FL}^b$-generic over $V[L \restriction \kappa]$,
and let $\psi$ be the  universal indestructible Laver function
added by $L^b$ on the interval $(\lambda^a_{\omega+1}, \lambda^b_\omega)$.

\smallskip

\noindent Global notation: $\psi$\index[Notation]{$\psi$}

\smallskip

Working in $V[(L \restriction \kappa) * L^b]$, we define posets
$\A^b$, $\B^b$, $\C^b$, $\S^b =  ({\C}^b)^{+A^b * U^b}$ following the recipe
in Section \ref{abcu} with the parameters set as follows:
\begin{itemize}
\item $\mu_0 = \lambda^a_{17}$, $\mu_1 = \lambda^a_{\omega+1}$, $\mu_2= \lambda^a_{\omega+2}$,  $\mu_{n+3} = \lambda^b_n$ for
  $n < \omega$.
\item The universal indestructible Laver function is the function $\psi$ which we just added using $L^b$.
\end{itemize}

\smallskip

\noindent Global notation: $\A^b$\index[Notation]{$\A^b$},
$\B^b$\index[Notation]{$\B^b$}, $\C^b$\index[Notation]{$\C^b$},
$\S^b$\index[Notation]{$\S^b$}

\smallskip

\begin{remark} \label{abzeroone} Since $\FL^b$ is defined on the interval $(\lambda^a_{\omega+1}, \lambda^b_\omega)$,
  it is $<\lambda^a_{\omega+1}$-closed, so that
  $\A^b_0 = \Add^{V[L \restriction \kappa]}(\mu_0, [\mu_1, \mu_2))$ and $\A^b_1 = \Add^{V[L \restriction \kappa]}(\mu_1, [\mu_2,\mu_3))$.
\end{remark}

Let $I^b$ be generic over $V[(L \restriction \kappa) * L^b]$ for
$\I^b = \A^b * \U^b * \S^b$,
 where $I^b$ decomposes in the obvious way as $A^b * U^b * S^b$. 
 For the record, in $V[L \restriction \kappa * L^b * I^b]$ we have the
 following situation:
 \begin{itemize}
 \item The cardinals $\lambda^a_{\omega+1}, \lambda^a_{\omega+2}, \lambda^b_0, \lambda^b_1, \ldots \lambda^b_\omega, \lambda^b_{\omega+1}$ form
   a block of $\omega + 2$ consecutive cardinals.
 \item
    $2^{\lambda^a_{17}} = \lambda^a_{\omega+2}$, $2^{\lambda^a_{\omega+1}} = \lambda^b_0$,
   $2^{\lambda^a_{\omega+2}} = \lambda^b_1$, $2^{\lambda^b_n} = \lambda^b_{n+2}$ for $n < \omega$.
\end{itemize} 

\smallskip
 
\noindent Global notation:
 $\I^b$\index[Notation]{$\I^b$},
 $I^b$\index[Notation]{$I^b$}, 
 $A^b$\index[Notation]{$A^b$},
 $U^b$\index[Notation]{$U^b$},
 $S^b$\index[Notation]{$S^b$}

\smallskip
 
 Working over the model $V[(L \restriction \kappa) * L^b * I^b]$
 (but using some Cohen posets defined in inner models of this model)
 we will define a poset $\A_e \times \J^c$,
 where $\A_e = \Add^{V[(L \restriction \kappa) * L^b]}(\lambda^b_{17}, \lambda^b_{\omega+3})$.
 We digress briefly to prove that $\A_e$ has reasonable chain condition and distributivity
 properties in $V[(L \restriction \kappa) * L^b * I^b]$. The point of defining
 $\A_e$ in the submodel $V[(L \restriction \kappa) * L^b]$ is that after forcing with
 $\I^b$ we have $2^{\lambda^b_{16}} = \lambda^b_{18}$, so that $\Add^{V[(L \restriction \kappa) * L^b * I^b]}(\lambda^b_{17}, \lambda^b_{\omega+3})$
 collapses $\lambda^b_{18}$. 
 \begin{lemma} \label{samplecohen} 
   $\A_e$ is $<\lambda^b_{17}$-distributive and $\lambda^b_{18}$-Knaster in $V[(L \restriction \kappa) * L^b * I^b]$. 
\end{lemma}  
\begin{proof} By item \ref{portmanteau3} of Lemma \ref{portmanteaulemma}, we may force to expand
  $V[(L \restriction \kappa) * L^b * I^b][A_e]$ to $V[(L \restriction \kappa) * L^b][P^b_0 \times P^b_1 \times A_e]$,
  where $\P^b_0, \P^b_1 \in  V[(L \restriction \kappa) * L^b]$, with
  $\P^b_0$ a  $\lambda^b_{17}$-cc initial segment of $\I^b$
  and $\P^b_0$ being $<\lambda^b_{17}$-closed  in $V[(L \restriction \kappa) * L^b]$. 
  By Easton's Lemma applied to $\P^b_0$ and $\P^b_1 \times \A_e$,
  all $<\lambda^b_{17}$-sequences of ordinals in $V[(L \restriction \kappa) * L^b * I^b][A_e]$ lie in
  $V[(L \restriction \kappa) * L^b][P^b_0] \subseteq V[(L \restriction \kappa) * L^b * I^b]$. 

  Since $\lambda^b_{18}$ is supercompact in $V[L \restriction \kappa][L^b]$,
  it follows that in this model $\eta^{< \lambda^b_{17}} < \lambda^b_{18}$ for all
  $\eta < \lambda^b_{18}$. By item \ref{aus-covering} of Lemma \ref{rightafterportmanteau},
  $V[(L \restriction \kappa) * L^b * I^b]$ is a $(\lambda^b_{17}, \lambda^b_{18})$-good extension of
  $V[(L \restriction \kappa) * L^b]$. Appealing to Lemma \ref{Cohenrobust},
  $\A_e$ is $\lambda^b_{18}$-cc in $V[(L \restriction \kappa) * L^b * I^b]$.
\end{proof}

In the sequel, there will be many situations where we use Cohen conditions
chosen from inner models, for example the Cohen posets $\A^c_0$ and $\A^c_1$ used
below in the definition of $\J^c$. We generally leave the verification
of the needed chain condition and distributivity properties,
which can all be proved along the lines of of the proof of Lemma \ref{samplecohen}, to the reader.

 The generic functions added
by $\A_e$ will be used below in the lifting arguments of Section \ref{auxcomp}.
The poset $\J^c$ will be an initial segment of the kind of ``two-phase $\A * \U * \S$ construction'' discussed in Remark \ref{twostepaus},
using different cardinal parameters
from the ones we used for $\I^b$. We first force with a poset $\A^c_0 * \U^c_0 * \S^c_0$, and then do the rest of the
construction over the extension by $\A^c_0 * \U^c_0 * \S^c_0$: an important new point is that the remainder of the construction
now involves $S^c_0$. This will be used to get some extra closure in Lemma \ref{q0info} below. 

  To define $\J^c$ we proceed as follows: 
\begin{itemize}
\item $\mu_0 = \lambda^b_{17}$, $\mu_1 = \lambda^b_{\omega+1}$, $\mu_2 = \lambda^b_{\omega+2}$, $\mu_3 = \lambda^b_{\omega+3}$.

\item  $\J^c$ will have the form $(\A^c_0 * \U^c_0 * \S^c_0) * (\A^c_1 * \U^c_1 * \S^c_1)$, so all its components have supports
  contained in $\mu_3$. 

\item $\A^c_0 = \Add^{V[(L \restriction \kappa) * L^b]}(\mu_0, [\mu_1, \mu_2))$

\item  To define $\U^c_0$ we use a Laver function on $(\mu_1, \mu_2)$ derived from $\phi$
      as in Fact \ref{laversmallfact}, that is the function $\alpha \mapsto \phi(\alpha)[(L \restriction \kappa) * L^b * I^b]$
      defined at those $\alpha$ where this makes sense.

 \item $\B^c \restriction \mu_2$ and $\U^c_0 = \U^c \restriction \mu_2$ are defined as in Section \ref{abcu}.

 \item As we noted in Remark \ref{nzerospecial},
   $\S^c_0 = ({\C}^c_0)^{+A^c_0 * U^c_0 \restriction \mu_1} = ({\C}^c_0)^{+A^c_0}$,
   so that $\A^c_0 * \S^c_0$ is just the standard Mitchell forcing to force that $2^{\mu_0} = \mu_2 = \mu_1^+$
   and $\mu_2$ has the tree property. 

 \item  For the definition of $\A^c_1 * \U^c_1 * \S^c_1$, we work in $V[(L \restriction \kappa) * L^b * I^b][A^c_0 * U^c_0 * S^c_0]$.
   We use the Laver function $\alpha \mapsto \phi(\alpha)[(L \restriction \kappa) * L^b * I^b][A^c_0 * U^c_0 * S^c_0]$
   on the interval $(\mu_2, \mu_3)$. The posets $\B^c_1$ and $\C^c_1$ are defined
   in $V[(L \restriction \kappa) * L^b * I^b][A^c_0 * U^c_0 * S^c_0]$, so that for example a condition
   $b \in B^c_1$ has domain a subset of $[\mu_2, \mu_3)$ lying in $V[(L \restriction \kappa) * L^b * I^b][A^c_0 * U^c_0 * S^c_0]$,
   and $b(\alpha)$ is a name which lies in this model.  
   
 \item  $\A^c_1 = \Add^{V[(L \restriction \kappa) * L^b]}(\mu_1, [\mu_2,\mu_3))$, and we define $\U^c_1$ and $\S^c_1$ by feeding
   in information from $A^c_1$ working over the model $V[(L \restriction \kappa) * L^b * I^b][A^c_0 * U^c_0 * S^c_0]$.
   In particular  $\U^c_1$ and $\S^c_1$ are both defined in the model 
$V[(L \restriction \kappa)*L^b*I^b][A^c_{[0,1]} * U^c_0 * S^c_0]$.

\end{itemize}

 The last stage of the preparation forcing $\FL$ at $\kappa$ is to force with $\A_e \times \J^c$ over $V[L \restriction \kappa * L^b * I^b]$, where
   $\J^c = (\A^c_0 * \U^c_0 * \S^c_0) * (\A^c_1 * \U^c_1 * \S^c_1)$.
We write the generic object as $A_e \times J^c$ where $J^c = (A^c_0 * U^c_0 * S^c_0) * (A^c_1 * U^c_1 * S^c_1)$.
We note that $\FL$ is $\lambda^b_{\omega+3}$-cc.

 Again we record some information about cardinals and cardinal arithmetic.
 In $V[L \restriction \kappa * L^b * I^b * (A_e \times J^c)]$ we have:
\begin{itemize}
\item The cardinals $\lambda^a_{\omega+1}$, $\lambda^a_{\omega+2}$, $\lambda^b_0$, $\lambda^b_1, \ldots \lambda^b_\omega$,
  $\lambda^b_{\omega+1}$, $\lambda^b_{\omega+2}$, $\lambda^b_{\omega+3}$
  form    a block of $\omega + 4$ consecutive cardinals.
 \item
    $2^{\lambda^a_{17}} = \lambda^a_{\omega+2}$, $2^{\lambda^a_{\omega+1}} = \lambda^b_0$,
     $2^{\lambda^a_{\omega+2}} = \lambda^b_1$, $2^{\lambda^b_n} = \lambda^b_{n+2}$ for $n < 17$,
     $2^{\lambda^b_n} = \lambda^b_{\omega+3}$ for $17 \le n \le \omega +2$. 
\end{itemize}

Let $V[L]$ be the model obtained after forcing with $\FL$.
The generic object added by $L$  at a stage $\tau \in Y$ is written as
$L(\tau) = L^b(\tau) * I^b(\tau) * (A_e(\tau) \times J^c(\tau))$ with the obvious notation for the components
 of $I^b(\tau)$ and $J^c(\tau)$. 

 \smallskip

 \noindent Global notation: 
 $L(\tau)$\index[Notation]{$L(\tau)$},
 $L^b(\tau)$\index[Notation]{$L^b(\tau)$}, 
 $I^b(\tau)$\index[Notation]{$I^b(\tau)$},
 $A^e(\tau)$\index[Notation]{$A^e(\tau)$},
 $J^c(\tau)$\index[Notation]{$J^c(\tau)$}

\smallskip

We will ultimately do a lifting argument to show that $\kappa$ is still a large cardinal in $V[L]$.
This will enable us to choose some supercompactness measures and other data,
which will be ultimately be used to define the Prikry
forcing $\bar \P$ in Section \ref{prikryforcing}. The lifting argument involves some 
objects introduced in Section \ref{successiveprikry}, so we defer it until the start of Section \ref{auxcomp}.

\begin{remark} \label{Lclosureremark} It follows readily from Lemma \ref{modestclosure} that
  $\FL(\tau)$ is $<\Lambda^a_{17}(\tau)$-closed in $V[L \restriction \tau]$.
  For the purposes of Section \ref{GroupIV} we note that as
   a consequence all initial segments of $\FL$ are $\rho$-closed in $V$. 
\end{remark}

\section{The interleaved forcing posets} \label{interleaved} 

\subsection{Between successive Prikry points} \label{successiveprikry}

  We now work in $V[L]$ to define the forcing ${\Q}(\tau, \tau^*)$ which the Prikry-type forcing
  $\bar\P$ will interleave between successive Prikry points $\tau$ and $\tau^*$.
  A few points to note:
  \begin{itemize}
  \item The points $\tau$ and $\tau^*$ will be elements of $Y$.  
  \item The poset $\bar\P$ is defined in a certain generic extension $V[L][\Agg]$ of $V[L]$, but 
   each poset ${\Q}(\tau, \tau^*)$ will actually be 
   defined in the extension of $V$ by a  certain initial segment of $L$ which we specify shortly.
  \item The filter on ${\Q}(\tau, \tau^*)$ added by forcing with $\bar\P$
    will be generic over $V[L * \Agg * E]$, where $E$ is the product of the finitely many generic
    objects added by $\P$ for the preceding interleaved forcing posets.
  \item ${\Q}(\tau, \tau^*)$ will be quite large (bigger than $\tau^*$) and will have an effect on the
    universe past $\tau^*$, and by the same token $E$ will have an effect on the universe past $\tau$.
    On the other hand $\Q(\tau, \tau^*)$ does not start to have an effect till some way past $\tau$, so
    that if $\tau, \tau^*, \tau^{**}$ are successive points on the Prikry sequence then there is a large
    gap between the intervals where $\Q(\tau, \tau^*)$ and $\Q(\tau^*, \tau^{**})$ are each doing their work:
    this is crucial to later arguments, particularly in Section \ref{treeprop}.
  \item $\bar\P$ will also have to act between $\omega$ and the first Prikry point.
  This will require special treatment, see Section \ref{firstprikry} below.
  \end{itemize}

  Ultimately the Prikry-type forcing of Section \ref{prikryforcing} will add
  (mutually) generic objects over $V[L]$ for the posets $\Q(\tau_n, \tau_{n+1})$ where $\tau_n$ and $\tau_{n+1}$ are
  successive points on the Prikry sequence, together with a generic object for $\Q^*(\tau_0)$ where $\Q^*(\tau)$ is
  defined in Section \ref{firstprikry}. 

  Notation: In the sequel it will be convenient to have a compact notation for certain initial
  segments of $V[L]$. For $\tau \in Y \cup \{ \kappa \}$ we will let:
  \begin{itemize}
  \item $\Vl(\tau) = V[L \restriction \tau]$.
  \item $\Vlb(\tau) = V[L \restriction \tau][L^b(\tau)]$.
  \item  $V^{lbi}(\tau) = V[L \restriction \tau][L^b(\tau)][I^b(\tau)]$.
   \end{itemize}

  Global notation: $\Vl(\tau)$\index[Notation]{$\Vl(\tau)$},
  $\Vlb(\tau)$\index[Notation]{$\Vlb(\tau)$},
  $\Vlbi(\tau)$\index[Notation]{$\Vlbi(\tau)$}.

  Recall that:
  \begin{itemize}
  \item  Part of the the final step of the preparation at stage $\tau$ was a forcing $\J^c(\tau)$,
    which added a generic object $J^c(\tau)$ for an initial segment of
    the kind of two-phase construction discussed in Remark \ref{twostepaus}.
    The cardinal parameters were
    $\Lambda^b_{17}(\tau)$, $\Lambda^b_{\omega+1}(\tau)$, $\Lambda^b_{\omega+2}(\tau)$, $\Lambda^b_{\omega+3}(\tau)$.
  \item One of the first steps of the preparation at stage $\tau^*$ was to add a generic object
    $I^b(\tau^*)$ for a version of the construction of Section \ref{abcu} whose first few cardinal
    parameters were
    $\Lambda^a_{17}(\tau^*), \Lambda^a_{\omega+1}(\tau^*), \Lambda^a_{\omega+2}(\tau^*),  \Lambda^b_0(\tau^*)$: this
    was computed in $\Vlb(\tau^*)$. In particular we added a generic
    object $A^b_0(\tau^*)$ 
    where
    $\A^b_0(\tau^*) = \Add^{\Vlb(\tau^*)}(\Lambda^a_{17}(\tau^*), [\Lambda^a_{\omega+1}(\tau^*), \Lambda^a_{\omega+2}(\tau^*)))$.

  \item  $L^b(\tau^*)$ is generic over $\Vl(\tau^*)$ for a forcing which is sufficiently closed that actually
    $\A^b_0(\tau^*) = \Add^{\Vl(\tau^*)}(\Lambda^a_{17}(\tau^*), [\Lambda^a_{\omega+1}(\tau^*), \Lambda^a_{\omega+2}(\tau^*)))$.
  \end{itemize}

  As we construct ${\Q}(\tau, \tau^*)$, we will keep track of the models in which its various components 
  are computed. This information will be used later in Lemma \ref{fordistrib1variant}.
  We will also keep track of some closure properties of the components.
  This is mostly for use in Section \ref{GroupIV}, where 
  all we will need  is that certain components are $\rho$-closed. 
  ${\Q}(\tau, \tau^*)$ is the product of three factors $\Q_i(\tau, \tau^*)$ for $i < 3$.

\smallskip
  
\noindent Global notation: $\Q(\tau, \tau^*)$\index[Notation]{$\Q(\tau, \tau^*)$}

\smallskip

  The first factor $\Q_0(\tau, \tau^*)$ completes 
  $J^c(\tau)$ to a generic object for a certain forcing poset defined in the model $\Vlbi(\tau)$,
  which we now describe: 
\begin{itemize}
\item   The forcing poset is a two-phase $\A * \U * \S$ construction of the type discussed in Remark \ref{twostepaus}, 
  and it has the form $(\A^c_0(\tau) * \U^c_0(\tau) * \S^c_0(\tau)) * (\A^c_{[1, \omega)}(\tau, \tau^*)
  * \U^c_{[1, \omega)}(\tau, \tau^*) * \S^c_{[1, \omega)}(\tau, \tau^*))$, 
      where $\A^c_{[1, \omega)}(\tau, \tau^*) * \U^c_{[1, \omega)}(\tau, \tau^*) * \S^c_{[1, \omega)}(\tau, \tau^*)$
            is computed in the extension by $A^c_0(\tau) * U^c_0(\tau) * S^c_0(\tau)$.

\item The first two steps were added as the component $J^c(\tau)$ of $L(\tau)$,
  explicitly
  \[
  J^c(\tau) = J^c_0(\tau) * J^c_1(\tau) = (A^c_0(\tau) * U^c_0(\tau) * S^c_0(\tau)) * (A^c_1(\tau) * U^c_1(\tau) * S^c_1(\tau)).
  \]
            
\item  $\mu_0 = \Lambda^b_{17}(\tau)$,  $\mu_1 = \Lambda^b_{\omega+1}(\tau)$,  $\mu_2 = \Lambda^b_{\omega+2}(\tau)$,
  $\mu_3 = \Lambda^b_{\omega+3}(\tau)$, then $\mu_{4 + n} = \Lambda^a_n(\tau^*)$ for $n < \omega$.

\item  The forcing $\A^c_{[1, \omega)}(\tau, \tau^*) * \U^c_{[1, \omega)}(\tau, \tau^*) * \S^c_{[1, \omega)}(\tau, \tau^*)$ is computed  
 in the model $\Vlbi(\tau)[J^c_0(\tau)]$ with parameters set as follows:
\begin{itemize}

\item  $\A^c_1(\tau) = \Add^{\Vlb(\tau)}(\mu_1, [\mu_2,\mu_3))$.

\item $\A^c_n(\tau, \tau^*) = \Add^V(\mu_n, [\mu_{n+1}, \mu_{n+2}))$ for $2 \le n < \omega$.

\item We define $\B^c_{[1, \omega)}(\tau, \tau^*)$ and $\U^c_{[1, \omega)}(\tau, \tau^*)$ using the Laver function  
    $\alpha \mapsto \phi(\alpha)[L \restriction \tau * L^b(\tau) * I^b(\tau)][A^c_0(\tau) * S^c_0(\tau) * U^c_0(\tau)]$   
    on the interval $(\mu_2, \mu_\omega)$.
  \item  The supports of conditions in $\B^c_{[1, \omega)}(\tau, \tau^*)$ and $\C^c_{[1, \omega)}(\tau, \tau^*)$ are defined in
      $\Vlbi(\tau)[J^c_0(\tau)]$.
\end{itemize}
    
\end{itemize}   

\smallskip

\noindent Global notation: $\Q_0(\tau, \tau^*)$\index[Notation]{$\Q_0(\tau, \tau^*)$}

\smallskip

Keeping in mind that $\J^c(\tau)$ has already added 
 $J^c(\tau) = (A^c_0(\tau) * U^c_0(\tau) * S^c_0(\tau)) * (A^c_1(\tau)  * U^c_1(\tau) * S^c_1(\tau))$,
$\Q_0(\tau, \tau^*)$  will add a generic object $Q_0(\tau, \tau^*)$ composed of: $\A_n$-generic objects $A^c_n(\tau, \tau^*)$ for $n \ge 2$,
together with
generic objects $U^c_{[2, \omega)}(\tau, \tau^*)$ for  $\U^c_{[2,\omega)}$ and $S^c_{[2, \omega)}(\tau, \tau^*)$ for $\S^c_{[2, \omega)}$.

  The last claim in the following Lemma is similar to some closure facts from Neeman's paper \cite{NeemanUpto},
  notably Claim 4.7, but  the setting is a bit different and we give a few more details.      

\begin{lemma} \label{q0info}
  $\Q_0(\tau, \tau^*)$ is a forcing poset of cardinality
  $\Lambda^a_{\omega+1}(\tau^*)$ defined in the model $\Vlbi(\tau)[J^c(\tau)]$.
  $\A^c_{[2, \omega)}(\tau, \tau^*)$ is defined and $<\Lambda^b_{\omega + 2}(\tau)$-closed in $V$,
  and is $<\Lambda^b_{\omega + 2}(\tau)$-distributive in $\Vlbi(\tau)[J^c(\tau)]$.
   $\U^c_{[2, \omega)}(\tau, \tau^*) * \S^c_{[2, \omega)}(\tau, \tau^*)$ is defined and $<\Lambda^b_{\omega + 2}(\tau)$-closed
    in $\Vlbi(\tau)[J^c(\tau)][A^c_{[2, \omega)}(\tau, \tau^*)]$.     
\end{lemma}

\begin{proof}
  It is easy to see that $\FL \restriction \tau * \FL^b(\tau) * \I^b(\tau) * \J^c_0(\tau) * \A^c_1(\tau)$
  is $\Lambda^b_{\omega + 2}(\tau)$-cc in $V$. In the model
  $\Vlbi(\tau)[J^c_0(\tau)]$,
  $U^c_1(\tau) * S^c_1(\tau)$ is the projection of the $<\Lambda^b_{\omega+2}(\tau)$-closed poset
  $\B^c_1(\tau) \times \C^c_1(\tau)$. So by a suitable quotient-to-term forcing we may
  extend $\Vlbi(\tau)[J^c(\tau)]$ to
  $\Vlbi(\tau)[J^c_0(\tau) * A^c_1(\tau) \times T]$,
    where $T$ is generic for the term forcing $\termspace^V(\FL \restriction \tau * \FL^b(\tau) * \I^b(\tau) * \J^c_0(\tau), 
   \B^c_1(\tau) \times \C^c_1(\tau))$ which is $<\Lambda^b_{\omega+2}(\tau)$-closed in $V$. 
  By a standard application of  Easton's lemma, $\A^c_{[2, \omega)}(\tau, \tau^*)$
    is $<\Lambda^b_{\omega + 2}(\tau)$-distributive in $\Vlbi(\tau)[J^c(\tau)]$.
    
    For the closure of $\U^c_{[2, \omega)}(\tau, \tau^*) * \S^c_{[2, \omega)}(\tau, \tau^*)$,
    start by noting that by Corollary \ref{rightafterportmanteau} every $<\Lambda^b_{\omega + 2}(\tau)$ sequence of ordinals from
    $\Vlbi(\tau)[J^c(\tau)][A^c_{[2, \omega)}(\tau, \tau^*)]$ lies in the submodel
      $\Vlbi(\tau)[J^c_0(\tau) \times A^c_1(\tau)]$.

      Let $V' = \Vlbi(\tau)[J^c_0(\tau)]$.
      To lighten the notation we drop the parameters $\tau$ and $\tau^*$, and use the ``$\mu_i$ notation''
      for the cardinal parameters.  We will only prove closure for $\U^c_2$,
      since this proof contains all the ideas. Note that the underlying set $\B^c \restriction (\mu_3, \mu_4)$ of $\U^c_2$ lies in $V'$,
      the ordering on $\U^c_2$ is defined in $V'[A^c_{[1, 2]} * U^c_1]$, and the relevant decreasing
      $<\mu_2$-sequences from $\U^c_2$ lie in $V'[A^c_1]$.

      We work for the moment in $V'$.
      Let $\eta < \mu_2$ and let $(\dot b_i)_{i < \eta}$ be a  sequence of $\A^c_1$-names
      for elements of $\B^c \restriction (\mu_3, \mu_4)$,   where without
      loss of generality the trivial condition in $\A^c_{[1, 2]} * \U^c_1$ forces that
      the $b_i$'s form a decreasing sequence in $\U^c_2$. 

      We will construct $b \in \B^c \restriction (\mu_3, \mu_4)$ which is forced
      to be a lower bound for the $b_i$'s. We let $\dom(b)$ be the union over $i$ of the possible
      values of $\dom(b_i)$, where it is easy to see that this is an Easton subset of $(\mu_3, \mu_4)$. Suppose that
      $\alpha \in \dom(b)$, we have defined $b \restriction \alpha$, and
      $b \restriction \alpha$ is forced to be a lower bound for the $b_i \restriction \alpha$'s. 
      
      Force with $\A^c \restriction (\mu_2, \alpha) * \U^c \restriction (\mu_2, \alpha)$
      below the condition  $(0, b \restriction \alpha)$ to obtain
      a generic object $F^A_\alpha * F^U_\alpha$. Let
      $c_i = \dot b_i[F^A_\alpha \restriction (\mu_2, \mu_3)] \in \B^c \restriction (\mu_3, \mu_4)$,
      and let $d_i = c_i(\alpha)[F^A_\alpha * F^U_\alpha] \in \phi(\alpha)[F^A_\alpha * F^U_\alpha]$
      if $\alpha \in \dom(c_i)$. We note that $\dom(c_i)$ increases with $i$, so that
      either $d_i$ is never defined or it is defined for all large $i < \eta$. 

      Let $i < j < \eta$ where $d_i$ and $d_j$ are both defined, we claim that $d_j \le d_i$. 
           Since it is forced that the $b_i$'s are decreasing in $\U^c_2$,
      there is a condition in $F^A_\alpha * F^U_\alpha$ forcing that $c_j(\alpha) \le c_i(\alpha)$,
      and so $d_j \le d_i$ as required. Since $\phi(\alpha)$ is forced to be $<\alpha$-directed closed,
      we may choose $b(\alpha)$ as a name such that $(0, b \restriction \alpha)$ forces $b(\alpha) \le \dot d_i$
      for all $i$.

      Now let $\bar F = A^c_1 * A^c_2 * U^c_1$ be $\A^c_{[1, 2]} * \U^c_1$-generic over $V'$,
      let $\alpha \in \dom(b)$ and let $c_i = \dot b_i[A^c_1]$. By the induction hypothesis $b \restriction \alpha \le c_i \restriction \alpha$
      for all $i$ in the version of $\U^c_2$  computed by $V'[\bar F]$.
      If $\alpha \notin \dom(c_i)$ there is nothing to do, so assume that $\alpha \in \dom(c_i)$ and choose a condition
      $(a_1, a_2, u_1) \in \bar F$ where $a_1 \restriction (\mu_2, \mu_3)$ forces $\dot b_i = \check c_i$. 
      Consider the condition $(a_1, a_2 \restriction (\mu_3, \alpha), u_1, b \restriction \alpha)$:
      it forces that $\dot d_i = c_i(\alpha)$ by the choice of $a_1$, and so forces that
      $b(\alpha) \le c_i(\alpha)$ because it refines $(0, b \restriction \alpha)$.
      So $b \restriction \alpha +1 \le c_i \restriction \alpha + 1$, with
      $(a_1, a_2 \restriction (\mu_3, \alpha), u_1)$ as the witnessing condition at coordinate $\alpha$.
\end{proof} 

As we mentioned earlier in Section \ref{moreprep}, Lemma \ref{q0info} depends critically on the definition of
$\J^c$ as a two-phase construction where we defined everything past stage zero using $S^c_0$.

  Recall 
  that $\A^b_0(\tau^*) = \Add^{\Vl(\tau^*)}(\Lambda^a_{17}(\tau^*), [\Lambda^a_{\omega+1}(\tau^*), \Lambda^a_{\omega+2}(\tau^*)))$.
    By the discussion in Section \ref{termforcing}, we may force over
  $\Vl(\tau^*)[A^b_0(\tau^*)]$ to produce a generic object $A^V_0(\tau^*)$ for 
  $\Add^V(\Lambda^a_{17}(\tau^*), \Lambda^a_{\omega+2}(\tau^*))$ so that
  $\Vl(\tau^*) \subseteq \Vl(\tau^*)[A^b_0(\tau^*)] \subseteq V[(L \restriction \tau^*) \times A^V_0(\tau^*)]$.
  $\Q_1(\tau, \tau^*)$ is the ``quotient to term'' forcing to produce $A^V_0(\tau^*)$ from $A^b_0(\tau^*)$.

\smallskip
  
\noindent Global notation: $\Q_1(\tau, \tau^*)$\index[Notation]{$\Q_1(\tau, \tau^*)$}

\smallskip

  \begin{lemma} \label{q1info}
    $\Q_1(\tau, \tau^*)$  is defined in $\Vl(\tau^*)[A^b_0(\tau^*)]$ and
    has size $\Lambda^a_{\omega + 2}(\tau^*)$.
    
 \end{lemma}

\begin{proof} 
  Clearly the definition of  $\Q_1(\tau, \tau^*)$ only needs $L \restriction \tau^*$ and $A^b_0(\tau^*)$, and
  we can compute its cardinality by counting terms.
\end{proof}

  To define $\Q_2(\tau, \tau^*)$, recall that at stage $\tau$ in $\FL$ we force with
  $\FL(\tau) = \FL^b(\tau) * \I^b(\tau) * (\A_e(\tau) \times \J^c(\tau))$ over $V[L \restriction \tau]$.
  Of course $\FL \restriction (\tau, \tau^*) \in V[L \restriction \tau + 1]$, and
  is easily seen to be  $\tau^*$-cc forcing of cardinality $\tau^*$ which is at least $<\Lambda^b_{\omega+3}(\tau)$-directed closed.
  Working in  $\Vlbi(\tau)[J^c(\tau)]$ we can compute the term forcing
  $\T(\tau, \tau^*) = \termspace(\A_e(\tau), \FL \restriction (\tau, \tau^*))$.
  We note that $\T(\tau, \tau^*)$ is a $<\Lambda^b_{\omega+3}(\tau)$-directed closed 
  and  $\tau^*$-cc forcing poset of cardinality $\tau^*$. 
  $\Q_2(\tau, \tau^*)$ 
    is the quotient to term forcing (see Section \ref{termforcing}) to produce a generic object $T(\tau, \tau^*)$ for the term forcing
    such that $A_e(\tau) \times T(\tau, \tau^*)$ induces $A_e(\tau) * L \restriction (\tau, \tau^*)$:
    we write $Q_2(\tau, \tau^*)$ for the generic object for $\Q_2(\tau, \tau^*)$.
 
\smallskip

\noindent Global notation: $\Q_2(\tau, \tau^*)$\index[Notation]{$\Q_2(\tau, \tau^*)$}
$\T(\tau, \tau^*)$\index[Notation]{$\T(\tau, \tau^*)$}

\smallskip

\begin{lemma} \label{q2info}
  $\Q_2(\tau, \tau^*)$ is a forcing poset of cardinality $\tau^*$ defined in $\Vl(\tau^*)$.
\end{lemma}   
    
\begin{proof}
  Clearly the definition of  $\Q_2(\tau, \tau^*)$ only needs $L \restriction \tau^*$, and
  the cardinality can be calculated by counting terms.
\end{proof}   

\begin{remark} \label{q1and q2closure} By Remark \ref{Lclosureremark} and
  Lemma \ref{qtot-closed}, $\Q_1(\tau, \tau^*)$ and $\Q_2(\tau, \tau^*)$ are both $\rho$-closed. 
\end{remark}

It is clear from Lemmas \ref{q0info}, \ref{q1info} and \ref{q2info} that $\Q(\tau, \tau^*)$
is a forcing poset of cardinality $\Lambda^a_{\omega+2}(\tau^*)$
defined in $V[L \restriction \tau^* + 1]$.
    To clarify what $\Q(\tau, \tau^*)$ is doing we record some information about cardinals and cardinal arithmetic after
    forcing with this poset.
    Since $\Q(\tau, \tau^*) \in V[L \restriction \tau^*][A^b_0(\tau^*)]$
    and it has cardinality less than $\Lambda^b_0(\tau^*)$, to analyse the extension of $V[L]$ by $\Q(\tau, \tau^*)$
    it is sufficient to analyse the extension of $V[L \restriction \tau^* +1]$ by $\Q(\tau, \tau^*)$.

    Recall from Section \ref{moreprep} that in $V[L \restriction \tau^* +1]$:
\begin{itemize}
\item
   $\Lambda^a_{\omega+1}(\tau)$, $\Lambda^a_{\omega+2}(\tau)$, $\Lambda^b_0(\tau)$, $\Lambda^b_1(\tau), \ldots \Lambda^b_\omega(\tau)$,
  $\Lambda^b_{\omega+1}(\tau)$, $\Lambda^b_{\omega+2}(\tau)$, $\Lambda^b_{\omega+3}(\tau)$
  form   a block of $\omega + 4$ consecutive cardinals, and similarly for $\tau^*$.
 \item
   $2^{\Lambda^a_{17}(\tau)} = \Lambda^a_{\omega+2}(\tau)$, $2^{\Lambda^a_{\omega+1}(\tau)} = \Lambda^b_0(\tau)$,
     $2^{\Lambda^a_{\omega+2}}(\tau) = \Lambda^b_1(\tau)$, $2^{\Lambda^b_n(\tau)} = \Lambda^b_{n+2}(\tau)$ for $n < 17$,
   $2^{\Lambda^b_n(\tau)} = \Lambda^b_{\omega+3}(\tau)$ for $17 \le n \le \omega +2$, and similarly for $\tau^*$.
\end{itemize} 
After forcing with $\Q(\tau, \tau^*)$ we have that:
\begin{itemize}
\item  The cardinals $\Lambda^b_{\omega+3}(\tau)$ , $\Lambda^a_0(\tau^*)$, $\Lambda^a_1(\tau^*)$,
  $\Lambda^a_2(\tau^*), \ldots   \Lambda^a_\omega(\tau^*)$ form a block of $\omega$ successive cardinals.
\item $2^{\Lambda^b_{\omega+2}(\tau)} = \Lambda^a_0(\tau^*)$, $2^{\Lambda^b_{\omega+3}(\tau)} = \Lambda^a_1(\tau^*)$,
   $2^{\Lambda^a_n(\tau^*)} = \Lambda^a_{n+2}(\tau^*)$ for $0 \le n < 17$. 
\end{itemize}

To help analyse $\Q(\tau, \tau^*)$, we embed the generic extension by this poset into something
more tractable. This will be useful immediately in the proof of distributivity for $\Q(\tau, \tau^*)$
in Lemma \ref{qdistrib}, and again in Section \ref{auxcomp}.
The poset $\Q(\tau, \tau^*)$ is defined in $V[L \restriction \tau^*][A^b_0(\tau^*)]$,
but for our purposes we work over the slightly larger model $V[L \restriction \tau^* + 1]$.

\begin{lemma} \label{fordistrib1variant}
 Let $\tau, \tau^* \in Y$ with $\tau < \tau^*$ and let $Q$ be $\Q(\tau, \tau^*)$-generic over $V[L \restriction \tau^* + 1]$.
 Let $A^c_2(\tau, \tau^*)$
 be the $\Add^V(\Lambda^b_{\omega+2}(\tau), [\Lambda^b_{\omega+3}(\tau), \Lambda^a_0(\tau^*)))$-generic filter added by $Q$ as the $\A^c_2(\tau, \tau^*)$-component.
 Let $\lambda = \Lambda^b_{\omega+3}(\tau)$. Then in some generic extension of $V[L \restriction \tau^* +1][Q]$ there exists $L'$ such that:
 \begin{enumerate}
 \item $V[L \restriction \tau^* + 1][Q] \subseteq V[L \restriction \tau + 1 \times A^c_2(\tau, \tau^*) \times L']$.  
 \item $A^c_2(\tau, \tau^*) \times L'$ is generic for the product of $\A^c_2(\tau, \tau^*)$
   and some $<\lambda$-closed forcing $\FL'$ lying in $V$, where $\FL'$ has cardinality $\Lambda^b_{\omega+3}(\tau^*)$.
 \end{enumerate}
\end{lemma}

\begin{proof} Decompose $Q$ in the natural way as $Q_0 \times Q_1 \times Q_2$.
  We recall that $\Q_0(\tau, \tau^*) \in \Vlbi(\tau)[J^c(\tau)]$,
  $\Q_1(\tau, \tau^*) \in \Vl(\tau^*)[A^b_0(\tau^*)]$, and 
  $\Q_2(\tau, \tau^*) \in \Vl(\tau^*) = \Vlbi(\tau)[J^c(\tau)]
  [A_e(\tau) * L \restriction (\tau, \tau^*)]$.
 
  We recall also that $Q_0 = A^c_{[2, \omega)}(\tau, \tau^*) * (U^c_{[2, \omega)}(\tau, \tau^*) * S^c_{[2, \omega)}(\tau, \tau^*))$
  where: 
\begin{itemize} 
\item $A^c_{[2, \omega)}(\tau, \tau^*)$ is generic for a product
  $\A^c_{[2, \omega)}(\tau, \tau^*) =  \A^c_2(\tau, \tau^*) \times  \A^c_{[3, \omega)}(\tau, \tau^*)$
  of Cohen posets defined in $V$.    
\item $\A^c_2(\tau, \tau^*) = \Add^V(\Lambda^b_{\omega+2}(\tau), [\Lambda^b_{\omega+3}(\tau), \Lambda^a_0(\tau^*)))$. 
\item $\A^c_{[3, \omega)}(\tau, \tau^*)$ is defined and $<\lambda$-closed in $V$.
\end{itemize}

  We will produce
  $V[L \restriction \tau +1 \times A^c_2(\tau, \tau^*) \times L']$ from $V[L \restriction \tau^* + 1][Q]$
  by a series of rearrangements and quotient to term forcings. We will be making several appeals to the
  Product Lemma and Lemma \ref{rearrangeqtot}, but we will not make these explicit: the point is that
  each generic object will be generic for the forcing which originally introduced it over various larger models
  than the model where that forcing was originally defined.

\begin{itemize}  

\item  We may rearrange $I^b(\tau^*)$ as $A^b_0(\tau^*) * I^b_-(\tau^*)$
  where $I^b_-(\tau^*)$ collects the remaining components of $I^b(\tau^*)$. Recall from
  Remark \ref{abzeroone} that $\A^b_0(\tau^*) \in V[L \restriction \tau^*]$. 
  Forcing with an appropriate
   series of quotient to term forcing posets
  we extend $V[L \restriction \tau^* + 1][Q]$ to $V[T_0][L \restriction \tau^*][A^b_0(\tau^*)[Q]$
    where $T_0$ is generic for the product $\T_0$ of the following term forcings:
    \begin{itemize}
    \item $\termspace^V(\FL \restriction \tau^*, \FL^b(\tau^*))$.    
    \item $\termspace^V(\FL \restriction \tau^* * \FL^b(\tau^*) * \A^b_0(\tau^*), \I^b_-(\tau^*))$.
    \item $\termspace^V(\FL \restriction \tau^* * \FL^b(\tau^*), \A^c_0(\tau^*))$
    \item $\termspace^V(\FL \restriction \tau^* * \FL^b(\tau^*), \A^c_1(\tau^*))$
    \item $\termspace^V(\FL \restriction \tau^* * \FL^b(\tau^*), \A_e(\tau^*))$
    \item $\termspace^V(\FL \restriction \tau^* * \FL^b(\tau^*) * \I^b(\tau^*) * \A^c_0(\tau^*), \U^c_0(\tau^*) * \S^c_0(\tau^*))$.
    \item $\termspace^V(\FL \restriction \tau^* * \FL^b(\tau^*) * \I^b(\tau^*) * \J^c_0(\tau^*) * \A^c_1(\tau^*), \U^c_1(\tau^*) * \S^c_1(\tau^*))$.
    \end{itemize}

\item $V[T_0][L \restriction \tau^*][A^b_0(\tau^*)[Q] = 
   V[T_0][(L \restriction \tau^*) * A^b_0(\tau^*) * Q_1][Q_0 \times Q_2]
   = V[T_0][(L \restriction \tau^*) \times A^V_0(\tau^*)][Q_0 \times Q_2]$,
   using the definition of $\Q_1$ as a quotient to term forcing.
   It will be convenient to reorganise $A^V_0(\tau^*)$ as $T_1$ which is generic
   for $\T_1 = \termspace(\FL \restriction \tau^*, \A^b_0(\tau^*))$, so our model becomes
   $V[T_0][T_1][L \restriction \tau^*][Q_0 \times Q_2]$,

 \item  $V[T_0][T_1][L \restriction \tau^*][Q_0 \times Q_2]
   = V[T_0][T_1][L \restriction \tau][L^b(\tau) * I^b(\tau) * J^c(\tau)][A_e(\tau) * L \restriction (\tau, \tau^*) * Q_2][Q_0]
   = V[T_0][T_1][L \restriction \tau][L^b(\tau) * I^b(\tau) * J^c(\tau)][A_e(\tau) \times T(\tau, \tau^*)][Q_0]
   = V[T_0][T_1][L \restriction \tau][L^b(\tau) * I^b(\tau) * J^c(\tau)][T(\tau, \tau^*)][A_e(\tau)][Q_0]$. 
   using the definition of $\Q_2$ as a quotient to term forcing.

 \item 
  By Lemma \ref{iteratedtermforcing} $\termspace^V(\FL \restriction \tau * L^b(\tau) * I^b(\tau) * J^c(\tau), \T(\tau, \tau^*)) 
  \simeq \termspace^V(\FL \restriction \tau + 1, \FL \restriction (\tau, \tau^*))$.
  So forcing with an appropriate quotient to term forcing we extend to obtain
  $V[T_0][T_1][T_2][L \restriction \tau][L^b(\tau) * I^b(\tau) * J^c(\tau)][A_e(\tau)][Q_0]$,
  where $T_2$ is generic for
  $\T_2 = \termspace^V(\FL \restriction \tau + 1, \FL \restriction (\tau, \tau^*))$.

\item     By the definition of $\Q_0(\tau, \tau^*)$,
  \begin{align*}
  & V[T_0][T_1][T_2][L \restriction \tau][L^b(\tau) * I^b(\tau) * J^c(\tau)][A_e(\tau)][Q_0]  = \\
  & V[T_0][T_1][T_2][L \restriction \tau][L^b(\tau) * I^b(\tau)][A_e(\tau)][A^c(\tau, \tau^*) * U^c(\tau, \tau^*) * S^c(\tau, \tau^*)]  \\
  \end{align*}

\item  We defined $\A^c_{[1, \omega)}(\tau, \tau^*) * \U^c_{[1, \omega)}(\tau, \tau^*) * \S^c_{[1, \omega)}(\tau, \tau^*)$
  as an $\A * \U * \S$ construction as in Section \ref{abcu}, 
  performed in the model  $V[L \restriction \tau * L^b(\tau) * I^b(\tau)][A^c_0(\tau) * U^c_0(\tau) * S^c_0(\tau)]$.
  In particular the construction involved auxiliary posets $\B^c_{[1,\omega)}(\tau, \tau^*)$ and $\C^c_{[1, \omega)}(\tau, \tau^*)$
      constructed in this model. $\C^c_{[1, \omega)}(\tau, \tau^*)$ breaks down as
        $\C^c_{[1, \omega)}(\tau, \tau^*) = \C^c_1(\tau, \tau^*) \times \C^c_{[2, \omega)}(\tau, \tau^*)$
            where $\C_{[2, \omega)}(\tau, \tau^*)$ is $<\lambda$-closed in      
              $V[L \restriction \tau * L^b(\tau) * I^b(\tau)][A^c_0(\tau) * U^c_0(\tau) * S^c_0(\tau)]$.
 The forcing poset $\S^c_{[2 , \omega)}(\tau, \tau^*)$ is defined from $\C^c_{[2, \omega)}(\tau, \tau^*)$
                  and $A^c_{[1, \omega)}(\tau, \tau^*) * U^c_{[1, \omega)}(\tau, \tau^*)$ as
     $(\C^c_{[2, \omega)}(\tau, \tau^*))^{+A^c_{[1, \omega)}(\tau, \tau^*) * U^c_{[1, \omega)}(\tau, \tau^*)}$ as in Section
         \ref{abcu}, so that in $V[L \restriction \tau * L^b(\tau) * I^b(\tau)][A^c_0(\tau) * U^c_0(\tau) * S^c_0(\tau)]$ we may view
         $\A^c_{[1, \omega)}(\tau, \tau^*) * \U^c_{[1, \omega)}(\tau, \tau^*) * \S^c_{[2, \omega)}(\tau, \tau^*)$ as a projection of
 the product $\A^c_{[1, \omega)}(\tau, \tau^*) * \U^c_{[1, \omega)}(\tau, \tau^*) \times \C^c_{[2, \omega)}(\tau, \tau^*)$.                               
    In $V[L \restriction \tau * L^b(\tau) * I^b(\tau)][A^c_0(\tau) * S^c_0(\tau) * U^c_0(\tau)][A^c_1(\tau) * U^c_1(\tau)]$
    we may view $A^c_{[2, \omega)}(\tau, \tau^*) * U^c_{[2, \omega)}(\tau, \tau^*)$ as a projection of 
        $A^c_{[2, \omega)}(\tau, \tau^*) \times (\B^c_{[2, \omega)}(\tau, \tau^*))^{+A^c_1(\tau) * U^c_1(\tau)}$. By Lemma \ref{Itay4.7} for
            $\A^c_{[1, \omega)}(\tau, \tau^*) * \U^c_{[1, \omega)}(\tau, \tau^*) * \S^c_{[1, \omega)}(\tau, \tau^*)$ with $\alpha = \alpha' = \lambda$
    (so that $F'$ in the lemma is $A^c_1 * U^c_1$),              
     $(\B^c_{[2, \omega)}(\tau, \tau^*))^{+A^c_1(\tau) * U^c_1(\tau)}$ is $<\lambda$-closed         
    in $V[L \restriction \tau * L^b(\tau) * I^b(\tau)][A^c_0(\tau) * U^c_0(\tau) * S^c_0(\tau)][A^c_1(\tau) * U^c_1(\tau)]$.       

  \item Since $S^c_{[1, \omega)}(\tau, \tau^*)$ is a product,
\begin{align*}    
    & V[T_0][T_1][T_2][L \restriction \tau][L^b(\tau) * I^b(\tau)][A_e(\tau)][A^c(\tau, \tau^*) * U^c(\tau, \tau^*) * S^c(\tau, \tau^*)]\\
  =  & V[T_0][T_1][T_2][L \restriction \tau][L^b(\tau) * I^b(\tau)][A_e(\tau)][A^c_0(\tau) * U^c_0(\tau) * S^c_0(\tau)]\\ 
    &[A^c_{[1, \omega)}(\tau, \tau^*) * U^c_{[1, \omega)}(\tau, \tau^*) * S^c_{[2, \omega)}(\tau, \tau^*)][S^c_1(\tau)] \\
\end{align*}
          
      Forcing with a suitable quotient to term forcing, we may extend to obtain
      a model
      \begin{align*}
      &V[T_0][T_1][T_2][L \restriction \tau][L^b(\tau) * I^b(\tau)][A_e(\tau)][A^c_0(\tau) * U^c_0(\tau) * S^c_0(\tau)]\\
      &[A^c_{[1, \omega)}(\tau, \tau^*) * U^c_{[1, \omega)}(\tau, \tau^*) \times C^c_{[2, \omega)}(\tau, \tau^*)][S^c_1(\tau)],\\
      \end{align*}
      and then reorganise as
      \begin{align*}
      &V[T_0][T_1][T_2][L \restriction \tau][L^b(\tau) * I^b(\tau)][A_e(\tau)][J^c(\tau)] \\
      &[A^c_{[2, \omega)}(\tau, \tau^*) * U^c_{[2, \omega)}(\tau, \tau^*) \times C^c_{[2, \omega)}(\tau, \tau^*)]. \\
      \end{align*}  
    Forcing with another quotient to term forcing, we may extend to obtain  a model            
    $V[T_0][T_1][T_2][L \restriction \tau][L^b(\tau) * I^b(\tau)][A_e(\tau)][J^c(\tau)]
    [A^c_{[2, \omega)}(\tau, \tau^*) \times X \times C^c_{[2, \omega)}(\tau, \tau^*)]$
        where $X$ is $(\B^c_{[2, \omega)}(\tau, \tau^*))^{+A^c_1(\tau) * U^c_1(\tau)}$-generic. 
    This model may be rewritten as 
    $V[T_0][T_1][T_2][L \restriction \tau + 1][A^c_{[2, \omega)}(\tau, \tau^*) \times X \times C^c_{[2, \omega)}(\tau, \tau^*)]$.

      \item Now let
     \begin{align*}
       & \T_3 \\
       & = \termspace^V(\FL \restriction \tau * \FL^b(\tau) * \I^b(\tau) * \J^c_0(\tau) * \A^c_1(\tau) * \U^c_1(\tau),
       (\B^c_{[2, \omega)}(\tau, \tau^*))^{+A^c_1(\tau) * U^c_1(\tau)}) \\
      &\times   
     \termspace^V(\FL \restriction \tau * \FL^b(\tau) * \I^b(\tau) * \A^c_0(\tau) * \U^c_0(\tau) * \S^c_0(\tau), \C^c_{[2, \omega)}(\tau, \tau^*)),\\
     \end{align*}      
       where $\T_3$ is $<\lambda$-closed in $V$. 
      With one more round of quotient to term forcing we may extend     
      $V[T_0][T_1][T_2][L \restriction \tau + 1][A^c_{[2, \omega)}(\tau, \tau^*) \times X \times C^c_{[2, \omega)}(\tau, \tau^*)]$
          to obtain
      $V[T_0][T_1][T_2][T_3][L \restriction \tau + 1][A^c_{[2, \omega)}(\tau, \tau^*)]$    
\end{itemize}

We set $\FL' = \A^c_{[3, \omega)}(\tau, \tau^*) \times \T_0 \times \T_1 \times \T_2 \times \T_3$.
  It is routine to check that $\FL'$ is $<\lambda$-closed and has cardinality
  $\lambda^b_{\omega + 3}(\tau^*)$.
\end{proof}

\begin{lemma} \label{qdistrib}
  The poset $\Q(\tau, \tau^*)$ is $<\Lambda^b_{\omega+2}(\tau)$-distributive in $V[L]$. 
\end{lemma}

\begin{proof}
  By the agreement between $V[L]$ and $V[L \restriction \tau^* + 1]$, it is enough to show that
  $\Q(\tau, \tau^*)$ is $<\Lambda^b_{\omega+2}(\tau)$-distributive in $V[L \restriction \tau^* + 1]$.
  Let $Q$ be $\Q(\tau, \tau^*)$-generic over $V[L \restriction \tau^* + 1]$, then
  by Lemma \ref{fordistrib1variant}
  $V[L \restriction \tau^* +1] [Q] \subseteq V[L \restriction \tau + 1 \times A^c_2(\tau, \tau^*) \times L']$,  
  where $A^c_2(\tau, \tau^*) \times L'$ is generic for the product of $\Add^V(\Lambda^b_{\omega+2}(\tau), \Lambda^a_0(\tau^*))$
  and some $<\Lambda^b_{\omega+3}(\tau)$-closed forcing $\FL'$ lying in $V$.
  The conclusion is now immediate by Easton's lemma.
\end{proof} 
A minor elaboration of this argument shows:
\begin{lemma} \label{qdistribmany}
  Let $\tau_0 < \ldots < \tau_n$ with $\tau_i \in Y$ for all $i$. Then
  $\prod_{0 \le i < n} \Q(\tau_i, \tau_{i+1})$ is $<\Lambda^b_{\omega+2}(\tau_0)$-distributive in $V[L]$. 
\end{lemma}

Since $\Q(\tau_n, \tau_{n+1})$ is distributive over the cardinality of $\prod_{0 \le i < n} \Q(\tau_i, \tau_{i+1})$,
we immediately deduce:

\begin{lemma} \label{qdistriboverprevious}
  Let $\tau_0 < \ldots < \tau_n < \tau_{n+1}$ with $\tau_i \in Y$ for all $i$,
  and let $E$ be $\prod_{0 \le i < n} \Q(\tau_i, \tau_{i+1})$-generic over $V[L]$
  Then $\Q(\tau_n, \tau_{n+1})$ is $<\Lambda^b_{\omega+2}(\tau_n)$-distributive in $V[L][E]$. 
\end{lemma}

\subsection{The first Prikry point} \label{firstprikry}

In this section we define a forcing poset ${\Q}^*(\tau)$ which will be used in the Prikry
forcing $\bar\P$ when $\tau$ is the first Prikry point. 

Recall from Section \ref{preparation} that $V = V_0[A^0  * U^0 \restriction \theta * L^0]$
where $A^0 * U^0 \restriction \theta$
forces that $2^\omega = \theta$
and makes $\theta$ indestructibly generically supercompact via Cohen reals, and $L^0$ is generic over $V_0[A^0 * U^0 \restriction \theta]$
for what is essentially a standard Laver indestructibility iteration in the interval $(\theta, \delta)$.
Let $V_1=V_0[A^0 * U^0 \restriction \theta]$. 
  Recall from the discussion preceding Lemma \ref{selection} that there is a unique stage $\bar \theta < \rho$
 of the preparation such that $\rho$ is a limit of supercompact cardinals in $V_0[A^0 \restriction \bar \theta * U^0 \restriction \bar \theta]$
 and  $\rho$ is an $\omega$-successor in $V_0[A^0 \restriction \bar \theta * U^0 \restriction \bar \theta + 1]$.
 
  Recall from Remark \ref{abzeroone} that $\A^b_0(\tau) 
 =\Add^{V[L \restriction \tau]}(\Lambda^a_{17}(\tau), [\Lambda^a_{\omega+1}(\tau), \Lambda^a_{\omega+2}(\tau)))$
 and is part of the component of $\FL$ at stage $\tau$. 
   The poset ${\Q}^*(\tau)$ will ultimately be defined in $V[L \restriction \tau][A^b_0(\tau)]$
and will have three components $\Q^*_i(\tau)$ for $i \in \{ 0, 1, 2 \}$.

 The idea for defining $\Q_0^*(\tau)$ is that we view the forcing $\A^0  * \U^0 \restriction \theta * \L^0$
 which produces $V$ from $V_0$ as the first phase of a two-phase $\A * \U * \S$ construction,
 and
 that $\Q_0^*(\tau)$ is defined in $V$ and implements the second phase. 
 Here are the details of the two-phase construction.

 \begin{itemize}
 \item  The cardinal parameters are $\mu_0 = \omega$, $\mu_1 = \rho^+$, $\mu_2 = \theta$, $\mu_{3 + n} = \Lambda^a_n(\tau)$ for all $n \in \omega$.
 \item  $\A_0 = \A^0 =  \A^0 \restriction \theta$, $\B_0 = \B^0 \restriction \theta$
   and $\U_0 = \U^0 \restriction \theta$ were already defined in $V_0$, and the construction
   of Section \ref{preparation} already gave us the generic object $A_0 * U_0$.
 \item   ${\C}_0$ is also defined in $V_0$  as in Section \ref{abcu}, in particular
     it adds generic objects for  $\Add(\rho^+, 1)^{V_0[A_0 \restriction \alpha * U_0 \restriction \alpha]}$ only
       for $\alpha$ with $\mu_1 < \alpha < \theta$.
     \item $\A_1 = \Add(\mu_1, [\mu_2,\mu_3))^{\bar W}$, where we recall that
       $\bar W = V_0[A_0 \restriction \bar \theta * U_0 \restriction \bar \theta + 1]$.      
 \item $\A_n = \Add(\mu_n,[\mu_{n+1},\mu_{n+2}))^V$ for $2 \le n < \omega$.
 \item The Laver function is the universal indestructible function $\phi$  from Lemma \ref{fromVzerotoV}.   
 \item For $n \ge 1$, $\B_n$ and $\C_n$ are defined over $V$. To be more precise conditions in $\B_n$
   are functions $b \in V$ with supports which are Easton subsets of $(\mu_{n+1}, \mu_{n+2})$, consisting of points $\alpha$
   where the Laver function returns an $\A \restriction \alpha * U \restriction \alpha$ name in $V$ for
   a forcing which is $<\alpha$-directed closed in $V[A \restriction \alpha * U \restriction \alpha]$. As usual $b(\alpha)$
   will name an element of this poset. The definition of $\C_n$ is similar.
 \end{itemize}

 The first component $\Q_0^*(\tau)$ of ${\Q}^*(\tau)$ prolongs $A_0 * U_0 * L^0$ to a generic object for this two=phase construction. 
 The second component $\Q_1^*(\tau)$ is defined over $V[L \restriction \tau][A^b_0(\tau)]$ and adds  
 $A^V_0(\tau)$ which is $\Add^V(\Lambda^a_{17}(\tau), \Lambda^a_{\omega+2}(\tau))$-generic over $V[L \restriction \tau]$ and is such that
 $V[L \restriction \tau] \subseteq V[L \restriction \tau][A^b_0(\tau)] \subseteq V[(L \restriction \tau) \times A^V_0(\tau)]$.
 The third component $\Q_2^*(\tau)$ is $\Coll(\omega, \rho)$.
We note that ${\Q}^*(\tau)$ has cardinality $\Lambda^a_{\omega + 2}(\tau)$.

\begin{remark}  By contrast with $\Q(\tau, \tau^*)$, not all components of $\Q^*(\tau)$ are $\rho$-closed
  posets defined in $V$ or some $\rho$-closed extension of $V$. For use in Section \ref{GroupIV},
  we categorise the components of $\Q^*(\tau)$.

  \begin{itemize}
  \item $\A_{[2, \omega)}$ is $\rho$-closed in $V$.
  \item $\A_1$ is $\rho$-closed in $\bar W$ which is a proper submodel of $V$.
  By the usual arguments with Easton's lemma, it is
    $\rho$-distributive in $V$. 
  \item $\U_{[1, \omega)} * \S_{[1, \omega)}$ is $\rho$-closed in $V[A_{[1, \omega)}]$. 
 \item As in the case of $\Q(\tau, \tau^*)$, $\Q^*_1(\tau)$ is $\rho$-closed in $V[L \restriction \tau][A^b_0(\tau)]$.
 \item Of course, $\Q_2^*(\tau)$ is not even $\omega$-distributive.    
  \end{itemize}
\end{remark}

\smallskip

\noindent Global notation:
  $\Q^*(\tau)$\index[Notation]{$\Q^*(\tau)$},
  $\Q^*_0(\tau)$\index[Notation]{$\Q^*_0(\tau)$},
  $\Q^*_1(\tau)$\index[Notation]{$\Q^*_1(\tau)$},
  $\Q^*_2(\tau)$\index[Notation]{$\Q^*_2(\tau)$}

\smallskip

  \subsection{Some auxiliary computations} \label{auxcomp}

Recall that $j: V \to M$ has critical point $\kappa$ and witnesses that
$\kappa$ is $\delta^+$-supercompact. We derive a supercompactness extender
$E$ from $j$ witnessing that $\kappa$ is $<\lambda^b_{\omega+3}$-supercompact:
to be more concrete, for each $\eta$ with  $\kappa \le \eta <\lambda^b_{\omega+3}$ we let
$W_\eta$ be the supercompactness measure on $P_\kappa \eta$ derived from $j$,
and let $E$ be the system of measures $\langle W_\eta: \kappa \le \eta <\lambda^b_{\omega+3} \rangle$,
with projection maps $\pi_{\eta \zeta}: P_\kappa \zeta \to P_\kappa \eta$ given
by $\pi_{\eta \zeta} : x \mapsto x \cap \eta$.

Let $j_E: V \to  Ult(V, E)$ be the limit ultrapower by $E$,
so that by standard arguments $\crit(j_E) = \kappa$ and
$Ult(V, E)$ is closed under $<\lambda^b_{\omega+3}$-sequences. As usual there is a an elementary embedding
$k_E: Ult(V, E) \rightarrow M$ such that $k_E \circ j_E = j$ and $\crit(k_E) \ge \lambda^b_{\omega + 3}$.
Using $k_E$ it is easy to see that
$j_E(\Lambda^z_i)(\kappa) = j(\Lambda^z_i)(\kappa) = \lambda^z_i$ for $z \in \{a, b \}$ and $i < \omega + 3$.

\smallskip

\noindent Global notation: $E$\index[Notation]{$E$}, $W_\eta$\index[Notation]{$W_\eta$}, $j_E$\index[Notation]{$j_E$}

\smallskip

We will need to iterate the ultrapower by $E$, but only for two steps. To simplify the notation
let $j_{0 1} = j_E$ and $M_1 = Ult(V, E)$. Then as usual $j_{1 2}: M_1 \rightarrow M_2$ is the ultrapower
map computed in $M_1$ using the extender $j_{0 1}(E)$, and $j_{0 2} = j_{1 2} \circ j_{0 1}$.
Note that by the usual chain condition argument, $V[L] \models {}^{<\lambda^b_{\omega+3}} M_1[L] \subseteq M_1[L]$.

We will use the identity $j_{01} \restriction M_1 \circ j_{01} = j_{02}$. The proof is quite easy:
by the elementarity of $j^V_E$ and the fact that $j^V_E$ is defined in $V$,
$j_{01}(j_{01}(x)) = j^V_E(j^V_E(x)) = j^{M_1}_{j^V_E(E)}(j^V_E(x)) = j_{12}(j_{01}(x)) = j_{02}(x)$.

\smallskip

\noindent Global notation: $j_{01}$\index[Notation]{$j_{01}$}, $j_{12}$\index[Notation]{$j_{12}$},  $j_{02}$\index[Notation]{$j_{02}$},
$M_0$\index[Notation]{$M_0$}, $M_1$\index[Notation]{$M_1$},  $M_2$\index[Notation]{$M_2$}

\smallskip

It is easy to see that:
\begin{itemize}
\item $\crit(j_{12}) = j_{01}(\kappa) > \lambda^b_{\omega+3}$. 
\item $M_1 \models {}^{<j_{01}(\lambda^b_{\omega + 3})} M_2 \subseteq M_2$.
\item For any function $g: \kappa \rightarrow \kappa$,
  $j_{02}(g)(\kappa) = j_{12}(j_{01}(g)(\kappa)) = j_{01}(g)(\kappa)$ and 
  $j_{02}(g)(j_{01}(\kappa)) = j_{01}(j_{01}(g)(\kappa))$. In particular
  $j_{02}(\Lambda^z_i)(\kappa) = \lambda^z_i$, and also
  $j_{02}(\Lambda^z_i)(j_{01}(\kappa)) = j_{01}(\lambda^z_i)$,
  for $z \in \{a, b \}$ and $i < \omega + 3$.
\end{itemize}


\begin{lemma} \label{technical}
 Let $\Aggforcing = \Add^V(\lambda^b_{\omega +2}, j_{01}(\lambda^a_0))$.
 There exists $L^* \in V[L]$ and $K \in V[L][\Agg]$ such that:
 \begin{enumerate}
 \item   $L^*$  is $j_{01}(\FL)$-generic over $M_1$.
 \item   $j_{01}[L] \subseteq L^*$.
 \item   $L^* \restriction \kappa + 1 = L$.
 \item   If we lift $j_{01}$ to obtain $j^*_{01}: V[L] \rightarrow M_1^* = M_1[L^*]$,
   let $j_{12}^* = j_{01}^*(j_{01}^*): M_1^* \rightarrow M_2^*$ and $j_{02}^* = j_{12}^* \circ j_{01}^*$,
   and define $\Q_\infty = \Q^{M^*_2}(\kappa, j_{01}(\kappa))$, then $K$ is $\Q_\infty$-generic
   over $M_1^*$. 
 \end{enumerate}  
\end{lemma}

\begin{proof} 
  We start with some easy remarks:
  \begin{enumerate}
  \item By Easton's Lemma, $\Aggforcing$ is $<\lambda^b_{\omega + 2}$-distributive
    in $V[L]$. 
  \item Since $L^* \restriction \kappa + 1 = L$,
    it will follow that $M_1^*$ is closed under $<\lambda^b_{\omega+3}$-sequences in $V[L]$.  
    By elementarity $M_2^*$ will be  closed under $<j_{01}(\lambda^b_{\omega+3})$-sequences in $M_1^*$,
    in particular $M_1^*$ and $M_2^*$ will agree for a long way past the rank of $\Q_\infty$. 
  \end{enumerate}  

We now appeal to Lemma \ref{fordistrib1variant} in the model $M_2$ with $\tau = \kappa$ and $\tau^* = j_{01}(\kappa)$.
Using the fact that $M_2$ is closed under $<j_{01}(\lambda^b_{\omega+3})$-sequences in $M_1$,
we get a projection in $M_1$ from $\FL \times \Aggforcing \times \FL'$
to $j_{01}(\FL) * \Q(\kappa, j_{01}(\kappa))$,
where $\Aggforcing = \Add^{M_2}(\lambda^b_{\omega +2}, j_{01}(\lambda^a_0)) =  \Add^V(\lambda^b_{\omega +2}, j_{01}(\lambda^a_0))$ and
$\FL'$ is the product of various term forcing posets.
The most relevant factors in $\FL'$ are:
\begin{itemize}
\item $\FL'_0  = \termspace^{M_1}(\FL, j_{01}(\FL) \restriction (\kappa, j_{01}(\kappa))$.
\item $\FL'_1  = \termspace^{M_1}(j_{01}(\FL \restriction \kappa), j_{01}(\FL^b))$.  
\item $\FL'_2  = \termspace^{M_1}(j_{01}(\FL \restriction \kappa), j_{01}(\A^b_0))$.
\item $\FL'_3  = \termspace^{M_1}(j_{01}(\FL \restriction \kappa * \A^b_0), \I^b_-)$, where $\I^b = \A^b_0 * \I^b_-$.
\item $\FL'_4  = \termspace^{M_1}(j_{01}(\FL \restriction \kappa * \FL^b), j_{01}(\A^c_0))$.
\item $\FL'_5  = \termspace^{M_1}(j_{01}(\FL \restriction \kappa * \FL^b), j_{01}(\A^c_1))$.
\item $\FL'_6  = \termspace^{M_1}(j_{01}(\FL \restriction \kappa * \FL^b), j_{01}(\A_e))$.
\item $\FL'_7  = \termspace^{M_1}(j_{01}(\FL \restriction \kappa * \FL^b * \I^b * \A^c_0), j_{01}(\U^c_0 * \S^c_0))$.      
\item $\FL'_8  = \termspace^{M_1}(j_{01}(\FL \restriction \kappa * \FL^b * \I^b * \J^c_0 * \A^c_1), j_{01}(\U^c_1 * \S^c_1))$.    
\end{itemize}
In the proof of Lemma \ref{fordistrib1variant} $\FL'_2$ corresponds to $\T_1$, $\FL'_0$ corresponds to $\T_2$,
and the remaining factors correspond to factors in $\T_0$.
The projection uses $\FL$ and the factors $\FL'_i$ listed above in the 
 obvious way to prolong the $\FL$-generic to a $j_{01}(\FL)$-generic object.

It is straightforward to verify that the set of maximal antichains of $\FL'$ which lie in $M_1$ has cardinality $\lambda^b_{\omega + 3}$ in $V$.
Since $\FL'$ is $<\lambda^b_{\omega + 3}$-closed in $M_1$, and $M_1$ is closed under $<\lambda^b_{\omega + 3}$-sequences in $V$,
we may readily work in $V$ to  build $L'$ which is $\FL'$-generic over $M_1$, but since we will ultimately use $L'$ to
build $L^*$ we need to build $L'$ more  carefully. The construction will involve successively lifting $j_{01}$
onto larger and larger initial segments of $V[L]$: to lighten the notation we will denote all the embeddings
by ``$j_{01}$'' and resolve any ambiguity by making the domain and codomain explicit.

To start we choose $L'_0 \in V$ which is $\FL'_0$-generic over $M_1$, and combine it with $L$
to construct $L^*_0 \in V[L]$ which is $j_{01}(\FL) \restriction (\kappa, j_{01}(\kappa))$-generic over $M_1[L]$.
Note that by the closure of $\FL'_0$,  $V \models {}^{<\lambda^b_{\omega+3}}M_1[L'_0] \subseteq M_1[L'_0]$. 
As usual we may lift to obtain $j_{01}:V[L \restriction \kappa] \rightarrow M_1[L^*_0]$. 
The next stage is slightly harder, because we must choose $L'_1$ so that it combines with $L * L^*_0$ to
produce $L^*_1$ so that we may lift $j_{01}$ to $V[L \restriction \kappa][L^b]$.

To this end, let $H$ be any filter which is $j_{01}(\FL \restriction \kappa)$-generic over $M_1$,
so that we may lift to obtain $j_{01}:V[H \restriction \kappa] \rightarrow M_1[H]$.
Since $\vert \FL^b \vert < \lambda^b_{\omega+3}$, it is easy to see that if $H'$ is the generic
filter on $\FL^b$ added by $H$ then $j_{01}[H'] \in M_1[H]$, and $j_{01}[H']$ has a lower bound
in $j_{01}(\FL^b)$. Let $\dot m$ be a $j_{01}(\FL \restriction \kappa)$-name for such a lower bound,
so that we may view $\dot m$ as a condition in $\FL'_1$ and build $L'_1 \in V$ which is $\FL'_1$-generic
over $M_1[L'_0]$ with $\dot m \in L'_1$. We combine $L'_1$ with $L * L^*_0$ to obtain $L^*_1 \in V[L]$
which is $j_{01}(\FL^b)$-generic over $M_1[L^*_0]$. By construction $j_{01}[L^b] \subseteq L^*_1$,
so that we may lift and obtain $j_{01}: V[L \restriction \kappa * L^b] \rightarrow M_1[L^*_0 * L^*_1]$.

Similar arguments will handle the other factors of size less than $\lambda^b_{\omega+3}$, but
the factors of size $\lambda^b_{\omega+3}$ will need more care because we do not have closure
under $\lambda^b_{\omega+3}$-sequences. We will handle this problem using ideas
of Magidor \cite{MagidorNonRegular}.

We will only do the argument for
$\A_e$, which has an extra twist: the arguments for $\A^c_1$ and $\U^c_1 * \S^c_1$
are similar but simpler. Recall that $\A_e = \Add^{V[L \restriction \kappa][L^b]}(\lambda^b_{17}, \lambda^b_{\omega+3})$:
forcing with $\A_e$ adds $\lambda^b_{\omega+3}$ many generic functions from $\lambda^b_{17}$ to
$\lambda^b_{17}$, and for $\alpha < \lambda^b_{\omega+3}$ we let $f_\alpha$ be the function
with index $\alpha$.

As we noted in the previous paragraph, $j_{01}[A_e \restriction \eta] \in M_1[j_{01}(L \restriction \kappa * L^b)]$
for all $\eta < \lambda^b_{\omega+3}$. We will use the following easy remark:
\begin{remark} \label{magidortrick}
  For every dense subset $D$ of $\termspace(\FL \restriction \kappa * \FL^b, \A_e)$,
  there is $f : \lambda^b_{\omega+3} \rightarrow \lambda^b_{\omega+3}$ 
  such that if $\gamma$ is an inaccessible closure point of $f$
  and $\forces \dot \sigma \in \dot \A_e \restriction \gamma$, there is
  $\dot \tau \in D$ such that $\forces \dot \tau \le \dot \sigma$ and
  $\forces \dot \tau \in \dot \A_e \restriction \gamma$. 
 \end{remark} 
  
Since $\vert j_{01}(\lambda^b_{17}) \vert = \lambda^b_{\omega+3}$,
we  enumerate the elements of $j_{01}(\lambda^b_{17})$ as $\gamma_j$ for $j < \lambda^b_{\omega+3}$.
  Let $\eta = \sup j_{01}[\lambda^b_{17}]$, and note that if $p \in \A_e$ then the support of $j_{01}(p)$ is contained in  $j_{01}(\lambda^b_{\omega+3}) \times \eta$.
  We will arrange the lifting construction so that in the end 
  $j_{01}^*: V[L] \rightarrow M_1^*$ has the property that $j_{01}^*(f_i)(\eta) = \gamma_i$
  for every $i < \lambda^b_{\omega+3}$. This idea originates in unpublished work of Woodin,
  and was used in a construction similar to ours by Gitik and Sharon \cite{GitikSharon}.

We will construct $L_2'$ as the upwards closure of a decreasing $\lambda^b_{\omega+3}$-sequence in
$\FL'_2$. View $\dot m$ as a condition in $j_{01}(\FL \restriction \kappa * L^b)$,
  and let $H$ be an arbitrary filter which is $j_{01}(\FL \restriction \kappa * L^b)$-generic over
  $M_1$ and contains this condition. Let $H'$ be the $\FL \restriction \kappa * L^b$-generic filter
  induced by $H$, so that $j_{01}[H'] \subseteq H$ by the choice of $\dot m$ and we may
  lift to obtain $j_{01}:V[H] \rightarrow M_1[H']$. Much as in the construction for $L'_1$,
  we will use this embedding to define suitable conditions in $\FL'_2$. 
    
  We will build a decreasing $\lambda^b_{\omega+3}$-sequence of conditions in $\FL'_2$, with the aim of generating a filter which is generic
  over $M_1$, and induces a filter $L^*_2$ which  is compatible with $j_{01}'$ and $A_e$
  and assigns the right values to $j_{01}^*(f_i)(\eta)$.
  Suppose that we have reached a stage of the construction where we built a condition $\dot q \in \FL'_2$ with the following
  properties:
  \begin{itemize}
  \item $\forces \dot q \in j_{01}(\A_e \restriction \alpha)$.
  \item $\forces \dot q \le j_{01}[A_e \restriction \alpha]$.
  \item $\forces \dot q(j_{01}(i), \eta) = \gamma_i$ for all $i < \alpha$.
  \end{itemize}
  
  Suppose that the next dense set in $\FL'_2$ to be handled is $D \in M_1$, and note that (since $E$ is a supercompactness extender)
  $D = j_{01}(d)(j_{01}[\sigma])$ for some $\sigma < \lambda^b_{\omega+3}$ and function $d \in V$ with $\dom(d) = P_\kappa \sigma$. We
  may assume that $d(x)$ is a dense subset of $\termspace(\FL \restriction \kappa * \FL^b, \A_e)$
  for all $x$: it is now easy to produce a function $f$ which satisfies the conclusion of Remark \ref{magidortrick}
  for all the dense sets $d(x)$ simultaneously.
   
  Let $\gamma > \alpha$ be an inaccessible closure point of $f$. We build a name $\dot r$
  for a condition extending $\dot q$ in stages,
  making sure that $\dot r$ names a condition in  $j_{01}(\A_e \restriction \gamma)$:
  \begin{itemize}
  \item Let $\dot r_1$ name   $q \cup \bigcup j_{01}[A_e \restriction [\alpha, \gamma)]$,
    so that $r_1$ names  a lower bound for $j_{01}[A_e \restriction \gamma]$.
  \item Let $\dot r_2$ name $r_1  \cup \{ (j_{01}(i), \eta, \gamma_i) : \alpha \le i < \gamma \}$. 
  \item Let $\dot r \in D$ with $\forces \dot r \le \dot r_2$,
    where it is possible to arrange  that $\forces \dot r \in j_{01}(\A_e \restriction \gamma)$ by the careful choice of $\gamma$.
  \end{itemize}   
  The condition $\dot r$ will be the next entry in our descending chain.

  By construction, if we induce $L^*_2$ using $L * L^*_0 * L^*_1$ then $j_{01}[A_e] \subseteq L^*_2$.
  We lift to obtain $j_{01}: V[L \restriction \kappa][L^b][A_e] \rightarrow M_1[L * L^*_0 * L^*_1 * L^*_2]$,
  where $j_{01}(f_i)(\eta) = \gamma_i$ for all $i$. Continuing in the same way we build the remainder of $L'$,
  induce $L^*$, and finally lift to get $j_{01}^*: V[L] \rightarrow M_1^* = M_1[L^*]$.

  Let $\Agg$ be $\Aggforcing = \Add^V(\lambda^b_{\omega +2}, j_{01}(\lambda^a_0))$-generic over $V[L]$, so that
  $\Agg$ is generic over $M_1[L \times L']$. Using the projection map in $M_1$ from
  $\FL \times \Aggforcing \times \FL'$ to $j_{01}(\FL) * \Q(\kappa, j_{01}(\kappa))$,
  we get $K \in V[L][\Agg]$  which is $\Q(\kappa, j_{01}(\kappa))$-generic over $M_1[L^*]$.
\end{proof}

\smallskip

\noindent Global notation: $j^*_{01}$\index[Notation]{$j^*_{01}$}, 
  $M_1^*$\index[Notation]{$M_1^*$}, $j^*_{02}$\index[Notation]{$j^*_{02}$}, $M^*_2$\index[Notation]{$M^*_2$},
$\Q_\infty$\index[Notation]{$\Q_{\infty}$}, $\Agg$\index[Notation]{$\Agg$},  $K$\index[Notation]{$K$}

\smallskip

  Working in $V[L]$ we derive for each $n \ge 17$ a supercompactness measure $U_n$
  on $P_\kappa \lambda^b_n$ using the embedding $j_{01}^*$. 
  We do some computations in $V[L]$ which will be useful when we define the Prikry forcing $\bar\P$ in
  Section \ref{prikryforcing}. For $n$ with $17 \le n < \omega$ let $N_n = Ult(V[L], U_n)$  
  and  $j_n = j_{U_n}^{V[L]}$, so that $j_n: V[L] \rightarrow N_n$ and
  we obtain as usual a factor map $k_n: N_n \rightarrow M_1^*$
  with $j^*_{01} = k_n \circ j_n$. 

  We will show that $k_n$ has a very large critical point, in fact
  $\crit(k_n) > j_{01}(\lambda^b_{17})$.
  To see this observe that the range of $k_n$ is the set of elements in $M_1^*$ of the form $j_{01}^*(f)(j_{01}[\lambda^b_n])$ where
  $f \in V[L]$ and $\dom(f) = P_\kappa \lambda^b_n$. If we let $f(x) = f_i(\sup(x \cap \lambda^b_{17}))$ then
  $j_{01}^*(f)(j_{01}[\lambda^b_n]) = j_{01}^*(f_i)(\eta) = \gamma_i$, so that easily $j_{01}(\mu_1) + 1 \subseteq \rge(k_n)$ and
  hence $\crit(k_n) > j_{01}(\lambda^b_{17})$.

  We will use  the observations that since $\crit(k_n) > j_{01}(\lambda^b_{17})$ for $n \ge 17$:
\begin{itemize}
\item  $j^*_{01}(\eta) = j_n(\eta)$ for all $\eta \le \lambda^b_{17}$.
\item  $j^*_{01}(\Lambda^b_k)(\kappa) = \lambda^b_k = j_n(\Lambda^b_k)(\kappa)$  for $k \le 17 \le n$.
\end{itemize}

\smallskip

\noindent Global notation: $U_n$\index[Notation]{$U_n$}, 
$N_n$\index[Notation]{$N_n$}, $j_n$\index[Notation]{$j_n$}, $k_n$\index[Notation]{$k_n$}

\smallskip

  We are interested in comparing the two-step iteration $j_{02}^*$ defined above, 
  and the iteration $i_n$ where we apply $j_n$ and then $j_n(j_{n+1})$. We use the easy equations
   $i_n = j_n(j_{n+1}) \circ j_n = j_n \circ j_{n+1}$ and $j_{02}^* = j_{01}^* \circ j_{01}^*$. 

\smallskip

\noindent Global notation: $i_n$\index[Notation]{$i_n$} 

\smallskip

\begin{lemma} 
   For all $n \ge 17$, $\Q_\infty = i_n({\Q})(\kappa, j_n(\kappa))$. 
\end{lemma}

\begin{proof}
We will produce a map $k$ such that $k \circ i_n = j_{02}^*$ and $\crit(k) > j_{01}(\lambda^b_{17})$. This will
suffice because $i_n({\Q})(\kappa, j_n(\kappa))$ can be coded as a bounded subset of $i_n(\Lambda^b_{17})(j_n(\kappa))$,
and $i_n(\Lambda^b_{17})(j_n(\kappa)) = j_n(j_{n+1}(\Lambda^b_{17}))(j_n(\kappa)) = j_n(j_{n+1}(\Lambda^b_{17})(\kappa)) 
    =j_n(\lambda^b_{17}) = j_{01}(\lambda^b_{17})$.

    Start by applying the embedding $j_n$ to the equation $j_{01}^* = k_{n+1} \circ j_{n+1}$, to get $j_n(j_{01}^*) = j_n(k_{n+1}) \circ j_n(j_{n+1})$.
    Here $j_n(j_{01}^*): N_n \rightarrow j_n(M^*_1)$, $j_n(j_{n+1}): N_n \rightarrow j_n(N_{n+1})$, and
    $j_n(k_{n+1}): j_n(N_{n+1}) \rightarrow j_n(M^*_1)$. By elementarity $\crit(j_n(k_{n+1})) > j_n(j_{01}(\lambda^b_{17}))$.

    Since $M^*_1$ is a class of $V[L]$, $j_n(M^*_1)$ is a class of $N_n$ and we may form the restriction
    $k_n \restriction j_n(M^*_1)$. Since $k_n \circ j_n = j_{01}^*$, it is routine to check that
    $k_n(j_n(M^*_1)) = j_{01}^*(M^*_1) = M_2^*$ and that $k_n \restriction j_n(M^*_1): j_n(M^*_1) \rightarrow M_2^*$ is elementary.

    To finish, we set $k = k_n \circ j_n(k_{n+1})$. To confirm this works,
    recall first that $j_n(j_{01}^*) = j_n(k_{n+1}) \circ j_n(j_{n+1})$.    
    Now
    \begin{align*}   k \circ i_n & = k_n \circ j_n(k_{n+1}) \circ j_n(j_{n+1}) \circ j_n \\
                             {} & = k_n \circ j_n(j_{01}^*) \circ j_n \\
                             {} & = k_n \circ j_n \circ j_{01}^* \\
                             {} & = j_{01}^* \circ j_{01}^*\\
                             {} & = j_{02}^*
    \end{align*}
    where the first equation holds because $k = k_n \circ j_n(k_{n+1})$ and $i_n = j_n(j_{n+1}) \circ j_n$, 
    the second equation holds because $j_n(j_{01}^*) = j_n(k_{n+1}) \circ j_n(j_{n+1})$,
    the third equation holds because $j_n \circ j_{01}^* = j_n(j_{01}^*) \circ j_n$, the fourth equation holds
    because $k_n \circ j_n = j_{01}^*$ and the last equation holds because $j_{02}^* = j_{01}^* \circ j_{01}^*$. 
    
    As for the critical point,
    $\crit(k_n) > j_{01}(\lambda^b_{17})$ for all $n$, so that $\crit(j_n(k_{n+1})) = j_n(\crit(k_{n+1})) > j_{01}(\lambda^b_{17})$,
    and since $k = k_n \circ j_n(k_{n+1})$ we have that $\crit(k) > j_{01}(\lambda^b_{17})$.

\smallskip

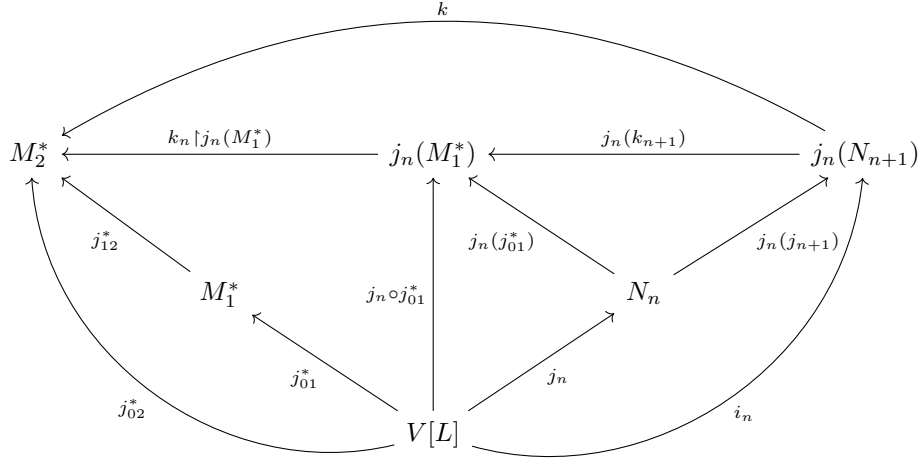
\begin{tikzcd}[row sep=huge,column sep=huge]
M_2^* &        &                    j_n(M_1^*)   \arrow{ll}[swap]{k_n \restriction j_n(M_1^*)}                        &                             &  j_n(N_{n+1}) \arrow{ll}[swap]{j_n(k_{n+1})} \arrow[bend right]{llll}[swap]{k} \\ 
&  M_1^* \arrow{lu}{j_{12}^*}  &                                               &   N_n  \arrow{ru}[swap]{j_n(j_{n+1})} \arrow{lu}{j_n(j_{01}^*)}  &              \\
&        &   V[L] \arrow{ru}[swap]{j_n} \arrow{lu}{j_{01}^*} \arrow[bend right=50]{ruru}[swap]{i_n} \arrow[bend left=50]{lulu}{j_{02}^*} \arrow{uu}{j_n \circ j_{01}^*} &                              &              \\
\end{tikzcd}

\end{proof}

The two-step iteration $i_n$ can be viewed as a one-step ultrapower by the measure
$U_n \times U_{n+1}$ on $P_\kappa \lambda^b_n \times P_\kappa \lambda^b_{n+1}$, where $A \in U_n \times U_{n+1}$
if and only if $\{ x : \{ y : (x, y) \in A \} \in U_{n+1} \} \in U_n$. We define product measures
$U_n \times U_{n+1} \times \ldots \times U_{n+i-1}$ with $i$ factors in a similar way. 
The following Lemma is an immediate consequence of the normality of the measures $U_j$. 
\begin{lemma} \label{Uprod}
  $A \in U_n \times U_{n+1} \times \ldots \times U_{n+i-1}$ if and only if 
  there exist sets $A_j \in U_j$ for $n \le j < n+i$ such that
  every $\prec$-increasing sequence from $\prod_{n \le j < n+i} A_i$ lies in $A$.
\end{lemma}

Motivated by the $i = 2$ case of Lemma \ref{Uprod} we define a modified version of the Cartesian product. 
\begin{definition}
  Let $A \subseteq P_\kappa \lambda^b_n$ and $B \subseteq P_\kappa \lambda^b_{n+1}$, then
  $A \times^\prec B = \{ (x, y) \in A \times B : x \prec y \}$. 
\end{definition}
With this definition, the $i = 2$ case of Lemma \ref{Uprod} states that
$U_n \times U_{n+1} = \{ X \subseteq P_\kappa \lambda^b_n \times P_\kappa \lambda^b_{n+1} :
\exists A \in U_n \; \exists B \in U_{n+1} \; A \times^\prec B \subseteq X \}$.

\smallskip

\noindent Global notation: $\times^\prec$\index[Notation]{$\times^{\prec}$}

\smallskip 

We will need the following version of Rowbottom's theorem, which also follows easily
from the normality of the measures $U_j$.
\begin{lemma} \label{rowbottom}
  Let $m < n < \omega$, let $(A_j)_{m \le j < n}$ be a sequence of sets with
  $A_j \in U_j$ and let $F$ be a colouring of the $\prec$-increasing sequences
  from $\prod_{m \le j < n} A_j$ in fewer than $\kappa$ colours. Then there exists a sequence $(B_j)_{m \le j < n}$ with
  $B_j \subseteq A_j$ and $B_j \in U_j$ such that $F \restriction \prod_{m \le j < n} B_j$
  is constant.
\end{lemma}

Now that we have defined $j_{01}^*$, we can define an auxiliary poset
that will be useful in Section \ref{prikryforcing}.
Recall that $Y$ is a measure one set of cardinals
which are less than $\kappa$ and reflect some properties of $\kappa$,
which we can think of as the ``potential Prikry points''. 
By  the agreement between $j_{01}^*$ and $j_n$,
$j_{01}^*(\Q)(\alpha, \kappa) = j_n(\Q)(\alpha, \kappa)$ for $\alpha \in Y$.
In a mild abuse of notation, we will write $\Q(\alpha, \kappa)$ for this poset. 

The following Lemma is immediate from Lemmas \ref{qdistrib} and \ref{qdistribmany}, together with the elementarity
of $j_{01}^*$ and the agreement between $V[L]$ and $M_{01}^*$.

\begin{lemma} \label{qdistribj}
  Let $\tau_0 < \ldots < \tau_n$ with $\tau_i \in Y$ for all $i$.
  Then $\prod_{0 \le i < n} \Q(\tau_i, \tau_{i+1}) \times \Q(\tau_n, \kappa)$ is
  $<\Lambda^b_{\omega+2}(\tau_0)$-distributive. In particular, in the case $n = 0$,
  $\Q(\tau_0, \kappa)$ is $<\Lambda^b_{\omega+2}(\tau_0)$-distributive. 
\end{lemma}

\smallskip

\noindent Global notation: $\Q(\alpha, \kappa)$\index[Notation]{$\Q(\alpha, \kappa)$}

\smallskip 

\section{Prikry forcing} \label{prikryforcing}

Let $\Aggforcing= \Add^V(\lambda^b_{\omega+2}, j_{01}(\lambda^a_0))$, let $\Agg$ be $\Aggforcing$-generic over $V[L]$,
and let $K \in V[L][\Agg]$ be the $\Q_\infty$-generic filter over $M_1^*$ constructed in Section \ref{auxcomp}.
Working in $V[L][\Agg]$ we will define a Prikry-type forcing $\bar \P$.
Conditions in $\bar \P$ will each lie in $V[L]$, but $K$ will be required to recognise the
set of conditions, so that $\bar\P \in V[L][\Agg]$. Since $\Agg$ is generic over $V[L]$ for $<\lambda^b_{\omega+2}$-distributive forcing, 
the models $V[L]$ and $V[L][\Agg]$ agree on bounded subsets of $\lambda^b_{\omega+2}$.
We will use this agreement without comment at several points below.

\subsection{Defining the forcing} \label{defining}  
The definition of $\bar \P$  will use the measures $U_n$ for $n \ge 17$.
A typical point for $U_n$ is a set $x \in P_\kappa \lambda^b_n$ with
$\kappa(x) = x \cap \kappa \in Y$. In a mild abuse of notation we
write (for example) ``$\Lambda^b_n(x)$'' as a shorthand for ``$\Lambda^b_n(\kappa(x))$''.  

The poset $\bar \P$ will add a sequence $\langle x_n : 17 \le n < \omega \rangle$ where:
\begin{itemize}
\item $x_n \in P_\kappa \lambda^b_n$.
\item $\kappa(x_n) \in Y$.
\item The sequence is $\prec$-increasing, that is $x_n \subseteq x_{n+1}$ and $\ot(x_n) < \kappa(x_{n+1})$.
\end{itemize}
We call the $x_n$'s the ``supercompact Prikry points'', and the associated cardinals
$\kappa(x_n)$ ``the Prikry points''. 

When $x$ and $y$ are successive supercompact Prikry points, the forcing poset $\bar \P$
will add a generic object for the poset $\Q(\kappa(x), \kappa(y))$
as defined in Section \ref{successiveprikry}. 
Recall from Section \ref{successiveprikry} that since $\kappa(x), \kappa(y) \in Y$ the preparation
forcing $\FL$ did some collapsing in a block of cardinals
associated
with $\kappa(x)$, and some more collapsing at a higher block of cardinals
associated with $\kappa(y)$:  
the point of forcing with $\Q(\kappa(x),\kappa(y))$ is to ``close the gap'' between these
two blocks of cardinals. When $x$ is the first supercompact Prikry point $\bar \P$ will add
a generic object for the poset $\Q^*(\kappa(x))$ as defined in Section \ref{firstprikry}.
In the sequel we will lighten the notation by writing ``$\Q(x, y)$'' for $\Q(\kappa(x),\kappa(y))$
and ``$\Q^*(x)$'' for $\Q^*(\kappa(x))$.

\smallskip

\noindent Global notation:  $\Q(x, y)$\index[Notation]{$\Q(x, y)$},
$\Q^*(x)$\index[Notation]{$\Q^*(x)$}

\smallskip

 Conditions in $\bar \P$ have the form
\[ p = \langle q_{17},x_{17}, \dots q_{n-1},x_{n-1},f_n,A_n,F_{n+1}, A_{n+1} ,F_{n+2},
A_{n+2}, \dots \rangle \]
where:
\begin{enumerate}
\item $n \ge 17$ (so that for $n = 17$ the condition $p$ is of the form $\langle f_{17}, A_{17}, F_{18}, \ldots \rangle$).
\item $A_j \in U_j$ for all $j \ge n$.   
\item For all $i \geq n+1$, $F_i$ is a function with domain
  $A_{i-1} \times^\prec A_i$, such that  $F_i(x, y) \in \Q(x, y)$ for all $(x, y) \in A_{i-1} \times^\prec A_i$
  and $[F_i]_{U_{i-1} \times U_i} \in K$.
\item $\langle x_i \mid 17 \leq i < n \rangle$ is a $\prec$-increasing sequence
where $x_i \in \mathcal{P}_\kappa(\lambda^b_i)$ and $\kappa(x_i) \in Y$. 
\item If $n > 17$, then
  \begin{enumerate}
  \item  For all $m \ge n$ and all $y \in A_m$,  $x_{n-1} \prec y$.   
  \item  $q_{17} \in \Q^*(x_{17})$.
  \item  $q_i \in \Q(x_{i-1},x_i)$ for all $i$ with $17 < i < n$.
  \item  $\dom(f_n) = A_n$ and $f_n(x) \in \Q(x_{n-1}, x)$ for all $x \in A_n$.
  \end{enumerate}    
\item If $n = 17$, then $f_n$ is a function with $\dom(f_n) = A_n$ such that
  $f_n(x) \in \Q^*(x)$ for all $x \in A_n$.  
\end{enumerate}

\smallskip

\noindent Global notation: $\bar\P$\index[Notation]{$\bar\P$}

\smallskip

The {\em length} $\lh(p)$ of $p$ is $1$ plus the index of the last $x_i$ entry in $p$, 
so that $\lh(p) = n$ for the condition displayed above. Note that the length
of a condition is the index of the measure one set from which the next
``$x$ point'' will be drawn when the condition is extended. 

For $p$ as above, the {\em lower part} of $p$
is the initial segment
\[
\langle q_{17},x_{17}, \dots q_{n-1}, x_{n-1} \rangle
\]
 and the
{\em stem} (written $\stem(p)$) of the condition $p$ is
\[
\langle q_{17},x_{17}, \dots q_{n-1}, x_{n-1}, [f_n]_{U_n} \rangle
\]
The {\em length} $\lh(h)$ of a stem $h$ is the length of the corresponding condition,
  so that the stem displayed above has length $n$.
The {\em upper part or constraint part} is
\[ 
\langle f_n, A_n, F_{n+1}, A_{n+1} ,F_{n+2}, \ldots  \rangle.
\]

\smallskip

\noindent  Global notation: $\lh(p)$\index[Notation]{$\lh(p)$}, $\stem(p)$\index[Notation]{$\stem(p)$}

\smallskip

\begin{remark} The point of distinguishing stems and lower parts is that the function $f_n$ can
  be a source of incompatibility between two conditions of the same length.
\end{remark}

\begin{remark} \label{countstems} 
  Since $\vert \Q(\tau, \tau^*) \vert, \vert \Q^*(\tau^*) \vert < \Lambda^b_0(\tau^*)$, there are fewer than $\lambda^b_0$ possibilities
  for $[f_n]_{U_n}$. Since $(\lambda^b_n)^{<\kappa} = \lambda^b_n$ for all $n$, it follows that there
  are $\lambda^b_n$ possible stems for conditions of length $n + 1$.
\end{remark}

Suppose that
\[
p' = \langle q'_{17},x'_{17}, \dots q'_{m-1},x'_{m-1},f_m',A_m'
F'_{m+1},A'_{m+1}, \dots \rangle
\]
is another condition.  Then $p' \leq p$
if:

\begin{enumerate}
\item $m \geq n$.
\item $\vec{x}'$ end-extends $\vec{x}$, that is to say $x_i = x_i'$ for $17 \le i < n$.
\item For all $i$ such that $n \le i < m$,  $x_i' \in A_i$.
\item For all $i \geq m$, $A_i' \subseteq A_i$.
\item If $m >n$, then
\begin{enumerate}
\item $q_n' \leq f_n(x_n)$,
\item for all $i$ such that $n < i < m$, $q_i' \leq F_i(x_{i-1},x_i)$ and 
\item for all $x \in A_m'$, $f_m'(x) \leq F_m(x_{m-1}',x)$.
\end{enumerate}
\item If $m=n$, then for all $x \in A_m'$, $f_m'(x) \leq f_m(x)$.
\item For all $i < n$, $q_i' \leq q_i$.
\item For all $i \geq m+1$ and all $(x,y) \in
A_{i-1}' \times^\prec A_i'$, $F_i'(x,y) \leq F_i(x,y)$.
\end{enumerate}

\begin{remark}
  Since the definition of $\bar \P$ includes the demands
  that $A_i \in U_i$ and $[F_i] \in K$, incompatibility between conditions of the same length
  can only arise from the stems.
\end{remark}

In the case when $q \le p$ with $\lh(q) = \lh(p)$ we say that $q$ is a {\em direct extension of $p$} and write $q \le^* p$.
When $\lh(q) - \lh(p) = t$ we say that $q$ is a {\em $t$-step extension} of $p$. 
As is typical for Prikry-type forcing posets, when $q \le p$ we may view $q$ as obtained by first adding the points
$x_i$ for $\lh(p) \le i < \lh(q)$, and then taking a direct extension
of the result.

More formally:
\begin{definition} \label{minimalextension}
  Let
\begin{align*}
  p & = \langle q_{17}, x_{17}, \dots q_{n-1}, x_{n-1}, f_n, \\
  & A_n, F_{n+1}, A_{n+1}, F_{n+2}, A_{n+2}, \dots \rangle \\
  \end{align*}
  and let $\vec x = (x_n, \ldots x_{n+t-1})$ be a $\prec$-increasing non-empty sequence such that $x_j \in A_j$ for $n \le j < n+t$ and
  $x_{n-1} \prec x_n$.
  Then $p\cat \vec x$ (the {\em minimal extension of $p$ by $\vec x$})
  is the condition 
  \begin{align*}
    & \langle q_{17}, x_{17}, \dots q_{n-1}, x_{n-1}, q_n, x_n, \ldots q_{n+t-1}, x_{n+t-1}, f_{n+t}, \\
    & A^*_{n+t},  F^*_{n+t+1}, A^*_{n+t+1} \dots \rangle \\
  \end{align*}
where $q_n = f_n(x_n)$, $q_{n + k} = F_{n+k}(x_{n+k-1}, x_{n+k})$ for $0 < k < t$,
$A^*_{n+k} = \{ y \in A_{n+k} : x_{n+t-1} \prec y \}$ for $k \ge t$,
$F^*_{n+k} = F_{n+k} \restriction A^*_{n+k-1} \times A^*_{n + k}$ for $k > t$,
and $f_{n+t}(y) = F_{n+t}(x_{n+t-1}, y)$ for $y \in A^*_{n+t}$.
\end{definition}

\smallskip

\noindent Global notation: $p\cat {\vec x}$\index[Notation]{$p\cat {\vec{x}}$}

\smallskip

By convention $p\cat \vec x = p$ for $\vec x$ empty, and we abuse notation by writing $p\cat x$ for $p\cat \langle x \rangle$
for sequences of length one. The following Lemma is routine:

\begin{lemma} $p\cat {\vec x} \le p$, and if $q \le p$ then there is a unique $\vec x$ such that $q \le^* p\cat \vec x$.
\end{lemma}

\begin{lemma} \label{samestemub} 
  Let $p, q \in \bar \P$ with $\stem(p) = \stem(q) = h$. Then
  there is a lower bound $r \le p, q$ with $\stem(r) = h$.
\end{lemma}

\begin{proof} Let the common length of $p$ and $q$ be $n$.
  We choose the lower part of $r$ to agree with the common lower part of $p$ and $q$.
  The main point is that 
  $[f^p_n]_{U_n} = [f^q_n]_{U_n}$, so that $\{ x : f^p_n(x) = f^q_n(x) \} \in U_n$. We may therefore choose
  $f^r_n$ such that $f^r_n(x) = f^p_n(x) = f^q_n(x)$ for all $x \in \dom(f^n_r)$.
  It is now easy to choose the remaining entries of $r$ to ensure that $r \le p, q$.
\end{proof}

\begin{lemma} \label{ccforPbar}
  In $V[L][\Agg]$ the poset $\bar \P$ is $\lambda^b_\omega$-centered, in particular it has the
  $(\lambda^b_\omega)^+$-cc.
\end{lemma}   

\begin{proof}
   It follows from Remark \ref{countstems} that the total number of stems is $\lambda^b_\omega$.
   The conclusion is now immediate from Lemma \ref{samestemub}.
\end{proof}    
   
Essentially the same proof as for Lemma \ref{samestemub} shows:
\begin{lemma} \label{omegasamestemub} 
  Let $h$ be a stem, let $\nu < \kappa$, and let $p_i \in \bar \P$ for $i < \nu$, with $\stem(p_i) = h$ for all $i$.  Then
  there is $r$ such that $\stem(r) = h$ and  $r \le p_i$ for all $i$. 
\end{lemma}

\begin{remark} We only need Lemma \ref{omegasamestemub} in the case where $\nu = \omega$.
  It will be used to verify Hypothesis \ref{preslemmahyp7} when we appeal
  to Lemma \ref{preslemma}.
\end{remark}

We define $\P$ to be the set of $p$  which satisfy all the conditions
for membership in $\bar \P$, except the condition that $[F_i]_{U_{i-1} \times U_i} \in K$.
Note that $\P \in V[L]$. We can view $\P$ as the set of potential elements of $\bar \P$.  

It will be convenient to factor the forcing poset $\bar \P \downarrow p$ for $p \in \bar \P$ in various ways. Let 
\[
p = \langle q_{17},x_{17}, \dots q_{n-1}, x_{n-1}, f_n, A_n, F_{n+1}, A_{n+1} ,F_{n+2}, \ldots \rangle
\]
 Let $\tau_j = \kappa(x_j)$ for $17 \le j < n$,
 and let $17 \le m < n -1$. Then $\bar \P$ below $p$ is isomorphic to $\P_{\rm low} \downarrow p_0 \times \P_{\rm high} \downarrow p_1$
 where:
 \begin{enumerate}
 \item $\P_{\rm low} = \Q^*(\tau_{17}) \times \prod_{17 < j \le m} \Q(x_{j-1}, x_j)$.
 \item $p_0 = (q_{17}, \ldots, q_m)$.    
 \item $\P_{\rm high}$ is defined in a similar way to $\bar \P$, with conditions of the form
   \[
   \langle q'_{m+1}, x'_{m+1}, \dots q'_{n'-1}, x'_{n'-1}, f'_{n'}, A'_{n'}, F'_{n'+1}, A'_{n'+1} ,F'_{n'+2}, \ldots \rangle
   \] 
 ordered in the natural way.  
 \item $p_1 = \langle q_{m+1}, x_{m+1}, \dots q_{n-1}, x_{n-1}, f_n, A_n, F_{n+1}, A_{n+1} ,F_{n+2}, \ldots \rangle$
 \end{enumerate}

 It follows that if $G$ is $\bar \P$-generic and $\langle \tau_j : 17 \le j < \omega \rangle$ is the
 Prikry sequence added by $G$, then for every $m \ge 17$ the generic object $G$ induces a
 $\Q^*(\tau_{17}) \times \prod_{17 < j \le m} \Q(x_{j-1}, x_j)$-generic filter.
 
 \begin{remark} Formally the posets $\P_{\rm low}, \P_{\rm high}$ and conditions $p_0, p_1$ depend on the choice
   of $m$. When we use this kind of factorisation in the sequel, the value of $m$ should
   always be clear from the context.
 \end{remark}

\subsection{The Prikry lemma}

Recall from section \ref{auxcomp} that in $V[L]$ we derived measures $U_n$ on $P_\kappa \lambda^b_n$ for $17 \le n < \omega$
from the embedding $j_{01}^*:V[L] \rightarrow M_1^*$, and formed ultrapower maps $j_n: V[L] \rightarrow N_n = Ult(V[L], U_n)$.
We arranged that if $k_n: N_n \rightarrow M_1^*$ is the natural factor map with $j_{01}^* = k_n \circ j_n$, then
$\crit(k_n) > j_{01}(\lambda^b_{17})$. It follows that for $\alpha \le \lambda^b_{17}$ we have
$j_{01}(\alpha) = j_n(\alpha)$.

Recall also that $i_n = j_n(j_{n+1}) \circ j_n$, and that
$\Q_\infty = i_n(\Q)(\kappa, j_n(\kappa))$ for all $n \ge 17$.
Now $\Q_\infty \in M_{01}^*$, and in $M_{01}^*$ we have  $\vert \Q_\infty \vert = j_{01}(\lambda^a_{\omega+2})$
and $2^{j_{01}(\lambda^a_{\omega+2})} = j_{01}(\lambda^b_1)$.
It follows that $\Q_\infty \in N_n$ and $K$ is $\Q_\infty$-generic over $N_n$ for all $n \ge 17$.
By similar arguments, if we let $N_n^+ = Ult(N_n, j_n(U_{n+1}))$, so that
$i_n: V[L] \rightarrow N_n^+$, then $\Q_\infty \in N_n^+$ and 
$K$ is $\Q_\infty$-generic over $N_n^+$.

\smallskip

\noindent  Global notation: $N_n^+$\index[Notation]{$N_n^+$}

\smallskip

For each $n$ with $17 \le n < \omega$, $\vert P_\kappa \lambda^b_n \vert = \lambda^b_n$ in $V[L]$, so that
$U_n$ is still a supercompactness measure on $P_\kappa \lambda^b_n$ in the $< \lambda^b_{\omega+2}$-distributive extension
$V[L][\Agg]$. It follows that $j_n$ lifts to the ultrapower map computed from $U_n$ in $V[L][\Agg]$,
and we write $j_n^A: V[L][\Agg] \rightarrow N_n^A = N_n[j_n^A(\Agg)]$. Similarly
$i_n$ lifts, and we obtain $i_n^A: V[L][\Agg] \rightarrow N_n^{A+} = N_n^+[i_n^A(\Agg)]$.
By distributivity it is easy to see that $K$ is still $\Q_\infty$-generic over the models
$N_n^A$ and $N_n^{A+}$. We also note that $\Q_\infty$ is still $<\lambda^b_{\omega + 2}$-distributive in each
of the models $N_n$, $N_n^+$, $N_n^A$ and $N_n^{A+}$.

\smallskip

\noindent   Global notation: $j_n^A$\index[Notation]{$j_n^A$}, $i_n^A$\index[Notation]{$i_n^A$},
$N_n^A$\index[Notation]{$N_n^A$}, $N_n^{A+}$\index[Notation]{$N_n^{A+}$}

\smallskip

Fix $E$ a dense open subset of $\bar \P$ with $E \in V[L][\Agg]$,
and let $E^{(k)}$ be the dense open set of conditions whose every $k$-step
extension lies in $E$. 
We describe a series of steps to ``canonise''
membership in $E$.  

\smallskip

\noindent Global notation: $E^{(k)}$\index[Notation]{$E^{(k)}$}

\smallskip

For each $n > 17$ we define ${\mathbf F}_n$ to be the set of functions of two variables $F$
such that $\dom(F) = A \times^\prec B$ for some $A \in U_{n-1}$ and $B \in U_n$,
and $F(x, y) \in \Q(\kappa(x), \kappa(y))$ for all
$(x, y) \in \dom(F)$: that is to say, ${\mathbf F}_n$ is the set of functions which can
appear as $F^p_n$ for some $p \in \P$. In this situation, for each $x \in A$  we define $F(x, -)$
to be the function with domain $\{ y \in B : x \prec y \}$ given by $F(x, -)(y) = F(x, y)$. 

\smallskip

\noindent Global notation: ${\mathbf F}_n$\index[Notation]{${\mathbf{F}}_n$},
$F(x, -)$\index[Notation]{$F(x, -)$}

\smallskip

We define $L_n$ to be the set of lower parts
$s$ of the form $q_{17} \ldots x_{n-1}$.
When $n > 17$ and $s = q_{17} \ldots x_{n-1} \in L_n$ we let $\kappa(s) = \kappa(x_{n-1})$,
and for $x \in P_\kappa \lambda^b_n$ we write  $s \prec x$ for $x_{n-1} \prec x$.
By convention $L_{17}$ is the singleton set containing the empty sequence,
$\langle \rangle \prec x$ for all $x \in P_\kappa \lambda^b_{17}$, and
$\Q(\langle \rangle, x) = \Q^*(\kappa(x))$ for all $x \in P_\kappa \lambda^b_{17}$.

\smallskip

\noindent Global notation:  $L_n$\index[Notation]{$L_n$}

\smallskip

It easy to
see that if $L \subseteq L_n$ and $(A_s)_{s \in L}$ is an $L$-indexed family of sets in
$U_n$, then $\{ x \in P_\kappa \lambda^b_n : \forall s \in L\; s \prec x \implies x \in A_s \} \in U_n$. In the sequel we use this
form of normality for $U_n$ without comment.

\begin{lemma} \label{Prikrystep1}
  There exist functions $(F^0_n)_{n > 17}$ and sets $(A^0_n)_{n \ge 17}$ such that:
 \begin{itemize}
 \item $A^0_n \in U_n$.    
 \item $\dom(F^0_n) = A^0_{n-1} \times^\prec A^0_n$.
\item  $[F^0_n]_{U_{n-1} \times U_n} \in K$.
\item
  For every $k$, every $n \ge 17$, every $x \in A^0_n$,  every
  lower part $s \in L_n$ with $s \prec x$, 
  and every condition $q \in \Q(\kappa(s), \kappa(x))$,
  one of the two mutually exclusive conditions holds:
  \begin{itemize}
  \item There is a condition in $E^{(k)}$ with an initial segment of the form
    \[
    s^\frown q^\frown x^\frown F^0_{n+1}(x, -).
    \]
  \item
    There is no condition in $E^{(k)}$ with an initial segment of the form
    \[
    s^\frown q^\frown x^\frown {f_{n+1}}
    \]
    where $f_{n+1} \le F_{n+1}^0(x,-)$.
  \end{itemize}  
 \end{itemize}
\end{lemma}

\begin{proof}
  Fix $n$ for the moment.
  Recall that $j_n^A:V[L][\Agg] \rightarrow N_n^A$ is the ultrapower map computed from $U_n$ in $V[L][\Agg]$,
  and $j_n^A$ is a lift of $j_n$. Let $x^1_n = j_n[\lambda^b_n]$, so that
  $U_n = \{ X \subseteq P_\kappa \lambda^b_n : x^1_n \in j_n(X) \}$.
   Observe that
  $\{ t \in j_n(L_n) : t \prec x^1_n \} = j_n[L_n] \in N_n$.

  The key point is now to observe that if $F \in {\mathbf F}_{n+1}$ then $j_n(F)(x^1_n, -) \in N_n$, 
  and is a function which can be integrated in $N_n$ with respect to $j_n(U_{n+1})$ to obtain $[F]_{U_n \times U_{n+1}} \in   \Q_\infty$.
  For each $k$, each $s \in L_n$ and each $Q \in \Q^{N_n}(\kappa(s), \kappa)$ we define in $N_n^A$ a dense open set of conditions
  in $\Q_\infty$, namely the set of conditions $r \in \Q_\infty$ such that
  one of the following mutually exclusive conditions
  holds:
  \begin{itemize}
  \item
    There is a condition in $j_n(E^{(k)})$ with an initial segment of the form
    \[
    j_n(s)^\frown Q^\frown {x^1_n}^\frown {f_{n+1}}
    \]
    where $[f_{n+1}]_{j_n(U_{n+1})} = r$. 
  \item
    There is no condition in $j_n(E^{(k)})$ with an initial segment of the form
    \[
    j_n(s)^\frown Q^\frown {x^1_n}^\frown {f_{n+1}}
    \]
    where $[f_{n+1}]_{j_n(U_{n+1})} \le r$. 
  \end{itemize}
  
  By the genericity of $K$ over $N_n^A$ and the distributivity of $\Q_\infty$ in $N_n^A$, there is
  $r_n \in K$ which is in the dense set defined above for every $s$ and $Q$. We choose
  $F'_{n+1} \in {\mathbf F}_{n+1}$ such that $[F'_{n+1}]_{U_n \times U_{n+1}} = r_n$, that is $[j_n(F'_{n+1})(x^1_n, -)]_{j_n(U_{n+1})} = r_n$.
  
  Working in $V[L][\Agg]$, let $A'_n$ be the set of $x \in P_\kappa \lambda^b_n$ such that
  for every $k$, every $s \in L_n$ with $s \prec x$, and every $q \in \Q(\kappa(s), \kappa(x))$
  one of the following mutually exclusive conditions
  holds:
  \begin{itemize}
  \item[$1_x$)]
    There is a condition in $E^{(k)}$ with an initial segment of the form
    \[
    s^\frown q^\frown x^\frown {f_{n+1}}
    \]
    where $[f_{n+1}]_{U_{n+1}} = [F'_{n+1}(x, -)]_{U_{n+1}}$.  
  \item[$2_x$)]
    There is no condition in $E^{(k)}$ with initial segment of the form
    \[
    s^\frown q^\frown x^\frown {f_{n+1}}
    \]
    where $[f_{n+1}]_{U_{n+1}} \le [F'_{n+1}(x, -)]_{U_{n+1}}$.  
  \end{itemize}
  By {\L}o\'s's theorem $A'_n \in U_n$.
  
  For each $k$, $x \in A'_n$, $s \prec x$, and $q \in \Q(\kappa(s), \kappa(x))$
  such that $1_x$ holds, let $B'_{n+1}(k, s, q, x) \in U_{n+1}$ be such that
  there is a condition in $E^{(k)}$ with initial segment
  \[
  s^\frown q^\frown x^\frown {F'_{n+1}(x, -) \restriction B'_{n+1}(k, s, q, x)}.
  \]
  Let $B'_{n+1}$ be the set of $y \in P_\kappa \lambda^b_{n+1}$ such that $y \in B'_{n+1}(k, s, q, x)$ for every $k$, every $x \in A'_n$ with
  $x \prec y$ and every relevant $s$ and $q$, so that $B'_{n+1} \in U_{n+1}$ by normality. 
  
  Now we choose $A^0_n \in U_n$ so that $A^0_n \subseteq A'_n \cap B'_n$ for every relevant $n$,
  let $F^0_n = F'_n \restriction A^0_{n-1} \times^\prec A^0_n$, and verify that this satisfies the desired property.
  Let $n \ge 17$ and suppose that $k < \omega$, $x \in A^0_n$,  $s \in L_n$ with $s \prec x$, and $q \in \Q(\kappa(s), \kappa(x))$.
  By construction $x \in A'_n$.

  Suppose first that $1_x$ holds, so that we defined $B'_{n+1}(k, s, q, x)$.
  By definition $\dom(F^0_{n+1}(x, -) ) \subseteq \dom(F'_{n+1}(x, -)) \subseteq \{ y \in B'_{n+1} : x \prec y \} \subseteq B'_{n+1}(k, s, q, x)$,
  so that $F^0_{n+1}(x, -) \le F'_{n+1}(x, -) \restriction B'_{n+1}(k, s, q, x)$ and hence there is a condition
  in $E^{(k)}$ with initial segment
  \[
  s^\frown q^\frown x^\frown {F^0_{n+1}(x, -)}.
  \]
  If alternatively $2_x$ holds then {\em a fortiori} there is no condition in $E^{(k)}$ with an initial segment of the form
  \[
  s^\frown q^\frown x^\frown {f_{n+1}}
  \]
  where
  $f_{n+1} \le F_{n+1}^0(x,-)$, because in this case we have \[ [f_{n+1}]_{U_{n+1}} \le [F'_{n+1}(x, -)]_{U_{n+1}} = [F^0_{n+1}(x, -)]_{U_{n+1}} .\]  
\end{proof}

\begin{lemma} \label{Prikrystep2}
  There exist functions $(F^1_n)_{n > 17}$ and sets $(A^1_n)_{n \ge 17}$ such that:
 \begin{itemize}
 \item $A^1_n \in U_n$ with $A^1_n \subseteq A^0_n$.
 \item $\dom(F^1_n) = A^1_{n-1} \times^\prec A^1_n$ with $F^1_n \le F^0_n$.
\item  $[F^1_n]_{U_{n-1} \times U_n} \in K$.
\item For every $k$, every $n \ge 17$, every $x \in A^1_n$, and every $t \in L_{n+1}$ with $t = s^\frown q^\frown x$, if there is a condition
  in $E^{(k)}$ with initial segment
  \[
  t \frown {F^0_{n+1}(x, -)}
  \]
  then
  \[
  t \frown {F^0_{n+1}(x, -)}^\frown (A^1_{m-1}, F^*_m)_{m > n + 1} \in E^{(k)},
  \]
  where
  $F^*_m = F^1_m \restriction \{ (y, z) \in A^1_{m-1} \times A^1_m : x \prec y \prec z \}$.
 \end{itemize}
\end{lemma}

\begin{proof}
  Fix $n$ for the moment. For every $k$, $x \in A^0_n$, $s \in L_n$ with $s \prec x$, and
  $q \in \Q(\kappa(s), \kappa(x))$, let $t = s^\frown q^\frown x$ (so that $t \in L_{n+1}$) and if
  there is a condition in $E^{(k)}$ with initial segment $t \frown {F^0_{n+1}(x, -)}$ then choose
  such a condition $p^{t, k}$. To lighten the notation let $F^{t,k}_m = F^{p^{t, k}}_m$ for $m > n + 1$.

  For all $k$, $t$, and  $m$, $[F^{t, k}_m]_{U_{m-1} \times U_m} \in K$. Since $K$ is generic over $M_1^*$,
  it follows from the closure properties of $M_1^*$ and the distributivity of $\Q_\infty$ that
  there exists a sequence $(G^n_m)_{m > n + 1}$ such that $G^n_m \in {\mathbf F_m}$,
  $[G^n_m]_{U_{m-1} \times U_m} \in K$ and $[G^n_m]_{U_{m-1} \times U_m} \le [F^{t, k}_m]_{U_{m-1} \times U_m}$
  for all $t$ and $k$. Using closure and distributivity again 
  there exists a sequence $(G_m)$ such that $G_m \in {\mathbf F_m}$,
  $[G_m]_{U_{m-1} \times U_m} \in K$, $[G_m]_{U_{m-1} \times U_m} \le [F^0_m]_{U_{m-1} \times U_m}$,
  and $[G_m]_{U_{m-1} \times U_m} \le [G^n_m]_{U_{m-1} \times U_m}$ for all $n > m + 1$.

  By taking appropriate diagonal intersections to define the sets $A^1_m$ and setting $F^1_m = G_m \restriction A^1_{m-1} \times_\prec A^1_m$,
  we may arrange that for every
  $k$, $n$,  $t \in L_{n+1}$, $m > n + 1$ and $(y, z) \in \dom(F^1_m)$ with $t \prec y \prec z$,
  we have $F^1_m(y, z) = G_m(y, z) \le G^n_m(y, z) \le F^{t, k}_m(y, z)$ and $F^1_m(y, z) \le F^0_m(y, z)$.
  To verify that this works, let $t = s^\frown q^\frown x \in L_{n+1}$
  with $x \in A^1_n$ and assume that there is a condition in $E^{(k)}$ with initial segment $t \frown {F^0_{n+1}(x, -)}$,
  so that we chose $p^{t, k} \in E^{(k)}$. The desired conclusion is immediate. 
\end{proof}

\begin{lemma} \label{Prikrystep3}
  There exist functions $(F^2_n)_{n > 17}$ and sets $(A^2_n)_{n \ge 17}$ such that:
 \begin{itemize}
 \item $A^2_n \in U_n$ with $A^2_n \subseteq A^1_n$.
 \item $\dom(F^2_n) = A^2_{n-1} \times^\prec A^2_n$ with $F^2_n \le F^1_n$.
\item  $[F^2_n]_{U_{n-1} \times U_n} \in K$.
\item  For every $k$, every $n \ge 17$, every  $(x, y) \in \dom(F^2_{n+1})$, every $s \in L_n$ with $s \prec x$, and every
  $q \in \Q(\kappa(s), \kappa(x))$, 
  one of the following mutually exclusive statements holds:
\begin{itemize}
\item There is a condition in $E^{(k)}$ with initial segment
  \[
  s^\frown q^\frown x^\frown {F^2_{n+1}(x, y)}^\frown y^\frown {F^1_{n+2}(y, -)}.
  \]
\item There is no condition in $E^{(k)}$ with initial segment of the form
  \[
  s^\frown q^\frown x^\frown {\bar r}^\frown y^\frown {F^1_{n+2}(y, -)}
  \]
  where $\bar r \le F^2_{n+1}(x, y)$.
\end{itemize}
   \end{itemize}
\end{lemma}

\begin{proof} 
  As in the proof of Lemma \ref{Prikrystep1}, let $x^1_n = j_n[\lambda^b_n]$.
  Let $x^2_n = j_n(j_{n+1})(x^1_n) = i_n[\lambda^b_n]$, and
  let $y^2_{n+1} = j_n(j_{n+1})[j_n(\lambda^b_{n+1})]$.  
  By a routine calculation
  $U_n \times U_{n+1} = \{ X \subseteq P_\kappa \lambda^b_n \times P_\kappa \lambda^b_{n+1} : (x^2_n, y^2_{n+1}) \in i_n(X) \}$.
  It is easy to see that $i_n[L_n] = \{ t \in i_n(L_n) : t \prec x^2_n \}$.

  Let $k < \omega$, $s \in L_n$, and $Q \in \Q^{N_n^+}(\kappa(s), \kappa)$. Working in $N_n^{+A}$ define the dense open set of conditions
  $r \in \Q_\infty$ such that
  one of the following mutually exclusive conditions holds:
  \begin{itemize}
  \item There is a condition in $i_n(E^{(k)})$ with initial segment
    \[
    i_n(s)^\frown Q^\frown {x^2_n}^\frown r^\frown {y^2_{n+1}}^\frown
    {i_n(F^1_{n+2})(y^2_{n+1}, -)}.
    \]
  \item There is no condition in $i_n(E^{(k)})$ with initial segment of the form
    \[
    i_n(s)^\frown Q^\frown {x^2_n}^\frown {\bar r}^\frown {y^2_{n+1}}^\frown {i_n(F^1_{n+2})(y^2_{n+1}, -)}
    \]
    where ${\bar r} \le r$.
  \end{itemize}  
  Using the genericity of $K$ over $N_n^{A+}$ and the distributivity of $\Q_\infty$ in this model, we find
  $r_n \in K$ which lies in this dense open set for every $k$, $s$, and $Q$, and fix $F''_{n+1} \in {\mathbf F}_{n+1}$ such that
  $[F''_{n+1}]_{U_n \times U_{n+1}} = i_n(F''_{n+1})(x^2_n, y^2_n) = r_n$.

  By {\L}o\'s's theorem there is a set $C_{n+1} \in U_n \times U_{n+1}$ such that for every $k$, $(x, y) \in C_{n+1}$,
   $s \in L_n$ with $s \prec x$, and every $q \in \Q(\kappa(s), \kappa(x))$
  one of the following mutually exclusive conditions holds:
  \begin{itemize}
  \item There is a condition in $E^{(k)}$ with initial segment
    \[
    s^\frown q^\frown x^\frown {F''_{n+1}(x, y)}^\frown y^\frown {F^1_{n+2}(y,-)}.
    \]
  \item There is no condition in $E^{(k)}$ with an initial segment of the form
    \[
    s^\frown q^\frown x^\frown {\bar r}^\frown y^\frown {F^1_{n+2}(y,-)}
    \]
    where $\bar r \le F''_{n+1}(x, y)$.
  \end{itemize}  

  Now we choose $F^2_n \le F''_n, F^1_n$ and $A^2_n \subseteq A^1_n$ so that
  $\dom(F^2_{n+1}) = A^2_n \times^\prec A^2_{n+1} \subseteq C_{n+1}$. 
  Clearly this satisfies the requirements. 
\end{proof}

\begin{lemma} \label{Prikrystep4}
   There exist sets $(A^3_n)_{n \ge 17}$ such that:
 \begin{itemize}
 \item $A^3_n \in U_n$ with $A^3_n \subseteq A^2_n$.
 \item For every $k$, every $n \ge 17$, and every  $t \in L_{n+1}$, one of the following mutually exclusive conditions holds:
\begin{itemize}
\item
  For every $y \in A^3_{n+1}$ with $t \prec y$,
  there is a condition in $E^{(k)}$ with initial segment
  \[
  t^\frown {F^2_{n+1}(x, y)}^\frown y^\frown F^2_{n+2}(y, -).
  \]
\item   
  For every $y \in A^3_{n+1}$ with $t \prec y$,
  there is no condition in $E^{(k)}$ with initial segment
  \[
  t^\frown {F^2_{n+1}(x, y)}^\frown y^\frown F^2_{n+2}(y, -).
  \]
\end{itemize} 
 \end{itemize}
\end{lemma}

\begin{proof}
 For every $k$, $n$ and $t = s^\frown q^\frown x \in L_{n+1}$ partition $\{ y \in A^2_{n+1} : t \prec y \}$ as follows:
 \begin{itemize}
 \item
   $A^+_{n+1}(t, k)$ is the set of $y \in A^2_{n+1}$ such that $t \prec y$ and
   there is a condition in $E^{(k)}$ with initial segment
   \[
   t^\frown {F^2_{n+1}(x, y)}^\frown y^\frown F^2_{n+2}(y, -).
   \]
 \item
   $A^-_{n+1}(t, k)$ is the set of $y \in A^2_{n+1}$ such that $t \prec y$ and there is no condition in $E^{(k)}$ with initial segment
   \[
   t^\frown {F^2_{n+1}(x, y)}^\frown y^\frown F^2_{n+2}(y, -).
   \]
 \end{itemize}  

 Let $A^3_{n+1}(t, k)$ be whichever of the sets $A^+_{n+1}(t, k)$ and $A^-_{n+1}(t, k)$ lies in $U_{n+1}$, 
 and then let $A^3_{n+1} = \{ y \in A^2_{n+1} : \forall k \; \forall t \in L_{n+1} \; t \prec y \implies y \in A^3_{n+1}(t, k) \}$.
 Clearly this satisfies the requirements.
\end{proof} 

   To keep the indices in step, we define $F^3_m = F^2_m \restriction A^3_{m-1} \times^\prec A^3_m$.

\begin{lemma} \label{Prikrystep5}
  Let $n \ge 17$, let $x \in P_\kappa \lambda^b_n$, and let $f_{n+1}$ be a function such that
  $A_{n+1} = \dom(f_{n+1}) \in U_{n+1}$, where  $x \prec y$ and $f_{n+1}(y) \in \Q(\kappa(x), \kappa(y))$
  for all $y \in A_{n+1}$.
  Then there exist $B_{n+1} \subseteq A_{n+1}$ and $f_{n+1}'$ with domain $B_{n+1}$
  such that:
  \begin{itemize}
  \item $f'_{n+1}(y) \le f_{n+1}(y)$ for all $y \in B_{n+1}$.
  \item For every $k$ and every $t \in L_{n+1}$ of the form $s^\frown q^\frown x$,  
 one of the two following mutually exclusive  conditions holds:
\begin{itemize} 
\item
  For every $y \in B_{n+1}$,  there is a condition in $E^{(k)}$ with initial segment
  \[
  t^\frown {f'_{n+1}(y)}^\frown y^\frown F^3_{n+2}(y, -).
  \]
\item  
  For every $y \in B_{n+1}$ and every $r \le f'_{n+1}(y)$,
  there is no condition in $E^{(k)}$ with initial segment
  \[
  t^\frown r^\frown y^\frown F^3_{n+2}(y, -).
  \]
\end{itemize}
\end{itemize}
  \end{lemma}

\begin{proof}
  Shrinking $A_{n+1}$ if necessary, we may assume that $A_{n+1} \subseteq A^3_{n+1}$.
  Note that $\Q(\kappa(x), \kappa(y))$ is $<\Lambda^b_{\omega+2}(\kappa(x))$-distributive, and the set of elements
  of $L_{n+1}$ of form $s^\frown q^\frown x$ is of cardinality at most $\Lambda^b_n(\kappa(x))$.
  It follows that for each $y \in A_{n+1}$ there is $r \le f_{n+1}(y)$ such that
  for every $k$ and every $t$ in $L_{n+1}$ of the form $s^\frown q^\frown x$ 
  one of the two following mutually exclusive  conditions holds:
  \begin{itemize}
  \item[$1_{t,k}$)] There is a condition in $E^{(k)}$ with initial segment
    \[
    t^\frown r^\frown y^\frown F^3_{n+2}(y, -).
    \]
  \item[$2_{t,k}$)] There is no condition in $E^{(k)}$ with an initial segment of the form
    \[
    t^\frown {\bar r}^\frown y^\frown F^3_{n+2}(y, -)
    \]
    where $\bar r \le r$.
  \end{itemize}
  
  For each $y \in A_{n+1}$ choose  $f_{n+1}'(y)$ be some $r \le f_{n+1}(y)$ as above. For each $t$ and $k$,
  let $A_{n+1}(t, k)$ be whichever of the sets $\{ y \in A_{n+1} : \mbox{$f_{n+1}(y)$ satisfies $1_{t,k}$} \}$
    and $\{ y \in A_{n+1} : \mbox{$f_{n+1}(y)$ satisfies $2_{t,k}$} \}$ is measure one for $U_{n+1}$.
  Let $B_{n+1} = \bigcap_{t, k} A_{n+1}(t,k)$.       
\end{proof}

\begin{remark} Lemma \ref{Prikrystep5} is only useful for conditions of length $n > 17$,
because for a condition $\langle f_{17}, A_{17}, F_{18} \rangle$ of length $17$,
$f_{17}(x) \in \Q^*(x)$ for all $x \in A_{17}$. This explains why the following Lemma \ref{strongPrikry}
is restricted to conditions of length greater than $17$.
\end{remark}

\begin{lemma}[Strong Prikry Lemma] \label{strongPrikry} 
  For every dense open subset $E$ of $\bar \P$ and every
  condition $p \in \bar\P$ of length greater than $17$, there exist $q$ a direct extension of $p$ and
  $k \in \omega$ such that $q \in E^{(k)}$. 
\end{lemma} 

\begin{proof}
  Let the condition $p$ be
$\langle q_{17}, x_{17} \ldots x_{n-1}, f_n, A_n, F_{n+1}, \ldots \rangle$ where $n > 17$, 
  and as usual $\dom(f^n) = A_n$ and $\dom(F_m) = A_{m-1} \times^\prec A_m$ for $n < m < \omega$.

  Appealing to Lemma \ref{Prikrystep5}, we refine $f_n$ to $f_n' \le f_n$ with $\dom(f_n') = A'_n \subseteq A_n$
  such that for every $k$ and every $t \in L_n$ with last entry $x_{n-1}$, one of the following holds:
  \begin{itemize}
  \item For every $y \in A_n'$ there is a condition in $E^{(k)}$ with initial segment
    \[t^\frown {f'_n(y)}^\frown y^\frown F^3_{n+1}(y, -).
    \]
  \item For every $y \in A_n'$ there is no condition in $E^{(k)}$ with an initial segment of the form
    \[
    t^\frown r^\frown y^\frown F^3_{n+1}(y, -)$ where $r \le f'_n(y).
    \]
  \end{itemize}  
  We then form a direct extension $p'$ of $p$, where $p'$ has the form  
  \[
  \langle q_{17}, x_{17} \ldots x_{n-1}, f'_n, A'_n, F'_{n+1}, \ldots \rangle
  \]
  with $F'_m \le F_m, F^3_m$ for all $m > n$.
Since $E$ is a dense open set, there is a condition $p'' \le p'$ such that $p'' \in E$.
Let $p''$ be a $k$-step extension of $p'$. If $k = 0$ we are done  setting $q = p''$, so assume that $k > 0$.

The condition $p''$ has the form
\[
\langle q''_{17}, x_{17} \ldots x_{n-1}, q''_n, x_n \ldots q''_{m-1}, x_{m-1}, f''_m, A''_m, F''_{m+1}, \ldots \rangle
\]
 where  $m = n + k > n$. We note that:
 \begin{itemize}
\item $q''_n \le f'_n(x_n)$.   
\item $x_j \in A^3_j$ for $n \le j < m$.  
\item $f''_m \le F^3_m(x_{m-1}, -)$.
\item $F''_j \le F^3_j$ for $j > m$.
\end{itemize}

\begin{claim}
  If $p^{**}$ is
  the condition 
  \[
  \langle q''_{17}, x_{17} \ldots x_{n-1}, q''_n, x_n, F'_{n+1}(x_n, -), A'_{n+1},  F'_{n+2},   F'_{n+3}, \ldots \rangle
  \]
  then $p^{**} \in E^{(k - 1)}$.
\end{claim}

\begin{proof}
  We will show by induction on $i$ that for $0 \le i \le k-1$,
  if $p^*$ is the condition
\[
\langle q''_{17}, x_{17}  \ldots q''_{m-i-1}, x_{m-i-1}, F'_{m-i}(x_{m-i-1}, -), A'_{m-i}, F'_{m-i+1}, \ldots \rangle
\]
then $p^* \in E^{(i)}$.

\noindent (Base case) $i = 0$: Since $f''_m \le F'_m(x_{m-1}) \le F^0_m(x_{m-1})$ and $x_{m-1} \in A^0_{m-1}$,
  it follows from Lemma \ref{Prikrystep1} that
  there is a condition in $E$ with initial segment
\[
\langle q''_{17}, x_{17}  \ldots q''_{m-1}, x_{m-1}, F^0_m(x_{m-1}, -) \rangle
\]
  Since $E$ is open, it follows from the choice of the functions $F'_j$ and Lemma \ref{Prikrystep2} that
  \[
  p^* = \langle q''_{17}, x_{17}  \ldots q''_{m-1}, x_{m-1}, F'_m(x_{m-1}), A'_{m+1}, F'_{m+1}, \ldots \rangle \in E=E^{(0)}.
  \]

  \noindent (Successor step) $i = i_0 + 1$ for $0 \le i_0 < k-1$.
  By the induction hypothesis if $p^{*-}$ is the condition
\[
\langle q''_{17}, x_{17}  \ldots q''_{m-i}, x_{m-i}, F'_{m-i+1}(x_{m-i}, -), A'_{m-i+1}, F'_{m-i+2}, \ldots \rangle
\]
  then $p^{*-} \in E^{(i-1)}$.

  Since $q''_{m-i} \le F'_{m-i}(x_{m-i-1}, x_{m-i}) \le F^2_{m-i}(x_{m-i-1}, x_{m-i})$,
  $x_{m-i-1} \in A^2_{m-i-1}$, and $x_{m-i} \in A^2_{m-i}$,
  it follows from Lemma \ref{Prikrystep3}
  that  there is a condition in $E^{(i-1)}$ with initial segment
  \[
  \langle q''_{17}, x_{17} \ldots x_{m-i-1}, F^2_{m-i}(x_{m-i-1}, x_{m-i}), x_{m-i}, F^1_{m-i+1}(x_{m-i},-) \rangle.
  \]
  Since $E^{(i-1)}$ is open 
  there is a condition in $E^{(i-1)}$ with initial segment
\[
  \langle q''_{17}, x_{17} \ldots x_{m-i-1}, F^2_{m-i}(x_{m-i-1}, x_{m-i}), x_{m-i}, F^2_{m-i+1}(x_{m-i},-) \rangle.
\]

  Since $x_{m-i} \in A^3_{m-i}$, it follows from Lemma \ref{Prikrystep4}
  that for every $y \in A^3_{m-i}$ with $x_{m-i-1} \prec y$, there is a condition in $E^{(i-1)}$ with initial segment
\[
\langle q''_{17}, x_{17} \ldots x_{m-i-1}, F^3_{m-i}(x_{m-i-1}, y), y, F^3_{m-i+1}(y) \rangle.
\]
   It follows from the choice of the functions $F'_j$ that for every such $y$ 
\begin{align*}
  & \langle q''_{17}, x_{17} \ldots x_{m-i-1}, F'_{m-i}(x_{m-i-1}, y), y, F'_{m-i+1}(y), A'_{m-i+1}, F'_{m-i+2}, \ldots \rangle \\
  & \in E^{(i-1)}. 
\end{align*}
 So every $1$-step extension of $p^*$ lies in $E^{(i-1)}$, so by definition $p^* \in E^{(i)}$.
\end{proof}

Since $q''_n \le f_n'(x_n)$, it follows from the choice of $f_n'$ that for every
$y \in A'_n$, there is a condition in $E^{(k-1)}$ with initial segment
\[
\langle q''_{17}, x_{17} \ldots x_{n-1}, f'_n(y), y, F^3_{n+1}(y, -) \rangle.
\]
Let $q$ be the condition
\[
\langle q''_{17}, x_{17} \ldots x_{n-1}, f'_n, A'_n, F'_{n+1}, \ldots \rangle
\]
By the choice of the functions $F'_j$ and Lemma \ref{Prikrystep5}, it follows that
${q}^\frown y \in E^{(k-1)}$ for all $y \in A_n'$, that is to
say $q \in E^{(k)}$.
\end{proof}

\begin{lemma}[Prikry Lemma] \label{prikry}
  Let $b$ be a Boolean value for $\bar\P$ and let $p \in \bar\P$ be a condition of length greater than
  $18$. Then there is $s \le^* p$ such that $s$ decides $b$.
\end{lemma}

\begin{proof}
  Let $E$ be the dense open set of conditions which decide $b$ and 
  let $q \leq^* p$ and $k$ be as in the conclusion of Lemma \ref{strongPrikry}.
  For each appropriate $k$-tuple $\vec x$, define $F(\vec x) = 0$ if $q\cat\vec x \forces \neg b$ and
  $F(\vec x) = 1$ if $q\cat\vec x \forces b$. By Lemma \ref{rowbottom}, we may find $r \le^* q$ such that
  all $k$-step extensions of $r$ decide $b$ the same way: since every extension of $r$ is compatible with some
  $k$-step extension, $r$ decides $b$.  
\end{proof}

Recall from the end of Section \ref{defining} that
given a condition 
\[
p = \langle q_{17},x_{17}, \dots q_{n-1}, x_{n-1}, f_n, A_n, F_{n+1}, A_{n+1} ,F_{n+2}, \ldots \rangle
\]
and $m < n -1$,
we factored $\bar \P \downarrow p$
as $\P_{\rm low} \downarrow p_0 \times \P_{\rm high} \downarrow p_1$, where
$\P_{\rm high}$ is defined in a very similar way to $\bar \P$ with the associated Prikry sequence
starting at $x_{m+1}$. 
The proofs of
Lemmas \ref{strongPrikry} and \ref{prikry} can easily be adapted
to prove the parallel assertions for $\P_{\rm high}$.

\begin{lemma} \label{Phighdist}
  Let $p$, $m$, $\P_{\rm low}$, $\P_{\rm high}$, $p_0$ and $p_1$ be as above. 
  Let $\tau = \kappa(x'_m)$ and $\lambda = \Lambda^b_{\omega + 2}(\tau)$. Then:
\begin{enumerate}
\item Forcing with $\P_{\rm high} \downarrow p_1$  adds no new bounded subsets of $\lambda$.
\item Forcing with $\bar \P \downarrow p$, all bounded subsets of $\lambda$ are in the intermediate extension by
  $\P_{\rm low}$. 
\end{enumerate}  
\end{lemma}

\begin{proof} To show the first claim, let $\gamma < \lambda$ and let $\dot X$ name a subset of $\gamma$.
  Let $p'$ be a condition in  $\P_{\rm high} \downarrow p_1$, and let
   \[
   p' = \langle q'_{m+1}, x'_{m+1}, \dots q'_{n'-1}, x'_{n'-1}, f'_{n'}, A'_{n'}, F'_{n'+1}, A'_{n'+1} ,F'_{n'+2}, \ldots \rangle,
   \] 
   where $x'_{j} = x_j$ for $m < j < n$. 

  Let $\tau_j = \kappa(x'_j)$ for $m < j < n'$.
  For each $\alpha < \gamma$ we will define a subset $D_\alpha$ of
  $\Q(\tau, \tau_{m+1}) \times \prod_{m + 1 < j < n'} \Q(\tau_{j-1}, \tau_j) \times \Q(\tau_{n'-1}, \kappa)$
  as follows:  $D_\alpha$ is the set of tuples $(q''_{m+1}, \ldots, q''_{n'-1}, q''_{n'})$
  such that there is a direct extension $p'' \le^* p'$ deciding $\alpha \in \dot X$ where
   \[
   p'' = \langle q''_{m+1}, x'_{m+1}, \dots q''_{n'-1}, x'_{n'-1}, f''_{n'}, A''_{n'}, F''_{n'+1}, A''_{n'+1} ,F''_{n'+2}, \ldots \rangle,
   \] 
   and $q''_{n'} = [f''_{n'}]$.

   Clearly $D_\alpha$ is open. 
   It follows by Lemma \ref{prikry} for $\P_{\rm high}$ that $D_\alpha$ is dense
   below $(q'_{m+1}, \ldots q'_{n'-1}, [f'_{n'}])$ for each $\alpha$.
   By Lemma \ref{qdistribj} $\bigcap_{\alpha < \gamma} D_\alpha$ is dense 
   below $(q'_{m+1}, \ldots q'_{n'-1}, [f'_{n'}])$, so we 
   find $(q''_{m+1}, \ldots q''_{n'-1}, q''_{n'}) \le (q'_{m+1}, \ldots q'_{n'-1}, [f'_{n'}])$ 
   with $(q''_{m+1}, \ldots q''_{n'-1}, q''_{n'}) \in \bigcap_{\alpha < \gamma} D_\alpha$.

   For each $\alpha < \gamma$, choose a condition
   \[
   p^\alpha = \langle q''_{m+1}, x'_{m+1}, \dots q''_{n'-1}, x'_{n'-1}, f^\alpha_{n'}, A^\alpha_{n'}, F^\alpha_{n'+1}, A^\alpha_{n'+1} ,F^\alpha_{n'+2}, \ldots \rangle
   \] 
   witnessing that $(q''_{m+1}, \ldots q''_{n'-1}, [f''_{n'}]) \in D_\alpha$. Since $[f^\alpha_{n'}] = q''_{n'}$ for each
   $\alpha$, by $\kappa$-completeness there is a large set where all the functions $f^\alpha_{n'}$ agree,
   so refining their domains we may as well assume that there is a fixed function $f^*_{n'}$ with
   $f^\alpha_{n'} = f^*_{n'}$ for all $\alpha < \gamma$.

   For each $j > n'$ and $\alpha < \gamma$,  $[F^\alpha_j]_{U_{j-1} \times U_j} \in K$. Since $K$ is generic over a highly closed inner model,
   we may find $F^*_j$ for $j > n$ such that $[F^*_j]_{U_{j-1} \times U_j} \in K$ and $[F^*_j] \le [F^\alpha_j]$ for all $\alpha$.
   Since $\gamma < \kappa$, by $\kappa$-completeness we may refine the domains of the functions $F^*_j$ and
   assume that $F^*_j \le F^\alpha_j$ for all $\alpha$.
   In summary we have constructed a condition
   \[
  p^* = \langle q''_{m+1}, x'_{m+1}, \dots q''_{n'-1}, x'_{n'-1}, f^*_{n'}, A^*_{n'}, F^*_{n'+1}, A^*_{n'+1} ,F^*_{n'+2}, \ldots \rangle
  \]
  which refines $p'$ and decides $\alpha \in \dot X$ for all $\alpha < \gamma$.

  For the second claim we observe that $\vert \P_{\rm low} \vert < \lambda$, so that all $\P_{\rm low}$-names for bounded subsets
  of $\lambda$ are coded by bounded subsets of $\lambda$, and we are done by the factorisation of $\bar \P \downarrow p$ and
  the first claim. 
\end{proof}

   The following corollary is immediate. 

\begin{corollary} \label{bounded_sets} 
   Let $G$ be $\bar \P$-generic and let $\langle \tau_j : 17 \le j < \omega \rangle$ be the
   Prikry sequence added by $G$. Let $\gamma < \kappa$, and let $m > 17$ be least such that
   $\gamma < \Lambda^b_{\omega +2}(\tau_m)$, and let $X \in P(\gamma)^{V[L][\Agg][G]}$.
   Then $X \in V[L][G_0]$ where $G_0$ is the 
 $\Q^*(\tau_{17}) \times \prod_{17 < j \le m} \Q(x_{j-1}, x_j)$-generic filter induced by $G$.
\end{corollary} 

For the purposes of the analysis
in Section \ref{treeprop}, we record some more refined information about how much of the various generic objects
we need to define some bounded subsets of $\kappa$. We remind the reader that $\Q(\tau, \tau^*)$
is $<\Lambda^b_{\omega+2}(\tau)$-distributive, and that only the $\A^b_0$ component of
$\I^b(\tau)$ adds any subsets of $\Lambda^a_\omega(\tau)$.
The proof of the following easy Lemma uses these facts and Corollary \ref{bounded_sets}. 

\begin{lemma} \label{longboringlemma} In the generic extension by $\bar \P$:
  \begin{itemize}
  \item   If $\tau$ and $\tau^*$ are successive Prikry points, the cardinals between $\tau$ and $\tau^*$
    are $\Lambda^a_j(\tau)$ for $j < \omega+3$ and $\Lambda^b_k(\tau)$ for $k < \omega+4$.
    If $E$ is the generic object added by the interleaved forcing between Prikry points up to
    $\tau$, and $Q(\tau, \tau^*)$ is the generic object added by the forcing between $\tau$ 
    and $\tau^*$, then:
    \begin{itemize}
    \item All bounded subsets of $\Lambda^b_\omega(\tau)$ lie in $\Vlbi(\tau)[E][A_e(\tau) \times J^c_0(\tau)]$.
    \item All subsets of $\Lambda^b_{\omega+1}(\tau)$ lie in
    $\Vlbi(\tau)[E][J^c_0(\tau)][A_e(\tau) \times A^c_1(\tau)]$.
    \item All bounded subsets of $\Lambda^a_\omega(\tau^*)$ lie in
      $V[L \restriction \tau^*][E][Q(\tau, \tau^*)][A^b_0(\tau^*)]$.
    \item All subsets of $\Lambda^a_{\omega+1}(\tau^*)$ lie in
      \[
      V[L \restriction \tau^*][E][Q(\tau, \tau^*)][L^b(\tau^*) \restriction \lambda^a_{\omega+2}(\tau^*)][A^b_{[0,1]}(\tau^*) * U^b_0(\tau^*) * S^b_0(\tau^*)].
      \]
    \end{itemize}
  \item   If $\tau$ is the first Prikry point then the infinite cardinals below $\tau$ are
    $\omega$, $\rho^+ = \omega_1$, $\theta= \omega_2$. If $Q^*(\tau)$ is the generic object added by the first interleaved forcing then:
    \begin{itemize}
    \item All  bounded subsets of $\Lambda^a_\omega(\tau)$ lie in  $V[L \restriction \tau][Q^*(\tau)][A^b_0(\tau)]$. 
    \item All subsets of $\Lambda^a_{\omega+1}(\tau)$ lie in
      \[
      V[L \restriction \tau][Q^*(\tau)][L^b(\tau) \restriction \lambda^a_{\omega+2}(\tau)][A^b_{[0,1]}(\tau) * U^b_0(\tau) * S^b_0(\tau)].
      \]
    \end{itemize}
  \end{itemize}
\end{lemma}

\begin{lemma} \label{cardinals_in_Prikry_extension}
  $\kappa = (\aleph_{\omega^2})^{V[L][\Agg][\bar P]}$ and
  $(\lambda^b_\omega)^+ = (\kappa^+)^{V[L][\Agg][\bar P]}$.
\end{lemma}

\begin{proof} It follows immediately from \ref{longboringlemma} that $\kappa = \aleph_{\omega^2}$
  in $V[L][\Agg][\bar P]$. By Lemma \ref{ccforPbar}, $(\lambda^b_{\omega})^+$ is a cardinal
  in this model. An easy density argument shows that $\lambda^b_\omega = \bigcup_{n \ge 17} x_n$
  where the $x_n$'s are the supercompact Prikry points added by $\bar P$, and it follows
  immediately that $\lambda^b_\omega$ is collapsed to have cardinality $\kappa$
  in $V[L][\Agg][\bar P]$. 
\end{proof}

The following Lemma gives an analysis of names for sequences of ordinals, in a similar spirit to  
Lemma \ref{Phighdist} and Corollary \ref{bounded_sets}.

\begin{lemma} \label{shortsequencenames} 
  Let $p \in \bar \P$, where  
\[
p = \langle q_{17},x_{17}, \dots q_{n-1}, x_{n-1}, f_n, A_n, F_{n+1}, A_{n+1} ,F_{n+2}, \ldots \rangle.
\]
Let $\tau_j = \kappa(x_j)$, let $\gamma < \Lambda^b_{\omega + 2}(\tau_{n-1})$, and let
$\dot f$ be a $\P$-name for a function from $\gamma$ to $ON$. 
Let $\P_{\rm low} = \Q^*(\tau_{17}) \times \prod_{17 < j < n} \Q(x_{j-1}, x_j)$,
and let $p_0 = (q_{17}, \ldots q_{n-1})$. Then there are a direct extension
\[
p' = \langle q_{17},x_{17}, \dots q_{n-1}, x_{n-1}, f'_n, A'_n, F'_{n+1}, A'_{n+1} ,F'_{n+2}, \ldots \rangle.
\]
 of $p$, conditions $(p^\alpha_0)_{\alpha < \gamma}$ in $\P_{\rm low}$ and natural numbers $(k_\alpha)_{\alpha < \gamma}$ such that for all $\alpha < \gamma$:
 \begin{itemize}
 \item $p^\alpha_0 \le p_0$.
 \item If $p^\alpha_0 = (q^\alpha_{17}, \ldots q^\alpha_{n-1})$, then every
   $k_\alpha$-step extension of 
\[
p^\alpha = \langle q^\alpha_{17},x_{17}, \dots q^\alpha_{n-1}, x_{n-1}, f'_n, A'_n, F'_{n+1}, A'_{n+1} ,F'_{n+2}, \ldots \rangle
\]
decides ${\dot f}(\alpha)$.
\end{itemize}
\end{lemma}

\begin{proof} 
  For each $\alpha$, let $D_\alpha$ be the set of $q \in \Q(\tau_{n-1}, \kappa)$ such that there
  exist a direct extension $\bar p$ of $p$ with $[f^{\bar p}_n] = q$ and $k < \omega$ such that
  every $k$-step extension of $\bar p$ decides ${\dot f}(\alpha)$. Clearly $D_\alpha$ is open,
  and by Lemma \ref{strongPrikry} $D_\alpha$ is dense below $[f_n]$. Since
  $\gamma < \Lambda^b_{\omega + 2}(\tau_{n-1})$ and $\Q(\tau_{n-1}, \kappa)$
  is $<\Lambda^b_{\omega + 2}(\tau_{n-1})$-distributive, we may find $q \le [f_n]$
  with $q \in \bigcap_{\alpha < \gamma} D_\alpha$.

  For each $\alpha < \gamma$ we choose ${\bar p}^\alpha \le^* p$ witnessing that $q \in D_\alpha$.
  Arguing exactly as in the proof of Lemma \ref{Phighdist}, we may assume that all the entries
  of $p^\alpha$ past $x_{n-1}$ are independent of $\alpha$. This defines a suitable
  condition $p'$.
\end{proof}

\section{The tree property below $\aleph_{\omega^2}$ in the final model} \label{treeprop}

We now establish the various instances of the tree property below $\aleph_{\omega^2}$  needed to prove Theorem \ref{mainthm}.
The instances above $\aleph_{\omega^2}$ require different techniques and will be discussed in Section \ref{treeprop2}.

Let $\langle \tau_i : 17 \le i < \omega \rangle$ be the Prikry-sequence added by $\bar \P$,
that is $\tau_i = \kappa(x_i)$ in the notation of Section \ref{prikryforcing}.
As we noted in Lemma \ref{longboringlemma}, it follows from the Prikry lemma that bounded subsets of $\kappa$ in the final model
live in extensions of $V[L]$ by posets of the form
$\Q^*(\tau_{17}) \times \Q(\tau_{17}, \tau_{18}) \times \ldots \times \Q(\tau_i, \tau_{i+1})$
for some $i < \omega$.

\smallskip

\noindent Global notation: $\tau_i$\index[Notation]{$\tau_i$} 
\smallskip

The general plan is based on the fact that all the cardinals of interest are either double successor cardinals
or successors of singular cardinals. To handle the double successors we will use the fact that the
forcing posets $\Q^*(\tau_0)$, $\I^b(\tau_i)$ and $\J^c(\tau_i) * \Q_0(\tau_i, \tau_{i+1})$ establish
instances of the tree property at all the cardinals which concern us, but only in submodels of our final model.
We will use Lemma \ref{indestructible} to show that these instances of the tree property persist into our
final model. To handle the successors of singulars we will exploit the fact that all the cardinals
of concern have the form $\lambda_{\omega+1}$ where $\lambda = \Lambda^{a,b}(\tau_i)$, and that
it is forced over $V$ by $\laux(\rho,\lambda) \times \raux(\lambda)$ that the tree property holds at $\lambda_{\omega+1}$:
the extension by $\laux(\rho,\lambda) \times \raux(\lambda)$ collapses so many cardinals that $\lambda_{\omega+1} = \aleph_{\omega+1}$,
the point is that this extension absorbs enough of our final model to argue that the tree property also holds at $\lambda_{\omega+1}$
in our final model.

Let $\tau_i$ and $\tau_{i+1}$ be successive Prikry points. To lighten notation we make the following definitions:
\begin{itemize}
\item $\tau = \tau_i$.
\item $\tau^* = \tau_{i+1}$.
\item $\formerlyrho^a_n = \Lambda^a_n(\tau)$, $\formerlyrho^b_n = \Lambda^b_n(\tau)$.
\item $\formerlyrho^{a*}_n = \Lambda^a_n(\tau^*)$, $\formerlyrho^{b*}_n = \Lambda^b_n(\tau^*)$.   
\end{itemize}

We will discuss the cardinals in groups,  roughly corresponding to the various instances of the 
$\A * \U * \S$ construction which are used below $\kappa$. Recall that one instance
of this construction was done entirely by the $\I^b(\tau)$ component of  $\FL(\tau)$ at stage $\tau$ with
cardinal parameters $\mu_0 = \formerlyrho^a_{17}$, $\mu_1 = \formerlyrho^a_{\omega+1}$, $\mu_2= \formerlyrho^a_{\omega+2}$,  $\mu_{n+3} = \formerlyrho^b_n$ for $n < \omega$.
Another ``two -phase'' version was done partly by the $\J^c(\tau)$ component of $\FL(\tau)$ and partly by $\Q(\tau, \tau^*)$ with
cardinal parameters $\mu_0 = \formerlyrho^b_{17}$, $\mu_1 = \formerlyrho^b_{\omega+1}$, $\mu_2= \formerlyrho^b_{\omega+2}$, $\mu_3= \formerlyrho^b_{\omega+3}$,
$\mu_{n+4} = \formerlyrho^{a*}_n$ for $n < \omega$.  When $\tau$ is the first Prikry point yet a third version
was done partly by the construction of $V$ and partly by the forcing $\Q^*(\tau)$,
this time with cardinal parameters $\mu_0 = \omega$, $\mu_1 = \rho^+$, $\mu_2 = \theta$, $\mu_{3 + n} = \formerlyrho^a_n$ for $n < \omega$.

We let $\E = \Q^*(\tau_{17}) \times \prod_{17 \le k < i} \Q(\tau_k, \tau_{k+1})$, so
that $\E$ accounts for the forcing posets interleaved between the Prikry points
up to and including $\tau$. Let $E$ be the $\E$-generic object added to $V[L]$
by $\bar P$.

\subsection{Group I: $\formerlyrho^b_{\omega +2}$, $\formerlyrho^b_{\omega +3}$, $\formerlyrho^{a*}_n$ for $n < \omega$} \label{GroupI}

Recall from Section \ref{moreprep} that $L(\tau) = L^b(\tau) * I^b(\tau) * (A_e(\tau) \times J^c(\tau))$
and is generic over $V[L \restriction \tau]$. 
Recall also from Section \ref{successiveprikry} that $\Q(\tau, \tau^*) \in V[L \restriction \tau^*][A^b_0(\tau^*)]$ and adds
$Q(\tau, \tau^*) = \prod_{i< 3} Q_i(\tau, \tau^*)$ which is generic over $V[L][E]$. 
The generic objects $J^c(\tau)$ and $Q_0(\tau, \tau^*)$ combine as described in Section \ref{successiveprikry}.
The cardinal parameters are $\mu_0 = \formerlyrho^b_{17}$, $\mu_1 = \formerlyrho^b_{\omega+1}$, $\mu_2= \formerlyrho^b_{\omega+2}$, $\mu_3= \formerlyrho^b_{\omega+3}$,
$\mu_{n+4} = \formerlyrho^{a*}_n$ for $n < \omega$. 

To help the reader keep track, here is a picture of some the most relevant cardinals
for this group under the various names that they go by in the proof . Cardinals in each 
row are equal, cardinals in each column are strictly increasing.  
\[
\begin{array}{cccc}
\vdots & \vdots & \vdots & \vdots \\
\vdots    & \Lambda^a_1(\tau^*) & \sigma^{a*}_1 & \mu_5 \\  
\vdots    & \Lambda^a_0(\tau^*) & \sigma^{a*}_0 & \mu_4 \\  
\tau_{i+1} & \tau^* & \vdots & \vdots \\
\vdots    &  \vdots & \vdots & \vdots \\
\vdots    & \Lambda^b_{\omega+3}(\tau) & \sigma^{b}_{\omega+3} & \mu_3 \\  
\vdots    & \Lambda^b_{\omega+2}(\tau) & \sigma^{b}_{\omega+2} & \mu_2 \\  
\vdots    & \Lambda^b_{\omega+1}(\tau) & \sigma^{b}_{\omega+1} & \mu_1 \\  
\vdots    & \vdots & \vdots & \vdots \\
\vdots    &  \Lambda^b_{17}(\tau) & \sigma^b_{17} & \mu_0 \\
\vdots    &  \vdots & \vdots & \vdots \\
\tau_i    &  \tau   & \vdots & \vdots \\
\end{array} 
\]

As we noted in Lemma \ref{longboringlemma} above, all the relevant trees for cardinals in Group I exist
in the model $W_{I} = V[L \restriction \tau^*][A^b_0(\tau^*)][E][Q(\tau, \tau^*)]$.
So to cover the cardinals in Group I it will suffice to prove: 
\begin{lemma} \label{WIlemma} 
  For all $n < \omega$, $\mu_{n+2}$ has the tree property in $W_{I}$.
\end{lemma}

\begin{proof} 
 Expanding $L \restriction \tau^*$,
\[
W_{I} = \Vlbi(\tau)[A_e(\tau) \times J^c(\tau)][L \restriction(\tau, \tau^*)][A^b_0(\tau^*)][E][Q(\tau, \tau^*)].
\]
With a view to rearranging $W_{I}$ we note that:
\begin{itemize}
\item By similar considerations as for $\Q(\tau, \tau^*)$, $\E \in \Vl(\tau)[A^b_0(\tau)]$. 
\item $\A_e(\tau) \in \Vlb(\tau)$.
\item $\J^c(\tau) \in \Vlbi(\tau)$.
\item $\Q_0(\tau, \tau^*) \in \Vlbi(\tau)[J^c(\tau)]$.
\item $\Q_1(\tau, \tau^*) \in \Vlbi(\tau)[A_e(\tau) \times J^c(\tau)][L \restriction(\tau, \tau^*)][A^b_0(\tau^*)]$.
\item $\Q_2(\tau, \tau^*) \in \Vlbi(\tau)[I^b(\tau)][A_e(\tau) \times J^c(\tau)][L \restriction(\tau, \tau^*)]$.
\end{itemize}  
So we may rearrange $W_{I}$ as 
\[
\Vlbi(\tau)[E][J^c(\tau)][A_e(\tau)][L \restriction(\tau, \tau^*)] 
     [Q_0(\tau, \tau^*) \times Q_2(\tau, \tau^*)][A^b_0(\tau^*) * Q_1(\tau, \tau^*)].
\]

By the definition of $\Q_1(\tau, \tau^*)$, we may rearrange $A^b_0(\tau^*) * Q_1(\tau, \tau^*)$ as $A^V_0(\tau^*)$ which is generic
 for $\A^V_0(\tau^*) = \Add^V(\formerlyrho^{a*}_{17}, \formerlyrho^{a*}_{\omega +2})$.  We note that $\formerlyrho^{a*}_{17} = \mu_{21}$ in
our list of cardinal parameters.

So $W_{I}$ is 
\[
   \Vlbi(\tau)[E][J^c(\tau)][A_e(\tau)][L \restriction(\tau, \tau^*)][Q_0(\tau, \tau^*) \times Q_2(\tau, \tau^*)][A^V_0(\tau^*)].
\]  
Since $\A^V_0(\tau^*) \in V$ we may rearrange $W_{I}$ as 
\[
  \Vlbi(\tau)[E][J^c(\tau)][Q_0(\tau, \tau^*)][A^V_0(\tau^*)][A_e(\tau)]
  [L \restriction(\tau, \tau^*)][Q_2(\tau, \tau^*)].
\]
Recalling that $Q_2(\tau, \tau^*)$ adds a term generic $T(\tau, \tau^*)$ such that $A_e(\tau) \times T(\tau, \tau^*)$ projects to $A_e(\tau) * L \restriction (\tau, \tau^*)$,
we may rearrange $W_{I}$ as
\[
\Vlbi(\tau)[E][J^c(\tau)][Q_0(\tau, \tau^*)][A^V_0(\tau^*)][A_e(\tau) \times T(\tau, \tau^*)],
\]
and then as
\[
\Vlbi(\tau)[J^c(\tau) * Q_0(\tau, \tau^*)][T(\tau, \tau^*)][A^V_0(\tau^*)][A_e(\tau)][E]
\]

We recall that  $J^c(\tau) = (A^c_0(\tau) * U^c_0(\tau) * S^c_0(\tau)) * (A^c_1(\tau) * U^c_1(\tau) * S^c_1(\tau))$,
 while $Q_0(\tau, \tau^*)$ adds $A^c_{[2, \omega)}(\tau, \tau^*)$,
   $U^c_{[2, \omega)}(\tau, \tau^*)$ and $S^c_{[2, \omega)}(\tau, \tau^*)$.
      The reader is advised to keep in mind that $\mu_j$ for $j \le 3$ depends on $\tau$ while $\mu_j$ for $j \ge 4$
       depends on $\tau^*$, so there is a ``seam'' between $\mu_3$ and $\mu_4$. 
       
       Bearing in mind that
       $\A^c_{[1, \omega)}(\tau, \tau^*) * \U^c_{[1, \omega)}(\tau, \tau^*) * \S^c_{[1, \omega)}(\tau, \tau^*)$
               is defined
   in the extension by $\J^c_0(\tau) = \A^c_0(\tau) * \U^c_0(\tau) * \S^c_0(\tau)$,
   we reorganise $J^c(\tau) * Q_0(\tau, \tau^*)$ as
   $J^c_0(\tau) * (A^c_{[1, \omega)}(\tau, \tau^*) * U^c_{[1, \omega)}(\tau, \tau^*) * S^c_{[1, \omega)}(\tau, \tau^*))$.
So $W_{I}$ is          
\[
 \Vlbi(\tau)[J^c_0(\tau) *
  (A^c_{[1, \omega)}(\tau, \tau^*) * U^c_{[1, \omega)}(\tau, \tau^*) * S^c_{[1, \omega)}(\tau, \tau^*)]
       [T(\tau, \tau^*)][A^V_0(\tau^*)][A_e(\tau)][E].
\]

   The general idea is now to use the indestructibility guaranteed by Lemma \ref{indestructible},
   but there are a couple of obstacles:
   \begin{itemize}
   \item Since  $\A^c_{[1, \omega)}(\tau, \tau^*) * \U^c_{[1, \omega)}(\tau, \tau^*) * \S^c_{[1, \omega)}(\tau, \tau^*)$
         was defined in an extension
         by $\A^c_0(\tau) * \U^c_0(\tau) * \S^c_0(\tau)$, we need to treat $\mu_2$ separately.
   \item For $n=2$,  $\T(\tau, \tau^*) \times \A^V_0(\tau^*) \times \A_e(\tau) \times \E$
     does not fit perfectly into the hypotheses of Lemma \ref{indestructible} as applied to $\mu_4$ and 
     $\A^c_{[1, \omega)}(\tau, \tau^*) * \U^c_{[1, \omega)}(\tau, \tau^*) * \S^c_{[1, \omega)}(\tau, \tau^*)$, and extra arguments are required.
   \end{itemize}
   
   With a view to applying Lemma \ref{indestructible} to
   $\A^c_{[1, \omega)}(\tau, \tau^*) * \U^c_{[1, \omega)}(\tau, \tau^*) * \S^c_{[1, \omega)}(\tau, \tau^*)$
   in $\Vlbi(\tau)[J^c_0(\tau)]$,
   recall that:
   \begin{itemize}
   \item $E$ is  generic for a poset $\E \in \Vlbi(\tau)$,
     where $\vert \E \vert < \mu_0$. 
   \item $A_e(\tau)$ is generic for a Cohen poset adding $\mu_3$ subsets of $\mu_0$, defined in
     $\Vlb(\tau)$.
   \item  $A^V_0(\tau^*)$ is generic for a Cohen poset adding many subsets of $\mu_{21}$, defined in $V$.
   \item $T(\tau, \tau^*)$ is generic for a $<\mu_3$-directed closed poset $\T(\tau, \tau^*)$
     defined in $\Vlbi(\tau)[J^c(\tau)]$, with
     $\mu_3 < \vert \T(\tau, \tau^*) \vert < \mu_4$.
   \item $\A^c_{[2, \omega)}(\tau, \tau^*)$ is a product of Cohen posets defined in $V$.
   \item
     The cardinals $\mu_j$ for $j \ge 3$ are indestructibly supercompact in $V$, and
     $\phi$ is an indestructible Laver function there. They remain
     supercompact in $\Vlbi(\tau)[J^c_0(\tau)]$,
     and $\A^c_{[1, \omega)}(\tau, \tau^*) * \U^c_{[1, \omega)}(\tau, \tau^*) * \S^c_{[1, \omega)}(\tau, \tau^*)$
           was defined using a Laver function
           derived from $\phi$.
   \item $\Vlbi(\tau)[J^c_0(\tau)]$,
           is an extension of $V$ by
      a poset of size $\mu_2$.  
   \end{itemize}
   
   Now we verify that $\mu_{n+2}$ has the tree property in $W_{I}$ for all $n$.
   For most $n$ we can directly apply Lemma \ref{indestructible} to
   $\A^c_{[1, \omega)}(\tau, \tau^*) * \U^c_{[1, \omega)}(\tau, \tau^*) * \S^c_{[1, \omega)}(\tau, \tau^*)$
         in $\Vlbi(\tau)[J^c_0(\tau)]$.
         In this context   
         $\Vd = \Vlbi(\tau)[J^c_0(\tau)]$, and
  $\Vi = V$. The reader is warned that since we are working with indices in the interval $[1, \omega)$,
   $\mu_{k+1}$ in our current context plays the role of $\mu_k$ in Lemma \ref{indestructible}.  
    In most cases, our appeals to Lemma \ref{indestructible} are justified by
    Lemma \ref{meetingindyreqs}.

The proof involves various auxiliary models, which we have sought to name in a consistent way.
The models $W_{I}^{x}$ where $x = i, ii, iii$ are submodels of $W_{I}$ which isolate
some families of trees, and $W_I^{x *}$ is a generic extension of $W_{I}^{x}$ obtained by
some form of quotient to term forcing.

  \begin{claim} $\mu_{n+2}$ has the tree property in $W_{I}$ for all $n \ge 21$. 
  \end{claim}  

  \begin{proof}
    Use Lemma \ref{indestructible} with $\D^0 = \A^V_0(\tau^*)$,  
    $\D^{\rm small} = \T(\tau, \tau^*) \times \A_e(\tau) \times \E$, 
    and the other factors trivial. The hypotheses of Lemma \ref{indestructible} are satisfied
    by appealing to Lemma \ref{meetingindyreqs}, where we note that
    $\A^V_0(\tau^*)$ is a Cohen poset defined in $V$ (which is $\Vi$) adding subsets of $\mu_{21}$,
    so that it is a reasonable value for $\D^0$. 
\end{proof} 

\begin{claim} $\mu_{n+2}$ has the tree property in $W_{I}$ for $n = 20$ 
\end{claim}  

\begin{proof}  
     Use Lemma \ref{indestructible} with $\D^1 = \A^V_0(\tau^*)$, 
     $\D^{\rm small} = \T(\tau, \tau^*) \times \A_e(\tau) \times \E$, 
     and the other factors trivial. Again we use Lemma \ref{meetingindyreqs} to justify the appeal
     to Lemma \ref{indestructible}, where this time $\A^V_0(\tau^*)$ is a Cohen poset defined in $V$
     adding subsets of $\mu_{n+1}$, so that it is a reasonable value for $\D^1$.  
     
\end{proof}

\begin{claim} $\mu_{n+2}$ has the tree property in $W_{I}$ for  $3 \le n \le 19$. 
\end{claim}

\begin{proof}
    Use Lemma \ref{indestructible} with $\D^2 = \A^V_0(\tau^*)$, 
    $\D^{\rm small} = \T(\tau, \tau^*) \times \A_e(\tau) \times \E$, 
    and the other factors trivial. In this range of values of $n$, $\A^V_0(\tau^*)$ is $<\mu_{n+2}$-directed closed forcing
    defined in $V$, hence it is a reasonable value for $\D^2$. 
\end{proof}

\begin{claim} $\mu_{n+2}$ has the tree property in $W_{I}$ for  $n =2$. \label{GroupInequals2} 
\end{claim}

\begin{proof}
 This case is slightly harder because we need the factor $T(\tau, \tau^*)$, but this
    doesn't fit smoothly into  Lemma \ref{indestructible}. We will use the   mutual
   genericity idea from Remark \ref{mutualgenericity1}. 
 
  All the relevant $\mu_4$-trees lie in the model 
  $
  W^{i}_{I}= \Vlbi(\tau)[J^c_0(\tau)]
  [\A^c_{[1, \omega)}(\tau, \tau^*) * \U^c_{[1, \omega)}(\tau, \tau^*) * \S^c_{[1, \omega)}(\tau, \tau^*)]
        [E \times A_e(\tau) \times T(\tau, \tau^*)]
  $.
  Let $T \in W^{i}_{I}$ be a $\mu_4$-tree. 

  Now while $\vert \E \vert < \mu_0$ (so $\E$ would a reasonable value for $\D^{\rm small}$),
  and $\A_e(\tau)$ is Cohen forcing defined in a model between $\Vi$ and $\Vd$ adding Cohen subsets to $\mu_0$
  (so would be a reasonable value for $\D^0$), the poset $\T(\tau, \tau^*)$ does not fit into our indestructibility scheme.

  Proceeding exactly as in the proof of Lemma \ref{indestructible}, we construct a generic embedding
  $j$ with critical point $\mu_4$ in an extension $W^{i}_{I}[P_{1-3}]$. Since $\vert \T(\tau, \tau^*) \vert < \mu_4$,
  no additional forcing is needed to handle $T(\tau, \tau^*)$. Using $j$ we obtain a branch
  $b$ of the tree $T$ with $b \in W^{i}_{I}[P_{1-3}]$, and aim to show that $b \in W^{i}_{I}$.

  To this end we force over $W^{i}_{I}[P_{1-3}]$ with a ``quotient to term'' forcing $\QTT$ to remove the dependence of $T(\tau, \tau^*)$
  on $S^c_1$,
  obtaining a generic object $TT$ for the term forcing poset
  $\TT = \termspace^{\Vlbi(\tau)[J^c_0(\tau)]
    [\A^c_1(\tau) * \U^c_1(\tau)]}(\S^c_1(\tau), \T(\tau, \tau^*))$
  such that $S^c_1(\tau) \times TT$ induces $S^c_1(\tau) * T(\tau, \tau^*)$.

  Note that:
  \begin{enumerate}
  \item $TT$ is generic for the term forcing poset $\TT$  which is  $<\mu_3$-closed in the model
    $\Vlbi(\tau)[J^c_0(\tau)][A^c_1(\tau) * U^c_1(\tau)]$, and by
    routine calculations this term poset is still $<\mu_3$-distributive
    in $\Vlbi(\tau)[J^c_0(\tau)] 
      [A^c_{[1, \omega)}(\tau, \tau^*) * U^c_{[1, \omega)}(\tau, \tau^*) * S^c_{[1, \omega)}(\tau, \tau^*)]$.
   
     \item
    It follows from Lemma \ref{qtot-closed} that $\QTT$ is defined and $<\mu_2$-closed in the model
    $\Vlbi(\tau)[A^c_{[0,1]}(\tau) * U^c_{[0,1]}(\tau) * S^c_{[0,1]}(\tau)][T(\tau, \tau^*)]$, and by the usual
    distributivity arguments $\QTT$ remains $<\mu_2$-closed in the model
    $\Vlbi(\tau)[A^c(\tau, \tau^*) * U^c(\tau, \tau^*) * S^c(\tau, \tau^*)][T(\tau, \tau^*)]$.
  \end{enumerate}

  Let $W^{i*}_{I} = W^{i}_{I}[QTT] = \Vlbi(\tau)[A^c(\tau, \tau^*) * U^c(\tau, \tau^*) * S^c(\tau, \tau^*)]
        [E \times A_e(\tau) \times TT]$, so that $b \in W^{i*}_{I}[P_{1-3}]$. We now proceed to argue that $b \in W^{i*}_{I}$ by a
        similar line of argument to that in Lemma \ref{indestructible}.
        
        Let $M_0 = W^{i*}_{I}$, $M_1 = M_0[P_{2 b}]$, $M_2 = M_1[P_{1 a} \times P_{1 b} \times P_3]$, $M_3 = M_2[P_{2 a}]$.
        The arguments that $b \in M_1 \implies b \in M_0$ and $b \in M_3 \implies b \in M_2$ work exactly as before.
        To complete the argument  we need only to argue to argue that $M_2$ is an extension of $M_1$ by ``formerly $<\mu_3$-closed''
        forcing in the
        sense of Fact \ref{formerlyclosed}.

        Arguing as before,  $\P_{1 a} \times \P_{1 b} \times \P_3$ is $<\mu_3$-closed in 
        $\Vlbi(\tau)[A^c_0(\tau) * U^c_0(\tau) * S^c_0(\tau)]
        [A^c_{[1, \omega)}(\tau, \tau^*) * U^c_{[1, \omega)}(\tau, \tau^*) * S^c(\tau, \tau^*) \restriction [\mu_3, \mu_\omega)]$,
        and it remains $<\mu_3$-closed in
        $M_{-} = \Vlbi(\tau)
        [A^c_0(\tau) * U^c_0(\tau) * S^c_0(\tau)]
        [A^c_{[1, \omega)}(\tau, \tau^*) * U^c_{[1, \omega)}(\tau, \tau^*) * S^c(\tau, \tau^*) \restriction [\mu_3, \mu_\omega)][TT]$.       
        Now $M_0 = M_{-}[S^c(\tau, \tau^*) \restriction [\mu_2, \mu_3) \times E \times A_e(\tau)]$ and
          $M_1 = M_{-}[P_{2 b}][S^c(\tau, \tau^*) \restriction [\mu_2, \mu_3) \times E \times A_e(\tau)]$.
        Since $S^c(\tau, \tau^*) \restriction [\mu_2, \mu_3) \times E \times A_e(\tau)$ is $\mu_3$-cc in $M_{-}[P_{2 b}]$,
        $\P_{1 a} \times \P_{1 b} \times \P_3$ is formerly  $<\mu_3$-closed, and we see that $b \in M_2 \implies b \in M_1$.

        We have shown that $b \in M_0 = W^{i*}_{I} = W^{i}_{I}[QTT]$. Since $QTT$ is mutually generic with $P_{1-3}$ and
        $b \in W^{i}_{I}[P_{1-3}]$, $b \in W^{i}_{I}$ and we are done. 
\end{proof}        
        
\begin{claim} $\mu_{n+2}$ has the tree property in $W_{I}$ for  $n=1$. \label{GroupInequals1}  
\end{claim}

\begin{proof}
  Again this case needs a slightly different argument, using some of the ingredients
  from the proof of Claim \ref{GroupInequals2},
  but appealing directly to Lemma \ref{indestructible} and avoiding the use of mutual genericity.    

 As in the preceding case,  all the relevant $\mu_3$-trees lie in the model 
  $
  W^{ii}_{I}=\Vlbi(\tau)[A^c(\tau, \tau^*) * U^c(\tau, \tau^*) * S^c(\tau, \tau^*)]
        [E \times A_e(\tau) \times T(\tau, \tau^*)]
  $,
    and the troublesome factor is $T(\tau, \tau^*)$.   Let $T \in W^{ii}_{I}$ be a $\mu_3$-tree. 
 
  Exactly as in the proof of Claim \ref{GroupInequals2}, we force over $W^{ii}_{I}$ with the quotient to term forcing $\QTT$ 
  to obtain a term forcing generic $TT$ such that $S^c_1(\tau, \tau^*) \times TT$ induces $S^c_1(\tau, \tau^*) * T(\tau, \tau^*)$.
  Since $TT$ is generic for a term forcing poset $\TT$ which is defined and  $<\mu_3$-directed closed in  
  $\Vlbi(\tau)[A^c_{[0,1]}(\tau, \tau^*)  * U^c_{[0,1]}(\tau, \tau^*) * S^c_0(\tau, \tau^*)]$,
  we may appeal
  to Lemma \ref{indestructible} for $\A^c_{[1, \omega)}(\tau, \tau^*) * \U^c_{[1, \omega)}(\tau, \tau^*) * \S^c_{[1, \omega)}(\tau, \tau^*)$
      with $\D^{\rm small} = E$, $\D^0 = \A_e(\tau)$, and $\D^3 = \TT$.
  As before, Lemma \ref{meetingindyreqs} ensures that we satisfied the hypotheses of Lemma \ref{indestructible}.   
  We conclude that $\mu_3$ has the tree property in  $W^{ii}_{I}[QTT]$, so that
    our tree has a branch $b$ in $W^{ii}_{I}[QTT]$.

   Since $\QTT$ is  $<\mu_2$-closed in 
   $
   \Vlbi(\tau)[A^c(\tau, \tau^*) * U^c(\tau, \tau^*) * S^c(\tau, \tau^*)][T(\tau, \tau^*)],
   $
         and $\E \times \A_e(\tau)$ is $\mu_2$-cc in this model,
    $\QTT$ is formerly $<\mu_2$-closed in $W^{ii}_{I}$ and so $b \in W^{ii}_{I}$ by Fact \ref{formerlyclosed}.      
\end{proof}

\begin{claim} $\mu_{n+2}$ has the tree property in $W_{I}$ for  $n = 0$. \label{GroupInequals0}
\end{claim}

\begin{proof}
     Routine calculation shows that all the relevant $\mu_2$-trees lie in $W^{iii}_I$, where
     $W^{iii}_{I}=
     \Vlbi(\tau)[J^c(\tau)]
     [A^c_2(\tau, \tau^*) \times E \times  A_e(\tau)]$.     
   The key point is that  $A^c_{[0, 2]}(\tau, \tau^*) *  U^c_0(\tau) * S^c_0(\tau)$    
   is a forcing poset which is in the scope of Lemma \ref{indestructible2}.

   We need to extend $W^{iii}_I$ before applying Lemma \ref{indestructible2}.
   Let
   \[
   \TBC = \termspace^{\Vlbi(\tau)[A^c_0(\tau) * U^c_0(\tau)]}(\S^c_0(\tau), \B^c_1(\tau) \times \C^c_1(\tau)),
   \]
   so by Lemma \ref{standardtermforcinglemma} $\TBC$ is $<\mu_2$-directed closed in
   $\Vlbi(\tau)[A^c_0(\tau) * U^c_0(\tau)]$.

   Forcing over $\Vlbi(\tau)[J^c(\tau)]$
   with an appropriate quotient forcing $\QTT_0$, we may obtain an extension
   of the form  $\Vlbi(\tau)[A^c_{[0,1]}(\tau) * U^c_0(\tau) * S^c_0(\tau)][B^c_1(\tau) \times C^c_1(\tau)]$.
   Since $\A^c_1(\tau) * \U^c_1(\tau) * \S^c_1(\tau)$ is the first stage of an $\A * \U * \S$ construction defined in
   $\Vlbi(\tau)[J^c_0]$, it follows from
   Remark \ref{useful} that
   $\QTT_0$ is $<\mu_1$-closed in $\Vlbi(\tau)[J^c(\tau)]$.
   Forcing with another quotient forcing $\QTT_1$, we may further extend to obtain a model
   $\Vlbi(\tau)[A^c_{[0,1]}(\tau) * U^c_0(\tau) *S^c_0(\tau)][TBC]$:
   by Lemma \ref{qtot-closed} we see that $\QTT_1$ is $<\mu_1$-closed in   
   $\Vlbi(\tau)[J^c_0(\tau)][B^c_1(\tau) \times C^c_1(\tau)]$.
   By the distributivity of $\A^c_1(\tau)$, 
   $\QTT_1$ is $<\mu_1$-closed in 
   $\Vlbi(\tau)[A^c_{[0,1]}(\tau) * U^c_0(\tau) * S^c_0(\tau)][B^c_1(\tau) \times C_c^1(\tau)]$,
   so that if we set $\QTT = \QTT_0 * \QTT_1$ then $\QTT$ is
   $<\mu_1$-closed in $\Vlbi(\tau)[J^c(\tau)]$.

   Forcing with $\QTT$ over $W^{iii}_{I}$, we get
   \begin{align*}
    & W^{iii}_{I} = \Vlbi(\tau)[J^c(\tau)][A^c_2(\tau, \tau^*) \times E \times  A_e(\tau)] \\
     \subseteq 
    & W^{iii*}_{I} =
    \Vlbi(\tau)[A^c_{[0, 2]}(\tau, \tau^*) * U^c_0(\tau) * S^c_0(\tau)][TBC \times E \times  A_e(\tau)]. \\
    \end{align*}

   Now we use Lemma \ref{indestructible2} to show that $\mu_2$ has the tree property
   in $W^{iii*}_{I}$. To save the reader some work we record how the parameters from that Lemma should be set:
   \begin{itemize}
   \item $n$ is $0$. 
   \item $\eta$ is $\mu_4$, so that $A \restriction \eta$ is $A^c_0(\tau) \times A^c_1(\tau) \times A^c_2(\tau, \tau^*)$.   
   \item $\Vi$ is $V$.
   \item $\Vd$ is $\Vlbi(\tau)$.
   \item $\D^{small}$ is $E$. 
    \item $\D^3$  is $\TBC$.
   \item $\D^0$ is $\A_e(\tau)$.
   \end{itemize}
     
   We claim that the quotient-to-term forcing which we used to
   obtain $W^{iii*}_{I}$ from $W^{iii}_{I}$ is formerly $<\mu_1$-closed in $W^{iii}_{I}$. To see this note that
   $W^{iii}_{I}$ is obtained from $\Vlbi(\tau)[J^c(\tau)]$
   by adding $A^c_2(\tau, \tau^*)$ (which is generic for highly distributive forcing and preserves the closure) and then
   $E \times A_e(\tau)$ (which is generic for $\mu_1$-cc forcing). It follows that $\mu_2$ has the tree property in
   $W^{iii}_{I}$. 
\end{proof}

This concludes the proof of Lemma \ref{WIlemma}
\end{proof}

   \subsection{Group II: $\formerlyrho^a_{\omega+2}$, $\formerlyrho^b_n$ for $n < \omega$.} \label{GroupII}
Recall from Section \ref{moreprep} that $L(\tau) = L^b(\tau) * I^b(\tau) * (A_e(\tau) \times J^c(\tau))$
and is generic over $V[L \restriction \tau]$. Here $\FL^b(\tau)$ is making the cardinals
$\Lambda^b_n(\tau) = \formerlyrho^b_n$ for $n < \omega$ indestructible, and $\I^b(\tau)$ is a forcing of the form $\A * \U * \S$
defined in $V[L \restriction \tau][L^b(\tau)]$ with parameters set as follows: $\mu_0 = \formerlyrho^a_{17}$, $\mu_1 = \formerlyrho^a_{\omega + 1}$,
$\mu_2 = \formerlyrho^a_{\omega+2}$, $\mu_{3 + n} = \formerlyrho^b_n$. The poset $\I^b(\tau)$ uses the indestructible Laver function added
by $L^b(\tau)$.

\[
\begin{array}{cccc}
\vdots & \vdots & \vdots & \vdots \\
\vdots    & \Lambda^b_1(\tau) & \formerlyrho^b_1 & \mu_4 \\  
\vdots    & \Lambda^b_0(\tau) & \formerlyrho^b_0 & \mu_3 \\  
\vdots    & \Lambda^a_{\omega+2}(\tau) & \formerlyrho^{a}_{\omega+2} & \mu_2 \\  
\vdots    & \Lambda^a_{\omega+1}(\tau) & \formerlyrho^{a}_{\omega+1} & \mu_1 \\  
\vdots    & \vdots & \vdots & \vdots \\
\vdots    &  \Lambda^a_{17}(\tau) & \formerlyrho^a_{17} & \mu_0 \\
\vdots    &  \vdots & \vdots & \vdots \\
\tau_i    &  \tau   & \vdots & \vdots \\
\end{array} 
\]

By the design of $\I^b(\tau)$, all the cardinals in Group II have the tree property in the model $V[L \restriction \tau][L^b(\tau)][I^b(\tau)]$.
As in Group I, to see that the cardinals in Group II have the tree property in our final model $V[L][\Agg][\bar P]$ we have to
account for various generic objects added by $L \restriction [\tau, \kappa)$ and by $\bar P$.
The objects of potential concern are:
\begin{itemize}
\item $E$, which we recall is added by the interleaved posets between Prikry points up to and including $\tau$.  
  This object is slightly more troublesome in this group because it is generic for forcing of size
  $\mu_2 = \formerlyrho^a_{\omega+2}$.  
\item $A_e(\tau) \times J^c(\tau)$, where $A_e(\tau)$ is adding subsets to $\formerlyrho^b_{17} = \mu_{20}$
  and $J^c(\tau)$ is doing the first two steps of an $\A * \U * \S$ construction whose first few
  cardinal parameters are $\formerlyrho^b_{17}$, $\formerlyrho^b_{\omega+1}$, $\formerlyrho^b_{\omega+2}$: since we only
  care about trees up to $\formerlyrho^b_\omega$, the only relevant part of $J^c(\tau)$ is $A^c_0$. 
\item $Q(\tau, \tau^*)$, which we can safely ignore since it is generic for  $\formerlyrho^b_\omega$-distributive forcing.  
\end{itemize}
So all the relevant trees lie in the model
$W_{II} = \Vlbi(\tau)[E][A_e(\tau) \times A^c_0(\tau)]$.

We need a slightly finer analysis of $E$:
\begin{itemize}
\item If $\tau$ is not the first Prikry point, let $\tau^-$ be the preceding Prikry point.
  Then $E = E_0 \times Q(\tau^-, \tau)$ where
  $\E_0$ represents the product of interleaved forcing posets up to $\tau^-$. It is easy to see that $\E_0 \in V[L \restriction \tau]$
  and $\vert \E_0 \vert < \tau < \mu_0$. 
  Furthermore $\Q(\tau^-, \tau) = \prod_{i < 3} \Q_i(\tau^-, \tau)$ where: 
  \begin{itemize}
  \item  $\Q_0(\tau^-, \tau)$ is defined in $\Vlbi(\tau^-)][J^c(\tau^-)]$
(which is a submodel
    of $V[L \restriction \tau]$) and $Q_0(\tau^-, \tau)$ is generic for a version of $\A * \U * \S$ forcing of size
    $\formerlyrho^a_{\omega+1} = \mu_1$ defined in $\Vlbi(\tau^-)][J^c(\tau^-)]$.
    Appealing to Lemma \ref{tailforcinglemma}, in $\Vlbi(\tau^-)][J^c(\tau^-)]$
    the poset  $\Q_0(\tau^-, \tau)$ embeds into a two-step iteration
    where the first step is $<\mu_1$-distributive and the second step is $\mu_1$-cc.  
  \item $\Q_1(\tau^-, \tau)$ is a ``quotient to term'' forcing defined in $V[L \restriction \tau][A^b_0(\tau)]$,
    refining $A^b_0(\tau)$ to  $A^V_0(\tau)$ which is a generic object for $\Add^V(\formerlyrho^a_{17}, \formerlyrho^a_{\omega+2}) = \Add^V(\mu_0, \mu_2)$, so
    that $V[L \restriction \tau][A^b_0(\tau)][Q_1(\tau, \tau^*)] = V[(L \restriction \tau) \times A^V_0(\tau)]$.
  \item  $\Q_2(\tau^-, \tau)$ is generic for a forcing of cardinality $\tau < \mu_0$, defined in $V[L \restriction \tau]$. 
  \end{itemize}
\item   If $\tau$ is the first Prikry point then $E = \prod_{i < 3} Q^*_i(\tau)$ where:
\begin{itemize} 
\item  $Q_0^*(\tau)$ is again generic for a version of $\A * \U * \S$ forcing of size $\mu_1$,
  which embeds into an iteration where $<\mu_1$-distributive forcing is followed by $\mu_1$-cc forcing.  
\item  $Q_1^*(\tau)$ is generic for  a quotient-to-term forcing defined exactly as above.
\item  $Q_2^*(\tau)$ is generic for $\Coll(\omega, \rho)$, where we note that $\rho < \tau < \mu_0$.
\end{itemize}  
\end{itemize} 
In summary, many of the factors in $\E$ are forcing posets which lie in $\Vlb(\tau)$ and have
size less than $\mu_0$.

   Now we argue that all the cardinals in Group II have the tree property 
in $W_{II}$. We mostly do this by applying Lemma \ref{indestructible} to $\I^b(\tau)$,
with  $\Vi = \Vd = \Vlb(\tau)$. 
As before, most appeals to Lemma \ref{indestructible} are justified
 by Lemma \ref{meetingindyreqs}.

\begin{lemma}  For all $n < \omega$, $\mu_{n+2}$ has the tree property in $W_{II}$. \label{WIIlemma}
\end{lemma}

\begin{proof} As before we break the proof into several claims. 

\begin{claim} 
  $\mu_{n+2}$ has the tree property in $W_{II}$ for  $n \ge 20$. 
\end{claim}

\begin{proof}
  Appeal to Lemma \ref{indestructible} with $\D^{small} = \E$,
  $\D^0 = \A_e(\tau) \times \A^c_0(\tau)$, and the remaining factors trivial.
\end{proof}

\begin{claim} 
  $\mu_{n+2}$ has the tree property in $W_{II}$ for  $n= 19$.
\end{claim}

\begin{proof}
  Appeal to Lemma \ref{indestructible} with $\D^{small} = \E$,
  $\D^1 = \A_e(\tau) \times \A^c_0(\tau)$, remaining factors trivial.
\end{proof}

\begin{claim} 
  $\mu_{n+2}$ has the tree property in $W_{II}$ for $2 \le n \le 18$.   
\end{claim}

\begin{proof}
  Appeal to Lemma \ref{indestructible} with $\D^{small} = \E$,
  $\D^2 = \A_e(\tau) \times \A^c_0(\tau)$, remaining factors trivial.
\end{proof}

Before we handle the last two cases, we need to discuss some  issues involving the
$\Q_0$ and $\Q_1$ factors in the posets $\Q(\tau^{-}, \tau)$ and $\Q^*(\tau)$.

\begin{itemize}

\item The $\Q_0$ factor: Let $\Q_0 = \Q_0^*(\tau)$ if $\tau$ is the first Prikry point,
  and $\Q_0 = \Q_0(\tau^{-}, \tau)$ otherwise. In either case
$\Q_0$ is defined in some model $V[L']$  intermediate between 
  $V$ and $V[L \restriction \tau]$. By the analysis from
 Section  \ref{further}, in $V[L']$ we may write $\Q_0$ as the projection
  of a two-step iteration $\Q_0^\mathrm{dist} * \Q_0^\mathrm{cc}$,
  where:
  \begin{itemize}
  \item  $\vert \Q_0^\mathrm{dist} * \Q_0^\mathrm{cc} \vert = \mu_1 = \formerlyrho^a_{\omega+1}$. 
  \item For all large $n < \omega$, $\Q_0^\mathrm{dist}$ is the projection
    of some $\formerlyrho^a_n$-closed forcing poset (so that 
    $\Q_0^\mathrm{dist}$ is $<\mu_1$-distributive).
  \item  It is forced by $\Q_0^\mathrm{dist}$ that  $\Q_0^\mathrm{cc}$ is the
    union of $\formerlyrho^a_\omega$ filters, so that it is $\mu_1$-cc in any outer
    model of $V[L'][\Q_0^\mathrm{dist}]$ where $\mu_1$ is still a cardinal. 
  \end{itemize}

\item The $\Q_1$ factor:
  Let $\Q_1 = \Q_1^*(\tau)$ if $\tau$ is the first Prikry point,
  and $\Q_1 = \Q_1(\tau^{-}, \tau)$ otherwise. In either case
 $\Q_1$ is a quotient to term poset defined
in $V[L \restriction \tau][A^b_0(\tau)]$, refining $A^b_0(\tau)$ to $A^V_0(\tau)$
which is generic for $\Add^V(\mu_0, \mu_2)$: in the notation from Lemma \ref{oldlemma}
$\P$ is $\FL \restriction \tau$ and $\Q$ is $\A^b_0(\tau)$. We would like to
set $\D^0$ equal to $\Q_1$ but there are some obstacles:
\begin{itemize}
\item The definition of the poset involves $A^b_0(\tau)$, so it is not
  in $\Vd$ (which is  $V[L \restriction \tau][L^b(\tau)]$) and thus hypothesis
  \ref{indhyp6}  of Lemma \ref{indestructible} is definitely not satisfied. 
\item Hypotheses \ref{indhyp6}  and \ref{indhyp6.5} require that $\D^0$ and some related posets have quite a robust chain condition,
  which in our context should be the $\mu_1$-cc since we plan to deal with $\mu_2$ as well as $\mu_3$.
  We need to verify versions of these hypotheses, appropriately modified to handle
  the dependence on $A^b_0(\tau)$,  for the poset $\Q_1$.
\end{itemize}

The cure for the first of these issues is to modify the statement and proof of Lemma \ref{indestructible}
to permit some dependence of $\D^0$ on $\A$. In the version appropriate for Claim \ref{GroupIInequals1} (resp. \ref{GroupIInequals0})
we modify the hypotheses concerning $\D^0$ as follows:
\begin{enumerate}

\item   $\D^0 \in \Vd[A_{n-1}]$ (resp. $\D^0 \in \Vd[A_n]$),
  and $\D^0$ is $\mu_{n+1}$-Knaster in $\Vd[A * U * S][D^{1,2,3}][P_{2b}]$. 

\item   
    For any $W'$ which is 
    an extension of $W[P_{2 b}]$ by a forcing which is $<\mu_{n+1}$-closed in $\Vd[A * U * S \restriction [\mu_{n+1}, \mu_\omega)][D^2]$,  
    and any $j$ as in hypotheses \ref{indhyp5} of Lemma \ref{indestructible},
    if $\P_{2a} = j(\A_{n-1} * \D^0)/j[A_{n-1} \times D^0]$)  (resp.  $\P_{2a} = j(\A_n \times \D^0)/j[A_n \times D^0]$) then 
    $\P_{2a}$ is $\mu_{n+1}$-Knaster in $W'$.  
\end{enumerate}

It is straightforward to modify the
proof of Lemma \ref{indestructible} for these slightly more general hypotheses.  
 
In our specific context we can use the following observations to satisfy these hypotheses:

\begin{itemize}

\item By Lemmas \ref{refinement1} and \ref{oldlemma}, to verify that $\Q_1$ is $\mu_1$-Knaster in some outer model
  it is sufficient to verify that $\A^V_0(\tau) = \Add^V(\mu_0, \mu_2)$ is $\mu_1$-Knaster in the same model.

\item By Lemma \ref{obviousinhindsight}, $\A^b_0(\tau) * \Q_1$ is equivalent to $\A^V_0(\tau)$ in $V[L \restriction \tau]$,
  so as for $\Q_1$ it is enough to check that $\A^V_0(\tau)$ is $\mu_1$-Knaster.

\item $\vert \A^b_0(\tau) * \Q_1 \vert = \mu_2$, so that in the case $n =1$ this poset is fixed by
  $j$ and the technical hypothesis involving stretching by $j$ is vacuously true.

\item  In the case $n=0$, $j(\A^b_0(\tau) * \Q_1)/j[A^b_0(\tau) * \Q_1)$ is easily seen to be equivalent to
  $\Add^V(\mu_0, j(\mu_2) \setminus \mu_2)$, so that again verifying $\mu_1$-Knasterness
  amounts to verifying this property for a Cohen poset adding subsets of $\mu_0$ and defined in $V$.

\end{itemize}

\end{itemize}

\begin{claim} \label{GroupIInequals1}
  $\mu_{n+2}$ has the tree property in $W_{II}$ for $n=1$.   
\end{claim}

\begin{proof}  Appeal to Lemma \ref{indestructible} (modified as above to permit $\D^0$ to depend on $A$) where 
  $\D^2$ is $\A_e(\tau) \times \A^c_0(\tau)$, $\D^0$ is $\Q_1(\tau^-, \tau)$ or $\Q_1^*(\tau)$,
  and $\D^{\rm small}$ is the  product of the remaining factors in $\E$.
\end{proof}

\begin{claim} \label{GroupIInequals0}
  $\mu_{n+2}$ has the tree property in $W_{II}$ for $n=0$.   
\end{claim}

\begin{proof}
   Appeal to Lemma \ref{indestructible} in the more general
  version from Remark \ref{mutualgenericity2}. Here $\D^2 = \A_e(\tau) \times \A^c_0(\tau)$,
  the factor $\D$ from Remark \ref{mutualgenericity2} is $\Q_0(\tau^-, \tau)$ or $\Q_0^*(\tau)$,
  $\D^0$ is $\Q_1(\tau^-, \tau)$ or $\Q_1^*(\tau)$,
  and $\D^{\rm small}$ is the
  product of the remaining factors in $\E$.

  To see that this is legitimate we need to verify that the hypotheses from Remark \ref{mutualgenericity2} are satisfied.
  Since $\D^1$ is trivial and $\vert \D^0 \vert < \mu_0$,  $\P_{2a}$ and $\P_{2b}$ are Cohen posets computed in
  $\Vd$ adding subsets to $\mu_0$ and $\mu_1$ respectively, so that by the usual arguments we can establish
  the necessary Knasterness and distributivity hypotheses. As we already discussed,
  $\Q_0^\mathrm{cc}$ has a very robust form of $\mu_1$-cc, 
  so it remains to show that
  $\Q_0^\mathrm{dist}$ is $<\mu_1$-distributive
 in   $\Vlbi(\tau)[A_e(\tau) \times A^c_0(\tau)]$.  

 Since $\vert \Q_0^\mathrm{dist} \vert = \mu_1$ and
 $A_e(\tau) \times A^c_0(\tau)$ is generic for highly distributive forcing,
    it is enough to verify the distributivity in 
$\Vlbi(\tau)$.
 In fact by article \ref{portmanteau3} of  Lemma \ref{portmanteaulemma} it will be enough to verify    
 it in $\Vlb(\tau)[A^b_{[0,1]}(\tau) * U^b_0(\tau) * S^b_0(\tau)]$.

 Recall from Remark \ref{abzeroone} that since $\FL^b(\tau)$ is $\mu_1$-closed, $\A^b_0(\tau)$ and $\A^b_1(\tau)$
 are Cohen posets defined in $V[L \restriction \tau]$, moreover
 $\A^b_0(\tau)$ is $\mu_0^+$-cc (where $\mu_0 = \formerlyrho^a_{17}$)
 and $\A^b_1(\tau)$ is $<\mu_1$-closed. In $\Vlb(\tau)$,
 $\A^b_{[0,1]}(\tau) * \U^b_0(\tau) * \S^b_0(\tau)$ is the projection of
  $\A^b_{[0,1]}(\tau) \times \B^b_0(\tau) \times \C^b_0(\tau)$,
  where  $\B^b_0(\tau) \times \C^b_0(\tau)$ is $<\mu_1$-closed.

    By the usual
    methods we may extend 
    $\Vlb(\tau)[A^b_{[0,1]}(\tau) * U^b_0(\tau) * S^b_0(\tau)]$.
    to $V[L \restriction \tau][A^b_0(\tau) \times T]$, where $A^b_0(\tau)$ is generic
    over $V[L \restriction \tau]$ for $\mu_0^+$-cc forcing, and $T$ is generic
    over $V[L \restriction \tau]$ for some $<\mu_1$-closed term forcing.
    By Easton's lemma, for all large $n$ we have that $\Q_0^\mathrm{dist}$
    is $\formerlyrho^a_n$-distributive in $V[L \restriction \tau][A^b_0(\tau) \times T]$.
    So $\Q_0^\mathrm{dist}$ is $<\mu_1$-distributive in 
    $\Vlbi(\tau)[A_e(\tau) \times A^c_0(\tau)]$, as required for an
    appeal to Remark \ref{mutualgenericity2}.

\end{proof}

This concludes the proof of Lemma \ref{WIIlemma}
\end{proof}

\subsection{Group III (At the first Prikry point): $\theta$, $\formerlyrho^a_n$ for $n < \omega$} \label{GroupIII}
Let $\tau'$ be the first Prikry point, and define $\formerlyrho^a_n = \Lambda^a_n(\tau')$ and so on as usual. 
We recall some salient facts and definitions from
Section \ref{firstprikry}.
\begin{itemize}

\item   $V=V_0[A_0 * U_0 * L^0]$,  $V_1=V_0[A_0 * U_0]$.

\item $L \restriction \tau'$ is generic over $V$ for $\theta^+$-directed  $\tau'$-cc forcing of cardinality
  $\tau'$. 

\item $L^0 * L \restriction \tau'$ is generic for $\theta^+$-directed closed forcing defined in $V_1$. 
  
\item $\bar \theta < \rho < \theta < \tau' < \formerlyrho^a_0$.

\item $\rho$ is a limit of supercompact cardinals in
$V_0[A_0 \restriction \bar \theta *  U_0 \restriction \bar \theta]$, 
but becomes an $\omega$-successor cardinal in
$V_0[A_0 \restriction \bar \theta *  U_0 \restriction \bar \theta + 1]$,

\item $A^b_0(\tau')$ is generic for $\Add^{V[L \restriction \tau']}(\formerlyrho^a_{17}, [\formerlyrho^a_{\omega+1}, \formerlyrho^a_{\omega+2}))$.
  This is added as part of $L(\tau')$. 

\item $\Q^*_0(\tau') \in V$. 
  $A_0 * U_0$ combines with $Q^*_0(\tau')$  to give us a generic object
  $A * U * S$ for a two-phase $\A * \U * \S$ construction with cardinal
  parameters $\mu_0 = \omega$, $\mu_1=\rho^+$, $\mu_2=\theta$, and $\mu_{3 +n} =  \formerlyrho^a_n$ for $n < \omega$.
  So $Q^*_0(\tau')$ adds $A_{[1, \omega)}$, $U_{[1, \omega)}$ and $S = S_0 * S_{[1, \omega)}$. 

\item  $\C_0$ is defined in $V_0$, while $\B_n$ and $\C_n$ for $n \ge 1$ are defined in $V$.
  The definition of $\S_0$ does not depend on $U^0$. 
        
\item $\theta$ is indestructibly supercompact in $V_0$.
  $A_0 * U_0 * S_0$ may be viewed as generic for the first stage of an $\A * \U * \S$ construction defined over $V_0$
  with parameters $\omega, \rho^+, \theta$ and using the indestructible Laver function $\phi_0$ from 
  $V_0$. 
  
\item        The cardinals $\mu_k$ for $k \ge 3$ are indestructibly supercompact
  in $V$, with an indestructible Laver function $\phi$ which was added by $L^0$.
  It was $\phi$ which was used to define $\A_{[1, \omega)} * \U_{[1, \omega)} * \S_{[1, \omega)}$. 
  So  $A_{[1, \omega)} * U_{[1, \omega)} * S_{[1, \omega)}$ may be viewed as generic for a version 
  of the  $\A * \U * \S$ construction defined in $V$ with parameters $\mu_1, \mu_2, \mu_3 \ldots$ and $\phi$.      
  
\item $\A_1$ is defined in $V_0[A_0 \restriction \bar\theta * U_0 \restriction \bar \theta +1]$,
$\A_n$ for $n \ge 2$ is defined in $V$.

\item $\Q^*_1(\tau') \in V[L \restriction \tau'][A^b_0(\tau')]$.  
  $\Q^*_1(\tau')$ is a ``quotient to term'' forcing poset, and refines $A^b_0(\tau')$ to
  $A^V_0(\tau')$ which is $\A^V_0(\tau')$-generic where  $\A^V_0(\tau') = \Add^V(\formerlyrho^a_{17}, \formerlyrho^a_{\omega + 2})$-generic.
  
\item $\Q^*_2(\tau') =\Coll(\omega, \rho)$, we call the generic object $h$.

\end{itemize}  

\[
\begin{array}{cccc}
\vdots & \vdots & \vdots & \vdots \\
\vdots    & \Lambda^a_1(\tau') & \sigma^a_1 & \mu_4 \\  
\vdots    & \Lambda^a_0(\tau') & \sigma^a_0 & \mu_3 \\  
\tau_0    &  \tau'   & \vdots & \vdots \\
\vdots    &  {} & \theta & \mu_2 \\  
\vdots    &  {} & \rho^+ & \mu_1 \\  
\vdots    &  {} & \omega & \mu_0 \\
\end{array} 
\]

Arguing as in Sections \ref{GroupI} and \ref{GroupII}, all the relevant trees for Group III lie in
$W_{III}$ where $W_{III} = V[L \restriction \tau][A_{[1, \omega)} * U_{[1, \omega)} * S][A^V_0(\tau')][h]$.  

\begin{lemma} For all $n < \omega$, $\mu_{n+2}$ has the tree property in $W_{III}$. \label{WIIIlemma}
\end{lemma}

\begin{proof}
We will break the problem of establishing the tree property at $\mu_{n+2}$ into two parts.
For $n \ge 1$ we will apply Lemma \ref{indestructible} to $\A_{[1, \omega)} * \U_{[1, \omega)} * \S_{[1, \omega)}$,      
      working over $V$, which makes sense because $\A_{[1, \omega)} * \U_{[1, \omega)} * \S_{[1, \omega)}$ is an $\A * \U * \S$
  construction defined in $V$ which establishes the tree property from $\mu_3$ onwards. 
  In this setting we have to account for the effects of $L \restriction \tau$, $S_0$, $A^V_0(\tau')$ and $h$.

\begin{claim}
  $\mu_{n+2}$ has the tree property in $W_{III}$ for $n \ge 20$.   
\end{claim}

\begin{proof}
  Appeal to Lemma \ref{indestructible} with $\D^{\rm small} = \FL \restriction \tau \times S_0 \times h$,
  $\D^0 = \A^V_0(\tau')$.
\end{proof}

\begin{claim}
  $\mu_{n+2}$ has the tree property in $W_{III}$ for $n = 19$.   
\end{claim}

\begin{proof}
  Appeal to Lemma \ref{indestructible} with $\D^{\rm small} = \FL \restriction \tau \times S_0 \times h$,
  $\D^1 = \A^V_0(\tau')$.
\end{proof}

\begin{claim}
  $\mu_{n+2}$ has the tree property in $W_{III}$ for $2 \le n \le 18$. 
\end{claim}

\begin{proof}
  Appeal to Lemma \ref{indestructible} with $\D^{\rm small} = \FL \restriction \tau' \times S_0 \times h$,
  $\D^2 = \A^V_0(\tau')$.
\end{proof}

\begin{claim}
  $\mu_{n+2}$ has the tree property in $W_{III}$ for $n = 1$. 
\end{claim}

\begin{proof}
  Note that $A^V_0$ is generic for highly distributive forcing, so adds no $\mu_3$-trees and
  can be ignored in this context. 
  Appeal to Lemma \ref{indestructible} with $\D^{\rm small} = \Coll(\omega, \rho)$, $\D^0 = \S_0$, $\D^1 = L \restriction \tau'$.
\end{proof}

\begin{claim} 
 $\mu_{n+2}$ has the tree property in $W_{III}$ for $n = 0$. 
\end{claim}

\begin{proof} 
  We are concerned with $\mu_2 = \theta$.
  By the usual distributivity arguments, the only relevant part of
  $A_{[1, \omega)} * U_{[1, \omega)} * S_{[1, \omega)}$ is $(A_1 * U_1 * S_1) \times A_2$.
  Note that $\A_2 = \Add(\mu_2, [\mu_3, \mu_4))^V = \Add(\mu_2, [\mu_3, \mu_4))^{V_1}$, where the last equality
  holds by the distributivity of $L^0$.     

  We aim to apply Lemma \ref{indestructible2}  but we need to be careful because
  $\A_1$ is only defined in the model $V_0[A^0 \restriction \bar \theta * U^0 \restriction \bar \theta + 1]$,
  so we view this as the ground model $\Vd$ in our appeal to the lemma. This is not a problem because
  the remaining part of $A^0 * U^0$ is $A^0 \restriction [\bar\theta, \theta) * U^0 \restriction (\bar\theta, \theta)$,
   and by the definition of $\Q^*_0(\tau)$ the definition of $\S_0$ only uses $A^0 \restriction \alpha$
  where $\alpha > \mu_1 > \rho > \bar\theta$.  

  Working over $V_0[A^0 * U^0 * S_0]$,
  we have to account for the effects of
  $L^0$ (which prolongs $V_1$ to $V$) plus $L \restriction \tau$, $h$ and    
  $A_{[1, \omega)} * U_{[1, \omega)} * S_{[1, \omega)}$.

  As usual we may force over $V[A_1 * U_1 * S_1]$ with a suitable quotient to term forcing
  to remove the dependence of $U_1 * S_1$ on $A_1$. We obtain $B_1 \times C_1$
  which is $\B_1 \times \C_1$-generic over $V[A_1]$, so that $A_1 \times B_1 \times C_1$
  induces $A_1 * U_1 * S_1$ and 
  $V[A_1 * U_1 * S_1] \subseteq V[A_1 \times B_1 \times C_1]$.
  $B_1 \times C_1$ is generic for $<\mu_2$-directed closed forcing
  defined in $V$, and by Remark \ref{useful} the forcing which produces $B_1 \times C_1$ is generic for $<\mu_1$-closed forcing
  defined in $V[A_1 * U_1 * S_1]$.

   At this point we recall that $S_0$ is added
  as part of $Q^*(\tau)$, is generic for $\S_0 \in V_0[A_0]$, and is mutually generic with
  $U_0 * L^0 * L \restriction \tau * (A_{[1, 2]} * U_1 * S_1 \times h)$.
  We will use Lemma \ref{indestructible2} with the non-trivial parameters set as follows:
  \begin{itemize}
  \item $\Vd = \Vi = V_0[A_0 \restriction \bar \theta * U_0 \restriction \bar \theta + 1]$.
  \item $\eta = \mu_3$   
  \item $\D^3 = \FL^0 * (\FL \restriction \tau \times \A_2 \times \B_1 \times \C_1)$ 
  \item $\D^0 = \Coll(\omega, \rho)$.
  \end{itemize}

  It follows from Lemma \ref{indestructible2} that $\theta$ has the tree property in
\begin{align*}
 &  V_0[A_0 * U_0 * S_0][A_1][L^0][L \restriction \tau \times A_2 \times B_1 \times C_1][h] \\
 = 
 &V[L \restriction \tau][S_0][A_1 \times A_2 \times B_1 \times C_1][h].\\
\end{align*} 
  The forcing which produces $B_1 \times C_1$ is
  $<\mu_1$-closed in $V[A_1 * U_1 * S_1]$, and retains this closure in $V[L \restriction \tau][A_1 * U_1 * S_1 \times A_2]$.
  Since $\A_0 * \S_0$ is the projection of $\A_0 \times \C_0$ where $\C_0$ is $<\mu_1$-closed in $V_0$,
  the usual arguments show that $\S_0$ is $<\mu_1$-distributive in 
  $V[L \restriction \tau][S_0][A_1 * U_1 * S_1 \times A_2]$, so that the forcing which produces
  $B_1 \times C_1$ is $<\mu_1$-closed in $V[L \restriction \tau'][S_0][A_1 * U_1 * S_1 \times A_2]$,
  and so is formerly $<\mu_1$-closed in $V[L \restriction \tau'][S_0][A_1 * U_1 * S_1 \times A_2][h]$ because $\Coll(\omega, \rho)$ trivially
  has $\mu_1$-cc. So by Fact \ref{formerlyclosed} $\mu_2$ has the tree property in 
  $V[L \restriction \tau][S_0][A_1 * U_1 * S_1 \times A_2][h]$ and we are done. 
\end{proof}

This concludes the proof of Lemma \ref{WIIIlemma}.
\end{proof}

\subsection{Group IV (Successors of singulars)}: $\formerlyrho^a_{\omega+1}$, $\formerlyrho^b_{\omega+1}$  \label{GroupIV}

Finally we treat the cardinals below $\kappa$ which are successors of singular cardinals. Such cardinals
are either of the form $\formerlyrho^a_{\omega+1} = \Lambda^a_{\omega+1}(\tau)$ or $\formerlyrho^b_{\omega+1} = \Lambda^b_{\omega+1}(\tau)$ for some Prikry point $\tau$.
The case of $\Lambda^a_{\omega+1}(\tau')$ for $\tau'$ the first Prikry point, which becomes $\aleph_{\omega+1}$ in the final
model, requires special attention. 

Recall from Section \ref{choiceofrho} that we chose the cardinals, $\rho$, $\lambda^a$ and $\lambda^b$ so
that for $\lambda \in \{ \lambda_a, \lambda_b \}$ the cardinal
$\lambda_{\omega+1}$ has the tree property in the extension of $V$ by $\laux(\rho, \lambda) \times \raux(\lambda)$, where
$\lambda_{\omega+1}$ becomes $\aleph_{\omega+1}$. We defined
the reflected versions $\Lambda^a(\tau)$ and $\Lambda^b(\tau)$ of $\lambda^a$ and
$\lambda^b$ to secure the following key property: for every potential Prikry point $\tau$,
if $\lambda \in \{ \Lambda^a(\tau), \Lambda^b(\tau) \}$ then
$\lambda_{\omega+1}$ is forced to have the tree property in the extension of $V$ by $\laux(\rho, \lambda) \times \raux(\lambda)$.

The key idea in this section is that for every relevant $\lambda$,
we can establish the tree property at $\lambda_{\omega+1}$ in our final model $V[L][\bar P]$ by embedding an appropriate submodel
of $V[L][\bar P]$ into the extension of $V$ by $\laux(\rho, \lambda) \times \raux(\lambda)$. It is important that the quotient to term posets
 that accomplish this embedding are $\rho$-closed, and this consideration played a role in the design of our construction.
The posets $\laux(\rho, \lambda)$ and $\raux(\lambda)$
were designed to absorb the many different posets which will appear in these embedding arguments.
We will make repeated use of the term forcing and absorption arguments from Sections \ref{termforcing}
and Section \ref{projections}. 

At this point it becomes important that all of the Laver functions used in our construction were
derived from the initial Laver function $\phi_0$. It is for this reason that various products of
term forcings which will be used in the absorption arguments fit the hypotheses of
Lemma \ref{Easton-collapse-absorb}.

\subsubsection{The cardinal $\formerlyrho^b_{\omega+1}$}: \label{GroupIVa} 
By the usual arguments, all relevant trees in the final model lie in the submodel
$\Vlbi(\tau)[E][J^c_0(\tau)][A_e(\tau) \times A^c_1(\tau)]$.
Since any particular $\formerlyrho^b_{\omega+1}$-tree only involves at most $\formerlyrho^b_{\omega + 1}$
coordinates in the (highly homogeneous) generic object $A_e(\tau) \times A^c_1(\tau)$, it will suffice to establish the tree
property in $M$, where 
$M=\Vlbi(\tau)[E][J^c_0 (\tau)]
[A_e(\tau) \restriction \formerlyrho^b_{\omega+2} \times A^c_1(\tau)']$
and
 $A^c_1(\tau)'$ is $Add^{V[L \restriction \tau][L^b(\tau)]}(\formerlyrho^b_{\omega+1}, \formerlyrho^b_{\omega+2})$-generic.  

For the purposes of absorbing $M$ into an extension of $V$ by $\laux(\rho, \formerlyrho^b_0) \times \raux(\formerlyrho^b_0)$, we note that:
\begin{itemize}
\item $V = V_0[A_0 * U_0* L^0] = {\bar W}[A_0 \restriction [\bar\theta, \theta) * U_0 \restriction (\bar \theta + 1, \theta)][L^0]$,
  where we note that the first element in the support of $\U_0 \restriction (\bar \theta + 1, \theta)]$ is much larger than
  $\rho$.
\item $E$ breaks down as $h \times A_1 \times E'$, where:
\begin{itemize}  
\item $h \times  A_1$ is part of the generic object for the forcing at the first Prikry point $\tau'$ as described in Section \ref{firstprikry}.
  $h$ is $\Coll(\omega, \rho)$-generic and $A_1$ is $\Add^{\bar W}(\rho^+, \mu')$-generic, where $\mu' = \Lambda^a_0(\tau')$.   
\item  $E'$ collects the rest of $E$, that is the remainder of $Q^*(\tau')$ together with generic objects for the
  interleaved forcing posets up to and including $\tau$.  $E'$  is generic  for a poset of cardinality $\formerlyrho^a_{\omega+2}$.    
\end{itemize}
\item $I^b(\tau)$ is generic for an $\A * \U * \S$ construction defined in $\Vlb(\tau)$
  with parameters $\formerlyrho^a_{17}, \formerlyrho^a_{\omega+1}, \formerlyrho^a_{\omega+2}, \formerlyrho^b_0, \ldots$.
\item $J^c_0(\tau)$ is the first phase of a two-phase $\A * \U * \S$ construction defined in
  $\Vlbi(\tau)$
  where the relevant parameters are $\formerlyrho^b_{17}$, $\formerlyrho^b_{\omega+1}$ and $\formerlyrho^b_{\omega+2}$.
  $A^c_0(\tau)$ is generic for $\Add^{\Vlb(\tau)}(\formerlyrho^b_{17}, [\formerlyrho^b_{\omega+1}, \formerlyrho^b_{\omega+2}))$,
    and the support of $U^c_0(\tau) * S^c_0(\tau)$ is contained in the interval $(\formerlyrho^b_{\omega+1}, \formerlyrho^b_{\omega+2})$.
\end{itemize} 

To help the reader keep track of the indices, we note that $\sigma^b_n$ is playing the role of $\mu_{n+3}$ in
$\I^b(\tau)$, and that the components of this forcing with index $n$ have supports in the interval
$[\mu_{n+1}, \mu_{n+2})$  as usual. 

As a first step we set aside $h$ (which will eventually be absorbed by the $\Coll(\omega, \rho)$ component of $\FL(\rho, \formerlyrho^b_0)$),
to obtain a model
$
\bar M =   \Vlbi(\tau)[A_1][E'] 
     [J^c_0(\tau)] [A_e(\tau) \restriction \formerlyrho^b_{\omega+2} \times A^c_1(\tau)'].
$
We note that $\bar M$ is a $\rho$-distributive extension of $V$. 
Then we isolate the generic objects which we plan to absorb into $\raux(\formerlyrho^b_0)$: these are
$L^b(\tau) \restriction [\formerlyrho^b_0, \formerlyrho^b_\omega)$, $A^b_{[3, \omega)}(\tau)$,
      $U^b_{[2, \omega)}(\tau) = U^b(\tau) \restriction [\formerlyrho^b_0, \formerlyrho^b_\omega)$,
          $S^b_{[2, \omega)}(\tau) = S^b(\tau) \restriction [\formerlyrho^b_0, \formerlyrho^b_\omega)$,
            $A^c_0(\tau)$, $U^c_0(\tau)$, $S^c_0(\tau)$, $A_e(\tau) \restriction \formerlyrho^b_{\omega+2}$
            and $A^c_1(\tau)'$.

            We now need to specify how these various objects are to be absorbed into $\raux(\formerlyrho^b_0)$
            by doing a series of quotient to term forcings, which in every
      case will be $\rho$-closed in the models where they are defined. The closure of these forcings will follow
      by appealing to Fact \ref{absorb}, Lemma \ref{Easton-collapse-absorb} and Lemma \ref{Levy-collapse-absorb}.

\begin{itemize}
\item $\FL^b(\tau) \restriction [\formerlyrho^b_0, \formerlyrho^b_\omega)$ is a Laver-type iteration defined
  in the model $V[L \restriction \tau][L^b(\tau) \restriction \formerlyrho^b_0)]$.

We claim that
$\FL \restriction \tau * \FL^b(\tau) \restriction \formerlyrho^b_0 * \FL^b(\tau) \restriction [\formerlyrho^b_0, \formerlyrho^b_\omega)$
  can be written as a projection of the product of
  $\FL \restriction \tau * \FL^b(\tau) \restriction \formerlyrho^b_0$
  and an Easton support product of term posets
  of the form $\termspace^V(\FL \restriction \tau * \FL^b(\tau) \restriction \alpha, \phi(\alpha))$,
  taken over $\alpha \in (\formerlyrho^b_0, \formerlyrho^b_\omega)$ such that $\alpha \in \dom(\phi)$
  and $\phi(\alpha)$ is an appropriate name. 
  The only tricky point is that since $\FL \restriction \tau * \FL^b(\tau) \restriction \formerlyrho^b_0$ has cardinality $\formerlyrho^b_0$,
  taking Easton supports in $V$ suffices.

  Since $\dom(\phi)$ consists of inaccessible closure points of $\phi$,
  and $\phi(\alpha)$ names a $<\alpha$-directed closed poset for all relevant $\alpha$, 
  it follows from Lemma \ref{Easton-collapse-absorb} 
  that the Easton product of term forcing posets can be absorbed by the component
  $\prod_{n < \omega} \East^{E_0}(\formerlyrho^b_n, <\formerlyrho^b_{n+1})$ of $\raux(\formerlyrho^b_0)$.
  
\item $\A^b_{[3, \omega)} = \prod_{n \ge 3} \A^b_n$, where 
  $\A^b_n = \Add^{V[L \restriction \tau * L^b(\tau)]}(\formerlyrho^b_n, [\formerlyrho^b_{n+1}, \formerlyrho^b_{n+2}))$,
    and so by closure of tails of $L^b(\tau)$ in fact 
    $\A^b_n =
  \Add^{V[L \restriction \tau * L^b(\tau) \restriction \formerlyrho^b_n]}(\formerlyrho^b_n, [\formerlyrho^b_{n+1}, \formerlyrho^b_{n+2}))$

  For each $n$, $\FL \restriction \tau * \FL^b(\tau) \restriction \formerlyrho^b_n * \A^b_n$
  is the projection of the product of $\FL \restriction \tau * \FL^b(\tau) \restriction \formerlyrho^b_n$
  and $\termspace^V(L \restriction \tau * L^b(\tau) \restriction \formerlyrho^b_n,
  \dot\Add(\formerlyrho^b_n, [\formerlyrho^b_{n+1}, \formerlyrho^b_{n+2}))$. 
  By Lemma \ref{oldlemma} the term poset at $n$ is equivalent to $\Add^V(\formerlyrho^b_n, \formerlyrho^b_{n+2})$,
  and can be absorbed by the component $\Add(\formerlyrho^b_n, \formerlyrho^b_{n+2})$ of $\raux(\formerlyrho^b_0)$.
  
\item We claim that  $\FL \restriction \tau * \FL^b(\tau) * \A^b(\tau) * \U^b(\tau)$
   is the projection of the product of $\FL \restriction \tau * \FL^b(\tau) * \A^b(\tau) * \U^b_{[0, 1]}(\tau)$ 
  and  an Easton support product of term posets 
  of the form
  $\termspace^V(\FL \restriction \tau * \FL^b(\tau) \restriction \gamma_\alpha * \A^b(\tau) \restriction \alpha * \U^b(\tau) \restriction \alpha, \dot \U^b(\tau)(\alpha))$,
  where  $\formerlyrho^b_0 < \alpha < \formerlyrho^b_\omega$,
  $\gamma_\alpha < \alpha^*$ and $\gamma_\alpha$ is chosen large enough that $\A^b(\tau) \restriction \alpha * \U^b(\tau) \restriction \alpha +1$
  can be defined in $V[L \restriction \tau * L^b(\tau) \restriction \gamma_\alpha]$.

This is a variation on Lemma \ref{u_into_easton}, adjusted to take account of the fact that
$\I^b(\tau)$ is defined in the generic extension $V[L \restriction \tau][L^b(\tau)]$ of $V$. 
The only subtle points here are that we restricted $L^b(\tau)$ to make
 the term forcing small enough,  and that we used Lemma \ref{eastonsets} to ensure that we may take an Easton support
product in $V$.

  This product of term forcing posets can be absorbed by the component $\prod_{n < \omega} \East^{E_0}(\formerlyrho^b_n, <\formerlyrho^b_{n+1})$ of $\R(\formerlyrho^b_0)$.

\item 
  $\FL \restriction \tau * \FL^b(\tau) * \A^b(\tau) * \U^b(\tau) *  \S^b(\tau)_{[2, \omega)}$
  is the projection of the product of $\FL \restriction \tau * \FL^b(\tau) * \A^b(\tau) * \U^b(\tau)$ and
  a full support product of term posets defined for $n < \omega$, where
   at $n$ we take the product with $<\formerlyrho^b_n$ supports over $\alpha \in (\formerlyrho^b_n, \formerlyrho^b_{n+2})$
   of posets of the form $\termspace^V(\FL \restriction \tau * \FL^b(\tau) \restriction \gamma_\alpha * \A^b(\tau) \restriction \alpha *
   \U^b(\tau) \restriction \formerlyrho^b_n,
  \dot \Add(\formerlyrho^b_n, 1))$ and $\delta_\alpha < \alpha^*$ is chosen large enough that
  $A^b(\tau) \restriction \alpha * \U^b(\tau) \restriction \formerlyrho^b_{n+1} * \Add(\formerlyrho^b_{n+1}, 1)$ can be defined in $V[L \restriction \tau * L^b \restriction \delta_\alpha]$. 

  The issue about supports is easier here than for $\U^b(\tau)_{[2, \omega)}$. Given a set of ordinals of size less than
    $\rho^b_n$ in $V[L \restriction \tau][L^b(\tau)]$, we use closure to cover it by a small set in
    $V[L \restriction \tau][L^b(\tau) \restriction \rho^b_n]$, and then chain condition to cover
    the covering set by a small set in $V$. 
  
    Here the product of term forcing posets can be absorbed by the component
    $\prod_{n < \omega} \Coll(\formerlyrho^b_n, \formerlyrho^b_{n+1})$ of $\raux(\formerlyrho^b_0)$.

  \item $\A^c_0(\tau) = \Add^{V[L \restriction \tau][L^b(\tau)]}(\formerlyrho^b_{17}, [\formerlyrho^b_{\omega+}, \formerlyrho^b_{\omega+2}))$,
  and by the closure of tails of $L^b(\tau)$,     
  $\A^c_0(\tau)   = \Add^{V[L \restriction \tau][L^b(\tau) \restriction \formerlyrho^b_{17}]}(\formerlyrho^b_{17}, [\formerlyrho^b_{\omega+}, \formerlyrho^b_{\omega+2}))$.

   $\FL \restriction \tau * \FL^b(\tau) \restriction \formerlyrho^b_{17} * \A^c_0(\tau)$ 
    is a projection of the product of $\FL \restriction \tau * \FL^b(\tau) \restriction \rho^b_{17}$ and 
    $\termspace^V(\FL \restriction \tau * \FL^b(\tau) \restriction \formerlyrho^b_{17}, \dot \Add(\formerlyrho^b_{17}, [\formerlyrho^b_{\omega+}, \formerlyrho^b_{\omega+2})))$,
    which is equivalent to $\Add^V(\formerlyrho^b_{17}, \formerlyrho^b_{\omega+2})$ by Lemma \ref{oldlemma}.
    The term forcing poset can be absorbed by the component $\Add^V(\formerlyrho^b_{17}, \formerlyrho^b_{\omega+2})$ of $\raux(\formerlyrho^b_0)$.
    
  \item $\FL \restriction \tau * \FL^b(\tau) * \I^b(\tau) * \A^c_0(\tau) * \U^c_0(\tau)$ 
    is the projection of the product of $\FL \restriction \tau * \FL^b(\tau) * \I^b(\tau) * \A^c_0(\tau)$ 
    and an Easton support product of term forcing posets of the form
    $\termspace^V(\FL \restriction \tau * \FL^b(\tau) * \I^b(\tau) * \A^c_0(\tau) \restriction \alpha * \U^c_0(\tau) \restriction \alpha,
    \dot \U^c_0(\tau)(\alpha))$
  for relevant $\alpha \in (\formerlyrho^b_{\omega+1}, \formerlyrho^b_{\omega+2})$. By similar arguments to those above,
  this product of term forcing posets can be absorbed by the component  $\East^{E_0}(\formerlyrho^b_{\omega+1}, <\formerlyrho^b_{\omega+2})$
  of $\raux(\formerlyrho^b_0)$.
  
\item $\FL \restriction \tau * \FL^b(\tau) * \I^b(\tau) * \A^c_0(\tau) * \S^c_0(\tau)$
  is the projection of the product of $\FL \restriction \tau * \FL^b(\tau) * \I^b(\tau) * \A^c_0(\tau)$ 
  with $< \formerlyrho^b_{\omega+1}$ supports of term posets of the form
    $\termspace^V(\FL \restriction \tau * \FL^b(\tau) * \I^b(\tau) * \A^c_1(\tau) \restriction \alpha, \dot \Add(\formerlyrho^b_{\omega+1}, 1))$
  for $\alpha \in (\formerlyrho^b_{\omega+1}, \formerlyrho^b_{\omega+2})$. By similar arguments to those above,
  this product of term forcing posets can be absorbed by the component
  $\Coll(\formerlyrho^b_{\omega+1}, <\formerlyrho^b_{\omega+2})$ of $\raux(\formerlyrho^b_0)$.

  \item $\A_e(\tau) \restriction \formerlyrho^b_{\omega+2} = \A^c_0(\tau)$, and in exactly the same
    way it may be written as the projection of the product of $\FL \restriction \tau * \FL^b(\tau) \restriction \rho^b_{17}$ 
    and a term forcing,  and the term forcing may be absorbed into
  the component $\Add^V(\formerlyrho^b_{17}, \formerlyrho^b_{\omega+2})$ of $\raux(\formerlyrho^b_0)$.

\item  $\A^c_1(\tau)' = \Add^{V[L \restriction \tau * L^b(\tau)]}(\formerlyrho^b_{\omega+1}, \formerlyrho^b_{\omega+2})$,
  and  $\FL \restriction \tau * \FL^b(\tau) * \A^c_1(\tau)'$  
  is the projection of the product of $\FL \restriction \tau * \FL^b(\tau)$
  and $\termspace^V(\FL \restriction \tau * \FL^b(\tau), \dot \Add(\formerlyrho^b_{\omega+1}, \formerlyrho^b_{\omega+2}))$.
  By Lemma \ref{oldlemma} this term poset is equivalent to $\Add^V(\formerlyrho^b_{\omega+1}, \formerlyrho^b_{\omega+2})$,
  and so can be absorbed by the component $\Add^V(\formerlyrho^b_{\omega+1}, \formerlyrho^b_{\omega+2})$ of $\raux(\formerlyrho^b_0)$.
        
\end{itemize}

\begin{remark} 
  Note that we have used some components of $\raux(\formerlyrho^b_0)$ to absorb multiple term forcing posets.
  This is not problematic, we may easily write (for example) $\East^{E_0}(\formerlyrho^b_n, <\formerlyrho^b_{n+1})$
  as the product of two copies of itself and use it to absorb two Easton iterations in the interval
  $[\formerlyrho^b_n,  \formerlyrho^b_{n+1})$.
\end{remark}

Forcing over $\bar M$ with an appropriate iteration of quotient to term forcings, we may absorb all the generic objects
for these forcing posets which we isolated above into a generic object $\rauxgen$ for $\raux(\formerlyrho^b_0)$.
Since $\bar M$ is a $\rho$-distributive extension of $V$ and each step in our iteration is $\rho$-closed
in the model where it is defined, the whole iteration is $\rho$-closed.
We produce 
$M^*_0 = V[L \restriction \tau][L^b(\tau) \restriction \formerlyrho^b_0 ][A^b_{[0,2]}(\tau) * U^b_{[0,1]}(\tau) * S^b_{[0,1]}(\tau)][A_1][E'][\rauxgen]$,
where $M^*_0$ is a $\rho$-closed extension of $\bar M$.
Let $M^* = M^*_0[h]$, so that $M^*$ is an extension of $M = \bar M[h]$: by Lemma \ref{MS preservation}
the passage from $M$ to $M^*$ does not add branches to $\formerlyrho^b_{\omega+1}$-trees in $M$.

We need to absorb the generic objects other than $\rauxgen$ used to obtain $M^*_0$. 
The generic objects $L \restriction \tau$, $L^b(\tau) \restriction \formerlyrho^b_0$, $A^b_{[0,2]} * U^b_{[0,1]} * S^b_{[0,1]}$
and $E'$ are generic for $\rho$-closed forcing posets of cardinality at most $\formerlyrho^b_1$,
each defined in some (possibly trivial) generic extensions of $V$. We note that all these posets actually exist
in some generic extension of $V' =  V_0[A^0 * U^0 * L^0 \restriction \formerlyrho^b_1]
= {\bar W}[Q]$, where $\Q  =  \A^0 \restriction [\bar\theta, \theta) * \U^0 \restriction (\bar \theta + 1, \theta) * \FL^0 \restriction \rho^b_1$,
  and  that in $\bar W$ the poset $\Q$ is $\formerlyrho^b_1$-cc  with cardinality
  $\formerlyrho^b_1$.

Let $M^*_1 = V'[L \restriction \tau][L^b(\tau) \restriction \formerlyrho^b_0 ][A^b_{[0,2]}(\tau) * U^b_{[0,1]}(\tau) * S^b_{[0,1]}(\tau)][A_1][E']
= {\bar W}[Q][L \restriction \tau][L^b(\tau) \restriction \formerlyrho^b_0 ][A^b_{[0,2]}(\tau) * U^b_{[0,1]}(\tau) * S^b_{[0,1]}(\tau)][A_1][E']$. 
We may perform a series of quotient to term forcings
to embed $M^*_1$ into a model of the form ${\bar W}[Q \times T \times A_1]$, where $T$ is generic for
some $\rho$-closed product of term forcings defined in ${\bar W}$, and we may assume that $T$ is generic
for forcing of size $\formerlyrho^b_1$ (it was for this reason that we truncated $V$ to $V')$.
Since $A_1$ is generic for small  $\rho$-closed forcing defined in $\bar W$, we may
do more forcing to embed $M_1^*$ into a model ${\bar W}[Q \times \lauxgen^{\rm coll}]$, where $\lauxgen^{\rm coll}$ is generic for
$\Coll^{\bar W}(\rho^+, \formerlyrho^b_1)$ which forms part of $\laux(\rho, \formerlyrho^b_0)$.
The quotient to term forcing used to produce  ${\bar W}[Q \times L_{\rm aux}^{\rm coll}]$ from $M^*_1$ also
has cardinality at most $\formerlyrho^b_1$.

Performing the same quotient to term forcing over the larger model $M^*$, we obtain a model
$M^{**} = V[\lauxgen^{\rm coll} \times h \times \rauxgen] = V[\lauxgen \times \rauxgen]$,
where $\lauxgen = \lauxgen^{\rm coll} \times h$ is $\laux(\rho, \formerlyrho^b_0)$-generic. 
By construction $\formerlyrho^b_{\omega+1}$ has the tree property in $M^{**}$.
 By Lemma \ref{chain} 
the passage from $M^*$ to $M^{**}$ does not add branches to $\formerlyrho^b_{\omega+1}$-trees in $M^*$.
So $\formerlyrho^b_{\omega+1}$ has the tree property in $M$ and we are done.

\subsubsection{The cardinal $\formerlyrho^a_{\omega+1}$ above the first Prikry point}: \label{GroupIVb}

Let $\tau$ and $\tau^*$ be successive Prikry points, we will establish the tree property at
$\formerlyrho^{a*}_{\omega+1} = \Lambda^a_{\omega+1}(\tau^*)$ in the final model.
We can do a similar analysis to that in Section \ref{GroupI} to find a submodel $M$ of the final model in which all the relevant trees will
lie. As in Section \ref{GroupIVa} we decompose $E$ as $h \times A_1 \times E'$. 

We need slightly more of the generic object $L^b(\tau^*) * I^b(\tau^*)$ than we did in Section \ref{GroupI},
because there we only considered $\formerlyrho^{a*}_n$-trees for $n$ finite. 
Recalling that $\I^b(\tau^*)$ is an $\A * \U * \S$ construction whose first few parameters are
$\formerlyrho^{a*}_{17}$, $\formerlyrho^{a*}_{\omega+1}$, $\formerlyrho^{a*}_{\omega+2}$ we see that  
the relevant parts of $L^b(\tau^*) * I^b(\tau^*)$
are $A^b_{[0,1]}(\tau^*)$, $U^b_0(\tau^*)$, $S^b_0(\tau^*)$ and
$L^b(\tau^*) \restriction \formerlyrho^{a*}_{\omega+2}$ (which is enough of $L^b(\tau^*)$ to define
$A^b_{[0,1]}(\tau^*) * U^b_0(\tau^*) * S^b_0(\tau^*)$).
By the same considerations as in Section \ref{GroupIVa}, we may replace
$A^b_1(\tau^*)$ by $A^b_1(\tau^*)'$ which is generic for
$\Add^{V[L \restriction \tau^*]}(\formerlyrho^{a*}_{\omega+1}, \formerlyrho^{a*}_{\omega+2})$.

 We see that all the relevant trees lie in the model
\begin{align*}
M = & V[E][L \restriction \tau][A^c_{[2, \omega)}(\tau, \tau^*) * U^c_{[2, \omega)}(\tau, \tau^*) * S^c_{[2, \omega)}(\tau, \tau^*)]
      \\ & [A^V_0(\tau^*) \times T(\tau, \tau^*)]
      [L^b(\tau^*) \restriction \formerlyrho^{a*}_{\omega + 2}] 
      [A^b_1(\tau^*) \restriction \formerlyrho^{a*}_{\omega + 2} \times U^b_0(\tau^*) * S^b_0(\tau^*)] \\
\end{align*}

To help keep track of the indices, we note that $\formerlyrho^{a*}_n$ plays the role of $\mu_{n+4}$
in the construction $\A^c(\tau, \tau^*) * \U^c(\tau, \tau^*) * \S^c(\tau, \tau^*)$.

As in Section \ref{GroupIVa}, we start by breaking out the generic objects which may be absorbed
by $\raux(\formerlyrho^{a*}_0)$. In this case they are $A^c_{[4, \omega)}$, $U^c_{[3, \omega)}$, $S^c_{[3, \omega)}$
      $A^V_0(\tau^*)$, $L^b(\tau^*) \restriction \formerlyrho^{a*}_{\omega+2}$, $A^b_1(\tau^*)'$
      and $U^b_0(\tau^*) * S^b_0(\tau^*)$. The argument that these generic objects may be absorbed using
      a  $\rho$-closed quotient to term forcing into an $\raux(\formerlyrho^{a*}_0)$-generic object
      $\rauxgen$ are exactly as in Section \ref{GroupIVa}.
      After the absorption process we obtain a model
\[
M^* = V[E][L \restriction \tau][A^c_{[2, 3]}(\tau^*) * U^c_2(\tau^*) * S^c_2(\tau^*)]
      [T(\tau, \tau^*)][\rauxgen]   
\]
  and just as before the passage from $M$ to $M^*$ does not add branches to  $\formerlyrho^{a*}_{\omega+1}$-trees.

  Arguing as in Section \ref{GroupIVa}, the generic objects $E' \times A_1$, $L \restriction \tau$,
  $A^c_{[2, 3]}(\tau^*) * U^c_2(\tau^*) * S^c_2(\tau^*)$ and $T(\tau, \tau^*)$ may be absorbed
  into a generic object $\lauxgen^{\rm coll}$ for $\Coll^W(\rho^+, \formerlyrho^{a^*}_1)$ where the  quotient forcing has cardinality $\formerlyrho^{a^*}_1$.
  Exactly as before we obtain a model
  $M^{**} = V[\lauxgen^{\rm coll} \times h][\rauxgen] = V[\lauxgen \times \rauxgen]$,
    where the tree property
  holds at $\formerlyrho^{a*}_{\omega+1}$,  and the passage from $M^*$ to $M^{**}$ does not add branches to  $\formerlyrho^{a*}_{\omega+1}$-trees.
  So $\formerlyrho^{a*}_{\omega+1}$ has the tree property in $M$ and we are done.

\subsubsection{The cardinal $\formerlyrho^a_{\omega+1}$ at the first Prikry point}: \label{GroupIVc}

Let $\tau'$ be the first Prikry point, so that $\formerlyrho^a_{\omega+1} = \Lambda^a_{\omega+1}(\tau')$.
We use the same notation as in Section \ref{GroupIII}. 
By the usual analysis, all the relevant trees lie in the model
\begin{align*}
    M = &V[h][L \restriction \tau'][A_{[1, \omega)} * U_{[1, \omega)} * S][A^V_0(\tau')][L^b(\tau') \restriction \formerlyrho^a_{\omega+2}]\\
    &[A^b_1(\tau') \restriction \formerlyrho^a_{\omega+2} * U^b_0(\tau') * S^b_0(\tau')] \\
\end{align*}

To help keep track of the indices, recall that $A_{[1, \omega)} * U_{[1, \omega)} * S$ comes from an $\A * \U * \S$ construction
with parameters $\omega, \rho^+, \theta, \formerlyrho^a_0, \ldots$. So the generic objects that we will absorb into
$\raux(\formerlyrho^a_0)$ are $A_{[3, \omega)}$, $U_{[2, \omega)}$, $S_{[2, \omega)}$, $A^V_0(\tau')$, $L^b(\tau') \restriction \formerlyrho^a_{\omega+2}$,
  $A^b_1(\tau') \restriction \formerlyrho^a_{\omega+2}$, $U^b_0(\tau')$, and $S^b_0(\tau')$.  Forcing with a suitable quotient
 to term forcing to absorb these generic objects
into an $\raux(\formerlyrho^a_0)$-generic object $\rauxgen$, we obtain a model
\[
M^* = V[h][L \restriction \tau][A_{[1, 2]} * U_1 * S_{[0, 1]}][\rauxgen]
\]
such that the passage from $M$ to $M^*$ adds no branches to $\formerlyrho^a_{\omega+1}$-trees from $M$.

We then force to absorb $L \restriction \tau$, $A_{[1, 2]}$, $U_1$, $S_0$ and $S_1$ into
a $\Coll^{\bar W}(\rho^+, \formerlyrho^a_1)$-generic object $\lauxgen^{\rm coll}$. Most of these objects are generic
for $\rho$-closed forcing posets defined in $V$ or generic extensions of $V$. The
exception is $S_0$, which we may absorb because (by the careful choice of $\C_0$ in
Section \ref{firstprikry}) it is the projection of a $\rho$-closed term forcing defined in ${\bar W}$. 
As usual we have absorbed $M^*$ into 
\[
M^{**} = V[\lauxgen^{\rm coll} \times h \times \rauxgen] = V[\lauxgen \times \rauxgen],
\]
without adding branches to $\formerlyrho^a_{\omega+1}$-trees from $M^*$, and
$\formerlyrho^a_{\omega+1}$ has the tree property in $M^{**}$.
 So $\formerlyrho^{a}_{\omega+1}$ has the tree property in $M$ and we are done.

\section{The tree property above $\aleph_{\omega^2}$ in the final model} \label{treeprop2}

\subsection{The tree property at $\aleph_{\omega^2+1}$}  \label{plusone}

We  argue that in our final model the tree property holds at $\aleph_{\omega^2 + 1}$.
The arguments are parallel to those in Sinapova's paper \cite[Section 4]{sinapova-omega^2},
 and also use ideas from work of Sinapova and Unger \cite{SinapovaUnger},   
but there are some additional complications: the relevant generic supercompactness embeddings
are added by a more complex forcing poset, and there are extra issues with the
constraint functions in the Prikry-type poset $\bar \P$.

Recall that $\bar \P \in V[L][\Agg]$ where $\Agg$ is generic over $V[L]$ for
the auxiliary forcing $\Aggforcing = \Add^V(\lambda^b_{\omega+2}, j_{01}(\lambda^a_0))$,
and $\Aggforcing$ is $<\lambda^b_{\omega+2}$-distributive in $V[L]$. We
defined $\bar \P$ using a filter $K \in V[L][\Agg]$ which is $\Q_\infty$-generic over $M_1^*$.

Before starting the proof, we derive some auxiliary filters ${\mathcal F}_n$ from $K$, working in $V[L][\Agg]$.
We recall from the discussion at the end of Section \ref{auxcomp}
that for $\alpha \in Y$, $\Q(\alpha, \kappa) = j_{01}^*(\Q)(\alpha, \kappa)$.

\begin{definition}
  Let $B_n=\{  (x, q) :  x\in P_\kappa(\lambda_n^b), q \in \Q(\kappa(x), \kappa) \}$.
\end{definition}
Note that $\vert B_n \vert = \lambda^b_n$.

\begin{definition}
  Let $C_n$ be the set of functions $F$ such that $\dom(F) \in U_{n-1} \times U_n$,  $F(x, y) \in \Q(x, y)$ for all $x$ and $y$
  and $[F]_{U_{n-1} \times U_n} \in K$.
\end{definition} 
If $F \in C_{n+1}$, and $x \in P_\kappa \lambda^b_n$
is such that $\{ y : (x, y) \in \dom(F) \} \in U_{n+1}$, 
then $[F(x, -)]_{U_{n+1}} \in \Q(\kappa(x), \kappa)$. 

\begin{remark} In the sequel we will often drop the subscript for the measure in expressions like
  ``$[F(x, -)]_{U_{n+1}}$'', where the relevant measure should be clear from the context.
\end{remark}

Note that $C_n$ includes the functions which can appear as $F^p_n$ for some $p \in \bar \P$. 

\begin{definition} A subset $E$ of $B_n$ is {\em downwards closed} if
  whenever $(x, q) \in E$ and $q' \le q$ then $(x, q') \in E$
\end{definition}

 \begin{definition} Let $\mathcal{F}_n $ be the filter on $B_n$  defined as follows.
   $E \in \mathcal{F}_n$ if and only there exist a set $D \in U_n$ and a function $F \in C_{n+1}$
   such that $(x, q) \in E$ for all $x \in D$ and all $q \le [F(x, -)]_{U_{n+1}}$.
\end{definition}

 \smallskip
 
\noindent Global notation: $\mathcal{F}_n$\index[Notation]{${\mathcal F}_n$}
 
\smallskip

The following Lemma is immediate from the definition and the agreement between
$j_n$ and $j_{01}^*$:
 \begin{lemma} \label{menachem}
   $E \in {\mathcal F}_n$ if and only if there is a condition $r \in K$ such that
   $(j_{01}[\lambda^b_n], q) \in j_{01}^*(E)$ for all $q \le r$.
 \end{lemma}   

  The following Lemma should be viewed as expressing an ``ultrafilter-like'' property of ${\mathcal F}_n$.
 
  \begin{lemma} \label{ultrafilter-like}
    Let $E \subseteq B_n$ and let  $E$ be downwards closed. Then $E\in \mathcal{F}_n$ or $B_n \setminus E \in {\mathcal F}_n$.
\end{lemma}

\begin{proof}
  Let $D$ be the dense set of conditions $r$ in $\Q_\infty$ such that either $(j_{01}[\lambda^b_n], r) \in j_{01}^*(E)$
  or there is no $q \le r$ with $(j_{01}[\lambda^b_n], q) \in j_{01}^*(E)$. Since $D \in M_1^*$
  and $K$ is $\Q_\infty$-generic over $M_1^*$, there is $r \in K \cap D$, and the conclusion follows. 
\end{proof}

It is easy to see that $\mathcal{F}_n$ is a $\kappa$-complete filter on $B_n$.
We will also need a version of normality
for families of ${\mathcal F}_n$-large sets indexed by lower parts.

  Recall that for  $p=\langle q_{17}, x_{17}, . . . q_{n-1}, x_{n-1}, f_n, A_n, F_{n+1}, A_{n+1}, F_{n+2}, A_{n+2}, ...\rangle$ a condition in $\bar\P$,
  the stem of $p$ is
  $\stem(p)=\langle q_{17}, x_{17}, . . . q_{n-1}, x_{n-1}, [f_n]_{U_n}\rangle$,
  the lower part of $p$ is $\langle q_{17}, x_{17}, . . . q_{n-1}, x_{n-1} \rangle$,
  the length of $p$ is $\lh(p) = n$, and $L_n$ is the set of lower parts of conditions with length $n$.

\begin{definition} \label{diagint} 
  Let $(E_s)_{s \in L}$ be a family of subsets of $B_n$ indexed by some set $L \subseteq L_n$ of lower parts.
  Then the {\em diagonal intersection} $\Delta_{s \in L} E_s$
  is 
  $\{ (x, q) \in B_n : \forall s \in L  (s \prec x \implies (x, q) \in E_s) \}$.
\end{definition}

\begin{lemma} \label{normal-like}
Let $(E_s)_{s \in L}$ be such that $E_s \in {\mathcal F}_n$ for all $s \in L$. Then 
$\Delta_{s \in L} E_s \in {\mathcal F}_n$. 
\end{lemma}   

\begin{proof}
  By the characterisation of ${\mathcal F}_n$, for each $s \in L$ we choose $r_s \in K$ such that
  $(j_{01}[\lambda^b_n], q) \in j_{01}^*(E_s)$ for all $q \le r_s$. Since $(r_s)_{s \in L} \in M_1^*$ by closure,
  and $K$ is generic over $M_1^*$, there is $r \in K$ such that $r \le r_s$ for all $s$.
  As usual $j_{01}^*[L] = \{ t \in j_{01}(L) : t \prec j_{01}[\lambda^b_n] \}$.
  For $q \le r$ we have $(j_{01}[\lambda^b_n], q) \in j_{01}^*(E)_t$ for all $t \in j_{01}^*[L]$,
  so that $(j_{01}[\lambda^b_n], q) \in j_{01}^*(\Delta_{s \in L} E_s)$. It follows
  that $\Delta_{s \in L} E_s \in {\mathcal F}_n$. 
\end{proof}

\begin{definition} \label{fakeforcing}  Let $h$ be a stem.
  $h \forces^* \phi$ if and only if 
  there is a condition $p$ such that $stem(p) = h$ and $p \forces \phi$.
\end{definition}

\smallskip

\noindent Global notation: $\forces^*$\index[Notation]{$\forces^*$}

\smallskip

  Let $\nu = \lambda^b_\omega$, $\mu = \lambda^b_{\omega+1}$. The strategy of the proof will be to introduce
  an auxiliary forcing $\R \in V[L][\Agg]$,  show that every $\mu$-tree in $V[L][\Agg][\bar P]$ has a branch in
  $V[L][\Agg][R][\bar P]$, and use Lemma \ref{preslemma} to show there is a branch in $V[L][\Agg][\bar P]$.  
  Let $\dot{T}$ in $V[L][\Agg]$  be a $\bar\P$-name for a $\mu$-tree. We assume that
  $T \subseteq \mu \times \kappa$, and that level $\alpha$ is a subset of $\{ \alpha \} \times \kappa$:
  this makes sense because $\mu = \kappa^+$ in $V[L][\Agg][\bar P]$.

  For the following lemma we work in $V[L][\Agg]$.
  Before stating and proving the lemma we make a remark on compatibility of conditions
  in $\bar \P$ which explains some complications in the proof.

  \begin{remark} If $p$ and $q$ are compatible conditions in $\bar \P$ with
    $\lh(p) = \lh(q) = n$, it does not follow in general that $p$ and $q$ have a common lower
    bound $r$ with $\lh(r) = n$. The issue is that possibly $[f^p_n]_{U_n}$ and $[f^q_n]_{U_n}$ are incompatible,
    so there is no reasonable choice for $f^r_n$:  compatibility of $p$ and $q$ only guarantees that
    $f^p(x)$ and $f^q(x)$ are compatible for at least one value of $x$. If
    $p$ and $q$ have the same stem $h$, then they are compatible and there is $r \le p, q$
    with $\stem(r) = h$. 
  \end{remark}

\begin{lemma}\label{lemma1} Let $p \in \bar \P$.
  There are $n < \omega$ and cofinal $I \subseteq \mu$,
  such that for all $\alpha < \beta$ both in $I$,
  there are a condition $p' \le p$  in $\bar\P$ of length $n$, and $\xi, \delta<\kappa $ such that:
  $p' \forces \langle \alpha, \xi\rangle <_{\dot{T}} \langle\beta, \delta\rangle$.
\end{lemma}

\begin{proof}
  Recall from Section \ref{auxcomp} that we defined an embedding $j_{01}^*$ in $V[L]$ witnessing that $\kappa$ is
  $<\lambda^b_{\omega + 3}$-supercompact, and used this to derive the supercompactness measures $U_n$
  on $P_\kappa \lambda^b_n$. Let $U_\mu$ be the supercompactness measure on $P_\kappa \mu$
  derived from $j_{01}^*$.
  
  Observe that $\vert P_\kappa \mu \vert = \mu$ \cite{Solovay}. 
  Since $\Agg$ is generic for $<\lambda^b_{\omega+2}$-distributive forcing, it is easy to see that
  $U_\mu$ is still a supercompactness measure on $P_\kappa \mu$ in $V[L][\Agg]$,
  and that taking the ultrapower of $V[L][\Agg]$ by $U_\mu$ gives
  an embedding  $j^*_\mu: V[L][\Agg] \rightarrow N$ which lifts the ultrapower of $V[L]$ by $U_\mu$. 
  It follows that $j^*_\mu$ induces the measure $U_n$ for every $n$.
  
  Let $p \in \bar \P$ have length $m$, so that the first measure one set appearing
  in $p$ is $A^p_m \in U_m$. By the choice of $j^*_\mu$
  we have that $j^*_\mu[\lambda^b_m]  \in j^*_\mu(A^p_m)$,
  so we may form in $j^*_\mu(\bar \P)$ a one point extension $q$ of $j^*_\mu(p)$ which forces that
  $\kappa$ is the Prikry point with index $m$.

  Let $\tau_{17}, \ldots \tau_m$ be the Prikry points determined by $q$, so that
  $\tau_m = \kappa$, and let $\Q_{\rm low} = \prod_{17 \le i < m} {\Q}^N(\tau_i, \tau_{i+1})$.
  We observe that $\vert \Q_{\rm low} \vert = \lambda^a_{\omega+2} < \mu$, and that every extension $r$ of $q$ determines
  a condition $r_{\rm low} \in \Q_{\rm low}$.

  Now let $u = (\sup(j^*_\mu[\mu]), 0)$ and for all $\alpha < \mu$
  let $\dot \xi_\alpha$ name the unique  ordinal $\xi_\alpha < j^*_\mu(\kappa)$
  such that $(j^*_\mu(\alpha), \xi_\alpha) <_{j^*_\mu(\dot T)} u$.

  By elementarity and Lemma \ref{shortsequencenames}, 
  we may find $r \le^* q$ together with $(s_\alpha)_{\alpha < \mu}$, $(k_\alpha)_{\alpha < \mu}$
  such that:
  \begin{itemize}
  \item $r_{\rm low} = q_{\rm low}$.
  \item $s_\alpha \in \Q_{\rm low}$ with $s_\alpha \le r_{\rm low}$.
  \item $k_\alpha < \omega$.   
  \item If $r_\alpha$ is the condition obtained from $r$ by extending $r_{\rm low}$ to $s_\alpha$, then
    every $k_\alpha$-step extension of $r_\alpha$ decides $\dot \xi_\alpha$.
  \end{itemize}

  Let $I \subseteq \mu$ be an unbounded set such that for some $k < \omega$ and $s \in \Q_{\rm low}$,
  $k_\alpha = k$ and $s_\alpha = s$ for all $\alpha \in I$. Let $r'$ be obtained from $r$ by extending 
  $r_{\rm low}$ to $s$, let $r''$ be some $k$-step extension of $r'$, and for each $\alpha \in I$
  let $\xi_\alpha$ be the value of $\dot \xi_\alpha$ determined by $r''$.

  Let $r''$ have length $n$. By construction $r'' \le j^*_\mu(p)$, and for $\alpha < \beta$ both in $I$ we
  have $r'' \forces^N_{j^*_\mu(\bar \P)} (j^*_\mu(\alpha), \xi_\alpha) <_{j^*_\mu(\dot T)} (j^*_\mu(\beta), \xi_\beta)$.
  The desired conclusion follows by elementarity. 
\end{proof}


\begin{lemma} \label{keylifting}  There is a forcing poset $\R \in V[L]$ 
  such that, for all sufficiently large $\bar n < \omega$,
   there is a forcing poset $\P_1 \times \P_2 \times \P_3 \in V[L][R]$  such that:
  \begin{itemize}
  \item  $\lambda^b_{\bar n}$ is generically $\mu$-supercompact   in $V[L][\Agg][R]$ via
    $\P_1 \times \P_2 \times \P_3$.
  \item  $\R$ is countably closed and $<\mu$-distributive in $V[L][\Agg]$.
  \item  $\Aggforcing$ is $<\lambda^b_{\omega + 2}$-distributive in $V[L][R]$.
  \item  $\P_1 \times \P_{2b} \times \P_3$ is $<\lambda^b_{\bar n -1}$-distributive in
    $V[L][R][\Agg]$.
  \item   In $V[L][R][\Agg][P_1 \times P_{2b} \times P_3]$, $\mu$ has cardinality and cofinality
    $\lambda^b_{\bar n -1}$, and $\P_{2a}$ has the $\lambda^b_{\bar n -1}$-Knaster property. 
  \item  For  $\lambda' = \lambda^b_{\bar n - 3}$, there is a forcing poset in $V[L][R][\Agg]$ which adds
    $\lambda'$ mutually generic filters for $\P_1 \times \P_2 \times \P_3$, preserves the regularity
    of $\lambda'$, and forces that $\cf(\mu) > \lambda'$. 
  \end{itemize}   
\end{lemma}

Before starting the proof, we make a remark about how the generic embedding from the
conclusion interacts with the supercompactness measures $U_m$ for $17 \le m < \bar n$.
The key points are that $\vert P_\kappa \lambda^b_m \vert = \lambda^b_m$ and
$2^{\lambda^b_m} = \lambda^b_{\omega+3}$, so that subsets of $P_\kappa \lambda^b_m$ are fixed,
while the power set of $P_\kappa \lambda^b_m$ and the measure $U_m$ are stretched. 

  \begin{proof}
  We will use ideas from the proof of Lemma \ref{indestructible}, in a context which is quite similar
  to that of Section \ref{GroupII}. The analogy with Section \ref{GroupII} is slightly imperfect, because 
  there we only needed to define the embedding on a submodel which contains all the subsets of the critical point. 
  In the discussion below this means that we need to deal with the whole of the last component
  $J^c$ in $L(\kappa)$, rather than just $A^c_0$.

  Recall that $V[L] = V[L \restriction \kappa][L(\kappa)]$, where
  $L(\kappa) = L^b * I^b *(A_e \times J^c)$. Decomposing further:
  \begin{itemize}
  \item  $I^b = A^b * U^b * S^b$, and is generic for an $\A * \U * \S$ construction with
    parameters  $\mu_0 = \lambda^a_{17}$, $\mu_1 = \lambda^a_{\omega+1}$, $\mu_2= \lambda^a_{\omega+2}$,
    $\mu_{n+3} = \lambda^b_n$. 
  \item  $\A_e = \Add^{\Vlb(\kappa)}(\lambda^b_{17}, \lambda^b_{\omega+3})$.
  \item  $\J^c = \J^c_0 * \J^c_1 =  (\A^c_0 * \U^c_0 * \S^c_0) * (\A^c_1 * \U^c_1 * \S^c_1)$, where the parameters are 
    $\mu_0 = \lambda^b_{17}$, $\mu_1 = \lambda^b_{\omega+1}$, $\mu_2 = \lambda^b_{\omega+2}$, $\mu_3 = \lambda^b_{\omega+3}$.
  \item  $\A^c_1 * \U^c_1 * \S^c_1$ is defined over
  $\Vlbi(\kappa)[A^c_0 * U^c_0 * S^c_0]$.
  \item $\A^c_0 = \Add^{\Vlb(\kappa)}(\lambda^b_{17}, [\lambda^b_{\omega+1},\lambda^b_{\omega+2}))$.
  \item  $\A^c_1 = \Add^{\Vlb(\kappa)}(\lambda^b_{\omega+1}, [\lambda^b_{\omega+2},\lambda^b_{\omega+3}))$.
  \end{itemize}

  The idea is now to construct a generic embedding as in the proof of Lemma \ref{indestructible}, where:
  \begin{itemize}
  \item $\Vlb(\kappa)$ plays the role of $\Vd$ (which coincides with $\Vi$).
  \item $I^b$ plays the role of $A * U * S$.
  \item $\lambda^b_{\bar n}$ plays the role of $\mu_{n+2}$.
  \end{itemize}
  Before we can do this we have to embed $A_e \times J^c$ into a generic object for a poset which
  meets the specifications of Lemma \ref{indestructible}.
  This is the job of the auxiliary poset $\R$.
  In the terminology of Lemma \ref{indestructible} 
  we will embed $\A_e \times \J^c$ into $\D^0 \times \D^2$, where $\D^0, \D^2 \in \Vlb(\kappa)$.
  The poset $\D^0$ will be $\lambda^b_{\bar n -1}$-Knaster in $\Vlb(\kappa)$,
  and $\D^2$ will be $<\lambda^b_{\omega+1}$-directed closed in this model.  

  We let $\D^0 = \A_e \times \A^c_0$.
  Before defining $\D^2$, we recall from the analysis 
  in Section \ref{further} that in $\Vlbi(\kappa)$
  the poset $\J^c_0$ is a projection
  of $\A^c_0 \times \B^c_0 \times \C^c_0$, where $\B^c_0 \times \C^c_0$ is
  $<\lambda^b_{\omega+1}$-directed closed.
  Similarly in $\Vlbi(\kappa)[J^c_0]$
  the poset $\J^c_1$ is a projection
  of $\A^c_1 \times \B^c_1 \times \C^c_1$, where $\B^c_1 \times \C^c_1$ is
  $<\lambda^b_{\omega+2}$-directed closed.

  We let  $\D^2 = \A^c_1 \times \termspace^{\Vlb(\kappa)}(\I^b, \B^c_0 \times \C^c_0)
  \times 
  \termspace^{\Vlb(\kappa)}(\I^b * \J^c_0, \B^c_1 \times \C^c_1)$.

\begin{definition} \label{lastRdef}  
  Let $\R \in V[L]$ be the (iterated) quotient to term forcing to add
  generic objects for the second and third factors in $\D^2$  inducing generic filters
  $B^c_0 \times C^c_0$ and $B^c_1 \times C^c_1$, such that in turn
  $B^c_0 \times C^c_0$ induces $U^c_0 * S^c_0$ and
  $B^c_1 \times C^c_1$ induces $U^c_1 * S^c_1$.
\end{definition}

\begin{claim} \label{portmanteauclaim2} Let $\D^2$, $\R$ and $\Aggforcing$ be as above.
\begin{enumerate}
\item  $\D^2$ is $<\lambda^b_{\omega+1}$-directed closed in $\Vlb(\kappa)$.
\item  $\R$ is countably closed and $<\mu$-distributive in $V[L][\Agg]$.
\item  $\Aggforcing$ is $<\lambda^b_{\omega + 2}$-distributive in $V[L][R]$.
\end{enumerate}
\end{claim}   
  
\begin{proof}
  We take each claim in turn.
\begin{enumerate}
\item  $\D^2$ is the product of three factors.
  The first term $\A^c_1$ 
  is clearly $<\lambda^b_{\omega+1}$-directed closed in $\Vlb(\kappa)$.
  The second factor is $<\lambda^b_{\omega+1}$-directed closed by
  items \ref{imaginary5} and \ref{imaginary6} of Lemma \ref{portmanteaulemma} together with Lemma \ref{standardtermforcinglemma}.
  and similarly the third factor is $<\lambda^b_{\omega+2}$-directed closed.

\item   $\R$ can be viewed as the product of two (iterated) quotient-to-term forcing posets, and each factor
  is countably closed by Lemma \ref{qtot-closed}. Since $\lambda^b_\omega$ is singular, 
  it will suffice for distributivity to show that $\R$ is $\lambda^b_t$-distributive
  in $V[L][\Agg]$ for all large enough $t < \omega$.  

  By the definition of $\R$, $V[L][\Agg][R] = \Vlb(\kappa)[I^b \times D^0 \times D^2 \times \Agg]$.
  By item \ref{portmanteau3} of Lemma \ref{portmanteaulemma}, in $\Vlb(\kappa)$ we may view $\I^b$ as a projection
  of $\I^b_0 \times \I^b_1$ where $\I^b_0$ is an initial segment of $\I^b$ with $\lambda^b_{t+1}$-cc and
  $\I^b_1$ is $<\lambda^b_{t+1}$-closed. We extend $V[L][\Agg][R]$ to obtain 
  $\Vlb(\kappa)[I^b_0 \times I^b_1 \times D^0 \times D^2 \times \Agg]$.
  
  Since $L \restriction \kappa * L^b * (I^b_0 \times D^0)$ is generic for
  $\lambda^b_{\omega+2}$-cc forcing, by Easton's lemma $\Aggforcing$ is $<\lambda^b_{\omega+2}$-distributive in 
  $\Vlb(\kappa)[I^b_0 \times D^0]$, so that 
  $I^b_1 \times D^2$ is $\lambda^b_t$-distributive in $V[L \restriction \kappa][L^b][I^b_0 \times D^0 \times \Agg]$.
  It follows that every $\lambda^b_t$-sequence of ordinals from $V[L][\Agg][R]$ lies in 
  $\Vlb(\kappa)[I^b_0 \times D^0 \times \Agg]$, which is a submodel of $V[L][\Agg]$.

\item By items \ref{imaginary2} and \ref{imaginary3} of Lemma \ref{portmanteaulemma}, $\B^c_0 \times \C^c_0$ is
  $\lambda^b_{\omega+2}$-cc in $\Vlbi(\kappa)$.
  Also $\lambda^b_{\omega+2}$ is supercompact in $\Vlb(\kappa)$,
  and $\vert I^b \vert < \lambda^b_{\omega+2}$, 
  so it follows by Lemma \ref{termforcingcclemma}
  that $\termspace^{\Vlb(\kappa)}(\I^b, \B^c_0 \times \C^c_0)$
  is $\lambda^b_{\omega+2}$-cc in $\Vlb(\kappa)$. 
  So $\D^2$ is the product of a $\lambda^b_{\omega+2}$-cc poset $\D^2_{\rm low}$ and
  a $<\lambda^b_{\omega+2}$-closed poset $\D^2_{\rm high}$ in the model $\Vlb(\kappa)$. 

  Recall that $V[L][R][\Agg] = \Vlb(\kappa)[I^b \times D^0 \times D^2_{\rm low} \times D^2_{\rm high} \times \Agg]$.
  Since  $L \restriction \kappa * L^b$ is generic over $V$ for $\lambda^b_{\omega+2}$-cc forcing,
  by Easton's Lemma $\Aggforcing$ is $<\lambda^b_{\omega+2}$-distributive in $\Vlb(\kappa)$,
    so that $\D^2_{\rm high}$ is $<\lambda^b_{\omega+2}$-closed in $\Vlb(\kappa)[\Agg]$.
  Since  $L \restriction \kappa * L^b * (I^b \times D_0 \times D^2_{\rm low})$ is generic over $V$ for
  $\lambda^b_{\omega+2}$-cc forcing, by Easton's Lemma $\FL \restriction \kappa * \FL^b * (\I^b \times \D_0 \times \D^2_{\rm low})$ is
  $\lambda^b_{\omega+2}$-cc in $V[\Agg]$, so that $\I^b \times \D_0 \times \D^2_{\rm low}$ is
    $\lambda^b_{\omega+2}$-cc in $\Vlb(\kappa)[\Agg]$. 

  By Easton's Lemma  $\D^2_{\rm high}$ is $<\lambda^b_{\omega+2}$-distributive in  
  $\Vlb(\kappa)[I^b \times D^0 \times D^2_{\rm low} \times \Agg]$
  and $\Aggforcing$ is $<\lambda^b_{\omega+2}$-distributive in $\Vlb(\kappa)[I^b \times D_0 \times D^2_{\rm low}]$.
  It follows that every $<\lambda^b_{\omega+2}$-sequence of ordinals in 
  $\Vlb(\kappa)[I^b \times D^0 \times D^2_{\rm low} \times \Agg \times D^2_{\rm high}]$
  lies in $\Vlb(\kappa)[I^b \times D^0 \times D^2_{\rm low}]$, so
  that $\Aggforcing$ is $<\lambda^b_{\omega+2}$-distributive in $V[L][R]$.

\end{enumerate}
  
\end{proof}

For use later, we record some information about forcing with $\bar \P$ over
$V[L][\Agg][R]$. 
\begin{claim} \label{for_use_later} 
  $\mu = \kappa^+$ in $V[L][\Agg][R][\bar P]$.
\end{claim}   

\begin{proof}   By Lemma \ref{portmanteauclaim2}
   the poset $\R$ is $<\mu$-distributive and countably closed 
   in $V[L][\Agg]$, and the proof of Lemma \ref{ccforPbar} easily shows that $\bar \P$ is
   $\mu$-cc in $V[L][\Agg][R]$.
   The claim follows.
\end{proof}

    By construction $V[L][R] = \Vlbi(\kappa)[D^0 \times D^2]$.
    Choosing $\bar n$ large enough, we may arrange that in the model
    $\Vlb(\kappa)$, $D^0$ is generic for $\lambda^b_{\bar n -1}$-cc forcing,
   and  $D^2$ is generic for $<\lambda^b_{\bar n}$-directed closed forcing.
   In fact since $\D^0$ is adding Cohen subsets to $\lambda^b_{17}$ and
   $\D^2$ is $<\lambda^b_{\omega+1}$-directed closed, any $\bar n \ge 19$ will work.

   We will perform the
  construction from the proof of Lemma \ref{indestructible} to obtain a generic embedding with
  domain $V[L][R]$, and then derive a generic embedding with domain $V[L][\Agg][R]$.
  The construction in the proof involves a regular cardinal $\chi$ which in our case
  is $\max(\nu, \vert \D^0 \vert, \vert \D^2 \vert)^+$. We note for the record that
  $\chi > \mu$.

  We summarise the key features of the construction from Lemma \ref{indestructible} in our current setting: 
  \begin{itemize}
  \item Working in $\Vlb(\kappa)[D^2 \times A^b_{[\bar n, \omega)}]$, we use the indestructible Laver function
    added by $L^b$ to choose an embedding $j$ which witnesses that $\lambda^b_{\bar n}$ is $\chi$-supercompact
    and satisfies some technical conditions. We may assume that $j$ is the ultrapower by
    some supercompactness measure $W$ on $P_{\lambda^b_{\bar n}} \chi$, where we note that
    $(P_{\lambda^b_{\bar n}} \chi)^{\Vlb(\kappa)[D^2 \times A^b_{[\bar n, \omega)}]} =
      (P_{\lambda^b_{\bar n}} \chi)^{\Vlb(\kappa)}$.
    \item The lifting of $j$ to $V[L][R]$ takes place in a generic extension
      $V[L][R][P_1 \times P_2 \times P_3]$.
    \item $\P_1 \times \P_3 \in \Vlb(\kappa)[A^b * U^b * S^b \restriction [\lambda^b_{\bar n - 1}, \lambda^b_\omega)][D^2]$
      and is $<\lambda^b_{\bar n - 1}$-closed in this model. 
    \item $\P_2 = \P_{2 a} \times \P_{2 b}$ where
     $\P_{2 a} = j(\A^b_{\bar n-2} \times \D^0)/j[A^b_{\bar n - 2} \times D^0]$
      and $\P_{2 b} = j(\A^b_{\bar n -1})/j[A^b_{\bar n -1}]$.
    \item $\P_{2 a} = \P_{2a}^{\rm low} \times \P_{2 a}^{\rm high}$,
      where $\P_{2a}^{\rm high}$ is a Cohen poset adding subsets to $\lambda^b_{\bar n -2}$,
      and $\P_{2a}^{\rm low}$ is a Cohen poset adding subsets to $\lambda^b_{17}$,
      both defined in $\Vlb(\kappa)$.
    \item   $\P_{2 b}$ is a Cohen poset adding subsets to $\lambda^b_{\bar n - 1}$,
      again defined in $\Vlb(\kappa)$.
   \end{itemize}

  Let $\lambda' = \lambda^b_{\bar n - 3}$.
  With a view to using Fact \ref{Itay3.4} in the proof of Lemma \ref{lemma2}, we construct
  an auxiliary forcing $\P_1^* \times (\P_{2a}^{\rm low})^* \times (\P_{2a}^{\rm high})^*
 \times \P_{2b}^* \times \P_3^* \in V[L][R]$
 whose aim is to
  add $\lambda'$ many mutually generic filters for $\P_1 \times \P_2 \times \P_3$.

\begin{itemize} 
\item $\P_1^*$ (resp $\P_3^*$) is the product of $\lambda'$ copies of $\P_1$ (resp $\P_3$) computed
  with full support in
  $\Vlb(\kappa)[A^b * U^b * S^b \restriction [\lambda^b_{\bar n - 1}, \lambda^b_\omega)][D^2]$.
    As noted above $\P_1 \times \P_3$ is defined and $<\lambda^b_{\bar n -1}$-closed in this model,
    so $\P^*_1 \times \P^*_3$ is also $<\lambda^b_{\bar n -1}$-closed in this model,
    and since $\Agg$ is generic over this model for highly distributive forcing
    $\P^*_1 \times \P^*_3$ is $<\lambda^b_{\bar n -1}$-closed in
    $\Vlb(\kappa)[A^b * U^b * S^b \restriction [\lambda^b_{\bar n - 1}, \lambda^b_\omega)][D^2][\Agg]$.
    
  \item $(\P_{2a}^{\rm low})^*$ is the product of $\lambda'$ copies of $\P_{2a}^{\rm low}$
    computed with $<\lambda^b_{17}$ support in $\Vlb(\kappa)$.
    Since $\P_{2a}^{\rm low}$ is a Cohen poset adding subsets to $\lambda^b_{17}$ defined in $\Vlb(\kappa)$,
    $(\P_{2a}^{\rm low})^*$ is a similar poset in this model.

  \item $(\P_{2a}^{\rm high})^*$  (resp $\P_{2 b}^*$) 
    is the product of $\lambda'$ copies of $\P_{2a}^{\rm high}$ (resp $\P_{2 b}$)
    computed with full support in $\Vlb(\kappa)$.
    Since $\P_{2a}^{\rm high}$ (resp $\P_{2 b}$) is a Cohen poset adding subsets to $\lambda^b_{\bar n -2}$
    (resp $\lambda^b_{\bar n -1}$) defined in $\Vlb(\kappa)$,
    $(\P_{2a}^{\rm high})^*$ (resp $\P_{2 b}^*$) is a similar poset in this model.

\end{itemize}

\begin{claim} \label{34setup}  
  It is forced over $V[L][R][\Agg]$ by  $\P_1^* \times \P_2^* \times \P_3^*$ that
  $\lambda'$ remains regular and $\cf(\mu) > \lambda'$.
\end{claim}

\begin{proof}
 Recall that $V[L][\Agg][R] = \Vlb(\kappa)[I^b \times D^0 \times D^2 \times \Agg]$,
 $\D^0 = \A_e \times \A^c_0$ is a Cohen poset adding subsets to $\lambda^b_{17}$ defined in
 $\Vlb(\kappa)$, while $\D^2$ is $<\lambda^b_{\omega+1}$-directed closed forcing again defined in
 $\Vlb(\kappa)$. 
  
 Since $\P_1^* \times \P_3^*$ is
 defined and $<\lambda^b_{\bar n-1}$-closed in a generic extension
 of $\Vlb(\kappa)$,
 we may force to extend the model
 $\Vlb(\kappa)[I^b \times D^0 \times D^2 \times \Agg][P_1^* \times P_2^* \times P_3^*]$
 to a model $\Vlb(\kappa)[I^b \times D^0 \times D^2 \times \Agg][P_2^* \times T]$
 where $T$ is generic for a term forcing $\T$ which is 
 defined and $<\lambda^b_{\bar n-1}$-closed in  $\Vlb(\kappa)$.
 By more forcing we may use item \ref{portmanteau3} of Lemma \ref{portmanteaulemma}  to extend to a model  
$\Vlb(\kappa)[I^b_0 \times I^b_1 \times D^0 \times D^2 \times \Agg][P_2^* \times T]$
 where $\I^b_0$ and $\I^b_1$ are defined in $\Vlb(\kappa)$,
 and in that model  $\I^b_0$ is $\lambda^b_{\bar n -3}$-Knaster and
 $\I^b_1$ is $<\lambda^b_{\bar n -3}$-closed.
 
 We reorganise our expanded model as
 \[
 \Vlb(\kappa)[D^2 \times \Agg][I^b_0 \times D^0 \times (P_{2a}^{\rm low})^*]
 [I^b_1 \times (P_{2a}^{\rm high})^* \times (P_{2b})^*  \times T].
 \]
 Note that $\I^b_0 \times \D^0 \times (\P_{2a}^{\rm low})^*$ is $\lambda^b_{\bar n - 3}$-cc in $\Vlb(\kappa)$. 
 Since $\D^2 \times \Aggforcing$ is highly distributive, 
 $(\P_{2a}^{\rm high})^* \times (\P_{2b})^*  \times \T$
 is $<\lambda^b_{\bar n-3}$-closed in $\Vlb(\kappa)[D^2 \times \Agg]$,
 and (since its definition does not change)  $\I^b_0 \times \D^0 \times (\P_{2a}^{\rm low})^*$
  is $\lambda^b_{\bar n - 3}$-cc in $\Vlb(\kappa)[D^2 \times \Agg]$.
 
  It follows by Easton's Lemma that $\lambda^b_{\bar n -3}$ is still regular
  in  $\Vlb(\kappa)[I^b_0 \times I^b_1 \times D^0 \times D^2 \times \Agg][P_2^* \times T]$,
  and so {\em a fortiori} it is regular
  in $\Vlb(\kappa)[I^b \times D^0 \times D^2 \times \Agg][P_1^* \times P_2^* \times P_3^*]$.
  To finish we note that by taking  $\I^b$ as the projection of the product of a $\lambda^b_{\bar n -2}$-Knaster
  poset and a $<\lambda^b_{\bar n -2}$-closed poset, we may repeat the argument with
  $\lambda^b_{\bar n -2}$ in place of $\lambda^b_{\bar n - 3}$. This allows us to conclude that
  $\lambda^b_{\bar n -2}$ is also still regular
  in $\Vlb(\kappa)[I^b \times D^0 \times D^2 \times \Agg][P_1^* \times P_2^* \times P_3^*]$,
  and that $\cf(\mu) \ge \lambda^b_{\bar n -2}$ in this model.
\end{proof} 

\begin{claim} \label{moresetup}
  The poset $\P_1 \times \P_{2b} \times \P_3$ is $<\lambda^b_{\bar n -1}$-distributive
  in $V[L][R][\Agg]$. 
\end{claim}

\begin{proof} The argument is very similar to that for Claim \ref{34setup}
  so we just sketch it. We force to extend
 $\Vlb(\kappa)[I^b \times D^0 \times D^2 \times \Agg][P_1 \times P_{2b} \times P_3]$
 to a model $\Vlb(\kappa)[I^b \times D^0 \times D^2 \times \Agg][P_{2b} \times T]$
 where $T$ is generic for a term forcing $\T$ which is 
 defined and $<\lambda^b_{\bar n-1}$-closed in  $\Vlb(\kappa)$, and then to
 a model  
$\Vlb(\kappa)[D^2 \times \Agg][I^b_0 \times D^0][I^b_1 \times P_{2b} \times T]$
 where $\I^b_0$ and $\I^b_1$ are defined in $\Vlb(\kappa)$,
 and in that model  $\I^b_0$ is $\lambda^b_{\bar n -1}$-Knaster and
 and $\I^b_1$ is $<\lambda^b_{\bar n -1}$-closed.

 In the model $\Vlb(\kappa)[D^2 \times \Agg]$ we appeal to Easton's lemma to see that
 all $<\lambda^b_{\bar n -1}$-sequences of ordinals from $V[L][R][\Agg][P_1 \times P_{2b} \times P_3]$
 lie in $\Vlb(\kappa)[D^2 \times \Agg][I^b_0 \times D^0]$. 
\end{proof}   

It is now easy to see that $\mu$ has cardinality and cofinality $\lambda^b_{\bar n -1}$
in the model $V[L][R][\Agg][P_1 \times P_{2b} \times P_3]$, and that $\P_{2a}$ is
$\lambda^b_{\bar n -1}$-Knaster in this model.

At this point we have produced a generic $\chi$-supercompactness embedding
$j:V[L][R] \rightarrow N$ 
with critical point $\lambda^b_{\bar n}$ which  exists in $V[L][R][P_{1,2,3}]$. 
By construction $N = \{ j(F)(a): F \in V[L][R], \dom(F) = Z \}$ where
$a = j[\chi]$ and $Z = (P_{\lambda^b_{\bar n}} \chi)^{V[L \restriction \kappa][L_b]}$.

Now let $Z_0 = (P_{\lambda^b_{\bar n}} \mu)^{V[L \restriction \kappa][L_b]}$. Factoring $j$
in the standard way we obtain a generic $\mu$-supercompactness embedding
$j_0: V[L][R] \rightarrow N_0$, such that $N_0 = \{ j_0(F_0)(a_0): F_0 \in V[L][R], \dom(F) = Z_0 \}$
where $a_0 = j_0[\mu]$. Since $\Agg$ is generic for $\mu$-distributive forcing in $V[L][R]$
and $\vert Z_0 \vert = \mu$, it is easy to see that
$j_0[\Agg]$ generates an $N$-generic filter and we may lift $j_0$
onto $V[L][R][\Agg]$. Note that the lifted $j_0$ exists in 
$V[L][\Agg][R][P_{1,2,3}]$.
  \end{proof}

\begin{lemma}\label{lemma2}
  There exist in $V[L][\Agg][R]$ a set $J$, a stem $h$ and a sequence $\langle u_\alpha \mid \alpha\in J \rangle$
such that:
\begin{itemize}
\item $J \subseteq I$ and $J$ is unbounded.
\item $h$ has length $n$.
\item $u_\alpha$ is a node of level $\alpha$.
\item For all $\alpha, \beta \in J$ with $\alpha < \beta$, $h \forces^* u_\alpha < u_\beta$.
\end{itemize}  
\end{lemma}

\smallskip

\begin{proof}
  We work for the moment in $V[L][\Agg]$.  Let $\bar n$ be so large that an elementary embedding
with critical point $\lambda^b_{\bar n}$ is guaranteed to fix each stem for a condition of
length $n$, together with the set of all such stems. Let $j_0$ be an elementary embedding
with critical point $\lambda^b_{\bar n}$  constructed as in Lemma \ref{keylifting}.
  
  Let $S$ be the set of stems for conditions of length $n$ in $\bar \P$.
  Define relations $(R_h)_{h \in S}$ on $I \times \kappa$ as follows:
  $(\alpha, \eta) R_h (\beta, \zeta)$ if and only if there is a condition
  $p \in \bar \P$ with length $n$ and stem $h$ such that
  $p \forces (\alpha, \eta) <_{\dot T} (\beta, \zeta)$. Since $\dot T$ names a tree
  and conditions with the same stem are compatible, it is easy to verify that this
  set of relations forms a system.

  We will show that forcing with $\R * \P_{1, 2, 3}$ adds a system of branches
  $(b_{h, i})_{(h, i) \in S \times \kappa}$.
  Let $\gamma \in j_0(I)$ with $\sup j_0[\mu] < \gamma$.
  By the choice of $\bar n$ we have $\vert S \vert < \lambda^b_{\bar n}$ and $j_0(S) = S$.  
  Let $\alpha \in \dom(b_{h, i})$ if and only if there exist $\eta < \kappa$
  and $p \in j_0(\bar \P)$ with stem $h$ such that
  $p \forces (j_0(\alpha), \eta) <_{j_0(\dot T)} (\gamma, i)$, and in this case let
  $b_{h, i}(\alpha)$ be the unique $\eta$ for which this holds. It is easy to see that
  this is a system of branches with $b_{h, i}$ forming a branch through $R_h$.

  Using Claim \ref{34setup} and appealing to Fact \ref{Itay3.4} in the model
  $V[L][\Agg][R]$, there exists $(h, i) \in S \times \kappa$ such that
  $b_{h, i} \in V[L][\Agg][R]$ and $\dom(b_{h, i})$ is unbounded in $\mu$.
  Now let $J = \dom(b_{h, i})$, and for $\alpha \in J$ let $u_\alpha = (\alpha, b_{h, i}(\alpha))$.
  If $\alpha < \beta$ with $\alpha, \beta \in J$ there is a condition $p \in j_0(\bar \P)$ with
  stem $h$ such that $p \forces j_0(u_\alpha), j_0(u_\beta) <_{j(\dot T)} (\gamma, i)$, so
  $p \forces j_0(u_\alpha) <_{j_0(\dot T)}  j_0(u_\beta)$. Since $j_0(h) = h$, by elementarity there is
  $p \in \bar\P$ with stem $h$ such that $p \forces u_\alpha <_{\dot T}  u_\beta$.
\end{proof}


Suppose that a stem $h'$  has  the form
$\langle q_{17}, x_{17}, . . . q_{m-1}, x_{m-1}, [g]_{U_m}\rangle$,
and that $(x, q) \in B_m$.
We write $h' + (x, q)$ for the stem
$\langle q_{17}, x_{17}, . . . q_{m-1}, x_{m-1}, g(x), x, q \rangle$. 
This is technically illegal because it depends on the choice of $g$, but
we will only use this notation in a context where the choice of $g$ is explicit.  

\smallskip

\noindent Global notation: $h' + (x, q)$\index[Notation]{$h' + (x, q)$}

\smallskip

Fix $J$, $h$ and $\langle u_\alpha: \alpha \in J \rangle$ as in the conclusion of
Lemma \ref{lemma2}. 

\smallskip

\begin{lemma}\label{jlemma3}
  Let $h'$ be a stem
of the form $\langle q_{17}, x_{17}, . . . q_{m-1}, x_{m-1}, [g]_{U_m}\rangle$.
  Assume that there exists in $V[L][\Agg][R]$ an unbounded set $\bar J \subseteq J$,
  such that $h' \forces^* u_\alpha < u_\beta$ for all $\alpha, \beta \in \bar J$ with $\alpha < \beta$.

Then there exist $\rho < \mu$ and a sequence $(E_\alpha)_{\alpha \in \bar J \setminus \rho}$ in $V[L][R][\Agg]$ such that:
\begin{enumerate}
\item $E_\alpha \subseteq B_m$ and $E_\alpha \in {\mathcal F}_m$.
\item For all $\alpha, \beta \in \bar J \setminus \rho$ with $\alpha < \beta$ and all
  $(x, q) \in E_\alpha \cap E_\beta$,
  $h' + (x, q) \forces^* u_\alpha < u_\beta$, where $h' + (x, q)$ is computed using the function $g$. 
\end{enumerate}

\end{lemma}

Before proving Lemma \ref{jlemma3}, we rewrite the conclusion
in a way that is less concise but will be useful later. 
Refining $E_\alpha$ if necessary,  we may fix  $D_\alpha \in U_m$ and $F_\alpha \in C_{m+1}$
such that $E_\alpha  = \{   (x, q) \in B_m: x \in D_\alpha, q \le [F_\alpha(x,-)] \}$.
Now $([F_\alpha])_{\alpha \in \bar J \setminus \rho} \in M_1^*$ by the distributivity
of $\Agg$ and the closure of $M_1^*$, $[F_\alpha] \in K$ for all $\alpha$,
and $K$ is generic over $M_1^*$. It follows that there is $F^* \in C_n$ such that
$[F^*] \le [F_\alpha]$ for all $\alpha$, and shrinking $D_\alpha$ we may assume that
$[F^*(x. -)] \le [F_\alpha(x, -)]$ for all $x \in D_\alpha$.
Refining $E_\alpha$ again we may assume that $E_\alpha  = \{   (x, q) \in B_m: x \in D_\alpha, q \le [F^*(x,-)] \}$.

Then $(x, q) \in E_\alpha \cap E_\beta$ if and only if $x \in D_\alpha \cap D_\beta$ and $q \le [F^*(x,-)]$,
and the conclusion amounts to saying that
if $x \in D_\alpha \cap D_\beta$ then
$h' + (x, [F^*(x, -)]) \forces^* u_\alpha < u_\beta$.
Readers of \cite{ItaySCH} and \cite{sinapova-omega^2} will notice that 
Lemma \ref{jlemma3} is parallel to \cite[Lemma 3.5]{ItaySCH} and \cite[Lemma 16]{sinapova-omega^2}.

\begin{proof}
  Choose $n^*$ such that $m + 20 < n^* < \omega$, and let $j_1: V[L][R][\Agg] \rightarrow N_1$  be a generic $\mu$-supercompactness 
  embedding with critical point $\lambda^b_{n^*}$ constructed as in Lemma \ref{keylifting}
 The embedding is added by a certain product $\P_1 \times \P_2 \times \P_3$.
We will work for the moment in $V[L][R][\Agg][P_1 \times P_2 \times P_3]$.

Let $\gamma \in j_1(\bar J)$ be such that $\sup j_1[\mu] < \gamma$, and let $v = j_1(u)_\gamma$. 
By elementarity, for every $\alpha \in \bar J$ there is a condition $r_\alpha \in j_1(\bar \P)$
such that $r_\alpha$ has stem $h'$ and $r_\alpha \forces j_1(u_\alpha) <_{j_1(T)} v$.
Now $j_1(g) = g$, and $[g]_{U_m} = [j_1(g)]_{j_1(U_m)} = [f^{r_\alpha}_m]_{j_1(U_m)}$.
Shrinking $A^{r_\alpha}_m$ if necessary, we may assume that 
$A^{r_\alpha}_m \subseteq \dom(g)$ and $f^{r_\alpha}_m = g \restriction A^{r_\alpha}_m$.

For each $y \in A^{r_\alpha}_m$, the minimal one-point extension of
$r_\alpha$ by $y$ forces $j_1(u_\alpha) <_{j_1(T)} v$. Since $g(y) = j_1(g)(y) =  f^{r_\alpha}_m(y)$,
the stem of the minimal one-point extension is
$\langle q_{17}, x_{17}, . . . q_{m-1}, x_{m-1}, g(y), y, r \rangle$
where $r = [F^{r_\alpha}_{m+1}(y, -)]_{U_{m+1}}$.
We conclude that there is a $j_1({\mathcal F}_m)$-large set $X_\alpha$
such that $h' + (y, r) \forces^*_{j_1(\bar \P)} j_1(u_\alpha) <_{j_1(T)} v$ for all $(y, r) \in X_\alpha$.
Membership of $X_\alpha$ in $j_1({\mathcal F}_m)$ is witnessed by $A^{r_\alpha}_m$ and
$F^{r_\alpha}_{m+1}$.

For $(x, q) \in B_m$, let
${\bar J}_{x, q} = \{ \alpha \in \bar J : h' + (x, q) \forces^*_{j_1(\bar \P)}   j_1(u_\alpha) < v \}$.
It is easy to see that for $\beta \in {\bar J}_{x, q}$, we have that $\alpha \in {\bar J}_{x, q} \cap \beta$
if and only if $h' + (x, q) \forces^*_{\bar \P} u_\alpha < u_\beta$, so that
${\bar J}_{x, q} \cap \beta \in V[L][\Agg]$.

Since $\mu$ has cardinality and cofinality $\lambda^b_{n^* - 1}$ in
in $V[L][R][\Agg][P_1 \times P_{2b} \times P_3]$, and $\P_{2a}$ has the
$\lambda^b_{n^*-1}$-approximation property in this model,  it follows
 that whenever ${\bar J}_{x, q}$ is unbounded in $\mu$ we have ${\bar J}_{x, q} \in V[L][R][\Agg][P_1 \times P_{2b} \times P_3]$.
 It is important to notice that even in this case the definition of $j_1$ (and hence ${\bar J}_{x, q}$) requires $P_{2a}$.

 
 Working in $V[L][R][\Agg][P_1 \times P_{2b} \times P_3]$, let ${\mathcal J}_{x, q}$ be the set of unbounded subsets $C$ of $\mu$ such that
 some condition in $\P_{2a}$ forces ${\bar J}_{x, q} = C$. It is easy to see that
 \begin{itemize}
 \item $\vert {\mathcal J}_{x, q} \vert < \lambda^b_{n^* - 1}$.
 \item The function $(x, q) \mapsto {\mathcal J}_{x, q}$ is in  $V[L][R][\Agg][P_1 \times P_{2b} \times P_3]$.
 \item If $C \in  {\mathcal J}_{x, q}$ and $\beta \in C$, then $C \cap \beta$ is the set of $\alpha < \beta$ such that
   $h' + (x, q) \forces^*_{\bar \P} u_\alpha < u_\beta$.
 \item If $C_1, C_2 \in {\mathcal J}_{x, q}$ with $C_1 \neq C_2$ then $C_1 \cap C_2$ is bounded in $\mu$.
 \end{itemize}

 Let $\rho < \mu$ be such that $C_1 \cap C_2 \subseteq \rho$ for all $(x, q) \in B_m$ and all $C_1, C_2 \in {\mathcal J}_{x, q}$
 with $C_1 \neq C_2$.  For $(x, q) \in B_m$ and $\alpha \in \bar J \setminus \rho$,
 let $f(x, q, \alpha)$ be the unique $C \in {\mathcal J}_{x, q}$ such that $\alpha \in C$ if such a $C$ exists,
 and let it be undefined otherwise.

 Let $\alpha_0 = \min(\bar J \setminus \rho)$, and let $A^*_\alpha$ be the set of $(x, q) \in B_m$ such that:
 \begin{itemize}
 \item  $f(x, q, \alpha)$  and $f(x, q, \alpha_0)$ are both defined.
 \item  $f(x, q, \alpha) = f(x, q, \alpha_0)$.
 \end{itemize}

 \begin{claim} $A^*_\alpha \in {\mathcal F}_m$.
 \end{claim}

 \begin{proof}
   Otherwise $B_m \setminus A^*_\alpha \in {\mathcal F}_m^+$,
   so that applying $j_1$ we have
   $B_m \setminus A^*_\alpha = j_1(B_m \setminus A^*_\alpha) \in j_1({\mathcal F}_m)^+$. 
   For each $\beta \in \bar J \setminus \rho$,
   choose $(x_\beta, q_\beta) \in X_{\alpha_0} \cap X_\alpha \cap X_\beta \cap (B_m \setminus A^*_\alpha)$. 

   Since $\cf(\mu) > \vert B_m \vert$ in $V[L][\Agg][R][P_1 \times P_{2b} \times P_3]$,
   we may find $(x, q) \in B_m$ such that $(x_\beta, q_\beta) = (x, q)$
   for unboundedly many $\beta$. For all such $\beta$ we have by the choice of $X_\beta$ that
   $h' + (x, q) \forces^*_{j_1(\bar \P)} j_1(u_\beta) <_{j_1(T)} v$, so
   ${\bar J}_{x, q}$ is unbounded and hence ${\bar J}_{x, q} = C$
   for some $C \in {\mathcal J}_{x, q}$.

   Since $(x, q) \in A_\alpha$, $\alpha \in {\bar J}_{x, q} = C$ so that
   $f(x, q, \alpha)$ is defined and $f(x, q, \alpha) = C$. The same is
   true for $\alpha_0$, so 
   $f(x, q, \alpha) = f(x, q, \alpha_0) = C$, and $(x, q) \in A^*_\alpha$.
   This is a contradiction since by construction $(x, q) \in B_m \setminus A^*_\alpha$. 
\end{proof}

  Define relations $(R_{x, q})_{(x, q) \in B_m}$ on $(\bar J \setminus \rho) \times 1$
  as follows: $(\alpha, 0) R_{x, q} (\beta, 0)$ if and only if
  $h' + (x, q) \forces^* u_\alpha <_T u_\beta$.
  It is easy to see that these  
  relations  form a system on $(\bar J \setminus \rho) \times 1$
  in the sense of Definition \ref{systemdef}: the main point is that if $\alpha, \beta \in \bar J \setminus \rho$ with
  $\alpha < \beta$ then by hypothesis $h' \forces^* u_\alpha < u_\beta$, and any minimal one-point extension
  of a suitable condition witnessing this will witness that $h' + (x, q) \forces^* u_\alpha < u_\beta$
  for some $(x, q) \in B_m$.

  For every $(x, q) \in B_m$, let $b_{x, q} = \{ \alpha \in \bar J \setminus \rho: (x, q) \in A^*_\alpha \}$.
 If $\alpha, \beta \in b_{x,q}$ with $\alpha < \beta$,
then $(x, q) \in A^*_\alpha \cap A^*_\beta$, so
$f(x, q, \alpha) = f(x, q, \beta) = C \in {\mathcal J}_{x, q}$ and hence
$h' + (x, q) \forces^* u_\alpha < u_\beta$, that is $(\alpha, 0) R_{x, q} (\beta, 0)$. 
Let $b^*_{x, q}$ be the function with domain $b_{x, q}$ and $b^*_{x, q}(\alpha) = (\alpha, 0)$
for all $\alpha \in b_{x, q}$. 

\begin{claim} \label{plusone-SOB}
 $(b^*_{x, q})$ is a system of branches through $(R_{x, q})$ in the sense of Definition \ref{systemdef}.
\end{claim} 

\begin{proof} Let $\beta \in b_{x, q}$ and let $\alpha < \beta$ be such that
  $h' + (x, q) \forces^* u_\alpha <_T u_\beta$. As $(x, q) \in A^*_\beta$,
  we have $f(x, q, \beta) = f(x, q, \alpha_0) = C$ where $C \in {\mathcal J}_{x, q}$.
  By the properties of ${\mathcal J}_{x, q}$ we have $\alpha \in C$, so that
  $f(x, q, \alpha) = C = f(x, q, \alpha_0)$ and $(x, q) \in A_\alpha^*$,
  hence $\alpha \in b_{x, q}$. Finally for every $\alpha \in \bar J \setminus \rho$
  we have $\alpha \in b_{x, q}$ for any $(x, q) \in A^*_\alpha$. 
\end{proof}

Let $E$ be the set of $(x, q) \in B_m$ such that  $b_{x, q} \in V[L][\Agg][R]$
and $b_{x, q}$ is unbounded. By the distributivity of $\P_1 \times \P_{2b} \times \P_3$, we have $E \in V[L][\Agg][R]$.
We now work below a condition in $\P_1 \times \P_{2b} \times \P_3$ that determines the value of $E$. 


\begin{claim} \label{plusone-E-large}
  $E \in {\mathcal F}_m$.
\end{claim}

\begin{proof} 
  Suppose for a contradiction that $B_m \setminus E \in {\mathcal F}_m^+$.
  The set $\{ b^*_{x, q} : (x, q) \in B_m \setminus E \}$ is still a system of branches through $(R_{x, q})$,
  since for every $\alpha \in \bar J \setminus \rho$ we may choose $(x, q) \in (B_m \setminus E) \cap A^*_\alpha$
  to witness that $\alpha \in \dom(b^*_{x, q})$. 

  Now we appeal to Fact \ref{Itay3.4} with $\lambda = \lambda^b_{n^*-3}$, $\P = \P_1 \times \P_{2b} \times \P_3$,
  and $\Q = \P_1^* \times  \P_{2b}^* \times \P_3^*$ defined as in the discussion preceding Claim \ref{34setup}.
    It follows that there is $(x, q) \in B_m \setminus E$ such that $b_{x, q} \in V[L][\Agg][R]$ and
    $b_{x, q}$ is unbounded, an immediate contradiction.
\end{proof}

By distributivity, $(b_{x, q})_{(x, q) \in E} \in V[L][\Agg][R]$. 
For every $\alpha \in \bar J \setminus \rho$, let $E_\alpha = \{ (x, q) \in E : \alpha \in b_{x, q} \}$.
Since $E_\alpha = E \cap A^*_\alpha$, $E_\alpha \in {\mathcal F}_m$.
 For all $\alpha, \beta \in \bar J \setminus \rho$ with $\alpha < \beta$ and all
 $(x, q) \in E_\alpha \cap E_\beta$,
 $\alpha, \beta \in b_{x, q}$ and hence $h' + (x, q) \forces^* u_\alpha < u_\beta$.
\end{proof}

\begin{lemma} \label{endgame} 
  There exist $\rho < \mu$ and a sequence of conditions $(p_\alpha)_{\alpha \in J \setminus \rho}$
  in $V[L][\Agg][R]$
  such that:
  \begin{itemize}
  \item For all $\alpha$ the stem of $p_\alpha$ is $h$.
  \item For all $\alpha$ and $\beta$ with $\alpha < \beta$,
    $p_\alpha \wedge p_\beta \forces u_\alpha < u_\beta$.
  \end{itemize}
\end{lemma}

\begin{proof}
  Let the stem $h$ from the conclusion of Lemma \ref{lemma2} be $q_{17}, \ldots x_{n-1}, [g]$.
  
  We will construct an increasing sequence $(\rho_m)_{n \le m < \omega}$ of ordinals less than $\mu$,
  together with sets $(A^\alpha_m)_{n \le m, \alpha \in J \setminus \rho_{m+1}}$ and
  functions $(F^*_{m+1})_{n \le m}$ such that the following properties hold, along with another one to be stated
  below: 
  \begin{itemize}
  \item $A^\alpha_m \in U_m$. 
  \item  $F^*_{m+1} \in C_{m+1}$.
  \item $A^\alpha_m \times^\prec A^{\alpha+1}_{m+1} \subseteq \dom(F^*_{m+1})$.   
  \end{itemize}

  For   $m \ge n$, $\alpha \in  J \setminus \rho_m$ say that 
  a lower part $s$ is {\em good for $\alpha$ at $m$}
  if and only if:
  \begin{itemize}
  \item  $s$ has the the form $q'_{17}, \ldots x'_{m-1}$.
  \item $q'_{k} \le q_k$ and $x'_k = x_k$ for $17 \le k < n$.
  \item  $x'_k \in A^\alpha_k$ for $n \le k < m$.
  \item  $q'_k \le F^*_{k+1}(x'_k, x'_{k+1})$ for $n < k < m-1$.
  \item  $q'_n \le g(x'_n)$ if $m > n$.
  \end{itemize}

    The final key property is that:
  \begin{itemize}
  \item For $\alpha , \beta \in J \setminus \rho_m$, if a lower part $s$ of form $q'_{17}, \ldots x'_{m-1}$ 
    is good for both $\alpha$ and $\beta$ at $m$,
    then  $s + [g] \forces^* u_\alpha < u_\beta$ if $m =n$,
    and $s + [F^*_m(x'_{m-1}, -)] \forces^* u_\alpha < u_\beta$ if $m > n$.
  \end{itemize}
  
  To initialise the construction we set $\rho_n = 0$, and verify that the key property 
  holds for $m = n$.
  Suppose that $\alpha, \beta \in J$
  with $\alpha < \beta$, and $s$ is good for both $\alpha$ and $\beta$ at $n$.
  That is to say $s$ has the the form $q'_{17}, \ldots x_{n-1}$
  where $q'_k \le q_k$  for $17 \le k < n$.
  Therefore $s + [g] \le^* h$, and so by the conclusion of Lemma \ref{lemma2}
  $s + [g] \forces^* u_\alpha < u_\beta$.

  Continuing the initialisation apply Lemma \ref{jlemma3} to the stem $h$, set $J$ and sequence $(u_\alpha)_{\alpha \in J}$.
  Let $\rho_{n+1}$ be the ordinal $\rho$ from the conclusion of that lemma,
  and choose sets $(A^\alpha_n)_{\alpha \in J \setminus \rho_{n+1}}$ and a function $F^*_{n+1} \in C_{n+1}$ as in the discussion
  following the statement of Lemma \ref{jlemma3}. We verify that the key property holds
  for $m = n +1$.
  
  Suppose that $\alpha, \beta \in J \setminus \rho_{n+1}$ with $\alpha < \beta$,
  and $s$ is good for both $\alpha$ and $\beta$ at $n+1$.
  That is to say $s$ has the the form $q'_{17}, \ldots x_{n-1}, q, x$
  where $q'_k \le q_k$  for $17 \le k < n$, $q \le g(x)$, $x \in A^\alpha_n \cap A^\beta_n$.
  Let $t = q_{17}, \ldots x_{n-1}$, so that  
  by construction $h + (x, [F^*_{n+1}(x, -)]) = t + g(x) + x + [F^*_{n+1}(x, -)] \forces^* u_\alpha < u_\beta$.
  Since $s \le^* t + g(x) + x$,
  $s + [F^*_{n+1}(x, -)] \forces^* u_\alpha < u_\beta$ as required.  

  Now suppose that $m \ge n$ and we have constructed $\rho_k$ for $k \le m+1$,
  $(A^\alpha_k)_{\alpha \in J \setminus \rho_{k+1}}$ for $k \le m$ and $F^*_k$ for $k \le m+1$.
  Let $s$ be a lower part of form $q'_{17}, \ldots x'_m$ and let $J^s$ be the set of
  $\alpha \in J \setminus \rho_{m+1}$ such that $s$ is good for $\alpha$ at $m+1$.
  By construction $s + [F^*_{m+1}(x'_m, -)] \forces^* u_\alpha < u_\beta$  for $\alpha, \beta \in J^s$ with
  $\alpha < \beta$. 
  
  For every lower part $s$ such that $J^s$ is bounded, let $\rho^s = \sup(J^s)$. 
  For $s$ such that $J^s$ is unbounded, we apply Lemma \ref{jlemma3} to the
  unbounded set $J^s$ and the stem $s + [F^*_{m+1}(x'_m, -)]$. We obtain $\rho^s < \mu$ and
  sets $(E^s_\alpha)_{\alpha \in J^s \setminus \rho_s}$ in ${\mathcal F}_{m+1}$, such that
  $s + F^*_{m+1}(x'_m, x) + (x, q) \forces^* u_\alpha < u_\beta$ for all
  $\alpha, \beta \in J^s \setminus \rho^s$ with $\alpha < \beta$
  and all $(x, q) \in E^s_\alpha \cap E^s_\beta$.

  Now let $\rho_{m+2} = \sup_s \rho^s$, and for $\alpha \in J \setminus \rho_{m+2}$
  let $E^{m+1}_\alpha = \Delta_{s, \alpha \in J^s} E^s_\alpha$,
  that is $E^{m+1}_\alpha = \{ x : \forall s \prec x \; \alpha \in J^s \implies x \in E^s_\alpha \}$.
  It follows that
  for $\alpha, \beta \in J \setminus \rho_{m+2}$ with $\alpha < \beta$,
  $(x, q) \in E^{m+1}_\alpha \cap E^{m+1}_\beta$, and $s$ as above such that
  $s \prec x$ and $s$ is good for both $\alpha$ and $\beta$ at $m+1$
  we have $s + F^*_{m+1}(x'_m, x) + (x, q) \forces^* u_\alpha < u_\beta$.
  
  As in the discussion following Lemma \ref{jlemma3}, we now (shrinking $E^{m+1}_\alpha$ if necessary)
  choose $(A^\alpha_{m+1})_{\alpha \in J \setminus \rho_{m+2}}$ and $F^*_{m+2}$ such that
  $E^{m+1}_\alpha = \{ (x, q) : x \in A^\alpha_{m+1}, q \le [F^*_{m+2}(x, -)] \}$.
  To finish the construction we verify that we have maintained the key property.
  So let $s$ be a lower part of form $q'_{17}, \ldots x'_m, q'_{m+1}, x'_{m+1}$
  which is good for both $\alpha$ and $\beta$ at $m+2$. Let $t$ be the initial
  segment $q'_{17}, \ldots x'_m$ of $s$, so that $t \prec x'_{m+1}$ and  $t$ is good for $\alpha$ and $\beta$ at $m+1$. 
  By definition $q_{m+1} \le F^*_{m+1}(x'_m, x'_{m+1})$ and $x'_{m+1} \in A^{m+1}_\alpha \cap A^{m+1}_\beta$,
  by construction $t + F^*_{m+1}(x'_m, x'_{m+1}) + x + [F^*_{m+2}(x, -)] \forces^* u_\alpha < u_\beta$,
  so $s + [F^*_{m+2}(x, -)] \forces^* u_\alpha < u_\beta$ as required.

  Now let $\rho = \sup_{n \le m < \omega} \rho_n$, and for $\alpha \in J \setminus \rho$ define
  $p_\alpha$ as follows:
\begin{itemize}  
\item $p_\alpha$ has $q_{17}, \ldots x_{n-1}$ as an initial segment.
\item $f^{p_\alpha}_n = g \restriction A^\alpha_n$.
\item $F^{p_\alpha}_k = F^*_k \restriction A^\alpha_{k-1} \times^\prec A^\alpha_k$ for $n < k < \omega$.   
\end{itemize}

Let $\alpha, \beta \in J \setminus \rho$ with $\alpha < \beta$, and suppose for a contradiction that
there is $q \le p_\alpha, p_\beta$ such that $q \forces u_\alpha \nless u_\beta$.
Let the lower part of $q$ be $t = q'_{17}, \ldots x'_{m-1}$, where without loss of generality $m > n$.
By definition $t$ is good for both $\alpha$ and $\beta$ at $m$,
and $f^q_m \le F^*(x'_{m-1}, -)$, so that a fortiori $[f^q_m] \le [F^*(x'_{m-1}, -)]$
  The stem of $q$ is $h = t + [f^q_m]$, and so by construction
  $h \forces^* u_\alpha < u_\beta$ for an immediate contradiction.
\end{proof}

\begin{lemma} \label{almost}
  The tree property at $\mu$ holds in the model $V[L][\Agg][R][\bar P]$.
\end{lemma}

\begin{proof} Let $p$ be arbitrary.
  Our whole construction could have been done
  below $p$,
  so that the conditions $p_\alpha$ from the conclusion of Lemma \ref{endgame} can be taken to be
  extensions of $p$. Since $\bar \P$ is $\mu$-cc, by Lemma \ref{kappaccfact} there is $\alpha$ such that $p_\alpha$
  forces $\{ \beta : p_\beta \in G \}$ to be unbounded. But for $\beta < \gamma$ with
  $p_\beta, p_\gamma \in G$, we have $p_\beta \wedge p_\gamma \in G$, so that
  $u_\beta <_T u_\gamma$. It follows that $p_\alpha$ forces that the $u_\beta$ such that
  $p_\beta \in G$ form a cofinal branch.
\end{proof}   

\begin{lemma} \label{finallythere} 
  The tree property at $\mu$ holds in the model $V[L][\Agg][\bar P]$.
\end{lemma}

\begin{proof} 
  We need to verify that we have satisfied the hypotheses of Lemma \ref{preslemma}
  as listed at the start of Section \ref{newbranchlemma}. 
   $V[L][\Agg]$  plays the role of $V$, and  $\bar \P$ plays the role of $\P$.

  Hypothesis \ref{preslemmahyp1} is immediate.  By Lemma \ref{cardinals_in_Prikry_extension} $\kappa$ is a cardinal
   and $\mu = \kappa^+$  in $V[L][\Agg][\bar P]$, taking care of the first part
   of Hypothesis \ref{preslemmahyp2}. Claim \ref{for_use_later} takes care of the
   rest of Hypothesis \ref{preslemmahyp2}. Hypothesis \ref{preslemmahyp3} follows
   from Lemma \ref{portmanteauclaim2}. The remaining hypotheses follow from Lemmas
   \ref{samestemub}, \ref{ccforPbar} and \ref{omegasamestemub}.

  The conclusion is now immediate from Lemmas \ref{almost} and \ref{preslemma}. 
\end{proof}

\subsection{The tree property at $\lambda=\aleph_{\omega^2+2}$} \label{plustwo}

Recall from Section \ref{plusone} that $\nu = \lambda^b_\omega$ and $\mu = \lambda^b_{\omega +1}$.
Let $\lambda = \lambda^b_{\omega+2}$, so that in our final model the cardinal
$\lambda$ is destined to become $\aleph_{\omega^2 + 2}$. 
We will establish that the tree property holds at $\lambda$ in $V[L][\Agg][\bar P]$. 
The argument is quite similar to that from \cite[Theorem 3.1]{AIM},
but there are extra complications.

We start by constructing a suitable generic embedding with critical point $\lambda$
whose domain is a generic extension of $V[L][\Agg]$.
Now that the desired critical point is $\lambda$ the poset
$\FL \restriction \kappa * \FL^b * \I^b$ counts as small forcing,
and the main obstacle is to deal with $\J^c_0 * \J^c_1$. 
We note that the situation here is very like the $n =0$ case in Section \ref{GroupI},
 in particular the proof of Claim \ref{GroupInequals0} in that section.

We start by writing
$V[L][\Agg] = \Vlbi(\kappa)[J^c_0][J^c_1][A_e \times \Agg]$.
Recall that $\B^c_1$ is defined and $<\lambda$-closed in $\Vlbi(\kappa)[J^c_0]$,
and $\U^c_1 = (\B^c_1)^{+A^c_1}$. Similarly 
$\C^c_1$ is defined and $<\lambda$-closed in $\Vlbi(\kappa)[J^c_0]$,
$\S^c_1 = (\C^c_1)^{+A^c_1}$.

Parallel to the proof of Claim \ref{GroupInequals0}, let
$\TBC = {\termspace}^{\Vlbi(\kappa)[A^c_0 * U^c_0]}(\S^c_0, \B^c_1 \times \C^c_1)$,
so that $\TBC$ is defined and $<\lambda$-closed in $\Vlbi(\kappa)[A^c_0 * U^c_0]$.
Let $\QTT$ be the two-step iteration of term forcing which adds a $\TBC$-generic object
 inducing $U^c_1 * S^c_1$. Exactly as in the proof of Claim \ref{GroupInequals0},
$\QTT$ is $<\mu$-closed in $\Vlbi(\kappa)[J^c_0 * J^c_1]$. 
Forcing with $\QTT$ over $V[L][\Agg]$, we obtain a model
$V[L][\Agg][QTT] = \Vlbi(\kappa)[J^c_0][A^c_1 \times TBC][A_e \times \Agg]$.

We will do the construction for Lemma \ref{indestructible2} with appropriate parameter settings.
Start by recalling that in the context of $\A^c_0 * \U^c_0 * \S^c_0 \times \A^c_1$
we have $\mu_0 = \lambda^b_{17}$, $\mu_1 = \lambda^b_{\omega+1} = \mu$, $\mu_2 = \lambda^b_{\omega+2} = \lambda$,
$\mu_3 = \lambda^b_{\omega+3}$. Accordingly we will set the parameters for Lemma \ref{indestructible2}
as follows: 
   \begin{itemize}
   \item $n =0$.
   \item $\eta = \mu_3 = \lambda^b_{\omega+3}$.
   \item $\Vi$ is $V$.
   \item $\Vd$ is $\Vlbi(\kappa)$,
     so that $\Vd[A \restriction \mu_{n+2} * U \restriction \mu_{n+2}]$
     is $\Vlbi(\kappa)[A^c_0 * U^c_0]$. 
   \item $V'$ is $\Vlbi(\kappa)[A^c_0 * U^c_0 * S^c_0 * A^c_1]$. 
   \item $\D^2 =  \Aggforcing$.
   \item $\D^3  = \TBC$.
   \item $\D^0 =  \A_e$.
   \item $\D^{small}$ is trivial forcing.  
   \item $V'[D^{0,2,3}]$ is $V[L][\Agg][QTT]$.    
   \end{itemize}

   
   Following the argument for Lemma \ref{indestructible2}, we will start with a
   suitable embedding $j:V[\Agg] \rightarrow N$,
   defined in $V[\Agg]$ and witnessing that $\lambda$ is $\chi$-supercompact
   in $V[\Agg]$ for some large enough value of $\chi$. Since $L\restriction \kappa * L^b * I^b$
   is generic for small forcing, it is easy to lift
   $j$ to an embedding $j:\Vlbi(\kappa)[\Agg] \rightarrow N[L \restriction \kappa][L^b][I^b]$,
   defined in $\Vlbi(\kappa)[\Agg]$ and 
   witnessing that $\lambda$ is $\chi$-supercompact in this model.
   We note that the models $V$ and $V[\Agg]$ agree on $<\lambda$-sequences of ordinals,
   as do $\Vlbi(\kappa)$ and $\Vlbi(\kappa)[\Agg]$.
   
   We will force over $V[L][\Agg][QTT]$ with a product $\P_2 \times \P_3$
   defined as in the proof of Lemma \ref{indestructible2}. To analyse $\P_2$, we
   note that
   \begin{align*}
   \A^c_0 & = \Add^{\Vlb(\kappa)}(\lambda^b_{17}, [\lambda^b_{\omega+1}, \lambda^b_{\omega+2})) \\
         & = \Add^{V[\Agg][L \restriction \kappa][L^b]}(\lambda^b_{17}, [\lambda^b_{\omega+1}, \lambda^b_{\omega+2})). \\
   \end{align*}
     By elementarity,
   \[
       j(\A^c_0) = \Add^{N[L \restriction \kappa][L^b]}(\lambda^b_{17}, [\lambda^b_{\omega+1}, j(\lambda^b_{\omega+2}))).
    \]
    Since $V[\Agg][L \restriction \kappa][L^b] \models {}^\chi N[L \restriction \kappa][L^b] \subseteq N[L \restriction \kappa][L^b]$,
    \begin{align*}
        j(\A^c_0) & = \Add^{V[\Agg][L \restriction \kappa][L^b]}(\lambda^b_{17}, [\lambda^b_{\omega+1}, j(\lambda^b_{\omega+2}))) \\
                  & = \Add^{\Vlb(\kappa)}(\lambda^b_{17}, [\lambda^b_{\omega+1}, j(\lambda^b_{\omega+2}))).\\
    \end{align*}

   By this and similar arguments for $\A^c_1$ and $\A_e$ we have:
   \begin{itemize}
   \item  $j(\A^c_0)/\A^c_0 =
   \Add^{\Vlb(\kappa)}(\lambda^b_{17}, j(\lambda^b_{\omega+2}) \setminus \lambda^b_{\omega+2})$.
   \item  $j(\A_e)/A_e =
   \Add^{\Vlb(\kappa)}(\lambda^b_{17}, j(\lambda^b_{\omega+3}) \setminus j[\lambda^b_{\omega+3}])$.
   \item  $j(\A^c_1)/\A^c_1 =
     \Add^{\Vlb(\kappa)}(\lambda^b_{\omega+1}, j(\lambda^b_{\omega+3}) \setminus j[\lambda^b_{\omega+3}])$.
   \end{itemize}  
   In summary $\P_{2a}$ is a Cohen poset to add subsets of $\lambda^b_{17}$ defined in 
   $\Vlb(\kappa)$, and $\P_{2b}$ is a Cohen poset to add subsets of $\lambda^b_{\omega+1}$
   defined in the same model.

   To analyse $\P_3$, it is useful to recall that $\U^c \restriction \mu_1$ and $\S^c \restriction \mu_1$
   are both trivial. It follows that $\P_3$ is defined in
   $\Vlbi(\kappa)[A^c_0 \times A^c_1 \times \Agg]$
   and is $<\lambda^b_{\omega+1}$-closed in
   $\Vlbi(\kappa)[A^c_0*U^c_0*S^c_0][A^c_1 \times B^c_1 \times C^c_1 \times \Agg]$.
   It is important that, as we noted in the proof of Lemma \ref{indestructible},
   $\P_3$ collapses $\lambda$ to become an ordinal of cofinality $\mu$ and cardinality $\mu$.
   We note for use later that {\it a fortiori}
   $\P_3$ is $<\mu$-closed in 
   $\Vlbi(\kappa)[J^c_0][J^c_1 \times \Agg]$.
   
   As in Lemma \ref{indestructible2} we lift $j$ to obtain a generic embedding
   with critical point $\lambda$ which has domain $V[L][\Agg][QTT]$ and exists
   in $V[L][\Agg][QTT][P_2 \times P_3]$. In the current setting we may restrict
   the domain to $V[L][\Agg]$, so we have a generic embedding with domain $V[L][\Agg]$
   obtained by forcing over $V[L][\Agg]$ with $\P_2 \times (\P_3 \times \QTT)$.

   By Lemma \ref{Cohenrobust}, $\P_2 \times \P_2$ is $\lambda$-cc in $V[L][\Agg]$.
   It follows from Lemma \ref{chain} that $\P_2$ has the $\lambda$-approximation 
   property in $V[L][\Agg]$. By another similar appeal to Lemmas \ref{Cohenrobust} and \ref{chain},
   $\P_{2a}$ has the $\mu$-approximation property in $V[L][\Agg][P_{2b} \times P_3 \times QTT]$.

      By the preceding analysis
     $\QTT$ is $<\mu$-closed in $\Vlbi(\kappa)[J^c]$,
     and since $\Aggforcing$ is highly distributive the same is true in
     $\Vlbi(\kappa)[J^c][\Agg]$, which is the submodel
     of $V[L][\Agg]$ missing only $A_e$. As we noted above $\P_3$ is
     also $<\mu$-closed in this model, so that 
     $\P_3 \times \QTT$ is $<\mu$-closed in this model.
     
     We claim that $\P_{2 b}$
     is $<\mu$-distributive in $\Vlbi(\kappa)[J^c][\Agg]$.
     To see this we note that we need to show that $\P_{2 b}$ is $\lambda^b_n$-distributive for all $n$,
     and this will follow by the usual arguments using term forcing and Easton's Lemma.

     It follows that $\P_3 \times \QTT$ is $<\mu$-closed in
     $\Vlbi(\kappa)[J^c][\Agg][P_{2 b}]$.
     Therefore $\P_3 \times \QTT$ is formerly $<\mu$-closed in the sense of Fact
     \ref{formerlyclosed} in the model
     \[
     \Vlbi(\kappa)[J^c][\Agg][P_{2b}][P_{2a} \times A_e]
     = V[L][\Agg][P_2].
     \]
     The key points are that $A^c_0$ added $\lambda$ subsets of $\lambda^b_{17}$,
     and that $\P_{2 a} \times \A_e$ is $\mu$-cc
     in $\Vlbi(\kappa)[J^c][\Agg][P_{2 b}]$.

   Let $\dot T \in V[L][\Agg]$ be a $\bar \P$-name for a $\lambda$-tree. We assume that level $\alpha$
   is a subset of $\{ \alpha \} \times \mu$.

   \begin{lemma} \label{magic2} 
     In $V[L][\Agg]$ there exist a stem $h$, an unbounded set $I \subseteq \lambda$
     and $(u_\alpha)_{\alpha \in I}$ such that $u_\alpha$ is a node of level $\alpha$ for all $\alpha \in I$,
     and $h \forces^* u_\alpha <_{\dot T} u_\beta$ for all $\alpha, \beta \in I$ with $\alpha < \beta$.
   \end{lemma}
   
   \begin{proof} Let $j$ be the generic embedding with domain $V[L][\Agg]$ and critical point $\lambda$,
     added by forcing
     over $V[L][\Agg]$ with $\P_2 \times (\P_3 \times \QTT)$. Define a system on $\lambda \times \mu$
     indexed by stems as follows: $u R_h v \iff h \forces^* u <_{\dot T} v$.

     In $V[L][\Agg][P_2][P_3][QTT]$ define a branch $b_h$ through $R_h$ as follows:
     $\alpha \in \dom(b_h)$ if and only if there is $\eta < \mu$ such that
     $h \forces^*_{j(\bar P)} (\alpha, \eta) <_{j(\dot T)} (\lambda, 0)$.
     It is routine to check that the branches $b_h$ form a system of branches
     in the sense of Definition \ref{systemdef}. It is also routine that  if $\alpha \in \dom(b_h)$
     then $\beta \in \dom(b_h) \cap \alpha$ if and only if there is $\zeta$ such that
     $h \forces^*_{\bar \P} (\beta, \zeta) <_{\dot T} (\alpha, \eta)$, and in this
     case $b_h(\beta) = \zeta$ for the unique such $\zeta$. In particular
     $b_h \restriction \alpha \in V[L][\Agg]$ for all $\alpha \in \dom(b_h)$.

     We now appeal to Lemma \ref{widersystems} with $\P_3 \times \QTT$ in place of $\P$,
     $\P_{2a} \times \A_e$ in place of $\E$, $\Vlbi(\kappa)[J^c][\Agg][P_{2 b}]$
     in place of $V$, and $\lambda^b_{17}$ in place of $\delta$. 
     It follows that there
     is $h$ such that $b_h \in V[L][\Agg][P_2]$ and $\dom(b_h)$ is unbounded in $\lambda$.
     Since $\P_2$ has the $\lambda$-approximation property in $V[L][\Agg]$, it follows
    that $b_h \in V[L][\Agg]$. Now we set $I = \dom(b_h)$ and $u_\alpha = (\alpha, b_h(\alpha))$
     to finish.
\end{proof} 

   We can now describe the main idea of the proof that $\lambda$ has the tree property. As in
   the proof of the tree property for $\mu$ in Section \ref{plusone}, we will construct conditions $(p_\alpha)$ for all
   sufficiently large $\alpha$ in $I$, such that
   $p_\alpha \wedge p_\beta \forces^{V[L][\Agg]}_{\bar \P} u_\alpha < u_\beta$ for
   $\alpha < \beta$. This time we will construct this sequence of conditions
   in $V[L][\Agg][P_3 \times QTT]$: as in Section \ref{plusone} this will give a branch in
   $V[L][\Agg][P_3 \times QTT][\bar P]$, and we will need to use a suitable branch lemma
   to find a branch in $V[L][\Agg][\bar P]$.

   Let $h$ be of the form  $(\bar s , [g])$  for some lower part $\bar s$ and some
   $1$-variable function $g$. We note that if $\beta \in I$
   then $I \cap \beta = \{ \alpha : h \forces^* u_\alpha < u_\beta \}$.

Let $\bar q = [g]$ and let the length of $h = (\bar s, \bar q)$ be $t$. 
For all relevant $q$, fix $g_q$ such that $q = [g_q]$. We take care to choose $g_{\bar q} = g$.

\begin{remark} Let $F \in {\mathbf F}_n$, that is to say
  $F$ is a potential value of $F^p_n$ for some $p \in \P$. Then for all relevant $x$,
  $[F(x, -)] = [g_{[F(x, -)]}]$, that is to say $F(x, y) = g_{[F(x, -)]}(y)$
  for many $y$.  Taking a diagonal intersection, we may shrink the domain of $F$ to arrange that
  $F(x, y) = g_{[F(x, -)]}(y)$ for all $(x, y) \in \dom(F)$. In the sequel we will
  arrange that all $2$-variable constraint functions have been treated in this way. 
\end{remark}

As in Section \ref{plusone}, we can use the functions $g_q$ to prolong a stem
$(s, q)$ to stems $(s, q) + (x, r) = (s, g_q(x), x, r)$
for each $(x, r)$ with $x \in \dom(g_q)$. In the natural way we use this to define (recursively) 
a notion of {\em extension} for stems.

  In the generic extension $V[L][\Agg][P_2][P_3][QTT]$ where $j$ is defined,
 let $v = j(u)_{\bar \lambda}$ where ${\bar \lambda}$ is the $\lambda^{\rm th}$ point of $j(I)$. We work below some condition
  in $\P_2 \times \P_3 \times \QTT$ which fixes the values of $\bar \lambda$ and $v$. This condition forces
``for all $\alpha \in I$, $(\bar s, \bar q)  \forces^* u_\alpha < v$''.  So in $V[L][\Agg][P_2][P_3][QTT]$   
 we may choose a sequence $(r_\alpha)_{\alpha \in I}$, such that $r_\alpha \in j(\bar \P)$ with stem $(\bar s, \bar q)$
 and $r_\alpha \forces u_\alpha < v$.

 For all stems $(s, q)$ extending $(\bar{s}, \bar q)$, define
 \[
 J_{s, q} = \{\alpha \in I \mid (s,q) \forces^*_{j(\bar P)} u_\alpha <_{j(\dot T)} v \}.
 \]
 We note that in general the definition of  $J_{s,q}$ involves the generic embedding $j$,
 so it takes place in $V[L][\Agg][P_2][P_3][QTT]$. However, it is clear that $J_{\bar s, \bar q} = I$.
 As usual if  $\beta\in J_{s,q}$, then
 $J_{s,q} \cap \beta = \{ \alpha < \beta: (s,q) \forces^* u_\alpha < u_\beta \}$,
 so in particular $J_{s,q} \cap \beta \in V[L][\Agg]$.

 Recall that in $V[L][\Agg][P_3 \times QTT \times P_{2b}]$, $\lambda$ has cardinality and cofinality $\mu$, while
 $\P_{2a}$ is $\mu$-cc and has the $\mu$-approximation property. 
 It follows that  if  $J_{s,q}$ is unbounded then  $J_{\s,q} \in  V[L][\Agg][P_3 \times QTT \times P_{2b}]$.

 Working in $V[L][\Agg][P_3 \times QTT \times P_{2b}]$,
 let $\mathcal{J}_{s,q}$ be the set of all possible unbounded values
 for $J_{s,q}$. Since $\P_{2a}$ is $\mu$-cc, $\vert \mathcal{J}_{s,q} \vert < \mu$,
 and for any name $\dot C$ for a bounded subset of $\lambda$ there is $\beta < \lambda$
 such that $\forces_{\P_{2a}} \dot C \subseteq \beta$. 
 We use these facts to  choose $\rho < \lambda$ so large that:
\begin{itemize}
\item It is forced by $\P_{2a}$ that for all $(s, q)$,
  if $J_{s, q}$ is bounded in $\lambda$ then $J_{s, q} \subseteq \rho$. 
\item  For all $(s, q)$ and all distinct $C, D \in \mathcal{J}_{s,q}$,
  $C \cap D \subseteq \rho$.
\end{itemize}

 For $\alpha \in I \setminus \rho$, define a partial function $f$,
 by setting  $f((s,q),\alpha)$ equal to the unique $C \in \mathcal{J}_{s,q}$ such that $\alpha\in C$.
 We note that ${\mathcal J}_{\bar s, \bar q} = \{ I \}$, so that $f( (\bar s, \bar q), \alpha) = I$ for all $\alpha \in I \setminus \rho$. 
 
 Let $\alpha_0 = \min(I \setminus \rho)$. Fix a length $k \ge t$,
 we will consider stems $(s, q)$ of this length extending $(\bar s, \bar q)$.  
Let $\alpha \in I \setminus \rho$, and define
$B_{k, \alpha}$ as the set of pairs  $(s,q)$ such that $(s, q)$ extends $(\bar s, \bar q)$,
$(s, q)$ has length $k$, and $f((s,q), \alpha) = f( (s,q), \alpha_0)$. 

Given $(s, q) \in B_{k, \alpha}$, define
\[
F^{s, q}_\alpha =
\{ (x, r) : (s, q) + (x, r) \in B_{k+1 \alpha} \}
\]

We note that for all $\alpha \in I \setminus \rho$:
\begin{itemize}
\item Since $f( (\bar s, \bar q), \alpha) = f( (\bar s, \bar q), \alpha_0) = I$, 
  $(\bar s, \bar q) \in B_{t, \alpha}$.
\item   By the distributivity of $\P_3 \times \QTT \times \P_{2b}$,
  $(B_{k, \alpha})_{t \le k < \omega} \in V[L][\Agg]$ and  $(F^{s, q}_\alpha)_{\textup{$(s, q)$ extends $(\bar s, \bar q)$}} \in V[L][\Agg]$. 
\end{itemize}

\begin{remark}
The following Claim is an assertion in $V[L][\Agg][P_3 \times QTT \times P_{2b}]$ about
sets which all lie in $V[L][\Agg]$, but are defined in terms of the function $f$ which
only exists in $V[L][\Agg][P_3 \times QTT \times P_{2b}]$, and in turn is defined using the embedding
$j$ which only exists in $V[L][\Agg][P_3 \times QTT \times P_2]$.
We will prove it (as one would expect) by a forcing argument involving both
$\P_3 \times \QTT \times \P_{2b}$ and $\P_{2a}$. Similar remarks 
apply to Claim \ref{plustwokey} below. Throughout we will only discuss $f$ and $j$
 in appropriate generic extensions, or in formulae which are being forced to hold in such extensions. 
\end{remark}

\begin{claim} \label{Balphadomain}  For all $\alpha \in I \setminus \rho$ and all
  $(s, q) \in B_{k, \alpha}$, $F^{s, q}_\alpha \in {\mathcal F}_{k+1}$.
\end{claim}

\begin{proof}
  If not then let $p \in \P_3 \times \QTT \times \P_{2b}$ be such that
  \[
  p \forces^{V[L][\Agg]}_{\P_3 \times \QTT \times \P_{2b}} \mbox{``$(s, q) \in {\dot B}_{k, \alpha}$ and ${\dot F}^{s, q}_\alpha \notin {\mathcal F}_{k+1}$''}.
  \]
  Since $j$ fixes sets of rank below $\lambda$,
\[
(p, 1_{\P_{2a}}) \forces^{V[L][\Agg]}_{\P_3 \times \QTT \times \P_2} \mbox{``$B_{k+1} \setminus {\dot F}^{s, q}_\alpha \in j({\mathcal F}_{k+1})^+$''}.
\]
  Forcing below $p$ we obtain $P_3 \times QTT \times P_{2b}$, such that $(s, q) \in B_{k, \alpha}$ and $F^{s, q}_\alpha \notin {\mathcal F}_{k+1}$
  in $V[L][\Agg][P_3 \times QTT \times P_{2b}]$.
  
  Let $f( (s, q), \alpha) = f( (s, q), \alpha_0 ) = C$. Since $C \in {\mathcal J}_{s, q}$,
  $C$ is a possible value for $J_{s, q}$, and so we may choose $\bar p \in \P_{2a}$ such that
  \[
  \bar p \forces^{V[L][\Agg][P_3 \times QTT \times P_{2b}]}_{\P_{2a}} \mbox{``$C = {\dot J}_{s, q}$''}.
  \]
   Forcing below $\bar p$ we obtain $P_{2a}$, such that $C = J_{s, q}$ in
   $V[L][\Agg][P_3 \times QTT \times P_2]$.

   So $\alpha_0, \alpha \in J_{s, q}$, that is to say
   $(s, q) \forces^*_{j(\bar \P)} u_{\alpha_0}, u_\alpha < v$. 
   We choose $p' \in j(\bar \P)$ with stem $(s, q)$
    such that $p' \forces_{j(\bar \P)} u_{\alpha_0}, u_\alpha < v$.

    Take a minimal one-step extension $p''$ of $p'$,
    arranging that the stem of $p''$ is $(s, q) + (x, r)$
    and $(x, r) \in B_{k+1} \setminus F^{s, q}_\alpha$.
    $p'' \forces_{j(\bar \P)} u_{\alpha_0}, u_\alpha < v$, so that
    $(s, q) + (x, r) \forces^*_{j(\bar \P)} u_{\alpha_0}, u_\alpha < v$.

    We have $\alpha_0, \alpha \in J_{(s, q) + (x, r)}$. 
    Since $\alpha, \alpha_0 > \rho$ we see that 
    $J_{(s, q) + (x, r)}$ is unbounded. So $J_{(s, q) + (x, r)} \in {\mathcal J}_{(s, q) + (x, r)}$,
    say it is $D$.

    Returning to the model $V[L][\Agg][P_3 \times QTT \times P_{2b}]$, we have $\alpha_0, \alpha \in D$, so
    $f( (s, q) + (x, r), \alpha) = D = f( (s, q) + (x, r), \alpha_0)$,
    that is to say $(s, q) + (x, r) \in B_{k+1, \alpha}$.
    Therefore $(x, r) \in F^{s, q}_\alpha$ by definition, contradicting our choice of
    $(x, r)$ as an element of $B_{k+1} \setminus F^{s, q}_\alpha$.
\end{proof}

The following claim will ultimately be used to create a branch using
the nodes $u_\alpha$.

\begin{claim} \label{plustwokey}
  Let $\alpha, \beta \in I \setminus \rho$ and let
  $(s, q)$ have length $k$ with  $(s, q) \in B_{k, \alpha} \cap B_{k, \beta}$.
  Then   $(s, q) \forces^* u_\alpha < u_\beta$.
\end{claim}   

\begin{proof}
  We work in $V[L][\Agg][P_3 \times QTT \times P_{2b}]$. 
  By the definitions of $B_{k, \alpha}$ and $B_{k, \beta}$, 
  $f( (s, q), \alpha ) = f( (s, q) , \alpha_0) = f( (s, q), \beta ) = C$ say.
  There is $p \in \P_{2a}$ forcing that $C = J_{s, q}$: if we force below $p$
  then in the extension $\alpha, \beta \in J_{s, q}$ and we may choose $r \in j(\bar\P)$ with stem $(s, q)$
  such that $r \forces_{j(\bar \P)} u_\alpha, u_\beta < v$, from which it follows that $r \forces_{j(\bar \P)} u_\alpha < u_\beta$.
  By elementarity there is $r_0 \in \bar\P$ with stem $(s, q)$ such that
  $r_0 \forces_{\bar \P} u_\alpha < u_\beta$, so $(s, q) \forces^* u_\alpha < u_\beta$.
\end{proof}

Claim \ref{Monotonicity0}
exposes a  ``monotonicity'' property of the sets
   $B_{k, \alpha}$
   which will be crucial in the proof of Lemma \ref{endgame2} below.
   
\begin{claim} \label{Monotonicity0}
  Let $(s, q) \in B_{k, \alpha}$ and $(s', q')$ be a direct extension of $(s, q)$, then
  $(s', q') \in B_{k, \alpha}$.
\end{claim}

\begin{proof} Let $f((s,q), \alpha) = f( (s,q), \alpha_0) = C$ and let
  $p \in \P_{2a}$ force that $J_{s, q} = C$. Force below $p$,
  choose $r \in j(\bar \P)$ with stem $(s, q)$ such that $r \forces u_{\alpha_0}, u_\alpha < v$,
  and refine $r$ to a condition $r'$ with stem $(s', q')$, so that $r' \forces u_{\alpha_0}, u_\alpha < v$
  and hence $(s', q') \forces^* u_{\alpha_0}, u_\alpha < v$. So $\alpha_0, \alpha \in J_{s', q'}$,
  since $\rho < \alpha_0 < \alpha$  we see that $J_{s', q'}$ is unbounded, say $J_{s', q'} = D \in {\mathcal J}_{s', q'}$. 
  Then $f((s',q'), \alpha) = f( (s',q'), \alpha_0) = D$, so that 
  $(s', q') \in B_{k, \alpha}$.
\end{proof}

At this point we are ready to construct the conditions $p_\alpha$ for $\alpha \in I \setminus \rho$.
We will perform the construction of the entries in $p_\alpha$ in $V[L][\Agg][P_3 \times QTT \times P_{2b}]$,
and it will follow by distributivity that $p_\alpha \in V[L][\Agg]$ (so that $p_\alpha \in \bar \P$)
for each $\alpha$. However the sequence $(p_\alpha)_{\alpha \in I \setminus \rho}$ only exists in
$V[L][\Agg][P_3 \times QTT \times P_{2b}]$.

Define $p_\alpha$ for $\alpha \in I \setminus \rho$, where
\[
p_\alpha = \langle \bar s, g \restriction A^\alpha_t, A^\alpha_t, F^\alpha_{t+1}, A^\alpha_{t+1}, F^\alpha_{t+2}, \ldots \rangle.
\]

To start the construction of $p_\alpha$, recall that $(\bar s, \bar q) \in B_{t, \alpha}$, so that
$F^{\bar s, \bar q}_\alpha \in {\mathcal F}_{t+1}$ by Claim \ref{Balphadomain}. We will begin by choosing
$A^\alpha_t$ and ${\bar F}^\alpha_{t+1}$ so that:
\begin{itemize}
\item $A^\alpha_t$ and ${\bar F}^\alpha_{t+1}$ witness that $F^{\bar s, \bar q}_\alpha \in {\mathcal F}_{t+1}$,
  that is to say $(x, [{\bar F}^\alpha_{t+1}(x, -)]) \in F^{\bar s, \bar q}_\alpha$ for
  all $x \in A^\alpha_t$.
\item $A^\alpha_t \subseteq \dom(g)$.
\item $x \in A^\alpha_t$  and
  $g_{[{\bar F}^\alpha_{t+1}(x, -)]}(y) = {\bar F}^\alpha_{t+1}(x, y)$ 
  for all $(x, y) \in \dom({\bar F}^\alpha_{t+1})$.
\end{itemize}

We note for the record that for all $x \in A^\alpha_t$:
\begin{itemize}
\item $(\bar s, \bar q) + (x, [{\bar F}^\alpha_{t+1}(x, -)]) \in B_{t+1, \alpha}$. 
\item Since $g = g_{\bar q}$, the stem of the minimal extension of $p_\alpha$ by adding $x$ will be
  $(\bar s, \bar q) + (x, [{\bar F}^\alpha_{t+1}(x, -)])$.
\end{itemize}   
We will complete the choice of $F^\alpha_{t+1}$ once we have chosen $A^\alpha_{t+1}$,
by defining $F^\alpha_{t+1} = {\bar F}^\alpha_{t+1} \restriction A^\alpha_t \times_{\prec} A^\alpha_{t+1}$.
Note that $F^\alpha_{t+1}$ retains all the properties listed above for ${\bar F}^\alpha_{t+1}$.

Now assume that $k \ge t$ and we have defined sets $(A^\alpha_i)$ for $t \le i \le k$, together with
functions $F^\alpha_i$ for $t < i \le k$ and a function ${\bar F}^\alpha_{k+1}$, satisfying:
\begin{itemize}
\item $\dom(F^\alpha_i) = A^\alpha_{i-1} \times_\prec A^\alpha_i$ for $t \le i < k$,
\item $x \in A^\alpha_k$ for all $(x, y) \in \dom({\bar F}^\alpha_{k+1})$.
\item $g_{[{\bar F}^\alpha_{k+1}(x, -)]}(y) = {\bar F}^\alpha_{k+1}(x, y)$ 
    for all $(x, y) \in \dom({\bar F}^\alpha_{k+1})$.
\item $g_{[{\bar F}^\alpha_i(x, -)]}(y) = {\bar F}^\alpha_i(x, y)$ 
  for all $(x, y) \in A^\alpha_{i-1} \times_\prec A^\alpha_i$.
\item For all $\prec$-increasing sequences $\vec x = (x_i)_{t \le i \le k}$ with $x_i \in A^\alpha_i$,
  let $h'(\vec x)$ be the stem
  \[
     (\bar s, \bar q) + (x_t, [F^\alpha_{t+1}(x_t, -)]) + \ldots (x_{k-1}, [F^\alpha_k(x_{k-1}, -)])
  +  (x_k, [{\bar F}^\alpha_{k+1}(x_k, -)]),
  \]
  then:
  \begin{itemize}
  \item $h'(\vec x) \in B_{k+1, \alpha}$. 
  \item $h'(\vec x)$ is the stem of the minimal extension of $p_\alpha$ by $(x_i)_{t \le i \le k}$.   
  \end{itemize}
\end{itemize}

\begin{remark} By the choice of the functions $F^\alpha_i$, 
  \[
   h'(\vec x) =   (\bar s, g(x_t), F^\alpha_{t+1}(x_t, x_{t+1}), \ldots, x_{k-1}, F^\alpha_k(x_{k-1}, x_k),  x_k, [{\bar F}^\alpha_{k+1}(x_k, -)]).
  \]
\end{remark}

In this round of the construction we will choose $A^\alpha_{k+1}$ and ${\bar F}^\alpha_{k+2}$,
and will then define $F^\alpha_{k+1}$ as ${\bar F}^\alpha_{k+1} \restriction A^\alpha_k \times_{\prec} A^\alpha_{k+1}$.
For each $\prec$-increasing sequence $\vec x = (x_i)_{t \le i \le k}$ with $x_i \in A^\alpha_i$,
 $h'(\vec x) \in B_{k+1, \alpha}$ by our induction hypothesis, and
thus $F^{h'(\vec x)}_\alpha \in {\mathcal F}_{k+2}$   by Claim \ref{Balphadomain}.
By Lemma \ref{normal-like} we have $\Delta_{\vec x} F^{h'(\vec x)}_\alpha \in {\mathcal F}_{k+2}$,
where
\[
\Delta_{\vec x} F^{h'(\vec x)}_\alpha = \{ (x, r) : \forall \vec x \; x_k \prec x \implies (x, r) \in F^{h'(\vec x)}_\alpha \}.
\]

We
choose $A^\alpha_{k+1}$ and ${\bar F}^\alpha_{k+2}$ so that:
\begin{itemize}
\item $A^\alpha_{k+1}$ and ${\bar F}^\alpha_{k+2}$ witness that $\Delta_{\vec x} F^{h'(\vec x)}_\alpha \in {\mathcal F}_{k+2}$,
  or to be more explicit $(x, [{\bar F}^\alpha_{k + 2}(x, -)]) \in \Delta_{\vec x} F^{h'(\vec x)}_\alpha \in {\mathcal F}_{k+2}$
  for all $x \in A^\alpha_{k+1}$.
\item $x \in A^\alpha_{k+1}$  and
  $g_{[{\bar F}^\alpha_{k+2}(x, -)]}(y) = {\bar F}^\alpha_{k+2}(x, y)$ 
  for all $(x, y) \in \dom({\bar F}^\alpha_{k+2})$.
\end{itemize}
Then we define $F^\alpha_{k+1}$ as ${\bar F}^\alpha_{k+1} \restriction A^\alpha_k \times_{\prec} A^\alpha_{k+1}$.

It remains to check that we have propagated our induction hypotheses.
Only the last clause requires any work. Let $(x_i)_{t \le i \le k + 1}$ with $x_i \in A^\alpha_i$,
  and let $h'$ be the stem
  \[
     (\bar s, \bar q) + (x_t, [F^\alpha_{t+1}(x_t, -)]) + \ldots (x_k, [F^\alpha_{k+1}(x_k, -)])
  +  (x_{k + 1}, [{\bar F}^\alpha_{k+2}(x_{k+1}, -)]),
  \]
where we note that 
  \[
 h' = (\bar s, g(x_t), x_t, F^\alpha_{t+1}(x_t, x_{t+1}), \ldots, x_k, F^\alpha_{k+1}(x_k, x_{k+1}), x_{k + 1}, [{\bar F}^\alpha_{k+2}(x_{k+1}, -)]).
  \]

  Let $\vec x = (x_i)_{t \le i \le k}$, and note that $[{\bar F}^\alpha_{k+1}(x_k, -)] = [F^\alpha_{k+1}(x_k, -)]$,
  so that $h'(\vec x)$ as defined above is the stem obtained from $h'$ by deleting the last entry.
  Because $x_k \prec x_{k+1}$ and $(x_{k+1}, [{\bar F}^\alpha_{k + 2}(x_{k+1}, -)]) \in \Delta_{\vec y} F^{h'(\vec y)}_\alpha$,
  $(x_{k+1}, [{\bar F}^\alpha_{k + 2}(x_{k+1}, -)]) \in F^{h'(\vec x)}_\alpha$,
  so that $h' \in B_{k+2, \alpha}$.
  
  To finish we consider the stem of the minimal extension of $p_\alpha$ by the sequence $(x_i)_{t \le i \le k + 1}$,
  recalling that by our induction hypothesis $h'(\vec x)$ is the stem of the minimal
  extension of $p_\alpha$ by $(x_i)_{t \le i \le k}$. Recalling that the one-variable
  function in the minimal extension by $(x_i)_{t \le i \le k}$ is $F^\alpha_{k+1}(x_k, -)$,
  and that by our induction hypothesis
  $g_{[{\bar F}^\alpha_{k+1}(x_k, -)]}(x_{k+1}) = F^\alpha_{k+1}(x_k, x_{k+1})$,
  it is clear that $h'$ is the stem of the minimal extension of $p_\alpha$ by $(x_i)_{t \le i \le k + 1}$
  as required.

  We have now constructed in $V[L][\Agg][P_3 \times QTT \times P_{2b}]$ a sequence $(p_\alpha)_{\alpha \in I}$ such that
\[
p_\alpha = \langle \bar s, g \restriction A^\alpha_t, A^\alpha_t, F^\alpha_{t+1}, A^\alpha_{t+1}, F^\alpha_{t+2}, \ldots \rangle.
\]
and for every minimal extension $q$ of $p_\alpha$ by a sequence $(x_i)_{t \le i \le k}$,
the stem of $q$ is in $B_{k+1, \alpha}$. 
  
  The following Lemma is analogous to Lemma \ref{endgame}.

\begin{lemma} \label{endgame2} 
  For $\alpha, \beta \in I \setminus \rho$ with $\alpha < \beta$,
  $p_\alpha \wedge p_\beta \forces u_\alpha < u_\beta$.
\end{lemma}   
  
\begin{proof}
  Suppose for a contradiction that $p \le p_\alpha, p_\beta$ and
  $p \forces u_\alpha \not< u_\beta$.
  Assume that $p$ is an $s$-step extension of $p_\alpha$ and $p_\beta$ for
  some $s \ge 2$. 
Let
\[
p = \langle s', c_t, x_t \ldots c_{t+s-1}, x_{t+s-1}, f_{t+s}, A_{t+s}, F_{t+s+1}, A_{t+s+1}, \ldots \rangle.
\]

    Let $\vec x = (x_i)_{t \le i < t + s}$. 
    The stem of $p$ directly extends the stem of the minimal extension of $p_\alpha$ 
    by $\vec x$. By construction the stem of this minimal extension
    lies in $B_{t + s, \alpha}$, and so by Claim \ref{Monotonicity0}
     $stem(p) \in B_{t + s, \alpha}$. Similarly $stem(p) \in B_{t + s, \beta}$.
    By Claim \ref{plustwokey} $stem(p) \forces^* u_\alpha < u_\beta$,
    contradicting $p \forces u_\alpha \not< u_\beta$.
\end{proof}

\begin{lemma} \label{almost2}
  The tree $T$ has a cofinal branch in $V[L][\Agg][P_3 \times QTT \times P_{2b} ][\bar P]$.
\end{lemma}

\begin{proof} 
  The proof is essentially the same as the proof of Lemma \ref{almost}.
  The main difference is that $\lambda$ is no longer a cardinal in $V[L][\Agg][P_3 \times QTT \times P_{2b}]$,
  in fact it has become an ordinal of cofinality $\mu$. Since
  $\bar \P$ has only $\nu$ stems and conditions with the same
  stem are compatible,  $\bar \P$ still enjoys the $\mu$-cc in  $V[L][\Agg][P_3 \times QTT \times P_{2b}]$,
  and the argument goes through.
\end{proof}

\begin{lemma} \label{finallythere2}
  The tree $T$ has a cofinal branch in $V[L][\Agg][\bar P]$.
\end{lemma}

\begin{proof} 
  We start by claiming that  $\P_3 \times \QTT$ is formerly $<\mu$-closed
  in $V[L][\Agg][P_{2b}][\bar P]$, with a view to using  
  Fact \ref{formerlyclosed}.
  This is easy: because of the robust $\mu$-cc of $\bar \P$,
  $\A_e \times \bar \P$ is $\mu$-cc  
  in $\Vlbi(\kappa)[J^c][\Agg][P_{2 b}]$,
  and we can argue as in the discussion preceding
  Lemma \ref{magic2}.
  it follows that $T$ has a cofinal branch in $V[L][\Agg][P_{2b}][\bar P]$. 

  Now we claim that $\P_{2b}$ has the $\lambda$-approximation property in 
  $V[L][\Agg][\bar P]$. Again this is easy, because $\bar \P \times \P_{2b}$
  is $\lambda$-Knaster in $V[L][\Agg]$.
  It follows that $T$ has a cofinal branch in $V[L][\Agg][\bar P]$ as required.  
\end{proof}

\subsection{The tree property at $\aleph_{\omega^2+3}$} \label{plusthree}

The proof that the tree property holds at $\aleph_{\omega^2 + 3}$ in our final model is very similar to that
for $\aleph_{\omega^2 + 2}$, so we only sketch it.  The main point is to get a suitable generic embedding with domain
$V[L][\Agg]$ and critical point $\lambda^b_{\omega+3}$. This is much more straightforward than
it was for $\lambda^b_{\omega+2}$ in Section \ref{plustwo}, mostly because $\lambda^b_{\omega+3}$ is
supercompact in $\Vlbi(\kappa)[J^c_0]$, so that we only need to account for
$J^c_1$, $A_e$ and $\Agg$. Moreover $A_e$ and $\Agg$ are both adding Cohen subsets to
cardinals below the critical point. 

We recall that $J^c_1$ is a single round of the $\A * \U * \S$ construction
defined in $\Vlbi(\kappa)[J^c_0]$. The parameters are
$\lambda^b_{\omega+1}$, $\lambda^b_{\omega+2}$, and $\lambda^b_{\omega+3}$.
\begin{itemize}
\item  $\A^c_1 = \Add^{\Vlb(\kappa)}(\lambda^b_{\omega+1}, [\lambda^b_{\omega+2}, \lambda^b_{\omega+3}))$.
\item  $\Aggforcing = \Add^V(\lambda^b_{\omega +2}, \lambda^*)$,
  where $\lambda^* = j_{01}(\lambda^a_0)$.
\item    $\A_e = \Add^{\Vlb(\kappa)}(\lambda^b_{17}, \lambda^b_{\omega+3})$.
\end{itemize}

    We fix an embedding $j$ witnessing that $\lambda^b_{\omega+3}$ is highly supercompact in $V$,
    and lift it trivially to $\Vlbi(\kappa)[J^c_0]$.
    We will lift $j$ onto $V[L][\Agg]$ much as we would in Lemma \ref{indestructible},
    if our goal was to prove that $\lambda^b_{\omega+3}$ has the tree property in 
    $V[L][\Agg]$.

We will step through the construction from the proof of Lemma \ref{indestructible}, adapted to a situation
where the $\A * \U * \S$ construction only runs for one round.
In our current context the parameters are set as follows:
\begin{itemize}
\item $n = 0$.
\item $\mu_0 = \lambda^b_{\omega+1}$, $\mu_1 = \lambda^b_{\omega+2}$, $\mu_2 = \lambda^b_{\omega+3}$.
\item $\A_0 * \U_0 * \S_0 = \J^c_1$.
\item $\Vd = \Vlbi(\kappa)[J^c_0]$. There is no need for $\Vi$.
\item $\D^1 = \Agg$.
\item $\D^0 = \A_e$.  
\end{itemize} 
 
The generic embedding is added to $V[L][\Agg]$ by a product forcing
$\P_2 \times \P_3 = \P_{2a} \times \P_{2b} \times \P_3$, where:
\begin{itemize}
\item $\P_{2 b} = j(\Aggforcing)/\Agg$, so $\P_{2b}  = \Add^V(\lambda^b_{\omega +2}, j(\lambda^*) \setminus j[\lambda^*])$.
\item $\P_{2a} = j(\A_e \times \A^c_1)/A_e \times A^c_1$, so 
$\P_{2a} =  \Add^{\Vlb(\kappa)}(\lambda^b_{17}, j(\lambda^b_{\omega+3}) \setminus \lambda^b_{\omega+3})
  \times 
\Add^{\Vlb(\kappa)}(\lambda^b_{\omega+1}, j(\lambda^b_{\omega+3}) \setminus \lambda^b_{\omega+3})$. 
\item $\P_3$ is defined in
  and  $<\lambda^b_{\omega+2}$-closed in $\Vlbi(\kappa)[J^c]$,
  and retains this closure in $\Vlbi(\kappa)[J^c][\Agg]$. 
\end{itemize}  


   As in  the proof of Lemma \ref{indestructible}, $\P_3$ is still
   $<\lambda^b_{\omega+2}$-closed in
   \[
   \Vlbi(\kappa)[J^c][\Agg][P_{2b}]
   = \Vlbi(\kappa)[J^c][\hat \Agg],
   \]
   where $\hat\Agg$ is the $j(\Aggforcing)$-generic object obtained by combining
   $\Agg$ and $P_{2b}$.    
   In this model $\A_e \times \P_{2a}$ is $\lambda^b_{\omega+2}$-Knaster
   and $2^{\lambda^b_{\omega+1}} = \lambda^b_{\omega+3}$.

   To summarise the key points: 
\begin{itemize}

\item $\P_2$ is $\lambda^b_{\omega+3}$-Knaster in $V[L][\Agg]$. 
  
\item    $\P_3$ is formerly $< \lambda^b_{\omega+2}$-closed forcing
   in $V[L][\Agg][P_2]$.

\end{itemize}
With this information in hand, we may finish the proof exactly as in Section \ref{plustwo}.

\bibliographystyle{amsplain}
\bibliography{james_references}

\printindex[Notation]

\appendix

\section{A lifting argument} \label{appendixA}

As promised, we give here the details of the generic supercompactness
for $\theta$ in the statement of Lemma \ref{fromVzerotoV}.

    Let $\Q \in V$ be $<\theta$-directed closed and let $H$ be $\Q$-generic over $V$.    
    Decompose $V$ as $V_0[A^0 * U^0 \restriction \theta][L^0]$ where $L^0$ is generic over
    $V_0[A^0 * U^0 \restriction \theta]$ for the Laver preparation ${\FL}^0$. 
    Let $\dot R$ be an $\Add(\omega, \theta) * \dot \U^0 \restriction \theta$-name for
    the two-step iteration $\R = {\FL}^0 * \dot {\Q}$. Appealing to the
    properties of $\theta$ and $\phi_0$ in $V_0$ we fix $i:V_0 \rightarrow N_0$ such that
    for an appropriate $\gamma > \delta$ (which may be chosen arbitrarily large):
    \begin{itemize}
    \item $i$ witnesses that $\theta$ is $\gamma$-supercompact in $V_0$.
    \item $\gamma^{++}$ is a fixed point of $i$.
    \item $i(\psi_0)(\theta) = \dot R$.
    \item The first point of $\dom(i(\phi_0))$ past $\theta$ is greater than $\gamma$.
    \end{itemize}   

    Let $A'$ be $\Add(\omega, i(\theta) \setminus \theta)$-generic over $V[H]$. Our goal
    is to find a lifting of $i$ onto $V[H]$ defined in $V[H][A']$. 
    
    The main point is that $i(\B^0)$ agrees with $\B^0$ up to $\theta$,
    uses the name $\dot R$ at stage $\theta$, and then has nothing in its support 
    until  past $\gamma$.
    By a straightforward adaptation of the argument of \cite[Claim 4.7]{NeemanUpto}, and using
    the gap in the support which we just mentioned,  $i(\B^0) \restriction (\theta, i(\theta))$
    is $\gamma$-closed in $N_0$. An easy counting argument shows that the set of maximal antichains
    of $i(\B^0) \restriction (\theta, i(\theta))$ which lie in $N_0$ has cardinality
    at most $\gamma^+$ in $V_0$, so we may build $B \in V_0$ which is 
    $i(\B^0) \restriction (\theta, i(\theta))$-generic over $N_0$.

    Let $\hat {\R}$ be the term forcing $\termspace^{V_0}(\Add(\omega, \theta) * \dot\U^0 \restriction \theta, \dot {\R})$,
    so that $\hat {\R}$ is $<\theta$-closed in $V_0$, and hence $i(\hat {\R})$ is $<i(\theta)$-closed
    in $N_0$. Since $\gamma < i(\theta)$ the poset $i(\hat {\R})$ is $\gamma$-closed in $V_0$.
    By choosing $\gamma$ large enough we may assume that the set of maximal antichains of $i(\hat {\R})$
    which lie in $N_0$ has size $\gamma^+$ in $V_0$, and we may build $R^* \in V_0$ which
    is $i(\hat {\R})$-generic over $N_0$. We will eventually make sure that $R^*$ contains a term for
    a master condition but we defer the description of this term.

    Let $A^*$ be obtained by combining $A^0$ and $A'$ in the natural way, so that $A^*$ is
    $\Add(\omega, i(\theta))$-generic over $V_0$ and $i[A^0] \subseteq A^*$. Keep in mind that
    $A'$ was obtained by forcing over $V[H] = V_0[A^0 * U^0 * H]$, so it is mutually generic with $U^0 * H$ over $V_0[A^0]$.
    We note for use later that by this analysis:
    \begin{itemize}
    \item $A^* * U^0 \restriction \theta$ is generic over $N_0$ for $\Add(\omega, i(\theta)) * i(\U^0) \restriction \theta$.  
    \item $A'$ is mutually generic with $L^0 * H$ over $V_0[A^0 * U^0 \restriction \theta]$. 
    \end{itemize} 
    
    Recall that we built $B \in V_0$ to be generic over $N_0$
    for the forcing $i(\B^0) \restriction (\theta, i(\theta))$, which is $\gamma$-closed in $N_0$.
    It is easy to see that $\Add(\omega, i(\theta)) * i(\U) \restriction \theta$ is $\theta$-cc
    in $N_0$, and so by Easton's Lemma $B$ is $i(\B^0) \restriction (\theta, i(\theta))$-generic over $N_0[A^* * U \restriction \theta]$.

    Now we recall that $L^0 * H$ is ${\FL}^0 * {\Q}$-generic over $V_0[A^0 * U^0 \restriction \theta]$,
    so {\it a fortiori} it is ${\FL}^0 * {\Q}$-generic over $N_0[A^0 * U^0 \restriction \theta]$.  
    As we noted above $A'$ is mutually generic with $L^0 * H$ over $V_0[A^0 * U^0 \restriction \theta]$,
    so these objects are mutually generic over $N_0[A^0 * U^0 \restriction \theta]$ and hence
    $L^0 * H$ is ${\FL}^0 * {\Q}$-generic over $N_0[A^* * U^0 \restriction \theta]$.  
    Since $\dot R = i(\psi_0)(\theta)$ and it  names ${\FL}^0 * {\Q}$,
    we see that $A^* * (U^0 \restriction \theta * L^0 * H)$ is $\Add(\omega, i(\theta)) * i(\U^0) \restriction \theta + 1$-generic
    over $N_0$.

    Choosing $\gamma$ large enough we can arrange that $L^0 * H$ is generic for $\gamma^+$-cc forcing, so that by Easton's Lemma again
    $B$ is $i(\B^0)(\theta, i(\theta))$-generic over $N_0[A^* * (U^0 \restriction \theta * L^0 * H)]$.
    As in Fact \ref{Itay4.5} it follows that the upwards closure of $B$ in $i(\B^0)^{A^* * (U \restriction \theta * L * H)}(\theta, i(\theta))$ is
    generic for this forcing over $N_0[A^* * (U^0 \restriction \theta * L^0 * H)]$: combining the upwards closure of $B$ with
    $A^* * (U^0 \restriction \theta * L^0 * H)$
    we obtain $A^* * U^*_{i(\theta)}$ which is $\Add(\omega, i(\theta)) * i(\U^0 \restriction \theta)$-generic over $N_0$.
    Note that we can rearrange $A^* * (U^0 \restriction \theta * L^0 * H)$ as $A^0 * U^0 * H * A'$, and that
    $U^*_{i(\theta)} \in V_0[A^0 * U^0 * H * A'] = V[H][A']$. By standard arguments we can lift $i$ to obtain
    a generic embedding $i: V_0[A^0 * (U^0 \restriction \theta)] \rightarrow N_0[A^* * U^*_{i(\theta)}]$. 

    Recall that $\hat {\R}$ is defined in $V_0$ as the set of $\A^0 * \U^0 \restriction \theta$-names for elements
    of $\dot {\R}$. Choosing $\gamma$ large enough we may arrange that $i \restriction \hat {\R} \in N_0$,
    and it follows readily that $i \restriction {\R} \in N_0[A^* * U^*_{i(\theta)}]$. By the construction
    of $U^*_{i(\theta)}$ we have that $L^0 * H \in N_0[A^* * U^*_{i(\theta)}]$, so that $i[L^0 * H] \in N_0[A^* * U^*_{i(\theta)}]$. 
    Since $i(\R)$ is $<i(\theta)$-directed closed, $i[L^0 * H]$ has a lower bound in $i(\R)$ and
    we claim that we can choose a term $\dot r \in i(\hat {\R})$ which is forced to denote a lower bound:
    this is easy because $\dot r$ has a simple definition in terms of $i \restriction \hat {\R}$
    and the $i( \A^0 * \U^0 \restriction \theta)$-generic object.
    At this point we return to the choice of $R^*$, an object which has not been used up to now,
    and make sure that $\dot r \in {\R}^*$.

    Now we can realise the set of names $R^*$, and obtain a filter $R^+ \subseteq i(\R)$ such that
    $i[L^0 * H] \subseteq R^+$. In order to complete the lifting and obtain $i: V[H] \rightarrow N_0[A^* * U^*_{i(\theta)}][R^+]$,
    it only remains to verify that $R^+$ is generic over $N_0[A^* * U^*_{i(\theta)}]$.
    Recall that we chose $R^* \in V_0$ to be generic over $N_0$. Since $R^*$ is generic over $N_0$ for
    $< i(\theta)$-closed forcing and $A^* * U^*_{i(\theta)}$ is generic over $N_0$ for $i(\theta)$-cc forcing,
    by Easton's Lemma $R^*$ is generic over $N_0[A^* * U^*_{i(\theta)}]$. It follows that 
    $R^+$ is generic over $N_0[A^* * U^*_{i(\theta)}]$.

\end{document}